\documentclass[11pt]{article}
\usepackage[margin=1in,left=1in]{geometry}
\usepackage{amsmath}
\usepackage{mathpazo}
\usepackage{epigraph}
\usepackage{amsfonts}
\usepackage{amssymb}
\usepackage{amsthm}
\usepackage{eucal}
\usepackage{wasysym}
\usepackage{booktabs}
\usepackage{enumitem}
\usepackage{mathtools}
\usepackage{caption}
\usepackage{graphicx}
\usepackage{color}
\usepackage{url}
\usepackage[colorlinks]{hyperref}
\usepackage{stmaryrd}
\usepackage[all]{xy}
\usepackage{tikz-cd}
\usetikzlibrary{
  decorations.pathmorphing
  }

\usepackage[utf8]{inputenc} % Required for inputting international characters
\usepackage[T1]{fontenc}

\numberwithin{section}{part}

\definecolor{aqua}{rgb}{0, 1.0, 1.0}
\definecolor{fuschia}{rgb}{1.0, 0, 1.0}
\definecolor{gray}{rgb}{0.502, 0.502, 0.502}
\definecolor{lime}{rgb}{0, 1.0, 0}
\definecolor{maroon}{rgb}{0.502, 0, 0}
\definecolor{navy}{rgb}{0, 0, 0.502}
\definecolor{olive}{rgb}{0.502, 0.502, 0}
\definecolor{purple}{rgb}{0.502, 0, 0.502}
\definecolor{silver}{rgb}{0.753, 0.753, 0.753}
\definecolor{teal}{rgb}{0, 0.502, 0.502}
\definecolor{cite}{RGB}{143, 113, 218}
\definecolor{url}{RGB}{218, 113, 136}

\theoremstyle{plain}
\newtheorem{theorem}{Theorem}[section]
\newtheorem{lemma}[theorem]{Lemma}

\newtheorem{proposition}[theorem]{Proposition}
\newtheorem{corollary}[theorem]{Corollary}

\newtheorem{thmdef}[theorem]{Theorem/Definition}

\theoremstyle{definition}
\newtheorem{definition}[theorem]{Definition}
\newtheorem{example}[theorem]{Example}

\newtheorem{remark}[theorem]{Remark}

\newtheorem{construction}[theorem]{Construction}

\newcommand{\dslash}{/\!\!/}

\renewcommand{\owedge}{\varowedge}

\numberwithin{equation}{section}

\providecommand{\keywords}[1]{{\small{\textit{Keywords:}} #1}}

\hypersetup{
  citecolor =  cite,
  urlcolor = url,
  linkcolor = red
  }

\begin{document}

%-------------------------------------------------------------------
%

\title{Strict algebraic models for rational parametrised spectra I}
\author{V.~Braunack-Mayer\footnote{Mathematics, Division of Science, New York University Abu Dhabi, UAE.\newline\indent\indent
\emph{Email address:} \href{mailto://v.braunackmayer@gmail.com}{\tt v.braunackmayer@gmail.com}}}
\date{2020}
\maketitle

\begin{abstract}
Building on Quillen's rational homotopy theory, we obtain algebraic models for the rational homotopy theory of parametrised spectra.
For any simply-connected space $X$ there is a dg Lie algebra $\Lambda_X$ and a (coassociative cocommutative) dg coalgebra $C_X$ that model the rational homotopy type.
In this article, we prove that the rational homotopy type of an $X$-parametrised spectrum is completely encoded by a $\Lambda_X$-representation  and also by a $C_X$-comodule.
The correspondence between rational parametrised spectra and algebraic data is obtained by means of symmetric monoidal equivalences of homotopy categories that vary pseudofunctorially in the parameter space $X$.

Our results establish a comprehensive dictionary enabling the translation of topological constructions into homological algebra using Lie representations and comodules, and conversely.
For example, the fibrewise smash product of parametrised spectra is encoded by the derived tensor product of dg Lie representations and also by the derived cotensor product of dg comodules.
As an application, we obtain novel algebraic descriptions of rational homotopy classes of fibrewise stable maps, providing new tools for the study of section spaces.
\end{abstract}

\tableofcontents

\section{Introduction}
In the 1950's Serre introduced the idea of doing homotopy theory modulo a class of groups.
Working modulo the class of torsion groups amounts to studying the free parts of homotopy groups in isolation or, equivalently, to studying the rationalised homotopy groups $\pi_\ast (X)\otimes_\mathbb{Z}\mathbb{Q}$.
As the rational homotopy groups of spheres are so simple, rational homotopy theory turns out to be considerably simpler than ordinary homotopy theory while at the same time retaining a great deal of useful information.

What makes rational homotopy theory so tractable is that, under mild assumptions, it is completely encoded by algebraic data.
In his foundational work \cite{quillen_rational_1969}, Quillen proved that the rational homotopy category of simply-connected spaces is equivalent to the homotopy theory of reduced differential graded (dg) Lie algebras over $\mathbb{Q}$ and also to the homotopy theory of 2-reduced dg cocommutative coalgebras over $\mathbb{Q}$.
A powerful consequence of Quillen's result is that topological constructions involving simply-connected spaces (for instance homotopy limits and colimits) can be carried out in relatively simple algebraic categories, provided that we are willing to work modulo torsion.

In this article, we extend Quillen's algebraic models for rational homotopy theory to homotopy categories of parametrised spectra.
Given a space $X$, an $X$-parametrised spectrum is a homotopically-coherent spectrum-valued local system $P\colon x\mapsto P_x$ on $X$.
Parametrised spectra encode a subtle mixture of stable and unstable homotopy theory, and ought to be regarded as stable representations of the homotopy type of the parameter space.
This idea is perhaps easiest to see when the base space $X$ is connected, in which case homotopy types of $X$-parametrised spectra correspond to homotopy types of stable $\Omega X$-modules (for a fixed but ultimately immaterial choice of basepoint).
The analogy between parametrised stable homotopy theory and representation theory is further borne out by the identification of parametrised spectra with $\infty$-categorical Beck modules  \cite[Chapter 7]{lurie_higher_2017}.
%The general study of parametrised spectra was originally motivated by investigations of transfer maps and has enjoyed increasing activity in recent years, due in large part to the fact that parametrised spectra classify twisted homology and cohomology theories.

For each space $X$, there is a stable $\infty$-category $\mathrm{Sp}_X$ of $X$-parametrised spectra.
Consequently, for any pair of $X$-parametrised spectra $P$, $Q$ there is a $\mathbb{Z}$-graded abelian group $\{P,Q\}_X^\ast$ of homotopy classes of fibrewise stable maps $P\to Q$.
A map of $X$-parametrised spectra $f\colon P\to Q$ is an equivalence precisely if for all $X$-parametrised spectra $R$, the induced map of homotopy classes of fibrewise stable maps
$ 
{\{f, R\}_X^\ast}
\colon
{\{Q, R\}_X^\ast}
\to
{\{P, R\}_X^\ast}
$
is an isomorphism of graded abelian groups.
This condition is equivalent to the requirement that $f$ induces stable equivalences of fibre spectra $f_x\colon P_x \xrightarrow{\sim} Q_x$ at all points $x$ of $X$.
Working modulo torsion, $f$ is then a \emph{rational equivalence} if for all $X$-parametrised spectra $R$, the induced map
\[
{\{f, R\}_X^\ast}\otimes_\mathbb{Z}\mathbb{Q}
\colon
{\{Q, R\}_X^\ast}\otimes_\mathbb{Z}\mathbb{Q}
\longrightarrow
{\{P, R\}_X^\ast}\otimes_\mathbb{Z}\mathbb{Q}
\]
is an isomorphism of graded rational vector spaces;
this is equivalent to the condition that $f$ induces rational stable equivalences on fibre spectra.

If $X$ is simply-connected, Quillen identifies the rational homotopy type of $X$ with the homotopy type of a reduced dg Lie algebra $\Lambda_X$ and also with the homotopy type of a 2-reduced cocommutative coassociative dg  coalgebra $C_X$.
The homology of $\Lambda_X$ computes the Whitehead Lie algebra $\pi^\mathbb{Q}_{\ast+1}(X):=\pi_{\ast+1}(X)\otimes_\mathbb{Z}\mathbb{Q}$, whereas the homology of $C_X$ recovers the coalgebra structure of $H_\bullet (X;\mathbb{Q})$.
Both $\Lambda_X$ and $C_X$ have a natural notion of representation, namely dg Lie representations and dg comodules.
The main result of this article is that the homotopy categories of $\Lambda_X$-representations and $C_X$-comodules are both equivalent to the rational homotopy category of $X$-parametrised spectra.
More precisely we prove the following
\begin{theorem}
\label{thm:Main1}
For a simply-connected space $X$ with Quillen models $\Lambda_X$ and $C_X$, there are equivalences of categories
\[
Ho(\mathrm{Sp}_X)_\mathbb{Q}
\xrightarrow{\;\;\sim\;\;}
Ho(\Lambda_X\mathrm{-Rep})
\xrightarrow{\;\;\sim\;\;}
Ho(C_X\mathrm{-Comod}_{(t)})\,.
\]
\end{theorem}
\medskip
Some clarifying remarks are in order. 
Firstly, the subscript \lq\lq$(t)$'' indicates that we consider unbounded comodules up to \emph{$t$-equivalence}, which is a finer notion of equivalence than quasi-isomorphism (for bounded below comodules these notions coincide). 
This complication stems from the familiar fact that important spectral sequences fail to converge in the unbounded setting.

Secondly, we must clarify how Theorem \ref{thm:Main1} behaves with respect to algebraic invariants.
Let $P$ be an $X$-parametrised spectrum, then for any point $x\colon\ast \to X$ the fibre spectrum $P_x$ is a $\Omega_x X$-module spectrum
$
\Omega_x X_+ \wedge P_x \to P_x
$.
Passing to stable rational homotopy groups furnishes $\pi^\mathrm{st}_\ast (P_x)\otimes_\mathbb{Z}\mathbb{Q}$ with the structure of a $H_\bullet(\Omega X;\mathbb{Q})$-module.
As
$X$ is simply-connected, the Milnor--Moore theorem identifies the rational homology of $\Omega X$ with the universal enveloping algebra of the rational Whitehead Lie algebra $\pi^\mathbb{Q}_{\ast+1}(X) = \pi_{\ast +1} (X)\otimes_\mathbb{Z}\mathbb{Q}$.
Any parametrised $X$-spectrum $P$ thus determines a $\pi^\mathbb{Q}_{\ast+1}(X)$-representation $\mathbold{\theta}_\ast (P) = \pi^\mathrm{st}_\ast (P_x)\otimes_\mathbb{Z}\mathbb{Q}$.
\begin{theorem}
\label{thm:Main2}
The diagram of functors
\[
\begin{tikzcd}
Ho(\mathrm{Sp}_X)_\mathbb{Q}
\ar[r, "\sim"]
\ar[dr, bend right =15, "\mathbold{\theta}_\ast"']
&
Ho(\Lambda_X\mathrm{-Rep})
\ar[d, "\mathrm{homology}"]
\\
&
\pi^\mathbb{Q}_{\ast+1}(X)\mathrm{-Rep}
\end{tikzcd}
\]
commutes up to natural isomorphism.
\end{theorem}

On the other hand, pushforward along the terminal map $X\to \ast$ sends the $X$-parametrised spectrum to a $X_+$-comodule spectrum $X_! P$.
Taking stable rational homotopy groups we find that $\pi^\mathrm{st}_\ast (X_! P) \otimes_\mathbb{Z}\mathbb{Q}$ is naturally a $H_\bullet(X;\mathbb{Q})$-comodule.
Writing $\mathbold{H}^X_\bullet (P)$ for this rational homology comodule, we prove the following
\begin{theorem}
\label{thm:Main3}
The diagram of functors
\[
\begin{tikzcd}
Ho(\mathrm{Sp}_X)_\mathbb{Q}
\ar[r, "\sim"]
\ar[dr, bend right =15, "\mathbold{H}_\bullet^X"']
&
Ho(C_X\mathrm{-Comod}_{(t)})
\ar[d, "\mathrm{homology}"]
\\
&
H_\bullet(X;\mathbb{Q})\mathrm{-Comod}
\end{tikzcd}
\]
commutes up to natural isomorphism.
\end{theorem}

Much of the power and flexibility of parametrised stable homotopy theory stems from the base change adjunctions
\[
(f_!
\dashv 
f^\ast
\dashv 
f_\ast)
\colon
\begin{tikzcd}
\mathrm{Sp}_X
\ar[rr, shift left=2.2ex]
\ar[rr, leftarrow, "\bot"]
\ar[rr, "\bot", shift left=-2.2ex]
&&
\mathrm{Sp}_Y
\end{tikzcd}
\]
associated to a map of parameter spaces $f\colon X\to Y$.
These adjunctions facilitate the transport of information between different parametrised contexts, and between the parametrised and unparametrised settings (as we have already seen for fibre and pushforward spectra).
If $X$ and $Y$ are simply-connected with Quillen models $\Lambda_X, \Lambda_Y$ and $C_X, C_Y$ respectively, then $f\colon X\to Y$ is represented in rational homotopy theory by a map of dg Lie algebras $\Lambda_f \colon \Lambda_X\to \Lambda_Y$ and also by a map of dg coalgebras $C_f \colon C_X\to C_Y$.
The latter maps give rise to derived adjunctions
\[
\begin{tikzcd}
Ho(\Lambda_X\mathrm{-Rep})
\ar[r, shift left =1.1ex, "\mathbf{L}(\Lambda_f)_!"]
\ar[r, leftarrow, shift left =-1.1ex, "\bot", "\mathbf{R}\Lambda_f^\ast"']
&
Ho(\Lambda_Y\mathrm{-Rep})
\;\;
\mbox{ and }
\;\;
Ho(C_X\mathrm{-Comod}_{(t)})
\ar[r, shift left =1.1ex, "\mathbf{L}(C_f)_!"]
\ar[r, leftarrow, shift left =-1.1ex, "\bot", "\mathbf{R}C_f^\ast"']
&
Ho(C_Y\mathrm{-Comod}_{(t)})
\end{tikzcd}
\] 
which are identified with $
(\mathbf{L}f_! \dashv \mathbf{R}f^\ast)\colon Ho(\mathrm{Sp}_X)_\mathbb{Q}\to Ho(\mathrm{Sp}_Y)_\mathbb{Q}$ by Theorem \ref{thm:Main1}.
If $f$ is a rational homotopy equivalence then all of these adjunctions are adjoint equivalences of homotopy categories.
More precisely, we prove the following
\begin{theorem}
\label{thm:Main4}
The equivalences of categories 
$
Ho(\mathrm{Sp}_X)_\mathbb{Q}
\xrightarrow{\;\sim\;}
Ho(\Lambda_X\mathrm{-Rep})
\xrightarrow{\;\sim\;}
Ho(C_X\mathrm{-Comod}_{(t)})
$
are pseudonatural in $X$; that is for each map of simply-connected spaces $f\colon X\to Y$ with Quillen models $\Lambda_f \colon \Lambda_X \to \Lambda_Y$ and $C_f\colon C_X\to C_Y$, there is a diagram of categories that commutes up to natural isomorphism
\[
\begin{tikzcd}
Ho(\mathrm{Sp}_X)_\mathbb{Q}
\ar[d, "\mathbf{L}f_!"]
\ar[r]
&
Ho(\Lambda_X\mathrm{-Rep})
\ar[d, "\mathbf{L}(\Lambda_f)_!"]
\ar[r]
&
Ho(C_X\mathrm{-Comod}_{(t)})
\ar[d, "\mathbf{L}(C_f)_!"]
\\
Ho(\mathrm{Sp}_Y)_\mathbb{Q}
\ar[r]
&
Ho(\Lambda_Y\mathrm{-Rep})
\ar[r]
&
Ho(C_Y\mathrm{-Comod}_{(t)})
\end{tikzcd}
\]
in which the horizontal arrows are equivalences of categories.
If $f$ is a rational homotopy equivalence then all arrows in the above diagram are equivalences of categories.
\end{theorem}

The fibrewise smash product of $X$-parametrised spectra, $(P, Q)\mapsto P\wedge_X Q$ gives rise to a closed symmetric monoidal structure on the stable homotopy category $Ho(\mathrm{Sp}_X)$ and its rationalisation.
$Ho(\Lambda_X\mathrm{-Rep})$ and $Ho(C_X\mathrm{-Comod}_{(t)})$ have symmetric monoidal structures $\otimes^\mathbf{L}$ and $\square^\mathbf{R}_C$ defined by deriving the tensor and cotensor products respectively.
Completing the dictionary between rational parametrised stable  homotopy theory and categories of (co)algebraic representations, we prove the following
\begin{theorem}
\label{thm:Main5}
The equivalences of categories
\[
\begin{tikzcd}
(Ho(\mathrm{Sp}_X)_\mathbb{Q}, \wedge_X)
\ar[r, "\sim"]
&
(Ho(\Lambda_X\mathrm{-Rep}), \otimes^\mathbf{L})
\ar[r, "\sim"]
&
(Ho(C_X\mathrm{-Comod}_{(t)}), \square^\mathbf{R}_C)
\end{tikzcd}
\]
are strongly symmetric monoidal.
\end{theorem}
 
When $P= \Sigma^\infty_X Y_+$ and $Q= \Sigma^\infty_X Z_+$ are the fibrewise suspension spectra of maps $f\colon Y\to X$ and $g\colon Z\to X$, this result recovers the Eilenberg--Moore theorem.
Indeed, we have
\[
H_\bullet (Y\times_X Z; \mathbb{Q})\cong
\pi^\mathrm{st}_{\ast}\big(X_! (P\wedge_X Q)\big)\otimes_\mathbb{Z}\mathbb{Q} 
\cong 
H_\bullet (M \,\square_{C_X}^\mathbf{R}\, N) 
\cong 
\mathrm{Cotor}^{C_X}_\bullet (M,N)
\]
where $P$ and $Q$ are identified with $C_X$-comodules $M$ and $N$ presenting $H_\bullet(X;\mathbb{Q})$-coactions on $H_\bullet(Y;\mathbb{Q})$ and $H_\bullet(Z;\mathbb{Q})$ respectively.

An interesting consequence of Theorem \ref{thm:Main5} is the classification of rational homotopy classes of fibrewise stable maps in terms of $\mathrm{Ext}$ and $\mathrm{Coext}$ groups.
This allows for the computation of rationalised twisted cohomology purely in terms of algebraic data, generalising and contextualising previous results obtained using Sullivan's approach to rational homotopy theory in \cite{felix_fibrewise_2010}.
We show that under some additional hypotheses, rational homotopy classes of fibrewise stable maps can be computed either  by means of a hyper-$\mathrm{Ext}$ or a coalgebraic twisted Atiyah--Hirzebruch spectral sequence.

The approach to rational parametrised stable homotopy theory presented in this article, since it is based on Quillen's original arguments, is necessarily fairly involved.
Theorem \ref{thm:Main1} is proven by means of a long chain of intermediate equivalences, making specific identifications (of an $X$-parametrised spectrum with a $\Lambda_X$-representation, say) rather difficult.
We have worked out a dual approach to rational parametrised stable homotopy theory based on Sullivan's rational PL de Rham theory.
When some finiteness conditions are met, the dual approach identifies the rational homotopy type of an $X$-parametrised spectrum with a module over the Sullivan algebra $\mathcal{A}_X$ of $X$.
More specifically, the $X$-parametrised spectrum $P$ is identified with an $\mathcal{A}_X$-module $M$ for which $H_\bullet (M) \cong \mathrm{Hom}(\pi^\mathrm{st}_\ast (X_!P), \mathbb{Q})$ as $H^\bullet(X;\mathbb{Q})$-modules. This correspondence between dg modules and rational parametrised spectra is the subject of the  forthcoming article \cite{braunack-mayer_strict_2019}.
We emphasise that the relative simplicity of the dual setting is ideal for applications; for example in \cite{braunack-mayer_gauge_2019} where the rational twisted $K$-theory charge of $D$-branes in Type IIA string theory is derived directly from $M$-brane charge using parametrised stable homotopy theory. 
That computational ease comes at a price, however, as the dual approach does not capture as much information about parametrised homotopy theory as the methods considered here.

\paragraph*{Organisation.}
The article is organised as follows.
In Section \ref{sec:ParamSpecLoopspace} we begin by recalling the definitions and main properties of parametrised symmetric spectra from \cite{braunack-mayer_combinatorial_2019}. For a reduced simplicial set $X$ we prove a $\mathrm{Sp}^\Sigma$-enriched Quillen equivalence between the model categories of $X$-parametrised symmetric spectra and $\mathbb{G}X_+$-module spectra, where $\mathbb{G}X$ is Kan's simplicial loop group.
Working with a natural model of the Eilenberg--Mac Lane spectrum $H\mathbb{Q}$ as a commutative symmetric ring spectrum, we pass to $X$-parametrised $H\mathbb{Q}$-module spectra and $(\mathbb{G}X_+, H\mathbb{Q})$-bimodule spectra in order to model the rational homotopy theories of $X$-parametrised spectra and stable $\Omega X$-modules respectively.
Along the way, we prove that for any $\mathrm{Sp}^\Sigma$-model category $\mathcal{M}$ with a set of compact generators, the passage to rational homotopy theory is implemented by the functor $x\to H\mathbb{Q}\wedge x$; which is to say, \lq\lq rationalisation in $\mathcal{M}$ is smashing''.
We also record a proof of the fact that for any simplicial group $G$, the model category of $G_+$-module spectra and its rationalisation are symmetric monoidal model categories.
In the case that $G= \mathbb{G}X$ is the Kan simplicial loop group of $X$, this symmetric monoidal model structure is identified with the fibrewise smash product under the equivalence of homotopy categories $Ho(\mathrm{Sp}_X)\cong Ho(\Omega X\mathrm{-Mod})$.

Section \ref{sec:Rect} contains the first steps of our rectification program.
For any simplicial group $G$,
we prove $\mathrm{Sp}^\Sigma$-enriched Quillen equivalences between model categories of symmetric $(G_+,H\mathbb{Q})$-bimodule spectra and symmetric spectrum objects in the category of representations over the rational group ring $\mathbb{Q}[G]$.
For any simplicial rational Hopf algebra $H$, the model category of $H$-modules is symmetric monoidal, as is symmetric stabilisation.
When $H= \mathbb{Q}[G]$ is a rational group ring, we argue that the Quillen equivalence between $(G_+, H\mathbb{Q})$-bimodule spectra and $\mathbb{Q}[G]$-module spectra is strongly symmetric monoidal.

Based on work of Shipley in the integral case \cite{shipley_HZ_2007}, in Section \ref{sec:LieRep} we record a zig-zag of (weakly) monoidal Quillen equivalences between unbounded rational chain complexes on the one hand and symmetric spectra in simplicial rational vector spaces on the other.
This result is a stabilised version of the Dold--Kan correspondence relating rational simplicial vector spaces and connective rational chain complexes.
Using this stabilised Dold--Kan correspondence, for any simplicial rational Lie algebra $\mathfrak{g}$, we prove a zig-zag of (weakly) monoidal Quillen equivalences between the categories of symmetric spectra in $\mathfrak{g}$-representations and unbounded differential graded representation over the Lie algebra of normalised chains $N\mathfrak{g}$.
The main technical result of this section is Lemma \ref{lem:Chi}.

The subject of Section \ref{sec:Koszul} is Koszul duality between comodules over a 2-reduced coassociative cocommutative coalgebra $C$ and modules over the cobar construction $\Omega C$.
By means of familiar spectral sequences, we prove Quillen equivalences between the model categories of bounded below $C$-comodules and bounded below $\Omega C$-comodules.
In the unbounded case, we circumvent issues of non-convergent spectral sequences by working with $C$-comodules up to $t$-equivalence; two $C$-comodules are \emph{$t$-equivalent} precisely if the corresponding $\Omega C$-modules are quasi-isomorphic.
Any $t$-equivalence is a quasi-isomorphism and the two notions coincide for bounded below comodules.
The resulting \emph{$t$-local derived category} of $C$-comodules is equivalent to the derived category of $\Omega C$-modules.
We conclude the section with a brief study of derived cotensor products on the $t$-local derived category, which compute the $\mathrm{Cotor}$ functors of Eilenberg and Moore.

Section \ref{sec:RatStabParamHom} contains the proofs of our main results.
Following a brief summary of Quillen's approach to rational homotopy theory, we  prove Theorems \ref{thm:Main1}, \ref{thm:Main4}, and \ref{thm:Main5} by synthesising the results of Sections \ref{sec:ParamSpecLoopspace}--\ref{sec:Koszul}.
A careful examination of these arguments allows us to prove Theorem \ref{thm:Main2} by analysing fibrewise suspension spectra.
In contrast, the proof of Theorem \ref{thm:Main3} is rather involved and makes use of Brown--Szczarba twisting chains and bar-cobar duality.
Our translation of topological into algebraic data is concluded by demonstrating that rational homotopy classes of fibrewise stable maps (including the torsion-free part of twisted cohomology) is computed by $\mathrm{Ext}$-groups of dg Lie representations and also by $\mathrm{Coext}$ groups of dg comodules.
In good situations these algebraic invariants can be calculated by hyper-$\mathrm{Ext}$ and $\mathrm{Coext}$ spectral sequences.

\paragraph*{Notation and conventions.}
\begin{itemize}
  \item Our results are proven using the language of model categories, which are frequently enriched or monoidal (see \cite{hovey_model_1999} for a textbook account, as well as the next bullet point).
  
  \item 
  If $(\mathcal{M}, \otimes)$ is a monoidal category, the \emph{pushout-product} of morphisms $i\colon A\to B$ and $j\colon X\to Y$ is the induced morphism
  \[
  i\,\square\, j\colon (B\otimes X)\coprod_{(A\otimes X)}(A\otimes Y)\longrightarrow
  B\otimes Y\,,
  \]
  provided it exists.
  In the case that $\mathcal{M}$ is a model category, $\otimes$ is a \emph{Quillen bifunctor} if for all cofibrations $i$ and $j$, the pushout-product $i\,\square\, j$ is a cofibration which is moreover acyclic if either $i$ or $j$ is.
  A monoidal model category is a model category $\mathcal{M}$ equipped with a closed monoidal structure $\otimes$ which is a Quillen bifunctor, subject to the additional condition that for any cofibrant replacement of the monoidal unit $QI\to I$ and any cofibrant object $X\in \mathcal{M}$, the morphism $QI\otimes X\to I\otimes X\cong X$ is a weak equivalence. 
  We use the notions of weakly and strongly symmetric monoidal Quillen equivalences from \cite{schwede_monoidal_2003}---in both cases the induced equivalence of homotopy categories is strongly symmetric monoidal.

  \item We work with simplicial sets as models for homotopy types throughout. Stable homotopy types are modelled by symmetric spectra, which form a model category $\mathrm{Sp}^\Sigma$ \cite{hovey_symmetric_2000}.
  The symmetric sphere spectrum is denoted $\mathbb{S}$ and has underlying symmetric sequence 
  \[
  \mathbb{S}(n) = S^n = \underbrace{S^1\wedge \dotsb \wedge S^1}_{\text{$n$ factors}}\,,
  \]
  where $S^1 = \Delta^1/\partial
   \Delta^1$ is the simplicial circle and $\Sigma_n$ acts on $S^n$ by permuting smash factors.
 When working with symmetric spectra it is necessary to take fibrant replacements before computing stable homotopy groups (see \cite[Section 3.1]{hovey_symmetric_2000}) and we shall take this for granted throughout.
  More generally, we implement the passage to stable homotopy theory using the symmetric stabilisation machine \cite{hovey_spectra_2001}.
  
%  \item Parametrised stable homotopy theory is modelled by symmetric spectra in retractive spaces as studied in \cite{braunack-mayer_combinatorial_2019}.
%  For a simplicial set $X$, a retractive space is  a map of simplicial sets $Y\to X$ admitting a section$\mathrm{Sp}^\Sigma_X = \mathrm{Sp}^\Sigma(\mathrm{sSet}_{\dslash X})$ is the model category of $X$-parametrised symmetric  spectra.

  \item The category of simplicial rational vector spaces is denoted $\mathrm{sVect}_\mathbb{Q}$.
  It is a combinatorial model category with fibrations and weak equivalences created by forgetful functor to the Kan model structure on simplicial sets.
  The suspension endofunctor is $\Sigma_\mathbb{Q}\colon V\mapsto \widetilde{\mathbb{Q}}[S^1]\otimes V$, with $\widetilde{\mathbb{Q}}[S^1]$ is the simplicial rational vector generated by the non-basepoint simplices of $S^1 = \Delta^1/\partial\Delta^1$.
  The suspension spectrum functor is written $\Sigma^\infty_\mathbb{Q}\colon \mathrm{sVect}_\mathbb{Q}\to \mathrm{Sp}^\Sigma(\mathrm{sVect}_\mathbb{Q})$.
  
  \item The category of rational chain complexes is denoted $\mathrm{Ch}$. It is a cofibrantly generated model category for which weak equivalences are quasi-isomorphisms and fibrations are degreewise epimorphisms.
  $\mathrm{Ch}$ is a symmetric monoidal model category with respect to the tensor product of chain complexes.
  The \emph{shift} or \emph{suspension} endofunctor on $\mathrm{Ch}$ is determined by its action on objects is $s\colon M\mapsto M[-1] = M\otimes \mathbb{Q}[-1]$ and the \emph{desuspension} endofunctor is given by $s^{-1}\colon M\mapsto M[1]=M\otimes \mathbb{Q}[1]$, where $\mathbb{Q}[n]$ is the chain complex consisting of $\mathbb{Q}$ concentrated in degree $-n$.
  The symmetric suspension spectrum functor for chain complexes is $s^\infty\colon \mathrm{Ch}\to \mathrm{Sp}^\Sigma(\mathrm{Ch})$.
  
  \item 
  For $\mathcal{M}$ a $\mathrm{Sp}^\Sigma$-model category, write $\underline{\mathcal{M}}(x,y)$ for the symmetric spectrum of maps $x\to y$.
  When $x$ is cofibrant and $y$ is fibrant, the $\mathbb{Z}$-graded abelian group of homotopy classes of maps from $x$ to $y$ is defined as
  \[
  \{x,y\}_{\mathcal{M}}^\ast =
  \pi^\mathrm{st}_\ast \underline{\mathcal{M}}(x,y)\,.
  \]
  A morphism of cofibrant objects $f\colon x\to x'$ is a weak equivalence in $\mathcal{M}$ precisely if 
  \[
  \{f,y\}_{\mathcal{M}}^\ast\colon \{x',y\}_{\mathcal{M}}^\ast\longrightarrow \{x,y\}_{\mathcal{M}}^\ast
  \]
  is an isomorphism of graded abelian groups for all fibrant objects $y$. The dual statement characterising weak equivalences between fibrant objects also holds.
  
  A morphism of cofibrant objects $f\colon x\to x'$ is defined to be a \emph{rational stable equivalence} if
  \[
  \{f,y\}_{\mathcal{M}}^\ast
  \otimes_\mathbb{Z}\mathbb{Q}\colon \{x',y\}_{\mathcal{M}}^\ast
  \otimes_\mathbb{Z}\mathbb{Q}\longrightarrow
  \{x,y\}_{\mathcal{M}}^\ast
  \otimes_\mathbb{Z}\mathbb{Q}
  \]
  is an isomorphism of graded rational vector spaces for all fibrant objects $y$.
  
  For $\mathcal{M}=\mathrm{Sp}^\Sigma_X$ a model category of parametrised symmetric spectra we write $\{P,Q\}_X^\ast$ instead of $\{P,Q\}_{\mathrm{Sp}^\Sigma_X}^\ast$ 
  
  \item When working in a model category $\mathcal{M}$ we often and without comment use superscripts to denote fibrant or cofibrant replacements; for example $x\to x^f$ and $x^c\to x$ for, respectively, a fibrant and cofibrant replacement of $x$.
  
  \item Unless explicitly stated otherwise all actions and coactions are \lq\lq from the left'' and complexes are homologically graded.
  The ground field is $\mathbb{Q}$.
  
  \item Differential graded (dg) Lie algebras $L= L_\bullet$ are always connective ($L_{\bullet <0} =0$) and usually reduced $(L_0 = 0)$.
  Dg coalgebras $C= C_\bullet$ are always connective and strictly coassociative, frequently strictly cocommutative, and often 2-reduced ($C_0 =\mathbb{Q}$, $C_1 =0$).
  
  \item For $k\in \mathbb{Z}$, the \emph{$k$-connective cover} of a chain complex $M$ is the chain complex
  \[
  (c_k M)_n
  =
  \begin{cases}
  \qquad\quad  M_n & n>k \\
  \mathrm{ker}(M_k \xrightarrow{\partial} M_{k-1}) & n=k \\
  \qquad \quad \;0 & n< k
  \end{cases} 
  \]
  with differential inherited from $M$.
  The natural map $c_k M\to M$ is a homology isomorphism in degrees $\geq k$.
\end{itemize}

\paragraph*{Acknowledgements.}
This article is the second of a series \cite{braunack-mayer_combinatorial_2019, braunack-mayer_strict_2019} based on my PhD
thesis written under the supervision of Alberto Cattaneo and Urs Schreiber. I would like to thank 
Kathryn Hess, Christoph Schweigert, and Hisham Sati for useful comments and encouragement at various
stages during the writing of this article.
I would also like to thank the anonymous referee for their thoughtful and useful suggestions to improve the readability of this article.
I acknowledge the support of the Deutsche Forschungsgemeinschaft
RTG 1670 \lq\lq Mathematics inspired by String Theory and Quantum Field Theory''.

\section{Parametrised spectra and stable loop space modules}
\label{sec:ParamSpecLoopspace}
The main objects of study in this article are parametrised spectra.
An $X$-parametrised spectrum (for some space $X$) is a homotopically-coherent spectrum-valued local system on $X$.
Such objects have a natural homotopy theory, and the first goal of this section is to prove that for connected spaces $X$, the homotopy theory of $X$-parametrised spectra is equivalent to the homotopy theory of stable $\Omega X$-modules.
Throughout this article we study the homotopy theory of parametrised spectra using the combinatorial model categories of \cite{braunack-mayer_combinatorial_2019}, which are brielfy recalled in Section \ref{ss:Comb}).

In Section \ref{S:ratsmash} we study the rational homotopy theory of $X$-parametrised spectra.
In terms of the combinatorial models, passage to rational homotopy theory is implemented by taking smash products with the Eilenberg--Mac Lane spectrum $H\mathbb{Q}$ at each fibre over the parameter space.
We prove that the equivalence of homotopy theories between $X$-parametrised spectra and stable $\Omega X$-modules persists after rationalisation.
This is a specialisation of a general result characterising the rationalisation of any $\mathrm{Sp}^\Sigma$-model category in terms of $H\mathbb{Q}$-module objects.

\subsection{Combinatorial parametrised spectra}
\label{ss:Comb}
We briefly recall the combinatorial models for unstable and stable parametrised homotopy types from \cite{braunack-mayer_combinatorial_2019}.

\begin{definition}
For any simplicial set $X$, a \emph{retractive space over $X$} is a diagram
\[
\begin{tikzcd}
X
\ar[r, "i"]
\ar[rr, bend left=-20, "\mathrm{id}_X"']
&
Y
\ar[r, "r"]
&
X
\end{tikzcd}
\]
exhibiting $X$ as a retract. 
By abuse of notation, such a retractive space is often denoted simply by $Y$, omitting the projection $r$ and section $i$ from the notation.
A morphism of retractive spaces $\psi\colon Y\to Y'$ over $X$ is the data of a commuting diagram
\[
\begin{tikzcd}
X
\ar[r, "i'"]
\ar[d, "i"']
&
Y'
\ar[d, "r'"]
\\
Y
\ar[ur, "\psi"]
\ar[r, "r"]
&
X\,.
\end{tikzcd}
\]
The category of retractive spaces over $X$ is $\mathrm{sSet}_{\dslash X}$.
\end{definition}

Retractive spaces are combinatorial models for unstable parametrised homotopy types that are equipped with a choice of coherently parametrised basepoint.
For any simplicial set $X$, the category $\mathrm{sSet}_{\dslash X}$ is furnished with a combinatorial model structure for which the cofibrations, fibrations, and weak equivalences are all created by the forgetful functor $\mathrm{sSet}_{\dslash X}\to \mathrm{sSet}$.

The category $\mathrm{sSet}_{\dslash X}$ is also enriched, tensored, and powered over the symmetric monoidal category $(\mathrm{sSet}_\ast, \wedge)$.
The $\mathrm{sSet}_\ast$-tensor of $K\in \mathrm{sSet}_\ast$ and $Y\in \mathrm{sSet}_{\dslash X}$ is the colimit (computed in $\mathrm{sSet}$)
\[
K\owedge_X Y :=
\mathrm{colim}
\left\{
\begin{tikzcd}[sep=small]
X
\ar[r]
\ar[d]
&
X\times K
\ar[dd]
\ar[dr]
\\
Y
\ar[dr]
\ar[rr, crossing over]
&&
Y\times_X (X\times K)
\\
&
X
\end{tikzcd}
\right\}\,,
\]
equipped with obvious structure maps as a retractive space over $X$.
$\mathrm{sSet}_{\dslash X}$ is a $\mathrm{sSet}_\ast$-model category with respect to this $\mathrm{sSet}_\ast$-tensoring.

\begin{remark}
Consider the simplicial circle $S^1 := \Delta^1/\partial \Delta^1$.
The \emph{fibrewise suspension} on $\mathrm{sSet}_{\dslash X}$ is the endofunctor $\Sigma_X := S^1 \owedge_X(-)$ given by forming $\mathrm{sSet}_\ast$-tensors with $S^1$.
For each simplicial set $X$, the fibrewise suspension $\Sigma_X$ is a left $\mathrm{sSet}_\ast$-Quillen functor; the right adjoint $\Omega_X$ is the \emph{fibrewise loop space functor}.
\end{remark}

Much of the versatility of parametrised homotopy theory stems from the existence of base change functors, which mediate transfer between different parametrised contexts.
For each map of simplicial sets $f\colon X\to X'$ there is an associated triple of base change adjunctions
\begin{equation}
\label{eqn:BaseChangeRetSpace}
(f_!\dashv f^\ast \dashv f_\ast)
\colon 
\begin{tikzcd}
\mathrm{sSet}_{\dslash X}
\ar[rr, shift left=2.2ex]
\ar[rr, leftarrow, "\bot"]
\ar[rr, shift left=-2.2ex, "\bot"]
&&
\mathrm{sSet}_{\dslash X'}\,.
\end{tikzcd}
\end{equation}
The functor $f^\ast\colon \mathrm{sSet}_{\dslash X'}\to \mathrm{sSet}_{\dslash X}$ is defined by pulling back along $f$; its left adjoint $f_!$ is obtained by pushout along $f$.
Both $f^\ast$ and $f_!$ preserve $\mathrm{sSet}_\ast$-tensors and we have the following
\begin{proposition}
For any map of simplicial sets $f\colon X\to X'$, $(f_!\dashv f^\ast)\colon \mathrm{sSet}_{\dslash X}\to \mathrm{sSet}_{\dslash X'}$ is a $\mathrm{sSet}_\ast$-enriched Quillen adjunction.
Moreover
\begin{itemize}
  \item If $f$ is a weak equivalence then $(f_!\dashv f^\ast)$ is a $\mathrm{sSet}_\ast$-Quillen equivalence.
  
  \item If $f$ is a fibration or a projection to a factor of a product then $(f^\ast\dashv f_\ast)$ is a $\mathrm{sSet}_\ast$-Quillen adjunction.
  
  \item If $f$ is an acyclic fibration then $(f^\ast\dashv f_\ast)$ is a $\mathrm{sSet}_\ast$-Quillen equivalence.
\end{itemize}
\end{proposition}

The combinatorial models for parametrised spectra used in this article are obtained from retractive spaces using Hovey's symmetric stabilisation machine \cite{hovey_spectra_2001}.
\begin{definition}
A \emph{symmetric $X$-spectrum} is a sequence of pairs $P=\{(P(n), \sigma_n)\}_{n\geq 0}$ where each $P(n)$ is a retractive space over $X$ equipped with a (left) $\Sigma_n$-action, each $\sigma_n\colon S^1\owedge_X P(n) \to P(n+1)$ is a $\Sigma_n$-equivariant morphism of retractive spaces over $X$, and for each $p,q\geq 0$ the composite
\[
\begin{tikzcd}[column sep=large]
S^p\owedge_X P(q)
\ar[r, "S^{p-1}\owedge_X \sigma_q"]
&
S^{p-1}\owedge_X P(q+1)
\ar[r, "S^{p-2}\owedge_X \sigma_{q+1}"]
&
\dotsb
\ar[r, "\sigma_{p+q-1}"]
&
P(p+q)
\end{tikzcd}
\]
is $(\Sigma_p\times \Sigma_q)$-equivariant.
A morphism of symmetric $X$-spectra $f\colon P\to Q$ is a sequence of equivariant maps of retractive spaces $f(n)\colon P(n)\to Q(n)$ commuting with the structure maps.
The category of symmetric $X$-spectra is denoted $\mathrm{Sp}^\Sigma_X$.
\end{definition}

For each $k\geq 0$ there is an adjunction
\begin{equation}
\label{eqn:StabAdjunctions}
\begin{tikzcd}
\mathrm{sSet}_{\dslash X}
\ar[rr, shift left=1.1ex, "\Sigma^{\infty-k}_X"]
\ar[rr, leftarrow, shift right=1.1ex, "\widetilde{\Omega}_X^{\infty-k}"', "\bot"]
&&
\mathrm{Sp}^\Sigma_X
\end{tikzcd}
\end{equation}
relating symmetric $X$-spectra and retractive spaces over $X$\footnote{These adjunctions are denoted $(\mathbf{\Sigma}^{\infty-k}_X\dashv \widetilde{\mathbf{\Omega}}_X^{\infty-k})$ in \cite{braunack-mayer_combinatorial_2019} to distinguish from similar functors for sequential parametrised spectra.}.
The right adjoint
$
\widetilde{\Omega}^{\infty-k}_X
$
sends the symmetric $X$-spectrum $P$ to its $k$-th term $P(k)$, forgetting the $\Sigma_k$-action.
The left adjoint $\Sigma^{\infty-k}_X$ sends the retractive space $Y$ over $X$ to the symmetric $X$-spectrum
\[
\Sigma^{\infty-k}_X (n) =
\begin{cases}
\qquad\qquad\quad  X & n<k
\\
\big((\Sigma_n)_+ \owedge_X S^{n-k}\owedge_X Y\big)\big/ \Sigma_{n-k} & n\geq k\,,
\end{cases}
\]
where $\Sigma_{n-k}$ acts via the canonical inclusion on $\Sigma_n$ and on $S^{n-k} = S^1\wedge \dotsb \wedge S^1$ by permuting smash factors.

A morphism $f\colon P\to Q$ of symmetric $X$-spectra is a levelwise fibration or weak equivalence if each $f(n) \colon P(n)\to Q(n)$ is a fibration or weak equivalence respectively.
The levelwise fibrations and weak equivalences together determine the \emph{projective model structure} on $\mathrm{Sp}^\Sigma_X$.

The category $\mathrm{Sp}^\Sigma_X$ is enriched, tensored, and powered over the symmetric monoidal category $(\mathrm{Sp}^\Sigma, \wedge)$ of symmetric spectra \cite{hovey_symmetric_2000}.
The $\mathrm{Sp}^\Sigma$-tensor of $K\in \mathrm{Sp}^\Sigma$ and $P\in \mathrm{Sp}^\Sigma_X$ is denoted $K\owedge_X P$\footnote{This is denoted $K\odot_X P$ in the notation of \cite{braunack-mayer_combinatorial_2019}.}; tensoring with $\Sigma^\infty S^1$ gives rise to a left Quillen endofunctor $\Sigma_X:= \Sigma^\infty S^1 \owedge_X (-)$ of the projective model structure on $\mathrm{Sp}^\Sigma_X$.
The left Quillen endofunctor $\Sigma_X$ models suspension on the homotopy category $Ho(\mathrm{Sp}^\Sigma_X)_\mathrm{proj}$ and has right adjoint $\Omega_X$.

The \emph{stable model structure} on $\mathrm{Sp}^\Sigma_X$ is obtained as a left Bousfield localisation of the projective model structure.
We record the salient properties of the stable model structure on $\mathrm{Sp}^\Sigma_X$:
\begin{itemize}
  \item The stable model structure on $\mathrm{Sp}^\Sigma_X$ is a left proper combinatorial $\mathrm{Sp}^\Sigma$-model structure. With respect to the stable model, the adjunction  $(\Sigma_X\dashv \Omega_X)\colon \mathrm{Sp}^\Sigma_X\to \mathrm{Sp}^\Sigma_X$ is a Quillen equivalence. Since the derived functor of $\Sigma_X$ is suspension on the homotopy category, $\mathrm{S}^\Sigma_X$ is a stable model category.
  
  \item The fibrant objects of $\mathrm{Sp}^\Sigma_X$ are the \emph{fibrant $\Omega_X$-spectra}; the symmetric $X$-spectra $A$ such that each $A(n)\to X$ is a fibration and the adjoint structure maps $\sigma_n^\vee \colon A(n)\to \Omega_X A(n+1)$ are weak equivalences.
  The stable weak equivalences of fibrant $\Omega_X$-spectra are precisely the levelwise weak equivalences.
  
  \item For each $k\geq 0$, the adjunction $(\Sigma^{\infty-k}_X\dashv \widetilde{\Omega}^{\infty-k}_X)$ of \eqref{eqn:StabAdjunctions} is a $\mathrm{sSet}_\ast$-Quillen adjunction.
\end{itemize}
The base change adjunctions between categories of retractive spaces \eqref{eqn:BaseChangeRetSpace} prolong to categories of symmetric parametrised spectra.
For each map of simplicial spaces $f\colon X\to X'$ there is an associated triple of base change adjunctions
\[
(f_!\dashv f^\ast \dashv f_\ast)
\colon 
\begin{tikzcd}
\mathrm{Sp}^\Sigma_{X}
\ar[rr, shift left=2.2ex]
\ar[rr, leftarrow, "\bot"]
\ar[rr, shift left=-2.2ex, "\bot"]
&&
\mathrm{Sp}^\Sigma_{ X'}\,.
\end{tikzcd}
\]
The base change functors $f_!$ and $f^\ast$ preserve $\mathrm{Sp}^\Sigma$-tensors and we have the following
\begin{proposition}
For any map of simplicial sets $f\colon X\to X'$, $(f_!\dashv f^\ast)\colon \mathrm{Sp}^\Sigma_X\to \mathrm{Sp}^\Sigma_{X'}$ is a $\mathrm{Sp}^\Sigma$-enriched Quillen adjunction between stable model structures.
Moreover
\begin{itemize}
  \item If $f$ is a weak equivalence then $(f_!\dashv f^\ast)$ is a $\mathrm{Sp}^\Sigma$-Quillen equivalence.
  
  \item If $f$ is a fibration or a projection to a factor of a product then $(f^\ast\dashv f_\ast)$ is a $\mathrm{Sp}^\Sigma$-Quillen adjunction.
  
  \item If $f$ is an acyclic fibration then $(f^\ast\dashv f_\ast)$ is a $\mathrm{Sp}^\Sigma$-Quillen equivalence.
\end{itemize}
\end{proposition}

To connect with the picture of $X$-parametrised spectra as homotopically-coherent spectrum-valued local systems that was given in the introduction, we record the following
\begin{proposition}
\label{prop:StableFibreEquiv}
A map of symmetric $X$-spectra $f\colon P\to Q$ is a stable weak equivalence precisely if $\mathbf{R}x^\ast(f)\colon \mathbf{R}x^\ast (P)\to \mathbf{R}x^\ast(Q)$ is an isomorphism in the stable homotopy category for every $x\colon \ast \to X$.
\end{proposition}

\subsection{Stable loop space modules}
\label{SS:StableLoopSpace}
Stable equivalences of parametrised spectra are detected by homotopy fibers (Proposition \ref{prop:StableFibreEquiv}) so that for any space $X$ there is a faithful embedding
\[
\begin{tikzcd}
\mathrm{fib}\colon Ho(\mathrm{Sp}_X)
\ar[r, hookrightarrow]
& 
\displaystyle\prod_{\pi_0(X)} Ho(\mathrm{Sp})\,.
\end{tikzcd}
\]
When $X$ is connected, $X$-parametrised spectra are equivalently $\Omega X_+$-module spectra in such a manner that the above embedding is identified with the forgetful functor.
The goal of this section is to give a precise proof of this fact using the combinatorial models.

\begin{remark}[$r$-reduced simplicial sets]
\label{rem:RedsSet}
Since we work with simplicial sets, we are able to present $k$-connected homotopy types by combinatorial models that are \lq\lq minimal'' in dimensions $\leq k$.
Recall that a simplicial set $X$ is \emph{$r$-reduced} if it has a unique $n$-simplex for $0\leq n< r$.
In the case $r=1$, we simply say that $X$ is \emph{reduced}.
The category of $r$-reduced simplicial sets is written $\mathrm{sSet}_{\geq r}$
and for each $r\geq 1$ the evident inclusion $\mathrm{sSet}_{\geq r}\hookrightarrow \mathrm{sSet}_\ast$ participates in a Quillen adjunction
\begin{equation}
\label{eqn:RedSset}
\begin{tikzcd}
\mathrm{sSet}_{\geq 1}
\ar[rr, hookrightarrow, shift left=1.1ex]
\ar[rr, leftarrow, shift right=1.1ex, "E_r"', "\bot"]
&&
\mathrm{sSet}_\ast\,.
\end{tikzcd}
\end{equation}
The right adjoint $E_r$ sends a pointed simplicial set $(X,x)$ to its \emph{$r$-th Eilenberg subcomplex}, which is the pullback 
\[
\begin{tikzcd}
E_r(X,x)
\ar[r]
\ar[d]
\ar[dr, phantom, "\ulcorner", very near start]
&
X
\ar[d]
\\
\ast 
\ar[r, "x"]
&
\mathrm{cosk}_{r-1} X\,.
\end{tikzcd}
\]
An $r$-reduced simplicial set is $(r-1)$-connected and any $(r-1)$-connected simplicial set is weakly equivalent to an $r$-reduced simplicial set.
The derived adjunction of \eqref{eqn:RedSset} exhibits $Ho(\mathrm{sSet}_{\geq r})$ as the full subcategory of $Ho(\mathrm{sSet}_\ast)$ on the pointed $(r-1)$-connected homotopy types.
\end{remark}

\begin{construction}
For a reduced simplicial set $X$, we use Kan's simplicial loop group $\mathbb{G}X$ to model the based loop space $\Omega X$ (see \cite{goerss_simplicial_2009} for a contemporary account).
By construction, $\mathbb{G}X$ is a strict simplicial group and the assignment $X\mapsto \mathbb{G}X$ is the left adjoint of a Quillen equivalence
\[
\begin{tikzcd}
\mathrm{sSet}_{\geq 1}
\ar[rr, shift left=1.1ex, "\mathbb{G}"]
\ar[rr, leftarrow, shift right=1.1ex, "\overline{W}"', "\bot"]
&&
\mathrm{sGrp}
\end{tikzcd}
\]
between reduced simplicial sets and simplicial groups.
The right adjoint computes Kan's model of the classifying space, that is $\overline{W} G\cong BG$ for any simplicial group $G$.
$\overline{W}G$ is the quotient of an acyclic free $G$-space $WG$ and the fibration $WG\to \overline{W}G$ models the universal principal $G$-fibration $EG\to BG$.
The unit and counit of the $(\mathbb{G}\dashv \overline{W})$-adjunction are natural weak equivalences, so that the pullback
\[
\begin{tikzcd}
\mathbb{P}X \ar[d, twoheadrightarrow, "\pi_X"']\ar[r] 
\arrow[dr, phantom, "\ulcorner", very near start]
&
 W \mathbb{G}X\ar[d, twoheadrightarrow]
\\
X\ar[r, "\eta_X"]
& 
\overline{W}\mathbb{G}X
\end{tikzcd}
\]
furnishes a model $\pi_X \colon \mathbb{P}X\to X$ for the path fibration of $X$.
The map $\pi_X$ is natural in $X$, and $\mathbb{P}X$ is a particularly nice model for the path space; it is a free $\mathbb{G}X$-space such that $\mathbb{P}X/\mathbb{G}X$ is isomorphic to $X$.
\end{construction}

For a simplicial group $G$, adjoining a disjoint basepoint $G\mapsto G_+= G\coprod \ast$ produces a monoid in $(\mathrm{sSet}_\ast, \wedge)$.
Write $G_+\mathrm{-Mod}_u$ for the category of left $G_+$-modules in pointed simplicial sets,  equivalently $G$-spaces with a distinguished $G$-fixed point (such objects are unstable $G_+$-modules, hence the subscript \lq\lq$u$'').
The free-forgetful adjunction
\[
\begin{tikzcd}
\mathrm{sSet}_{\ast}
\ar[rr, shift left=1.1ex, "G_+\wedge(-)"]
\ar[rr, leftarrow, shift right=1.1ex, "U"', "\bot"]
&&
G_+\mathrm{-Mod}_u
\end{tikzcd}
\]
induces a proper combinatorial $\mathrm{sSet}_\ast$-model structure. The weak equivalences and fibrations in $G_+\mathrm{-Mod}_u$ are created by $U$ and $\mathrm{sSet}_\ast$-tensors are computed in terms of underlying pointed simplicial sets.

\begin{remark}
For any morphism of simplicial groups $\psi\colon G\to H$ there is an induced $\mathrm{sSet}_\ast$-Quillen adjunction $(\psi_!\dashv \psi^\ast)\colon G_+\mathrm{-Mod}_u\to  H_+\mathrm{-Mod}_u$.
The right adjoint $\psi^\ast$ simply regards the $H_+$-module $L$ as a $G_+$-module via $\psi_+\colon G_+\to H_+$, whereas the left adjoint sends
\[
K\longmapsto \psi_! K = H_+ \bigwedge_{G_+}K\,, 
\]  
with $G_+$ acting to the right on $H_+$ via $\psi_+$.
By Quillen invariance for $\mathrm{sSet}_\ast$ (cf~\cite{schwede_monoidal_2003}), $(\psi_!\dashv \psi^\ast)$ is a Quillen equivalence if $\psi\colon G\to H$ is a weak equivalence of simplicial groups.
\end{remark} 

When the simplicial group under consideration is the Kan simplicial loop group $\mathbb{G}X$ of a reduced simplicial set $X$, there is a Quillen equivalence relating $\mathbb{G}X_+$-modules and  retractive spaces over $X$.
\begin{lemma}
\label{lem:RetSpModUn}
For any reduced simplicial set $X$, there is a $\mathrm{sSet}_\ast$-Quillen equivalence
\[
\begin{tikzcd}
\mathrm{sSet}_{\dslash X}
\ar[rr, shift left=1.1ex, "\mathfrak{f}_X"]
\ar[rr, leftarrow, shift right=1.1ex, "\mathfrak{b}_X"', "\bot"]
&&
\mathbb{G}X_+\mathrm{-Mod}_u\,.
\end{tikzcd}
\]
For any map $f\colon X\to Y$ of reduced simplicial sets, there is a natural isomorphism of left $\mathrm{sSet}_\ast$-Quillen functors $\mathbb{G}f_! \circ \mathfrak{f}_X \cong \mathfrak{f}_Y\circ f_!$.
\end{lemma}
\begin{proof}
The first part of the result is \cite[Lemma 1.31]{braunack-mayer_combinatorial_2019}; let us recall the definitions of $\mathfrak{f}_X$ and $\mathfrak{b}_X$.
The left adjoint $\mathfrak{f}_X$ sends the retractive space $Z$ over $X$ to the pushout of $\mathbb{G}X$-spaces
\[
\begin{tikzcd}
\mathbb{P}X \ar[d]
\ar[r] 
\arrow[dr, phantom, "\lrcorner", very near end]
&
 \ast
\ar[d]
\\
\pi_X^\ast Z \ar[r, ]
& 
\mathbb{P}X_! \pi^\ast_X Z\,.
\end{tikzcd}
\]
Note that $\pi^\ast_X Z$ is a free $\mathbb{G}X$-space ($\mathbb{P}X$ is a retract), so that $\mathbb{P}X_! \pi^\ast_X Z$ is free away from the basepoint.
The right adjoint sends the $\mathbb{G}X_{+}$-module $K$, regarded as a diagram of $\mathbb{G}X$-spaces $\ast\to K\to \ast$, to the retractive space over $X$
\[
\mathbb{P}X/\mathbb{G}X\cong X  \longrightarrow 
(K\times \mathbb{P}X)/\mathbb{G}X
\longrightarrow 
\mathbb{P}X/\mathbb{G}X \cong X\,,
\]
where $K\times \mathbb{P}X$ is equipped with the diagonal $\mathbb{G}X$-action.

For a morphism of reduced simplicial sets $f\colon X\to Y$, recall that $f_!$ is computed by taking pushouts along $f$.
For any $Z\in \mathrm{sSet}_{\dslash X}$ there is a natural commuting diagram
\[
\begin{tikzcd}
\mathbb{P}X 
\ar[d] 
\ar[r]
&
\pi^\ast_Y X 
\ar[r]
\ar[d]
\ar[dr, phantom, "\lrcorner", very near end]
&
\mathbb{P}Y
\ar[r]
\ar[d]
\ar[dr, phantom, "\lrcorner", very near end]
&
\ast
\ar[d]
\\
\pi^\ast_X Z
\ar[r]
&
\pi^\ast_Y Z
\ar[r]
&
\pi^\ast_Y f_! Z
\ar[r]
&
\mathbb{P}Y_! \pi^\ast_Y f_! Z\,.
\end{tikzcd}
\]
The left-hand square comes from naturality of $\pi_X$ in $X$, the middle square is a pushout since colimits in $\mathrm{sSet}$ are universal (that is, stable under base change), and the right-hand square is a pushout by definition.
This diagram determines a natural morphism $\mathfrak{f}_X Z\to (\mathbb{G}f^\ast\mathfrak{f}_Y f_!)(Z)$ in $\mathbb{G}X_+\mathrm{-Mod}_u$, and taking adjoints we have a natural map of pointed $\mathbb{G}Y$-space $(\mathbb{G}f_! \circ \mathfrak{f}_X) (Z) \to( \mathfrak{f}_Y\circ  f_! )(Z)$.
By the above, the $\mathbb{G}Y$-actions on both domain and codomain are free away from the basepoint and a careful unravelling of the definitions shows that taking quotients by the $\mathbb{G}Y$-action yields the identity map on $X_! Z \cong Z/X$.
It follows that $\mathbb{G}f_!\circ \mathfrak{f}_X\Rightarrow \mathfrak{f}_Y\circ f_!$ is a natural isomorphism.
\end{proof}

We now implement the passage to stable homotopy theory by applying the symmetric stabilisation machine \cite{hovey_spectra_2001}. 
Since the Quillen adjunctions of Lemma \ref{lem:RetSpModUn} preserve $\mathrm{sSet}_\ast$-tensors, symmetric stabilisation produces stable model categories and Quillen adjunctions enriched over symmetric spectra.
This enrichment over $\mathrm{Sp}^\Sigma$ is crucial to our arguments below.

Applying the symmetric stabilisation machine to the $\mathrm{sSet}_\ast$-Quillen adjunctions of Lemma \ref{lem:RetSpModUn} proves the following
\begin{theorem}
\label{thm:ParamSpecandLoop}
For any reduced simplicial set $X$, there is a $\mathrm{Sp}^\Sigma$-Quillen equivalence
\[
\begin{tikzcd}
\mathrm{Sp}^\Sigma_X
\ar[rr, shift left=1.1ex, "\mathfrak{f}_X"]
\ar[rr, leftarrow, shift right=1.1ex, "\mathfrak{b}_X"', "\bot"]
&&
\mathbb{G}X_+\mathrm{-Mod}\,.
\end{tikzcd}
\]
For any map $f\colon X\to Y$ of reduced simplicial sets, there is a natural isomorphism of left $\mathrm{Sp}^\Sigma$-Quillen functors $\mathbb{G}f_! \circ \mathfrak{f}_X\cong \mathfrak{f}_Y\circ f_!$.
\end{theorem}

\begin{remark}
\label{rem:GModSpecvsGSpecInMod}
There is a potential ambiguity in how one ought to to define the stable model category of $G_+$-modules; we can consider either $G_+$-modules in symmetric spectra (identifying $G_+$ with the suspension sepctrum $\Sigma^\infty G_+$) or symmetric spectrum objects in $G_+\mathrm{-Mod}_u$.
It is not difficult to verify that the two approaches yield categories that are canonically isomorphic.
\end{remark}
\begin{remark}
\label{rem:FibSpec}
For a cofibrant symmetric $X$-spectrum $P$, the underlying symmetric spectrum of $\mathfrak{f}_X (P)$ is a cofibrant model for the homotopy fibre of $P$.
For \emph{sequential} parametrised spectra this is a consequence of the proof of \cite[Lemma 2.17]{braunack-mayer_combinatorial_2019} and essentially the same argument applies for symmetric parametrised spectra.
\end{remark}

\begin{example}
\label{exm:FibSuSpectraAndLoopModules}
A map of simplicial sets $p\colon Y\to X$ gives rise to a retractive space over $X$
\[
Y_{+X} = \left( X\longrightarrow X\coprod Y\longrightarrow X\right)\,.
\]
The fibrewise suspension spectrum $\Sigma^\infty_X Y_{+X}$ is then a symmetric $X$-spectrum whose homotopy fibres are the suspensions of the homotopy fibres of $p$.
If $X$ is reduced, by the previous remark
 $\mathfrak{f}_X\big(\Sigma^\infty_X Y_{+Y}\big)$ is stably equivalent to $\Sigma^\infty \mathrm{fib}(p)_+$.
We highlight two particular examples:
\begin{enumerate}[label=(\arabic*)]
  \item For $p = \mathrm{id}_X\colon X\to X$ we obtain a cofibrant $\mathbb{G}X_+$-module spectrum modelling the trivial $\Omega X_+$-action on the sphere spectrum $\mathbb{S}$.
  
  \item For $p = x\colon \ast \to X$ the unique point inclusion, we obtain $\mathbb{G}X_+$ as a left module over itself.
\end{enumerate}
\end{example}

\subsection{Rationalisation is smashing}
\label{S:ratsmash}
It is well-known that smashing with the Eilenberg--Mac Lane spectrum $H\mathbb{Q}$ localises the classical stable homotopy category at rational stable homotopy equivalences.
Modelling $H\mathbb{Q}$ as a fibrant-cofibrant commutative symmetric ring spectrum we present stable rational homotopy theories in terms of strict $H\mathbb{Q}$-module objects.
We work in some detail as it plays an important role in the rectification arguments to follow.
\begin{construction}
There is an adjunction
\begin{equation}
\label{eqn:psSetQVect}
\begin{tikzcd}
\mathrm{sSet}_\ast
\ar[rr, shift left=1.1ex, "\widetilde{\mathbb{Q}}"]
\ar[rr, leftarrow, shift right=1.1ex, "U"', "\bot"]
&&
\mathrm{sVect}_\mathbb{Q}\,,
\end{tikzcd}
\end{equation}
in which the left adjoint $\widetilde{\mathbb{Q}}$ sends the pointed simplicial set $(K,k)$ to the simplicial rational vector space with $n$-simplices
\[
\widetilde{\mathbb{Q}}[K]_n := \mathbb{Q}[K_n \setminus s_0^n(k)] \cong \mathbb{Q}[K_n] \big/ \mathbb{Q}[s_0^n(k)]
\]
the free vector space on the non-basepoint $n$-simplices of $K$.
The forgetful functor $U$ simply regards a simplicial rational vector space as a simplicial set pointed by the zero vector.
The adjunction \eqref{eqn:psSetQVect} equips $\mathrm{sVect}_\mathbb{Q}$ with a proper combinatorial $\mathrm{sSet}_\ast$-model structure.
$\mathrm{sVect}_\mathbb{Q}$ is moreover a symmetric monoidal model category with respect to the degreewise tensor product of simplicial rational vector spaces, and $\widetilde{\mathbb{Q}}\colon (\mathrm{sSet}_\ast, \wedge)\to (\mathrm{sVect}_\mathbb{Q}, \otimes)$ is a strongly monoidal Quillen functor.

The symmetric Eilenberg--Mac Lane spectrum is $H\mathbb{Q} := U \widetilde{\mathbb{Q}}[\mathbb{S}]$, with underlying symmetric sequence
\[
H\mathbb{Q}(n) = \widetilde{\mathbb{Q}}\big[\underbrace{S^1\wedge \dotsb \wedge S^1}_{\text{$n$ factors}}\big]\,,
\] 
with $\Sigma_n$-action permuting the smash factors.
The symmetric spectrum structure maps are the composites
\[
S^1\wedge H\mathbb{Q}(n) \longrightarrow \widetilde{\mathbb{Q}}[S^1] \otimes \widetilde{\mathbb{Q}}[H\mathbb{Q}(n)] \longrightarrow \widetilde{\mathbb{Q}}[S^1] \otimes H\mathbb{Q}(n) \cong H\mathbb{Q}(n+1)
\]
of the adjunction unit followed by the adjunction counit in the second tensor factor.
The symmetric spectrum $H\mathbb{Q}$ inherits the structure of symmetric right spectrum from the sphere spectrum $\mathbb{S}$, where the unit $\mathbb{S}\to H\mathbb{Q}$ is the unit of the adjunction \eqref{eqn:psSetQVect}.
Using the Dold--Kan correspondence, it is straightforward to check that $H\mathbb{Q}$ is a fibrant $\Omega$-spectrum with a single stable homotopy group $\mathbb{Q}$ concentrated in degree $0$.
Using the characterisation of cofibrant symmetric spectra given in \cite[Proposition 3.4.2]{hovey_symmetric_2000}, one checks that $H\mathbb{Q}$ is moreover cofibrant.
\end{construction}

We write $H\mathbb{Q}\mathrm{-Mod}$ for the category of symmetric $H\mathbb{Q}$-module spectra.
The free-forgetful adjunction
\[
\begin{tikzcd}
\mathrm{Sp}^\Sigma
\ar[rr, shift left=1ex, "H\mathbb{Q}\wedge (-)"]
\ar[rr, shift left=-1ex, leftarrow,  "\bot"]
&&
H\mathbb{Q}\mathrm{-Mod}
\end{tikzcd}
\]
furnishes $H\mathbb{Q}\mathrm{-Mod}$ with the structure of a symmetric monoidal $\mathrm{Sp}^\Sigma$-model category such that the forgetful functor creates weak equivalences and fibrations (cf~Construction \ref{con:Rationalisation}).

The following result is well known, however we provide a proof in order to clarify the results of this section.
\begin{proposition}[Rationalisation is smashing]
\label{prop:RatisSmashing}
The free-forgetful 
$\mathrm{Sp}^\Sigma$-Quillen adjunction
\[
\begin{tikzcd}
\mathrm{Sp}^\Sigma
\ar[rr, shift left=1ex, "H\mathbb{Q}\wedge (-)"]
\ar[rr, shift left=-1ex, leftarrow,  "\bot"]
&&
H\mathbb{Q}\mathrm{-Mod}
\end{tikzcd}
\]
exhibits $Ho(H\mathbb{Q}\mathrm{-Mod})$ as the localisation of $Ho(\mathrm{Sp})$ at rational stable homotopy equivalences.
\end{proposition}
\begin{proof}
As $H\mathbb{Q}$ is a cofibrant symmetric ring spectrum, $H\mathbb{Q}\wedge (-)$ preserves stable weak equivalences \cite[Corollary 5.3.10]{hovey_symmetric_2000} so that $H\mathbb{Q}\wedge E$ is equivalent to the derived smash product $H\mathbb{Q}\wedge^\mathbf{L}E$ for any symmetric spectrum $E$.
The pairing $\pi^\mathrm{st}_k(E)\otimes_\mathbb{Z}\pi^\mathrm{st}_l(F)\to \pi^\mathrm{st}_{k+l}(E\wedge^\mathbf{L} F)$ on (true) stable homotopy groups gives rise to a natural transformation $\rho_E\colon \mathbb{Q}\otimes_\mathbb{Z} \pi^\mathrm{st}_k(E)\to \pi_k^\mathrm{st}(H\mathbb{Q}\wedge E)$.
Let $\mathcal{E}\hookrightarrow Ho(\mathrm{Sp}^\Sigma)$ be the full subcategory of objects $E$ for which $\rho_E$ is an isomorphism.
Then $\mathcal{E}$ is a triangulated subcategory that is moreover localising (ie~closed under arbitrary copoducts) and contains the sphere spectrum by classical results of Serre.
Since the sphere spectrum is a compact generator of the classical stable homotopy category we must have $\mathcal{E}= Ho(\mathrm{Sp}^\Sigma)$.

For any symmetric $H\mathbb{Q}$-module spectrum $N$, the image of the action $H\mathbb{Q}\wedge N\to N$ under taking stable homotopy groups shows that the $\pi^\mathrm{st}_\ast (N)$ form a $\mathbb{Z}$-graded rational vector space.
The component of the derived counit at $N$ is stably equivalent to the action $H\mathbb{Q}\wedge N\to N$, hence is a stable weak equivalence.
It follows that $Ho(H\mathbb{Q}\mathrm{-Mod})\to Ho(\mathrm{Sp}^\Sigma)$ is fully-faithful and hence that $Ho(\mathrm{Sp}^\Sigma)$ is a reflective localisation.

For a morphism of cofibrant symmetric spectra $f\colon A\to B$, the following are equivalent:
\begin{enumerate}[label=(\roman*)]
  \item $f$ is a rational stable homotopy equivalence,
  \item $H\mathbb{Q}\wedge f$ is a stable weak equivalence,
  \item $f$ induces an isomorphism $\pi^\mathrm{st}_\ast(A)\otimes_\mathbb{Z}\mathbb{Q}\to \pi^\mathrm{st}_\ast(B)\otimes_\mathbb{Z}\mathbb{Q}$.
\end{enumerate}
The equivalence (ii)$\Leftrightarrow$(iii) is easily extracted from the above.
To show (i)$\Rightarrow$(ii) we use that the free functor $\mathrm{Sp}^\Sigma\to H\mathbb{Q}\mathrm{-Mod}$ commutes with $\mathrm{Sp}^\Sigma$-tensors, which implies a natural isomorphism of enriched hom spectra
\begin{equation}
\label{eqn:HQHomSpec}
\underline{\mathrm{Sp}}^\Sigma\big(E, N\big)
\xrightarrow{\;\;\cong\;\;}
\underline{H\mathbb{Q}\mathrm{-Mod}}\big(H\mathbb{Q}\wedge E, N\big)
\end{equation}
for $E\in \mathrm{Sp}^\Sigma$ and $N\in H\mathbb{Q}\mathrm{-Mod}$.
The $H\mathbb{Q}$-action on $N$ induces a $H\mathbb{Q}$-action on the domain as the smash-hom adjunct of the composite
\[
H\mathbb{Q}\wedge \underline{\mathrm{Sp}}^\Sigma\big(E, N\big)\wedge E
\longrightarrow
H\mathbb{Q}\wedge N
\longrightarrow
N\,.
\]
In particular, for $N$ fibrant and $E$ cofibrant we have $\{E,N\}^\ast\otimes_\mathbb{Z}\mathbb{Q}\cong \{E, N\}^\ast \cong \pi_\ast^\mathrm{st}\underline{\mathrm{Sp}}^\Sigma(E,N)$.
Assuming (i), it now follows from \eqref{eqn:HQHomSpec} that
\[
\underline{H\mathbb{Q}\mathrm{-Mod}}\big(H\mathbb{Q}\wedge B, N\big)
\longrightarrow
\underline{H\mathbb{Q}\mathrm{-Mod}}\big(H\mathbb{Q}\wedge A, N\big)
\]
is a stable weak equivalence for all fibrant $N$ so that (ii) holds.
Conversely, supposing (ii) holds let $F$ be a fibrant symmetric spectrum with $H\mathbb{Q}\wedge F\to F_\mathbb{Q}$ a fibrant replacement in $H\mathbb{Q}\mathrm{-Mod}$ (hence also in $\mathrm{Sp}^\Sigma$).
The free functor $\mathrm{Sp}^\Sigma\to H\mathbb{Q}\mathrm{-Mod}$ induces maps of enriched hom spectra
\[
\underline{\mathrm{Sp}}^\Sigma\big(A, F\big)
\longrightarrow 
\underline{H\mathbb{Q}\mathrm{-Mod}}\big(H\mathbb{Q}\wedge A, H\mathbb{Q}\wedge F\big)
\longrightarrow
\underline{H\mathbb{Q}\mathrm{-Mod}}\big(H\mathbb{Q}\wedge A, F_\mathbb{Q}\big) \cong
\underline{\mathrm{Sp}}^\Sigma\big(A, F_\mathbb{Q}\big)
\]
and similarly for $B$.
Taking stable homotopy groups, we get a commuting diagram of $\mathbb{Z}$-graded rational vector spaces
\[
\begin{tikzcd}[row sep=small]
\{A, F\}^\ast\otimes_\mathbb{Z}\mathbb{Q}
\ar[r]
\ar[d]
&
\{A, F_\mathbb{Q}\}^\ast
\ar[d, "\cong"]
\\
\{B, F\}^\ast\otimes_\mathbb{Z}\mathbb{Q}
\ar[r]
&
\{B, F_\mathbb{Q}\}^\ast
\end{tikzcd}
\]
in which the right-hand vertical arrow is an isomorphism by \eqref{eqn:HQHomSpec} and the hypothesis on $f$.
The horizontal arrows are also isomorphisms (consider the localising triangulated subcategory $\mathcal{E}_F\hookrightarrow Ho(\mathrm{Sp}^\Sigma)$ of symmetric spectra $E$ for which $\{E, F\}^\ast\otimes_\mathbb{Z}\mathbb{Q}\to  \{E, F_\mathbb{Q}\}^\ast$ is an isomorphism;
then $\mathcal{E}_F=Ho(\mathrm{Sp}^\Sigma)$ as it contains the compact generator $\mathbb{S}$).
As the fibrant symmetric spectrum $F$ was arbitrary we conclude that $f$ is a rational stable homotopy equivalence and so (ii)$\Rightarrow$(i).
The equivalence (i)$\Leftrightarrow$(ii) shows that $Ho(\mathrm{Sp}^\Sigma)\to Ho(H\mathbb{Q}\mathrm{-Mod})$ is the localisation at rational stable homotopy equivalences.
\end{proof}

\begin{construction}
\label{con:Rationalisation}
Let $\mathcal{M}$ be a cofibrantly generated $\mathrm{Sp}^\Sigma$-model category.
Since $H\mathbb{Q}$ is a cofibrant symmetric ring spectrum, tensoring with $H\mathbb{Q}$ defines a monad $T_\mathbb{Q}$ whose underlying functor is left Quillen.
The free-forgetful adjunction
\[
\begin{tikzcd}
\mathcal{M}
\ar[rr, shift left=1.1ex, "H\mathbb{Q}\wedge(-)"]
\ar[rr, shift right=1.1ex, "\bot" , leftarrow]
&&
T_\mathbb{Q}\mathrm{-Alg}(\mathcal{M})
\end{tikzcd}
\]
satisfies the conditions of Kan's transfer theorem (see \cite[Theorem 11.3.2]{hirschhorn_model_2003})  so that $T_\mathbb{Q}\mathrm{-Alg}(\mathcal{M})$ is a cofibrantly generated model category with weak equivalences and fibrations created by the forgetful functor to $\mathcal{M}$.
Since $H\mathbb{Q}$ is cofibrant and tensoring preserves colimits in both variables, the forgetful functor preserves colimits and cofibrations.
$T_\mathbb{Q}\mathrm{-Alg}(\mathcal{M})$ inherits $\mathrm{Sp}^\Sigma$-tensors from $\mathcal{M}$, that is for $E\in \mathrm{Sp}^\Sigma$ and $x\in T_\mathbb{Q}\mathrm{-Alg}(\mathcal{M})$, the object $E\wedge x$ of $\mathcal{M}$ becomes a $T_\mathbb{Q}$-algebra with respect to the structure map
\[
H\mathbb{Q}\wedge E\wedge x\cong E\wedge H\mathbb{Q}\wedge x\longrightarrow E\wedge x
\]
defined using the symmetry of the smash product in $\mathrm{Sp}^\Sigma$ and the $T_\mathbb{Q}$-algebra structure of $x$.
For $x,y \in T_\mathbb{Q}\mathrm{-Alg}(\mathcal{M})$, the enriched hom is the symmetric spectrum $\underline{\mathcal{M}}_{T_\mathbb{Q}}(x,y)$ defined as the equaliser of the two maps $
\begin{tikzcd}[cramped, sep=small]
\underline{\mathcal{M}}(x,y)
\ar[r, shift left=1]
\ar[r, shift left=-1]
&
\underline{\mathcal{M}}(H\mathbb{Q}\wedge x, y)
\end{tikzcd}
$; the first map is induced by the $T_\mathbb{Q}$-algebra structure map $H\mathbb{Q}\wedge x\to x$ on $x$ and the second is the composite of $\underline{\mathcal{M}}(x,y)\to \underline{\mathcal{M}}(H\mathbb{Q}\wedge x, H\mathbb{Q}\wedge y)$ followed by the $T_\mathbb{Q}$-algebra structure map of $y$.
The reader can check that these tensors and internal homs make $T_\mathbb{Q}\mathrm{-Alg}(\mathcal{M})$ a $\mathrm{Sp}^\Sigma$-model category.

The assignment $M\mapsto T_\mathbb{Q}\mathrm{-Alg}(\mathcal{M})$ is pseudofunctorial, that is any $\mathrm{Sp}^\Sigma$-Quillen adjunction of cofibrantly generated $\mathrm{Sp}^\Sigma$-model categories $(L\dashv R)\colon \mathcal{M}\to \mathcal{N}$ gives rise to a $\mathrm{Sp}^\Sigma$-Quillen adjunction
\[
\begin{tikzcd}
T_\mathbb{Q}\mathrm{-Alg}(\mathcal{M})
\ar[rr, shift left=1.1ex, "L_\mathbb{Q}"]
\ar[rr, shift right=1.1ex, "\bot" , "R_\mathbb{Q}"', leftarrow]
&&
T_\mathbb{Q}\mathrm{-Alg}(\mathcal{N})\,.
\end{tikzcd}
\]
The functors $L_\mathbb{Q}$ and $R_\mathbb{Q}$ coincide with $L$ and $R$ on underlying objects
\end{construction}

\begin{remark}
\label{rem:TQAlgHom}
If $(y,\rho_y)$ is a $T_\mathbb{Q}$-algebra in $\mathcal{M}$, the symmetric spectrum $\underline{\mathcal{M}}(x,y)$ is a $H\mathbb{Q}$-module spectrum.
The action $H\mathbb{Q}\wedge \underline{\mathcal{M}}(x,y)\to \underline{\mathcal{M}}(x,y)$ corresponds to
\[
\begin{tikzcd}
H\mathbb{Q}\wedge \underline{\mathcal{M}}(x,y) \wedge x
\ar[r, "\mathrm{ev}_x"]
&
H\mathbb{Q}\wedge y
\ar[r, "\rho_y"]
&
y
\end{tikzcd}
\]
under the natural isomorphisms $\mathrm{Sp}^\Sigma (E,\underline{\mathcal{M}}(x,y))\cong \mathcal{M}(E\wedge x, y)$.
For $m,n \in T_\mathbb{Q}\mathrm{-Alg}(\mathcal{M})$, the enriched hom $\underline{\mathcal{M}}_{T_\mathbb{Q}}(m,n)$ inherits a $H\mathbb{Q}$-action from that of $\underline{\mathcal{M}}(m,n)$. 
It is not too hard to show that $T_\mathbb{Q}\mathrm{-Alg}(\mathcal{M})$ is a $H\mathbb{Q}\mathrm{-Mod}$-model category, but we do not need this.
\end{remark}

\begin{construction}
For any objects $x,y\in \mathcal{M}$, the composite
\[
\underline{\mathcal{M}}(x,y)\wedge H\mathbb{Q} \wedge x
\cong 
H\mathbb{Q}\wedge \underline{\mathcal{M}}(x,y)\wedge x
\xrightarrow{\;\;H\mathbb{Q}\wedge \mathrm{ev}_x\;\;}
H\mathbb{Q}\wedge y
\]
determines a natural morphism $\underline{\mathcal{M}}(x,y)\to \underline{\mathcal{M}}_{T_\mathbb{Q}}(H\mathbb{Q}\wedge x, H\mathbb{Q}\wedge y)$, where the codomain is a $H\mathbb{Q}$-module spectrum by Remark \ref{rem:TQAlgHom}.

Supposing now that $x$ is cofibrant and $y$ is fibrant, let 
\begin{equation}
H\mathbb{Q}\wedge y\longrightarrow y_\mathbb{Q}
\end{equation}
be a functorial replacement in $T_\mathbb{Q}\mathrm{-Alg}(\mathcal{M})$ (hence also in $\mathcal{M}$).
Since the codomain is a $H\mathbb{Q}$-module spectrum, the composite morphism
$
\underline{\mathcal{M}}(x,y)
\to
\underline{\mathcal{M}}_{T_\mathbb{Q}}(H\mathbb{Q}\wedge x, H\mathbb{Q}\wedge y)
\to
\underline{\mathcal{M}}_{T_\mathbb{Q}}(H\mathbb{Q}\wedge x, y_\mathbb{Q})
$
gives rise to a morphism of $\mathbb{Z}$-graded rational vector spaces
\begin{equation}
\kappa_{x,y}\colon
\{x,y\}_\mathcal{M}^\ast \otimes_\mathbb{Z}\mathbb{Q}
\longrightarrow 
\{H\mathbb{Q}\wedge x, y_\mathbb{Q} \}^\ast_{T_\mathbb{Q}\mathrm{-Alg}(\mathcal{M})}\,,
\end{equation}
after taking stable homotopy groups.
\end{construction}

\begin{definition}
\label{defn:SmashinglyRat}
We say that a cofibrantly generated $\mathrm{Sp}^\Sigma$-model category $\mathcal{M}$ is \emph{smashingly rational} if $\kappa_{x,y}$ is a natural isomorphism for all cofibrant $x$ and fibrant $y$ in $\mathcal{M}$.
\end{definition}

The ur-example of a smashingly rational model category is $\mathrm{Sp}^\Sigma$ (Proposition \ref{prop:RatisSmashing}) and it turns out that this condition is satisfied in a very large class of examples (Theorem \ref{thm:RatSmash}).
The hypothesis that $\mathcal{M}$ is smashingly rational guarantees that passing to $T_\mathbb{Q}$-algebras by tensoring (or \emph{smashing}) with $H\mathbb{Q}$ implements rationalisation in the homotopy category.
\begin{lemma}
\label{lem:SmashRat}
If $\mathcal{M}$ is smashingly rational, the free-forgetful Quillen adjunction
\[
\begin{tikzcd}
\mathcal{M}
\ar[rr, shift left=1.1ex, "H\mathbb{Q}\wedge(-)"]
\ar[rr, shift right=1.1ex, "\bot" , leftarrow]
&&
T_\mathbb{Q}\mathrm{-Alg}(\mathcal{M})
\end{tikzcd}
\]
exhibits $Ho(T_\mathbb{Q}\mathrm{-Alg}(\mathcal{M}))$ as the localisation of $Ho(\mathcal{M})$ at rational stable homotopy equivalences.
\end{lemma}
\begin{proof}
A straightforward adjointness argument shows that there is an isomorphism of symmetric spectra $\underline{\mathcal{M}}(x,n)\cong \underline{\mathcal{M}}_{T_\mathbb{Q}}(H\mathbb{Q}\wedge x, n)$ for any $x\in \mathcal{M}$ and $n\in T_\mathbb{Q}\mathrm{-Alg}(\mathcal{M})$.
There is a diagram of $H\mathbb{Q}$-module spectra
\begin{equation}
\label{eqn:SmashRatTechnical1}
\begin{tikzcd}[row sep=small]
\underline{\mathcal{M}}(x,n)
\ar[r]
\ar[dr, bend right=15, "\cong"']
&
\underline{\mathcal{M}}_{T_\mathbb{Q}}(H\mathbb{Q}\wedge x, H\mathbb{Q}\wedge n)
\ar[d]
\\
&
\underline{\mathcal{M}}_{T_\mathbb{Q}}(H\mathbb{Q}\wedge x,n)
\end{tikzcd}
\end{equation}
in which the right-hand vertical arrow is induced by the adjunction counit.

Suppose that $x\in \mathcal{M}$ is cofibrant and that $n\in T_\mathbb{Q}\mathrm{-Alg}(\mathcal{M})$ is fibrant.
Since $T_\mathbb{Q}\mathrm{-Alg}(\mathcal{M})$ has functorial fibrant replacements there are fibrant objects $n_\mathbb{Q}, n^f \in T_\mathbb{Q}\mathrm{-Alg}(\mathcal{M})$ together with a commuting diagram of morphisms of $T_\mathbb{Q}$-algebras
\begin{equation}
\label{eqn:FibReplacementTQAlgCounit}
\begin{tikzcd}[row sep=small]
H\mathbb{Q}\wedge n
\ar[r, "\sim"]
\ar[d]
&
n_\mathbb{Q}
\ar[d]
\\
n
\ar[r, "\sim"]
&
n^f
\end{tikzcd}
\end{equation}
in which the horizontal arrows are weak equivalences.
Building upon the diagram \eqref{eqn:SmashRatTechnical1}, we obtain a commuting diagram of symmetric spectra
\begin{equation}
\label{eqn:SmashRatTechnical2}
\begin{tikzcd}[row sep=small]
\underline{\mathcal{M}}(x,n)
\ar[r]
\ar[dr, bend right=15, "\cong"']
&
\underline{\mathcal{M}}_{T_\mathbb{Q}}(H\mathbb{Q}\wedge x, H\mathbb{Q}\wedge n)
\ar[d]
\ar[r]
&
\underline{\mathcal{M}}_{T_\mathbb{Q}}(H\mathbb{Q}\wedge x,n_\mathbb{Q})
\ar[d]
\\
&
\underline{\mathcal{M}}_{T_\mathbb{Q}}(H\mathbb{Q}\wedge x,n)
\ar[r]
&
\underline{\mathcal{M}}_{T_\mathbb{Q}}(H\mathbb{Q}\wedge x,n^f)\,.
\end{tikzcd}
\end{equation}
As $H\mathbb{Q}\wedge x$ is cofibrant in $T_\mathbb{Q}\mathrm{-Alg}(\mathcal{M})$, $\underline{\mathcal{M}}_{T_\mathbb{Q}}(H\mathbb{Q}\wedge x,n)\to \underline{\mathcal{M}}_{T_\mathbb{Q}}(H\mathbb{Q}\wedge x,n^f)$ is a weak equivalence of symmetric spectra.
Passing to stable homotopy groups, \eqref{eqn:SmashRatTechnical2} therefore gives rise to a commuting diagram of graded rational vector spaces
\[
\begin{tikzcd}[row sep=small]
\{x,y\}_\mathcal{M}^\ast \cong \{x,y\}_\mathcal{M}^\ast\otimes_\mathbb{Z}\mathbb{Q}
\ar[r, "\kappa_{x,n}"]
\ar[d]
&
\{H\mathbb{Q}\wedge x, n_\mathbb{Q}\}^\ast_{T_\mathbb{Q}\mathrm{-Alg}(\mathcal{M})}
\ar[d]
\\
\{H\mathbb{Q}\wedge x, n\}^\ast_{T_\mathbb{Q}\mathrm{-Alg}(\mathcal{M})}
\ar[r]
&
\{H\mathbb{Q}\wedge x, n^f\}^\ast_{T_\mathbb{Q}\mathrm{-Alg}(\mathcal{M})}\,,
\end{tikzcd}
\]
where $\kappa_{x,n}$ is an isomorphism by hypothesis on $\mathcal{M}$.
Since the left-hand vertical and bottom horizontal arrows are also isomorphisms, it follows that for all cofibrant $x\in \mathcal{M}$ and fibrant $n\in T_\mathbb{Q}\mathrm{-Alg}(\mathcal{M})$ we have a zig-zag of isomorphisms
\begin{equation}
\label{eqn:SmashinglyRatCounitMapFree}
\{H\mathbb{Q}\wedge x, n_\mathbb{Q}\}_{T_\mathbb{Q}\mathrm{-Alg}(\mathcal{M})}^\ast \longrightarrow
\{H\mathbb{Q}\wedge x,n^f\}_{T_\mathbb{Q}\mathrm{-Alg}(\mathcal{M})}^\ast
\longleftarrow
\{H\mathbb{Q}\wedge x,n\}_{T_\mathbb{Q}\mathrm{-Alg}(\mathcal{M})}^\ast
\end{equation}
that is natural in $x$.

For any cofibrant $T_\mathbb{Q}$-algebra $m$, the bar construction of the monad $T_\mathbb{Q}$ gives rise to a Reedy cofibrant augmented simplicial object
\[
B_\bullet(T_\mathbb{Q};m)=\left(
\!
\begin{tikzcd}
\dotsb 
\ar[r]
\ar[r, shift left=1.6ex]
\ar[r, shift left=-1.6ex]
&
H\mathbb{Q}\wedge
H\mathbb{Q}\wedge m
\ar[r, shift left=0.8ex]
\ar[r, shift left=-0.8ex]
&
H\mathbb{Q}\wedge m
\end{tikzcd}
\!\right)
\xrightarrow{\;\;\rho_m\;\;}
m
\]
with face maps arising from the monoid structure of $H\mathbb{Q}$ and degeneracies given by inserting the unit of $H\mathbb{Q}$ into the different smash factors.
The unit $m \to H\mathbb{Q}\wedge m$ gives rise to a system of extra degeneracies, so that $B_\bullet(T_\mathbb{Q};m)$ is contractible as an augmented simplicial object and hence $\mathrm{colim}_{\Delta^\mathrm{op}} B_\bullet(T_\mathbb{Q};m)\to m$ is a weak equivalence of $T_\mathbb{Q}$-algebras.
Observe that $m$ is also cofibrant when regarded as an object of $\mathcal{M}$ and associated to the fibrant replacement \eqref{eqn:FibReplacementTQAlgCounit} of the adjunction counit there is a commuting diagram of symmetric spectra
\[
\begin{tikzcd}[row sep=small]
\underline{\mathcal{M}}_{T_\mathbb{Q}}(m,n_\mathbb{Q})
\ar[r]
\ar[d, "\sim"']
&
\underline{\mathcal{M}}_{T_\mathbb{Q}}(m, n^f)
\ar[d, "\sim"]
\\
\underset{\Delta}{\mathrm{holim}}\,
\underline{\mathcal{M}}_{T_\mathbb{Q}}(B_\bullet(T_\mathbb{Q};m),n_\mathbb{Q})
\ar[r]
&
\underset{\Delta}{\mathrm{holim}}\,\underline{\mathcal{M}}_{T_\mathbb{Q}}(B_\bullet(T_\mathbb{Q};m), n^f)
\end{tikzcd}
\]
with stable weak equivalences as marked.
In light of the natural isomorphisms \eqref{eqn:SmashinglyRatCounitMapFree}, we get that $\underline{\mathcal{M}}_{T_\mathbb{Q}}(B_\bullet(T_\mathbb{Q};m),n_\mathbb{Q})
\to\underline{\mathcal{M}}_{T_\mathbb{Q}}(B_\bullet(T_\mathbb{Q};m), n^f)$ is a levelwise weak equivalence of Reedy fibrant cosimplicial symmetric spectra. 
Taking limits, it follows that $\underline{\mathcal{M}}_{T_\mathbb{Q}}(m,n_\mathbb{Q})\to \underline{\mathcal{M}}_{T_\mathbb{Q}}(m,n^f)$ is a stable weak equivalence for any cofibrant $m$ and fibrant $n$.
It follows that the derived counit of the adjunction is a natural isomorphism in $Ho(T_\mathbb{Q}\mathrm{-Alg}(\mathcal{M}))$ and hence that $ Ho(T_\mathbb{Q}\mathrm{-Alg}(\mathcal{M}))\to Ho(\mathcal{M})$ is fully-faithful.

To complete the proof it is sufficient to show that a morphism $f\colon x\to x'$ of cofibrant objects in $\mathcal{M}$ is sent to an isomorphism in $Ho(T_\mathbb{Q}\mathrm{-Alg}(\mathcal{M}))$  if and only if it is a rational stable homotopy equivalence.
The \lq\lq if'' direction does not require the  assumption that $\mathcal{M}$ is smashingly rational; if $f$ is rational stable homotopy equivalence then the natural isomorphism  $\underline{\mathcal{M}}(x,n)\cong \underline{\mathcal{M}}_{T_\mathbb{Q}}(H\mathbb{Q}\wedge x, n)$ for $n\in T_\mathbb{Q}\mathrm{-Alg}(\mathcal{M})$ together with the fact that $\{x,n\}^\ast_\mathcal{M}$ is already a graded rational vector space imply that the induced map $\underline{\mathcal{M}}_{T_\mathbb{Q}}(H\mathbb{Q}\wedge x', n)\to \underline{\mathcal{M}}_{T_\mathbb{Q}}(H\mathbb{Q}\wedge x, n)$ is a stable weak equivalence for all fibrant $n\in H\mathbb{Q}\mathrm{-Mod}$, hence $f$ is sent to an isomorphism in $Ho(T_\mathbb{Q}\mathrm{-Alg}(\mathcal{M}))$.

For the converse, suppose that $f\colon x\to x'$ is sent to an isomorphism in $Ho(T_\mathbb{Q}\mathrm{-Alg}(\mathcal{M})$.
Fix a fibrant object $y\in \mathcal{M}$ and choose a fibrant replacement $H\mathbb{Q}\wedge y\to y_\mathbb{Q}$ in $T_\mathbb{Q}\mathrm{-Alg}(\mathcal{M})$ (hence also in $\mathcal{M}$).
Then there is a commuting diagram of graded rational vector spaces
\[
\begin{tikzcd}[row sep=small]
\{x',y\}^\ast_\mathcal{M}\otimes_\mathbb{Z}\mathbb{Q}
\ar[r]\ar[d]
&
\{H\mathbb{Q}\wedge x', y_\mathbb{Q}\}^\ast_{T_\mathbb{Q}\mathrm{-Alg}(\mathcal{M})}\cong \{ x', y_\mathbb{Q}\}^\ast_{\mathcal{M}}
\ar[d]
\\
\{x,y\}^\ast_\mathcal{M}\otimes_\mathbb{Z}\mathbb{Q}
\ar[r]
&
\{H\mathbb{Q}\wedge x, y_\mathbb{Q}\}^\ast_{T_\mathbb{Q}\mathrm{-Alg}(\mathcal{M})}
\cong \{ x, y_\mathbb{Q}\}^\ast_{\mathcal{M}}\,.
\end{tikzcd}
\]
The right-hand vertical arrow is an isomorphism by the hypothesis on $f$ whereas the horizontal arrows are isomorphisms since $\mathcal{M}$ is smashingly rational.
The left-hand vertical arrow is thus an isomorphism, and since $y$ was arbitrary it follows that $f\colon x\to x'$ is a rational stable homotopy equivalence.
\end{proof}
\begin{remark}
One can show that the converse to Lemma \ref{lem:SmashRat} holds; namely for a $\mathrm{Sp}^\Sigma$-model category $\mathcal{M}$, if $Ho(\mathcal{M})\to Ho(T_\mathbb{Q}\mathrm{-Alg}(\mathcal{M}))$ is the localisation at rational stable homotopy equivalences then $\mathcal{M}$ is smashingly rational.
\end{remark}

We are primarily interested in two families of examples:
\begin{definition}For any simplicial set $X$, write $H\mathbb{Q}\mathrm{-Mod}_X$ for the model category of $T_\mathbb{Q}$-algebras in $\mathrm{Sp}^\Sigma_X$.
  The objects of $H\mathbb{Q}\mathrm{-Mod}_X$ are thus symmetric $X$-spectra whose fibres are equipped with a coherent left $H\mathbb{Q}$-action.
\end{definition}
\begin{definition}
For $G$ a simplicial group, we write $(G_+,H\mathbb{Q})\mathrm{-Bimod}$ for the model category of $T_\mathbb{Q}$-algebras in $G_+\mathrm{-Mod}$.
  The category $(G_+,H\mathbb{Q})\mathrm{-Bimod}$ is isomorphic to the category of $(G_+,H\mathbb{Q})$-bimodules in symmetric spectra, justifying our mild abuse of notation.
\end{definition}
\begin{lemma}
\label{lem:XSpecisSmashinglyRat}
$\mathrm{Sp}^\Sigma_{X}$ is smashingly rational for any simplicial set $X$.
\end{lemma}
\begin{proof}
Fix a fibrant object $P\in \mathrm{Sp}^\Sigma_X$ and a fibrant replacement $H\mathbb{Q}\owedge_X P\to P_\mathbb{Q}$ in $H\mathbb{Q}\mathrm{-Mod}_X$.
Let $\mathcal{E}_P$ be the localising triangulated subcategory of $Ho(\mathrm{Sp}^\Sigma_X)$  of parametrised symmetric $X$-spectra $A$ for which 
\[
\kappa_{A,P}\colon \{A, P\}^\ast_{X}\otimes_\mathbb{Z}\mathbb{Q}\longrightarrow \{H\mathbb{Q}\wedge A, P_\mathbb{Q}\}_{H\mathbb{Q}\mathrm{-Mod}_X}^\ast
\]
is a natural isomorphism.

Taking $A = x_! \mathbb{S}$ to be the pushforward of the symmetric sphere spectrum along the point inclusion $x\colon\ast \to X$, we get 
\[
\underline{\mathrm{Sp}}^\Sigma_X\big(x_! \mathbb{S}, P\big) \cong x^\ast P
\quad
\mbox{ and }
\quad
\underline{H\mathbb{Q}\mathrm{-Mod}}_X\big(H\mathbb{Q}\wedge x_! \mathbb{S}, P_\mathbb{Q}\big)
\cong x^\ast P_\mathbb{Q}
\]
using the $\mathrm{Sp}^\Sigma$-Quillen adjunction $(x_!\dashv x^\ast)\colon \mathrm{Sp}^\Sigma\to \mathrm{Sp}^\Sigma_X$.
In this case, $\kappa_{A,P}$ is isomorphic to $
\pi^\mathrm{st}_\ast x^\ast P\otimes_\mathbb{Z}\mathbb{Q}\to \pi^\mathrm{st}_\ast x^\ast P_\mathbb{Q}
$. Since there is a stable weak equivalence $H\mathbb{Q}\wedge^\mathbf{L} x^\ast P\to x^\ast P_\mathbb{Q}$ (cf~\cite[Lemma 2.69]{braunack-mayer_combinatorial_2019}), Proposition \ref{prop:RatisSmashing} implies that this map is a natural isomorphism.

Pick a section $i\mapsto x_i$ for the projection $X\to \pi_0(X)$ so that $\{(x_i)_!\mathbb{S}\}$ is a set of compact generators for $Ho(\mathrm{Sp}^\Sigma_X)$ (cf~\cite[Remark 2.18]{braunack-mayer_combinatorial_2019}).
The localising triangulated subcategory $\mathcal{E}_P$ contains $x_! \mathbb{S}$ for any $x\colon \ast \to X$, from which it follows that $\mathcal{E}_P$ is the whole homotopy category.
As the fibrant object $P$ was arbitrary, the result is proven. 
\end{proof}
\begin{corollary}
For any simplicial set $X$, the derived adjunction 
\[
\begin{tikzcd}
Ho(\mathrm{Sp}^\Sigma_X)
\ar[rr, shift left=1.1ex, "{\mathbf{L}(H\mathbb{Q}\owedge_X(-))}"]
 \ar[rr, leftarrow, shift left=-1.1ex, "\bot"]
&&
Ho(H\mathbb{Q}\mathrm{-Mod}_X)
\end{tikzcd}
\]
is the localisation of $Ho(\mathrm{Sp}^\Sigma_X)$ at the class of rational stable homotopy equivalences.
\end{corollary}
\begin{remark}
The proof of Lemma \ref{lem:XSpecisSmashinglyRat} provides an important consistency check; a morphism of symmetric $X$-spectra $f\colon A\to B$ is a rational stable equivalence if and only if the induced maps 
$
\mathbf{R}x^\ast (f)\colon \mathbf{R} x^\ast (A)\to \mathbf{R} x^\ast (B)
$
of stable homotopy fibres are rational stable equivalences for all $x\colon \ast \to X$.
\end{remark}

\begin{lemma}
\label{lem:GmodSmashRat}
$G_+\mathrm{-Mod}$ is smashingly rational for any simplicial group $G$.
\end{lemma}
\begin{proof}
Fix a fibrant object $M\in G_+\mathrm{-Mod}$ and  a fibrant replacement $H\mathbb{Q}\wedge M \to M_\mathbb{Q}$ of $(G_+,H\mathbb{Q})$-bimodule spectra .
Let $\mathcal{E}_M$ be the localising triangulated subcategory of $Ho(G_+\mathrm{-Mod})$ of $G_+$-module spectra $N$ for which 
\[
\kappa_{N,M} \colon \{N, M\}^\ast_{G_+\mathrm{-Mod}}\otimes_\mathbb{Z}\mathbb{Q}\longrightarrow 
\{H\mathbb{Q}\wedge N, M_\mathbb{Q}\}_{(G_+, H\mathbb{Q})\mathrm{-Bimod}}^\ast
\]
is a natural isomorphism.
For $N = \Sigma^\infty G_+$, $\kappa_{N,M}$ takes the form $\pi^\mathrm{st}_\ast M\otimes_\mathbb{Z}\mathbb{Q}\to \pi^\mathrm{st}_\ast M_\mathbb{Q}$ hence is an isomorphism by Proposition \ref{prop:RatisSmashing}.
Since $\Sigma^\infty G_+$ is a compact generator of $Ho(G_+\mathrm{-Mod})$, it follows that $\mathcal{E}_M = Ho(G_+\mathrm{-Mod})$.
As the fibrant object $M$ was arbitrary, this shows that $G_+\mathrm{-Mod}$ is smashingly rational.
\end{proof}
\begin{corollary}
For any simplicial group $G$, the derived adjunction
\[
\begin{tikzcd}
Ho(G_+\mathrm{-Mod})
\ar[rr, shift left=1.1ex, "{\mathbf{L}(H\mathbb{Q}\wedge(-))}"]
\ar[rr, shift left=-1.1ex, "\bot", leftarrow]
&&
Ho((G_+,H\mathbb{Q})\mathrm{-Bimod})
\end{tikzcd}
\]
is the localisation of $Ho(G_+\mathrm{-Mod})$ at the class of rational stable homotopy equivalences.
\end{corollary}

\begin{lemma}
For a morphism of simplicial groups $\psi\colon G\to K$, the induced $\mathrm{Sp}^\Sigma$-Quillen adjunction
\[
\begin{tikzcd}
(G_+, H\mathbb{Q})\mathrm{-Bimod} 
\ar[rr, shift left=1.1ex, "\psi_!"]
\ar[rr, shift left=-1.1ex, "\psi^\ast"', "\bot", leftarrow]
&&
(K_+, H\mathbb{Q})\mathrm{-Bimod} 
\end{tikzcd}
\]
is a Quillen equivalence if and only if $\psi$ is an isomorphism on rational homology.
\end{lemma}
\begin{proof}
The  derived unit at $G_+\wedge H\mathbb{Q}$ is
$
G_+\wedge H\mathbb{Q}\to  K_+ \wedge \mathbb{Q} \to \big(K_+ \wedge H\mathbb{Q}\big)^f$,
which becomes $H_\bullet(\psi)\colon H_\bullet(G;\mathbb{Q})\to H_\bullet(K;\mathbb{Q})$ under passing to stable homotopy groups.
The \lq\lq only if'' part of the claim follows.

Conversely, if $\psi$ is a rational homology isomorphism then the derived unit is a natural isomorphism; this follows from an abstract argument using that $Ho\big((G_+,H\mathbb{Q})\mathrm{-Bimod}\big)$ is triangulated with compact generator $G_+\wedge H\mathbb{Q}$.
The right adjoint $\psi^\ast$ preserves and reflects weak equivalences and fibrations, hence $(\psi_!\dashv \psi^\ast)$ is a Quillen equivalence by \cite[Lemma 4.1.7]{hovey_symmetric_2000}.
\end{proof}

\begin{lemma}
\label{lem:RatParamSpec}
For any reduced simplicial set $X$, the $\mathrm{Sp}^\Sigma$-Quillen adjunction
\[
\begin{tikzcd}
H\mathbb{Q}\mathrm{-Mod}_X
\ar[rr, shift left=1.1ex, "\mathfrak{f}_X"]
\ar[rr, shift left=-1.1ex, "\mathfrak{b}_X"', "\bot", leftarrow]
&&
(\mathbb{G}X_+, H\mathbb{Q})\mathrm{-Bimod} 
\end{tikzcd}
\]
is a Quillen equivalence.
For any map $f\colon X\to Y$ of reduced simplicial sets, there is a natural isomorphism of left $\mathrm{Sp}^\Sigma$-Quillen functors $\mathbb{G}f_!\circ \mathfrak{f}_X\cong \mathfrak{f}_Y\circ f_!$.
\end{lemma}
\begin{proof}
Since the forgetful functors $H\mathbb{Q}\mathrm{-Mod}_X\to \mathrm{Sp}^\Sigma_X$ and $(\mathbb{G}X_+,H\mathbb{Q})\mathrm{-Bimod}\to \mathbb{G}X_+\mathrm{-Mod}$ preserve fibrations, cofibrations and weak equivalences the result follows from  Theorem \ref{thm:ParamSpecandLoop}.
\end{proof}

We conclude this section with a general result on smashing rationalisations.
Let $\mathcal{M}$ be a $\mathrm{Sp}^\Sigma$-model category with a set $\mathcal{G}$ of compact generators, which may be taken to be cofibrant-fibrant without loss of generality.
Let $\mathcal{E}(\mathcal{G})$ denote the spectral endomorphism category of $\mathcal{G}$, with objects $G\in\mathcal{G}$ and mapping spectra $\mathcal{E}(G,G) =\underline{\mathcal{M}}(G,G')$.
Schwede and Shipley prove a $\mathrm{Sp}^\Sigma$-Quillen equivalence
\[
\begin{tikzcd}
\mathcal{M}
\ar[rr,leftarrow, shift left=1.1ex, ""]
\ar[rr, shift left=-1.1ex, "\bot", ""]
&&
\mathcal{E}(\mathcal{G})\mathrm{-Mod}
\end{tikzcd}
\]
between $\mathcal{M}$ and the category of $\mathcal{E}(\mathcal{G})$-module spectra \cite[Theorem 3.9.3]{schwede_stable_2003}.
The proof of Lemma \ref{lem:GmodSmashRat} is readily adapted to show that $\mathcal{E}(\mathcal{G})\mathrm{-Mod}$ is smashingly rational and since the property of being smashingly rational is clearly preserved by $\mathrm{Sp}^\Sigma$-Quillen equivalences we have essentially proven the following
\begin{theorem}
\label{thm:RatSmash}
Let $\mathcal{M}$ be a $\mathrm{Sp}^\Sigma$-model category with a set $\mathcal{G}$ of compact generators.
Then $\mathcal{M}$ is smashingly rational.
In particular, the derived functor $Ho(\mathcal{M})\to Ho(T_\mathbb{Q}\mathrm{-Alg}(\mathcal{M}))$ is the localisation of $Ho(\mathcal{M})$ at the class of rational stable homotopy equivalences.
\end{theorem}
\begin{remark}
By the main result of \cite{schwede_stable_2003}, if $\mathcal{M}_0$ is a stable simplicial, cofibrantly generated, proper model category with a set of compact generators $\mathcal{G}$ then there is a chain of simplicial Quillen equivalences between $\mathcal{M}_0$ and $\mathcal{E}(\mathcal{G})\mathrm{-Mod}$.
The hypotheses on $\mathcal{M}_0$ are required in order to apply the symmetric stabilisation machine, for which the suspension spectrum functor $\Sigma^\infty\colon \mathcal{M}_0\to \mathrm{Sp}^\Sigma(\mathcal{M}_0)$ is a Quillen equivalence.
Theorem \ref{thm:RatSmash} then applies to $\mathrm{Sp}^\Sigma(\mathcal{M}_0)$.
\end{remark}

\subsection{Fibrewise smash products}
The fibrewise smash product of retractive spaces defines a symmetric monoidal structure on symmetric $X$-spectra.
However, unless $X=\ast$ this monoidal structure is not a Quillen bifunctor so does not readily descend to a monoidal structure on the homotopy category (cf~\cite[Remark 1.11]{braunack-mayer_combinatorial_2019}).

In this section, we show that both $G_+\mathrm{-Mod}$ and its rationalisation carry symmetric monoidal model structures.
In the case that $G=\mathbb{G}X$ for a reduced simplicial set $X$, we show that this enables us to present the fibrewise smash product of $X$-spectra via the Quillen equivalences of Theorem \ref{thm:ParamSpecandLoop} and Lemma \ref{lem:RatParamSpec}.
\begin{lemma}
\label{lem:GSpaceMonoidal}
For any simplicial group $G$, $G\mathrm{-sSet}$ is a symmetric monoidal model category with respect to the product of $G$-spaces, equipped with the diagonal action.
\end{lemma}
\begin{proof}
We expect this result is well-known but were unable to find a proof in the literature.
It suffices to verify that the product bifunctor $(M,N)\mapsto M\times N$ satisfies the unit and pushout-product axioms.

To verify the pushout-product axiom, we use that
\[
\mathcal{I}_G :=\{G\times i_n\mid i_n\colon \partial \Delta^n \hookrightarrow \Delta^n,\;n\geq 0\, \}
\qquad
\mbox{ and }
\qquad
\mathcal{J}_G :=\{G\times h^n_k\mid i_n\colon \Lambda^n_k \hookrightarrow \Delta^n,\;n\geq k\geq  0\, \}
\]
are, respectively, sets of generating cofibrations and generating acyclic cofibrations for $G\mathrm{-sSet}$.
Writing $\Delta_G = (G\times G)\big/ G$ for the quotient by the diagonal action, there is an isomorphism of $G$-spaces $(G\times K)\times (G\times L)\cong G\times \Delta_G \times K\times L$ for any $K,L \in \mathrm{sSet}$.
In particular, the product of free $G$-spaces is free.
Fixing $n,m\geq 0$, by using the standard description of the non-degenerate simplices of $\Delta^m \times \Delta^n$ in terms of $(m,n)$-shuffles we can write the pushout-product
\[
i_m\,\square\, i_n\colon (\partial\Delta^m \times \Delta^n)\coprod_{(\partial\Delta^m\times \partial\Delta^n)} (\Delta^m\times \partial\Delta^n)\longrightarrow \Delta^m \times \Delta^n
\]
as a finite sequence of pushouts along maps $i_k\colon \partial\Delta^k\to \Delta^k$.
Combining these facts, we obtain a presentation of the pushout-product 
$
(G\times i_m)\,\square\, (G\times i_n)$ as a finite sequence of pushouts along the cofibrations $G\times \Delta_G \times i_k \colon G\times \Delta_G \times \partial\Delta^k \to G\times \Delta_G \times \Delta^k$ of free $G$-spaces.
This shows that the set $\mathcal{I}_G\,\square\, \mathcal{I}_G$ consists of cofibrations, hence the pushout-product of any pair of cofibrations is a cofibration.
Forming products of $G$-spaces preserves weak equivalences in each variable, so that for any $i= (G\times i_m)\in\mathcal{I}_G$ and $j= (G\times h^n_k)\in \mathcal{J}_G$ the commuting diagram of cofibrations
\begin{equation}
\label{eqn:PPAcycDiag}
\begin{tikzcd}
(G\times \partial\Delta^m)\times (G\times \Lambda^n_k)
\ar[r, rightarrowtail]
\ar[d, rightarrowtail, "\sim"']
\ar[dr, phantom, "\lrcorner" very near end]
&
(G\times \partial\Delta^m)\times (G\times \Lambda^n_k)
\ar[d, rightarrowtail, "\sim"']
\ar[dr, rightarrowtail, bend left= 15, "\sim"]
\\
(G\times \partial\Delta^m)\times (G\times \Lambda^n_k)
\ar[r, rightarrowtail]
&
{\text{(pushout)}}
\ar[r, rightarrowtail, "i\,\square\, j"]
&
(G\times \partial\Delta^m)\times (G\times \Lambda^n_k)
\end{tikzcd}
\end{equation}
has weak equivalences as marked.
Hence $i\,\square\, j$ is a weak equivalence by the 2-out-of-3 property.
By symmetry and cofibrant generation of the model structure, this is sufficient to prove the pushout-product axiom.

Finally, to verify the unit axiom we observe that $WG\to \ast$ is a cofibrant resolution in $G\mathrm{-sSet}$ and that the map $WG\times N\to \ast \times N\cong N$ is a weak equivalence for any $G$-space $N$.
\end{proof}
\begin{corollary}
For any simplicial group $G$, $G_+\mathrm{-Mod}_u$ is a symmetric monoidal model category with respect to the smash product of pointed $G$-spaces, equipped with the diagonal action.
\end{corollary}
\begin{proof}
The category $G_+\mathrm{-Mod}_u$ is isomorphic to the category of pointed objects in $G\mathrm{-sSet}$, and adjoining a $G$-fixed point defines a left Quillen functor $(-)_+\colon G\mathrm{-sSet}\to G_+\mathrm{-Mod}_u$.
Since $(-)_+$ strong symmetric monoidal and $(\mathcal{I}_G)_+$ and $(\mathcal{J}_G)_+$ are sets of generating cofibrations and generating acyclic cofibrations for $G_+\mathrm{-Mod}_u$, the pushout-product axiom for the smash product of pointed $G$-spaces follows from the Lemma.
We also have that $WG_+ \to S^0$ is a cofibrant replacement of the monoidal unit in $G_+\mathrm{-Mod}_u$.
The smash product of pointed simplicial sets preserves weak equivalences in both arguments, hence $WG_+\wedge N \to S^0\wedge N\cong N$ is a weak equivalence for all $N\in G_+\mathrm{-Mod}_u$, establishing the unit axiom.
\end{proof}

\begin{corollary}
\label{cor:GModSpecMonoidal}
For any simplicial group $G$, $G_+\mathrm{-Mod}$ is a symmetric monoidal model category with respect to the smash product of $G_+$-module spectra, equipped with the diagonal $G_+$-action.
\end{corollary}
\begin{proof}
We first verify the pushout-product axiom for the smash product of $G_+$-module spectra.
For each $k\geq 0$ there is an adjunction $(\Sigma^{\infty-k}\dashv \widetilde{\Omega}^{\infty-k})\colon
G_+\mathrm{-Mod}_u\to 
G_+\mathrm{-Mod}$ arising from the symmetric stabilisation machine (these adjunctions are denoted $(F_n\dashv \mathrm{Ev}_n)$ in \cite{hovey_spectra_2001}).
These adjunctions jointly produce the set  $\mathcal{I}_G^\Sigma := \bigcup_{k\geq 0} \Sigma^{\infty-k} (\mathcal{I}_G)_+$ of generating cofibrations of $G_+\mathrm{-Mod}$ and for $k, l\geq 0$ there is a natural isomorphism of bifunctors 
\[
\Sigma^{\infty-k} (-)\wedge \Sigma^{\infty-l} (-) \cong \Sigma^{\infty-(k+l)}((-)\wedge (-))\colon G_+\mathrm{-Mod}_u\times G_+\mathrm{-Mod}_u\longrightarrow G_+\mathrm{-Mod}\,,
\]
which, in light of the previous result, implies that the set $\mathcal{I}_G^\Sigma\,\square\,\mathcal{I}_G^\Sigma$ consists of cofibrations.
As $G_+\mathrm{-Mod}$ is cofibrantly generated, it follows that the pushout-product of any pair of cofibrations is again a cofibration.

Identifying $G_+\mathrm{-Mod}$ with the category of $\Sigma^\infty G_+$-modules in $\mathrm{Sp}^\Sigma$ (Remark \ref{rem:GModSpecvsGSpecInMod}), there is a natural functor $U\colon G_+\mathrm{-Mod}\to \mathrm{Sp}^\Sigma$ forgetting the stable $G_+$-actions.
The forgetful functor $U$ is strictly monoidal and creates stable weak equivalences and fibrations.
Furthermore, as $\Sigma^\infty G_+$ is cofibrant the functor $U$ preserves cofibrations.
Hence for $i$ and $j$ respectively a cofibration and acyclic cofibration in $G_+\mathrm{-Mod}$, the morphism $U(i\,\square\, j) = U(i)\, \square\, U(j)$ is an acyclic cofibration of symmetric spectra.
Thus $i\,\square\, j$ is an acyclic cofibration in $G_+\mathrm{-Mod}$, which completes the proof of the pushout-product axiom.

To prove the unit axiom, we use that $\Sigma^\infty WG_+\to \Sigma^\infty S^0$  is a cofibrant resolution of the monoidal unit in $G_+\mathrm{-Mod}$, so in particular $\Sigma^\infty WG_+$ is cofibrant when considered as an object of $\mathrm{Sp}^\Sigma$.
For any cofibrant $N\in G_+\mathrm{-Mod}$, the comparison morphism $\Sigma^\infty WG_+\wedge N\to \Sigma^\infty S^0\wedge N\cong N$ is a stable weak equivalence by \cite[Corollary 5.3.10]{hovey_symmetric_2000}.
\end{proof}

\begin{remark}
\label{rem:FibSmashviaModules}
For a reduced simplicial set $X$, the symmetric monoidal  structure on $\mathbb{G}X_+\mathrm{-Mod}$ models the fibrewise smash product of $X$-parametrised spectra;
this follows from Remark \ref{rem:FibSpec} combined with the fact that the forgetful functor $\mathbb{G}X_+\mathrm{-Mod}\to \mathrm{Sp}^\Sigma$ is strictly monoidal.
The homotopically correct fibrewise smash product of symmetric $X$-spectra is thus presented by the bifunctor
$
(A,B) \mapsto \mathbf{R}\mathfrak{b}_X \big(\mathbf{L}\mathfrak{f}_X(A)\wedge \mathbf{L}\mathfrak{f}_X(B)\big)
$.
\end{remark}

Fix a simplicial group $G$ and a cofibrant commutative symmetric ring spectrum $R$.
Taking coequalisers of $R$-module structure maps gives rise to a bifunctor
\[
(M, N) \longmapsto
M\wedge_R N 
:= \mathrm{coeq}
\left(\!
\begin{tikzcd}
M\wedge R\wedge N
\ar[r, shift left=0.75ex]
\ar[r, shift left=-0.75ex]
&
M\wedge N
\end{tikzcd}\!
\right).
\]
As $R$ is commutative, $M\wedge_R N$ inherits a canonical $R$-module structure which commutes with the diagonal $G_+$-action.
This bifunctor defines a closed symmetric monoidal structure on the category of $(G_+, R)$-bimodules in symmetric spectra.
\begin{proposition}
For any simplicial group $G$ and cofibrant commutative symmetric ring spectrum $R$, the category $(G_+, R)\mathrm{-Bimod}$ of bimodule spectra is a symmetric monoidal model category with respect to the bifunctor $(M,N)\mapsto M\wedge_R N$.
\end{proposition}
\begin{proof}
Since $R$ is cofibrant, the forgetful functor $(G_+,R)\mathrm{-Bimod}\to \mathrm{Sp}^\Sigma$ preserves cofibrations.
Smash products with cofibrant symmetric spectra preserve stable weak equivalences, so the proof of Corollary \ref{cor:GModSpecMonoidal} applies \emph{mutatis mutandis}.
\end{proof}
\begin{remark}
\label{rem:FibrewiseRatSmash}
Setting $R=H\mathbb{Q}$, the last result implies that the rationalisation of $G_+\mathrm{-Mod}$ is a symmetric monoidal model category.
For any reduced simplicial set $X$ the symmetric monoidal model structure on $(\mathbb{G}X_+,H\mathbb{Q})\mathrm{-Bimod}$ models the fibrewise smash product in the rational homotopy category of $X$-parametrised spectra (cf~Remark \ref{rem:FibSmashviaModules}).
\end{remark}

\section{Rectification of rational bimodule spectra}
\label{sec:Rect}
In this section we prove the first first key result of our rectification program, Theorem \ref{thm:StrictGrpRing}.
For any simplicial group $G$, this result provides a Quillen equivalence between the model category of symmetric $(G_+, H\mathbb{Q})$-bimodule spectra considered previously and the model category of symmetric spectra in simplicial rational vector spaces equipped with a $\mathbb{Q}[G]$-action (discussed below).

We begin by recalling the definition and basic properties of symmetric spectra in simplicial rational vector spaces.
\begin{definition}
A \emph{symmetric spectrum in $\mathrm{sVect}_\mathbb{Q}$} is a sequence of pairs $V= \{(V(n), \sigma_n)\}_{n\geq 0}$, where $V(n)$ is a simplicial rational vector space with a (left) $\Sigma_n$-action, each $\sigma_n \colon \widetilde{\mathbb{Q}}[S^1]\otimes V(n)\to V(n+1)$ is a 
$\Sigma_n$-equivariant map of simplicial rational vector spaces, and for all $n,m$ the composite
\[
\widetilde{\mathbb{Q}}[S^n]\otimes V(m)
\xrightarrow{\;\;
\widetilde{\mathbb{Q}}[S^{n-1}]\otimes \sigma_m
\;\;} 
\widetilde{\mathbb{Q}}[S^{n-1}]\otimes V(m+1)
\xrightarrow{\;\;
\widetilde{\mathbb{Q}}[S^{n-2}]\otimes \sigma_{m+1}
\;\;} 
\dotsb
\xrightarrow{\;\; \sigma_{m+n-1}
\;\;} 
V(m+n)
\]
is $(\Sigma_n \times \Sigma_m)$-equivariant.
A morphism $V\to W$ in $\mathrm{Sp}^\Sigma(\mathrm{sVect}_\mathbb{Q})$ is a sequence of equivariant morphisms $f_n\colon V_n \to W_n$ commuting with the structure maps.
The category of symmetric spectra in simplicial rational vector spaces is denoted $\mathrm{Sp}^\Sigma(\mathrm{sVect}_\mathbb{Q})$.
\end{definition}

\begin{remark}
\label{rem:SymSeqinsVectQ}
The category $\mathrm{Sp}^\Sigma(\mathrm{sVect}_\mathbb{Q})$ is obtained from $\mathrm{sVect}_\mathbb{Q}$ by using Hovey's symmetric stabilisation machine \cite{hovey_spectra_2001}.
The stabilisation is performed so as to invert the left Quillen endofunctor
\begin{equation}
\label{eqn:RatSuspend}
\Sigma_\mathbb{Q} \colon V\longmapsto \widetilde{\mathbb{Q}}[S^1]\otimes V
\end{equation}
computing $\mathrm{sSet}_\ast$-tensors with the simplicial circle (which models the suspension endofunctor on $Ho(\mathrm{sVect}_\mathbb{Q})$.

We recall some details of the symmetric stabilisation procedure that we require in the sequel.
Let $\mathbf{\Sigma}$ be the symmetric monoidal groupoid with objects $n$ for $n\geq 0$ and morphisms $\mathbf{\Sigma}(n,n) = \Sigma_n$, with no morphisms between $m$ and $n$ for $m\neq n$.
The symmetric monoidal structure is given on objects by $(n,m)\mapsto n+m$ and is extended to morphisms by means of the canonical inclusions $\Sigma_n\times \Sigma_m\hookrightarrow \Sigma_{n+m}$.

A \emph{symmetric sequence} of simplicial rational vector spaces is a functor $V\colon \mathbf{\Sigma}\to\mathrm{sVect}_{\mathbb{Q}}$; that is, a sequence $\{V(n)\}_{n\geq 0}$ of simplicial rational vector spaces such that $V(n)$ has a (left) $\Sigma_n$-action.
The \emph{Day convolution product} of symmetric sequences $W, V$ is the symmetric sequence 
\[
(W\otimes^\mathrm{Day} V )(n)
= \bigoplus_{p+q=n} (\Sigma_n)_+ \times_{(\Sigma_p\times \Sigma_q)_+} (W(p)\otimes V(q))\,.
\]
The Day convolution product furnishes the category of symmetric sequences with a symmetric monoidal structure, with respect to which the symmetric sequence
\[
H\mathbb{Q}\colon n\longmapsto \widetilde{\mathbb{Q}}[S^n] \cong \underbrace{\widetilde{\mathbb{Q}}[S^1]\otimes \dotsb \otimes \widetilde{\mathbb{Q}}[S^1]}_{\text{$n$ factors}}
\]
is a commutative monoid.
The category $\mathrm{Sp}^\Sigma(\mathrm{sVect}_\mathbb{Q})$ is isomorphic to the category of $H\mathbb{Q}$-modules in $\mathrm{Fun}(\mathbf{\Sigma},\mathrm{sVect}_\mathbb{Q})$ (with respect to the Day convolution product).
Since $H\mathbb{Q}$ is a commutative monoid, $\mathrm{Sp}^\Sigma(\mathrm{sVect}_\mathbb{Q})$ inherits a symmetric monoidal structure; for $V, W\in \mathrm{Sp}^\Sigma(\mathrm{sVect}_\mathbb{Q})$ their \emph{smash-tensor product} is 
\begin{equation}
\label{eqn:SmashTensor}
V\otimes W := W\otimes^\mathrm{Day}_{H\mathbb{Q}} V
=
\mathrm{colim}
\left(
\!
\begin{tikzcd}
W\otimes^\mathrm{Day} H\mathbb{Q}\otimes^\mathrm{Day} V
\ar[r, shift left=0.75ex]
\ar[r, shift left=-0.75ex]
&
W\otimes^\mathrm{Day} V
\end{tikzcd}\!
\right)\,.
\end{equation}
The monoidal unit for the smash-tensor product is clearly $H\mathbb{Q}$, so that $\mathrm{Sp}^\Sigma(\mathrm{sVect}_\mathbb{Q})$ is canonically identified with the category of $H\mathbb{Q}$-modules in $\mathrm{Sp}^\Sigma(\mathrm{sVect}_\mathbb{Q})$.
\end{remark}

By general properties of the symmetric stabilisation machine, the adjunction \eqref{eqn:psSetQVect} prolongs to an adjunction of symmetric spectrum objects
\begin{equation}
\label{eqn:SpQVectSpec}
\begin{tikzcd}
\mathrm{Sp}^\Sigma
\ar[rr, shift left=1.1ex, "\widetilde{\mathbb{Q}}"]
\ar[rr, leftarrow, shift right=1.1ex, "U"', "\bot"]
&&
\mathrm{Sp}^\Sigma(\mathrm{sVect}_\mathbb{Q})
\end{tikzcd}
\end{equation}
by levelwise application.
There are now two obvious ways of equipping $\mathrm{Sp}^\Sigma(\mathrm{sVect}_\mathbb{Q})$ with a $\mathrm{Sp}^\Sigma$-model structure: directly via the symmetric stabilisation machine or alternatively by applying Kan's transfer theorem to the adjunction \eqref{eqn:SpQVectSpec}.
A minor adaptation of \cite[Proposition 4.1]{shipley_HZ_2007} shows that the stabilised and transferred model structures are identical.
An immediate consequence of this is the following
\begin{lemma}
\label{lem:ForgetCreate}
The forgetful functor $U\colon \mathrm{Sp}^\Sigma(\mathrm{sVect}_\mathbb{Q})\to \mathrm{Sp}^\Sigma$ creates fibrations and weak equivalences.
\end{lemma}
\begin{corollary}
\label{cor:AcycFibinSpSvect}
Acyclic fibrations in $\mathrm{Sp}^\Sigma(\mathrm{sVect}_\mathbb{Q})$ are levelwise weak equivalences.
\end{corollary}
\begin{proof}
Acyclic fibrations in $\mathrm{Sp}^\Sigma$ are levelwise weak equivalences \cite[Lemma 3.4.15]{hovey_symmetric_2000}.
\end{proof}

\begin{lemma}
\label{lem:RectMonoidal}
$\mathrm{Sp}^\Sigma(\mathrm{sVect}_\mathbb{Q})$ is a symmetric monoidal model category. The functor $\widetilde{\mathbb{Q}}\colon \mathrm{Sp}^\Sigma\to \mathrm{Sp}^\Sigma(\mathrm{sVect}_\mathbb{Q})$ is strongly monoidal left Quillen.
\end{lemma}
\begin{proof}
The levelwise tensor product of simplicial rational vector spaces makes $\mathrm{sVect}_\mathbb{Q}$ a symmetric monoidal model category. 
Upon forming symmetric spectrum objects, this gives rise to the symmetric monoidal model structure \eqref{eqn:SmashTensor} on $\mathrm{Sp}^\Sigma(\mathrm{sVect}_\mathbb{Q})$ (cf~\cite[Section 6]{hovey_spectra_2001}).
Since the rational chains functor is strongly monoidal left Quillen, so is its stabilisation.
\end{proof}

%\begin{remark}
%We shall call the symmetric monoidal model structure on $\mathrm{Sp}^\Sigma(\mathrm{sVect}_\mathbb{Q})$ the \emph{smash-tensor product}, which we denote by the symbol $\otimes$.
%For $V, W\in \mathrm{Sp}^\Sigma(\mathrm{sVect}_\mathbb{Q})$, their \emph{smash-tensor product} is 
%\[
%V\otimes W := W\otimes^\mathrm{Day}_{H\mathbb{Q}} V
%=
%\mathrm{colim}
%\left(
%\!
%\begin{tikzcd}
%W\otimes^\mathrm{Day} H\mathbb{Q}\otimes^\mathrm{Day} V
%\ar[r, shift left=0.75ex]
%\ar[r, shift left=-0.75ex]
%&
%W\otimes^\mathrm{Day} V
%\end{tikzcd}\!
%\right)\,.
%\] 
%Observe that the stabilised rational chains functor $\widetilde{\mathbb{Q}}\colon \mathrm{Sp}^\Sigma\to \mathrm{Sp}^\Sigma(\mathrm{sVect}_\mathbb{Q})$ sends $\mathbb{S}\mapsto H\mathbb{Q}$.
%\end{remark}
\begin{lemma}
\label{lem:SmashTensCofibObject}
If $V$ is a cofibrant object of $\mathrm{Sp}^\Sigma(\mathrm{sVect}_\mathbb{Q})$ the smash-tensor product $V\otimes (-)$ preserves weak equivalences.
\end{lemma}
\begin{proof}[Sketch of proof.]
The argument is modelled on \cite[Section 5.3]{hovey_symmetric_2000}.
%Any stable weak equivalence in $\mathrm{Sp}^\Sigma(\mathrm{sVect}_\mathbb{Q})$ factors as an acyclic cofibration followed by an acyclic fibration.
%Lemma \ref{lem:RectMonoidal} implies that $V\otimes (-)$ preserves acyclic cofibrations (this is a consequence of the pushout-product axiom for the smash-tensor product) so it suffices to show that acyclic fibrations are sent to stable weak equivalences.
%More precisely, we show that $V\otimes (-)$ preserves levelwise weak equivalences, which, by Corollary \ref{cor:AcycFibinSpSvect} and the fact that levelwise weak equivalences are stable weak equivalences, is sufficient to prove the result.
%
The Day convolution product with $H\mathbb{Q}$ defines a functor $\mathrm{Fun}(\mathbf{\Sigma},\mathrm{sVect})\to \mathrm{Sp}^\Sigma(\mathrm{sVect}_\mathbb{Q})$ (see Remark \ref{rem:SymSeqinsVectQ}).
Letting $\mathcal{I}$ denote the class of monomorphisms in $\mathrm{Fun}(\mathbf{\Sigma};\mathrm{sVect}_\mathbb{Q})$, the class of \emph{$H\mathbb{Q}$-cofibrations}
is $\mathrm{cof}(H\mathbb{Q}\otimes^\mathrm{Day} \mathcal{I})$; that is, the class of morphisms of $\mathrm{Sp}^\Sigma(\mathrm{sVect}_\mathbb{Q})$ obtained as retracts of transfinite compositions of pushouts of coproducts of elements in $H\mathbb{Q}\otimes^\mathrm{Day} \mathcal{I}$.

We claim that cofibrations in $\mathrm{Sp}^\Sigma(\mathrm{sVect}_\mathbb{Q})$ are $H\mathbb{Q}$-cofibrations.
To see this, observe that for each $k\geq 0$ we have a composite functor
\[
\begin{tikzcd}
\mathrm{sSet}_\ast
\ar[r, "\widetilde{\mathbb{Q}}"]
&
\mathrm{sVect}_\mathbb{Q}
\ar[r, "G_k"]
&
\mathrm{Fun}(\mathbf{\Sigma},\mathrm{sVect}_{\mathbb{Q}})\,,
\end{tikzcd}
\]
where $G_k$ is left adjoint to evaluation at $k\in \mathbf{\Sigma}$ and sends the rational simplicial vector space $A$ to the symmetric sequence given in degree $k$ by $(\Sigma_k)_+\wedge A$ and in all other degrees by $0$.
If $\mathcal{I}_\mathrm{Kan}$ denotes the standard set of generating cofibrations for $\mathrm{sSet}$, 
\[
H\mathbb{Q}\otimes^\mathrm{Day} \left(\bigcup_{k\geq 0} G_k \big(\widetilde{\mathbb{Q}}[\mathcal{I}_\mathrm{Kan}]\big)\right)
\]
comprises a set of generating cofibrations for $\mathrm{Sp}^\Sigma(\mathrm{sVect}_\mathbb{Q})$ (by general properties of the symmetric stabilisation machine, using that $\widetilde{\mathbb{Q}}[\mathcal{I}_\mathrm{Kan}]$ is a set of generating cofibrations for $\mathrm{sVect}_\mathbb{Q}$).
Each of the morphisms in $\bigcup_{k\geq 0} G_k \big(\widetilde{\mathbb{Q}}[\mathcal{I}_\mathrm{Kan}]\big)$ is a monomorphism so that cofibrations in $\mathrm{Sp}^\Sigma(\mathrm{sVect}_\mathbb{Q})$ are $H\mathbb{Q}$-cofibrations.

If $i\colon A\to B$ is a morphism in $\mathrm{Fun}(\mathbf{\Sigma},\mathrm{sVect}_\mathbb{Q})$ and $f\colon V\to W$ is a morphism in $\mathrm{Sp}^\Sigma(\mathrm{sVect}_\mathbb{Q})$, then
\begin{equation}
\label{eqn:PProd1}
(H\mathbb{Q}\otimes^\mathrm{Day} i)\, \square\, f = i\,\square^\mathrm{Day}\, f
\end{equation}
as maps of symmetric sequences, where the pushout-product on the left is taken with respect to the smash-tensor product and is taken on the right with respect to the Day convolution product.
If $i$ is a monomorphism and $f$ is a levelwise cofibration, the equation \eqref{eqn:PProd1} implies that $(H\mathbb{Q}\otimes^\mathrm{Day} i)\, \square\, f$ is a levelwise cofibration in $\mathrm{Sp}^\Sigma(\mathrm{sVect}_\mathbb{Q})$, which is moreover a levelwise acyclic cofibration if $f$ is.
An argument on adjoints (using that the monoidal structures $\otimes$ and $\otimes^\mathrm{Day}$ are both closed) shows that for any $H\mathbb{Q}$-cofibration $i$ and levelwise cofibration $j$, the pushout-product $i\,\square\, j$ is a levelwise cofibration which is moreover levelwise acyclic if $f$ is.

If $W$ is $H\mathbb{Q}$-cofibrant (that is, the initial morphism $0\to W$ is a $H\mathbb{Q}$-cofibration), the previous paragraph shows that the functor $W\otimes (-)$ preserves levelwise cofibrations and levelwise acyclic cofibrations in $\mathrm{Sp}^\Sigma(\mathrm{sVect}_\mathbb{Q})$.
This implies that $W\otimes (-)$ is a Quillen endofunctor for the \emph{injective} model structure on $\mathrm{Sp}^\Sigma(\mathrm{sVect}_\mathbb{Q})$ (for which the cofibrations and weak equivalences are levelwise).
As all objects of $\mathrm{Sp}^\Sigma(\mathrm{sVect}_\mathbb{Q})$ are injectively cofibrant, Ken Brown's lemma implies that $W\otimes (-)$ preserves \emph{all} levelwise weak equivalences for any $H\mathbb{Q}$-cofibrant $W$.

Suppose that $V$ is a cofibrant object of $\mathrm{Sp}^\Sigma(\mathrm{sVect}_\mathbb{Q})$.
Any weak equivalence $f\colon A\to B$ in $\mathrm{Sp}^\Sigma(\mathrm{sVect}_\mathbb{Q})$ factors as an acyclic cofibration $i$ followed by an acyclic fibration $p$.
Since $V$ is cofibrant, $V\otimes i$ is an acyclic cofibration by Lemma \ref{lem:RectMonoidal}.
Any cofibrant object is $H\mathbb{Q}$-cofibrant and $p$ is a levelwise equivalence by Corollary \ref{cor:AcycFibinSpSvect}, hence $V\otimes p$ is a levelwise weak equivalence by the above argument.
As levelwise weak equivalences are weak equivalences (for the stable model structure) the  composite $V\otimes f = (V\otimes p)\circ (V\otimes i)$ is a weak equivalence.
\end{proof}

For a symmetric spectrum object $V\in \mathrm{Sp}^\Sigma(\mathrm{sVect}_\mathbb{Q})$, the composite
\[
\begin{tikzcd}
H\mathbb{Q}\wedge V
\ar[r, "\eta_{H\mathbb{Q}\wedge V}"]
&
\widetilde{\mathbb{Q}}[H\mathbb{Q}\wedge V]
\cong
\widetilde{\mathbb{Q}}[H\mathbb{Q}]\otimes \widetilde{\mathbb{Q}}[V]
\ar[r, "\epsilon_{H\mathbb{Q}}\otimes \epsilon_V"]
&
H\mathbb{Q}\otimes V
\cong 
V
\end{tikzcd}
\]
furnishes the underlying symmetric spectrum with the structure of a $H\mathbb{Q}$-module.
The forgetful functor $\mathrm{Sp}^\Sigma(\mathrm{sVect}_\mathbb{Q})\to \mathrm{Sp}^\Sigma$ thus factors as
\[
\begin{tikzcd}
\mathrm{Sp}^\Sigma(\mathrm{sVect}_\mathbb{Q})
\ar[r, "\mathfrak{M}"]
&
H\mathbb{Q}\mathrm{-Mod}
\ar[r]
&
\mathrm{Sp}^\Sigma\,.
\end{tikzcd}
\]
The $(\widetilde{\mathbb{Q}}\dashv U)$-counit gives rise to a morphism $\zeta\colon \widetilde{\mathbb{Q}}[H\mathbb{Q}]\to H\mathbb{Q}$ of commutative monoids in $\mathrm{Sp}^\Sigma(\mathrm{sVect}_\mathbb{Q})$, and the rectification functor
\[
\mathfrak{R}\colon M \longmapsto \zeta_! (\widetilde{\mathbb{Q}}[M]) = H\mathbb{Q}\bigotimes_{\widetilde{\mathbb{Q}}[H\mathbb{Q}]} \widetilde{\mathbb{Q}}[M]
\]
is left adjoint to $\mathfrak{M}$.
The rectification functor $\mathfrak{R}$ is the composite
\[
\begin{tikzcd}
H\mathbb{Q}\mathrm{-Mod}
\ar[r, "\widetilde{\mathbb{Q}}"]
&
{\widetilde{\mathbb{Q}}[H\mathbb{Q}]\mathrm{-Mod}}
\ar[r, "\zeta_!"]
&
\mathrm{Sp}^\Sigma(\mathrm{sVect}_\mathbb{Q})\,,
\end{tikzcd}
\]
where $\zeta_!$ is extension of scalars along the morphism of commutative monoids $\zeta \colon \widetilde{\mathbb{Q}}[H\mathbb{Q}]\to H\mathbb{Q}$ and we have used that $\mathrm{Sp}^\Sigma(\mathrm{sVect}_\mathbb{Q})$ is isomorphic to the category of $H\mathbb{Q}$-modules in $\mathrm{Sp}^\Sigma(\mathrm{sVect}_\mathbb{Q})$ (cf~Remark \ref{rem:SymSeqinsVectQ}).
Since both $\widetilde{\mathbb{Q}}$ and $\zeta_!$ are strongly symmetric monoidal, so is $\mathfrak{R}$.
A useful consequence of this is that $\mathfrak{R}$ preserves $\mathrm{Sp}^\Sigma$-tensors.
\begin{lemma}
\label{lem:Rect1}
The rectification adjunction
\[
\begin{tikzcd}
H\mathbb{Q}\mathrm{-Mod}
\ar[rr, shift left=1.1ex, "\mathfrak{R}" ]
\ar[rr, shift left=-1.1ex, "\bot", "\mathfrak{M}"', leftarrow]
&&
\mathrm{Sp}^\Sigma(\mathrm{sVect}_\mathbb{Q})
\end{tikzcd}
\]
is a strongly symmetric monoidal $\mathrm{Sp}^\Sigma$-Quillen equivalence.
\end{lemma}
\begin{proof}
We have already observed that rectification $M\mapsto \mathfrak{R}M$ is a strongly symmetric monoidal functor, so in particular preserves $\mathrm{Sp}^\Sigma$-tensors.
Since the right adjoint $\mathfrak{M}$ creates fibrations and weak equivalences by Lemma \ref{lem:ForgetCreate}, it is sufficient to check that the unit $\eta_M \colon M\to \mathfrak{M}\mathfrak{R}(M)$ is a weak equivalence for all cofibrant $H\mathbb{Q}$-module spectra $M$.

To this end, let $\mathcal{E}\hookrightarrow Ho(\mathrm{Sp}^\Sigma(\mathrm{sVect}_\mathbb{Q}))$ be the full subcategory spanned by those cofibrant $M$ for which $\eta_M$ is a stable weak equivalence.
By stability, $Ho(\mathrm{Sp}^\Sigma(\mathrm{sVect}_\mathbb{Q}))$ is a triangulated category and the derived functors $\mathbf{L}\mathfrak{R}$ and $\mathbf{R}\mathfrak{M}$ are exact.
The functor computing stable homotopy groups of $H\mathbb{Q}$-module spectra is corepresented by $H\mathbb{Q}$, and since $\mathbf{L}\mathfrak{R}(H\mathbb{Q}) = H\mathbb{Q}$ we have natural isomorphisms of abelian groups
\begin{align*}
\Big[H\mathbb{Q}, (\mathbf{R}\mathfrak{M}\circ \mathbf{L}\mathfrak{R})\big(\bigoplus\!{\vphantom{\big)}}_{i\in \mathcal{I}}\, M_i\big)\Big]_{Ho(H\mathbb{Q}\mathrm{-Mod})}
&
\cong
\Big[H\mathbb{Q}, \mathbf{L}\mathfrak{R}\big(\bigoplus\!{\vphantom{\big)}}_{i\in \mathcal{I}}\, M_i\big)\Big]_{Ho(\mathrm{Sp}^\Sigma(\mathrm{sVect}_\mathbb{Q}))}
\\
&
\cong 
\bigoplus\!{\vphantom{\big)}}_{i\in \mathcal{I}}\,\big[H\mathbb{Q}, \mathbf{L}\mathfrak{R}( M_i)\big]_{Ho(\mathrm{Sp}^\Sigma(\mathrm{sVect}_\mathbb{Q}))}
\\
&
\cong 
\bigoplus\!{\vphantom{\big)}}_{i\in \mathcal{I}}\,\big[H\mathbb{Q}, (\mathbf{R}\mathfrak{M}\circ \mathbf{L}\mathfrak{R})( M_i)\big]_{Ho(H\mathbb{Q}\mathrm{-Mod})}
\\
&\cong
\Big[H\mathbb{Q}, \bigoplus\!{\vphantom{\big)}}_{i\in \mathcal{I}}\,(\mathbf{R}\mathfrak{M}\circ \mathbf{L}\mathfrak{R})\big( M_i\big)\Big]_{Ho(H\mathbb{Q}\mathrm{-Mod})}
\end{align*}
for any small set $\mathcal{I}$.
From this it follows that $\mathcal{E}$ is a localising triangulated subcategory.
Regarded as a module over itself, $H\mathbb{Q}$ is cofibrant and presents a compact generator of the homotopy category $Ho(H\mathbb{Q}\mathrm{-Mod})$.
The component of the adjunction unit at $H\mathbb{Q}$ is in fact an isomorphism, so that $H\mathbb{Q}\in \mathcal{E}$.
Since $\mathcal{E}$ is localising, we must have $\mathcal{E}=Ho(H\mathbb{Q}\mathrm{-Mod})$, which completes the proof.
\end{proof}

\begin{definition}
For a simplicial group $G$, $\Sigma^\infty G_+$ is a cofibrant symmetric ring spectrum and the $\mathrm{Sp}^\Sigma$-tensoring defines a monad
 $T_G = \Sigma^\infty G_+ \wedge(-) \cong \Sigma^\infty_\mathbb{Q} \mathbb{Q}[G]\otimes (-)$ on $\mathrm{Sp}^\Sigma(\mathrm{sVect}_\mathbb{Q})$.
Write $\mathbb{Q}[G]\mathrm{-Mod}$ for category of $T_G$-algebras, which is equipped with a $\mathrm{Sp}^\Sigma$-model structure via the free-forgetful adjunction
\[
\begin{tikzcd}
\mathrm{Sp}^\Sigma(\mathrm{sVect}_\mathbb{Q})
\ar[rr, shift left=1.1ex, "\Sigma^\infty G_+\wedge(-)"]
\ar[rr, shift left=-1.1ex, ""', "\bot", leftarrow]
&&
\mathbb{Q}[G]\mathrm{-Mod}\,.
\end{tikzcd}
\]
\end{definition}
\begin{remark}
Let $\mathbb{Q}[G]\mathrm{-Mod}_u$ denote the category of simplicial rational vector spaces equipped with an action of the rational group ring $\mathbb{Q}[G]$ of $G$.
A straightforward exercise in unravelling the definitions shows that the category $\mathbb{Q}[G]\mathrm{-Mod}$ is canonically isomorphic to the category of symmetric spectrum objects in $\mathbb{Q}[G]\mathrm{-Mod}$ with respect to the suspension endofunctor $\Sigma_\mathbb{Q}$ of \eqref{eqn:RatSuspend}.
\end{remark}

\begin{theorem}
\label{thm:StrictGrpRing}
For any simplicial group $G$ there is a $\mathrm{Sp}^\Sigma$-Quillen equivalence
\[
\begin{tikzcd}
(G_+,H\mathbb{Q})\mathrm{-Bimod}
\ar[rr, shift left=1.1ex, "\mathfrak{R}_G"]
\ar[rr, shift left=-1.1ex, "\mathfrak{M}_G"', "\bot", leftarrow]
&&
\mathbb{Q}[G]\mathrm{-Mod}\,.
\end{tikzcd}
\]
\end{theorem}
\begin{proof}
We identify the category of $(G_+,H\mathbb{Q})$-bimodule spectra with the category of $T_G$-algebras in $H\mathbb{Q}\mathrm{-Mod}$.
The adjunction $(\mathfrak{R}_G\dashv \mathfrak{M}_G)$ is thus obtained from the rectification adjunction $(\mathfrak{R}\dashv \mathfrak{M})$ by passing to categories of algebras over the monad $T_G= \Sigma^\infty G_+\wedge (-)$.
The functors $\mathfrak{R}_G$ and $\mathfrak{M}_G$ coincide with $\mathfrak{R}$ and $\mathfrak{M}$ respectively on underlying objects (that is, after forgetting $\Sigma^\infty G_+$-actions).
In particular, $\mathfrak{M}_G$ creates fibrations and weak equivalences so that the $(\mathfrak{R}_G\dashv \mathfrak{M}_G)$-adjunction is Quillen.

Fix a cofibrant object $B\in (G_+,H\mathbb{Q})\mathrm{-Bimod}$, noting that $B$ is also cofibrant when regarded as an object of $H\mathbb{Q}\mathrm{-Mod}$.
The map of $H\mathbb{Q}$-module spectra $B\to \mathfrak{M}\mathfrak{R}(B)$ is a weak equivalence per the proof of Lemma \ref{lem:Rect1}.
For an arbitrary $V\in \mathbb{Q}[G]\mathrm{-Mod}$ and morphism of $(G_+,H\mathbb{Q})$-bimodule spectra $\psi\colon B\to M_G (V)$, there is a commuting diagram of morphisms of $H\mathbb{Q}$-module spectra
\[
\begin{tikzcd}
B 
\ar[r, "\psi"]
\ar[d, "\eta_B"', "\sim"]
& 
\mathfrak{M}_G (V)
\\
\mathfrak{M}_G \mathfrak{R}_G (B) = \mathfrak{M}\mathfrak{R}(B)
\ar[ur, bend right=20, "\mathfrak{M}_G (\psi^\vee)"']
&
\end{tikzcd}
\]
with $\psi^\vee$ the $(\mathfrak{R}_G\dashv \mathfrak{M}_G)$-adjunct of $\psi$ and $\eta$ the adjunction unit.
$M_G$ creates weak equivalences, hence $\psi$ is a weak equivalence precisely if its adjunct $\psi^\vee$ is.
\end{proof}
\begin{remark}
For each morphism of simplicial groups $\phi\colon G\to K$, there is a base change  $\mathrm{Sp}^\Sigma$-Quillen adjunction 
$(\mathbb{Q}[\phi]_!\dashv \mathbb{Q}[\phi]^\ast)\colon \mathbb{Q}[G]\mathrm{-Mod}\to \mathbb{Q}[K]\mathrm{-Mod}$
obtained in the usual fashion.
The right adjoint $\mathbb{Q}[\phi]^\ast$, defined on objects by restriction along $\phi$, is the identity on underlying objects of $\mathrm{Sp}^\Sigma(\mathrm{sVect}_\mathbb{Q})$ so preserves and reflects weak equivalences and fibrations.
\end{remark}
\begin{lemma}
\label{lem:PseudoNatRect}
For any morphism of simplicial groups $\phi\colon G\to K$, there is a natural isomorphism of left $\mathrm{Sp}^\Sigma$-Quillen functors $\mathfrak{R}_K\circ \phi_!\cong \mathbb{Q}[\phi]_! \circ \mathfrak{R}_G$ from $(G_+,H\mathbb{Q})\mathrm{-Bimod}$ to $\mathbb{Q}[K]\mathrm{-Mod}$.
\end{lemma}
\begin{proof}
The reader can convince herself that there is a natural isomorphism $\mathfrak{M}_G\circ \mathbb{Q}[\phi]^\ast\cong \phi^\ast \circ \mathfrak{M}_K$, from which the result follows by uniqueness of adjoints.
\end{proof}

\subsection{Rectification and tensor products}
\label{ssec:RectandTens}
Any simplicial group $G$ may be regarded as a (coassociative cocommutative) Hopf algebra object in $(\mathrm{sSet}, \times)$ with  coproduct given by the diagonal map $\Delta\colon G\to G\times G$. 
The functor
\begin{align*}
(\mathrm{sSet},\times)&\longrightarrow (\mathrm{sVect}_\mathbb{Q},\otimes)\\
K&\longmapsto \mathbb{Q}[K] \cong \widetilde{\mathbb{Q}}[K_+]
\end{align*}
is strongly symmetric monoidal, so that in particular the rational group ring $\mathbb{Q}[G]$ is a Hopf algebra and we can equip the category of $\mathbb{Q}[G]$-modules with a symmetric monoidal structure by forming levelwise tensor products.
The goal of this section is to show that $\mathbb{Q}[G]\mathrm{-Mod}_u$ is a symmetric monoidal model category.
Applying the symmetric stabilisation machine, we can therefore equip $\mathbb{Q}[G]\mathrm{-Mod}$ with a symmetric monoidal model strucutre.
With respect to this monoidal structure, the rectification adjunction $(\mathfrak{R}_G\dashv \mathfrak{M}_G)\colon (G_+,H\mathbb{Q})\to \mathbb{Q}[G]\mathrm{-Mod}$ is a strongly monoidal Quillen equivalence.
We shall require the following
\begin{lemma}
\label{lem:FreeHopf}
Let $H$ be a Hopf algebra over $\mathbb{Q}$.
The tensor product of free left $H$-modules is free.
\end{lemma}
\begin{proof}
The tensor product $H\otimes H$ is a left Hopf module over $H$ with respect to the left $H$-action
$
h\cdot (a\otimes b) =  h_{(0)}\cdot a\otimes h_{(1)}\cdot b
$
and coaction 
$
\kappa \colon a\otimes b\mapsto a_{(0)}\otimes a_{(1)}\otimes b
$,
with $\Delta h =  h_{(0)}\otimes h_{(1)}$ in sumless Sweedler notation.
Writing $(H\otimes H)_H :=\{v\in H\otimes H\mid \kappa(v) =1\otimes v\}$ for the $H$-coinvariant submodule, the map
$
H\otimes (H\otimes H)_H \to H\otimes H
$, $
h\otimes v\mapsto h\cdot v
$,
is an isomorphism of left $H$-modules by the fundamental theorem of Hopf modules.
The tensor product of free $H$-modules $H\otimes V$ and $H\otimes W$ is thus
$
(H\otimes V)\otimes (H\otimes W)\cong H\otimes \left[ (H\otimes H)_H \otimes V\otimes W\right]
$, which is free on $(H\otimes H)_H \otimes V\otimes W$.
\end{proof}
\begin{corollary}
\label{cor:HopfPowers}
Let $H$ be a Hopf algebra over $\mathbb{Q}$.
Then $H^{\otimes n}$ is a free $H$-module for any $n\geq 1$.
\end{corollary}
\begin{proof}
This follows from the Lemma by an inductive argument.
Alternatively, $H^{\otimes n}$ is a left Hopf module over $H$ with respect to the coaction $\Delta\otimes \mathrm{id}$ and left $H$-action 
\[
h\cdot (a_1\otimes \dotsb \otimes a_n) =  h_{(0)}\cdot a_1\otimes \dotsb \otimes h_{(n-1)}\cdot a_n\,,
\]
so that $H^{\otimes n}\cong H\otimes H^{\otimes n}_H$ by the fundamental theorem of Hopf modules.
\end{proof}
\begin{corollary}
\label{cor:SimplicialHopfMod}
Let $H$ be a simplicial Hopf algebra over $\mathbb{Q}$.
Then the levelwise tensor product of free $H$-modules is free.
In particular, $H^{\otimes n}$ is a free $H$-module for any $n\geq 1$.
\end{corollary}
\begin{proof}
Apply Lemma \ref{lem:FreeHopf} and Corollary \ref{cor:HopfPowers} in each simplicial degree.
\end{proof}

Fixing a simplicial Hopf algebra $H$ over $\mathbb{Q}$, we write $H\mathrm{-Mod}_u$ for the category of $H$-modules in simplicial vector spaces.
Kan's transfer theorem applies to the free-forgetful adjunction
\begin{equation}
\label{eqn:uHmodAdj}
\begin{tikzcd}
\mathrm{sVect}_{\mathbb{Q}}
\ar[rr, "H\otimes (-)", shift left =1.1ex]
\ar[rr, "\bot", leftarrow, shift left =-1.1ex]
&&
H\mathrm{-Mod}_u\,,
\end{tikzcd}
\end{equation}
so that $H\mathrm{-Mod}_u$ is a combinatorial $\mathrm{sSet}_\ast$-model category.
Sets of generating cofibrations and generating acyclic cofibrations are provided by 
\[
\mathcal{I}_H :=\{H\otimes \mathbb{Q}[i_n]\mid i_n\colon \partial \Delta^n \hookrightarrow \Delta^n,\;n\geq 0\, \}
\quad
\mbox{ and }
\quad
\mathcal{J}_H :=\{H\otimes \mathbb{Q}[h^n_k]\mid i_n\colon \Lambda^n_k \hookrightarrow \Delta^n,\;n\geq k\geq  0\, \}
\]
respectively.
As all Hopf algebras we consider are assumed coassociative and cocommutative, the levelwise tensor product of simplicial vector spaces gives rise to a symmetric monoidal structure on $H\mathrm{-Mod}_u$.
That is, for $H$-modules $(M,\rho_M)$ and $(N,\rho_N)$, the levelwise tensor product $M\otimes N$ is regarded as a $H$-module via the action
\[
\begin{tikzcd}
H\otimes M\otimes N
\ar[r, "\Delta\otimes \mathrm{id}_{M\otimes N}"]
&
H\otimes H\otimes M\otimes N
\ar[r, "\cong"]
&
H\otimes M\otimes H\otimes N
\ar[r, "\rho_M\otimes \rho_N"]
&
M\otimes N\,.
\end{tikzcd}
\]
\begin{remark}
The symmetric monoidal structure on $H\mathrm{-Mod}_u$ is closed, with internal hom  objects given by the levelwise internal hom $\mathrm{Hom}(M,N)$ of simplicial vector spaces equipped with $H$-action
$
h\cdot \lambda \colon m\mapsto  h_{(0)}\cdot \lambda \big(Sh_{(1)}\cdot m\big)
$,
where $S$ is the antipode.
\end{remark}
\begin{lemma}
\label{lem:SimplHopfModMonoidal}
Let $H$ be a simplicial Hopf algebra over $\mathbb{Q}$.
Then $H\mathrm{-Mod}_u$ is a symmetric monoidal model category.
\end{lemma}
\begin{proof}
The unit axiom follows from the fact that the levelwise tensor product of simplicial rational vector spaces preserves weak equivalences in both variables.
Arguing on the sets $\mathcal{I}_H$ and $\mathcal{J}_H$, the pushout-product axiom following the same argument as in Lemma \ref{lem:GSpaceMonoidal}.
\end{proof}

\begin{construction}
\label{cons:HopfCofibUnit}
It will prove useful to a specfic cofibrant resolution of the monoidal unit of $H\mathrm{-Mod}_u$ in mind.
Iterated application of the comonad of the adjunction \eqref{eqn:uHmodAdj} to $\mathbb{Q}$ produces an augmented simplicial object
\[
\underline{\mathbb{Q}}_\bullet^H =
\left\{
\begin{tikzcd}
\dotsb 
\ar[r]
\ar[r,shift left=1.5ex]
\ar[r,shift left=-1.5ex]
&
H\otimes H
\ar[r,shift left=0.75ex]
\ar[r,shift left=-0.75ex]
&
H
\end{tikzcd}
\right\}
\xrightarrow{\;\;\varrho\;\;}
\mathbb{Q}\,,
\]
where face maps and degeneracies are given respectively by inserting the counit and comultiplication of $H$ in different tensor factors.
With reference to Corollary \ref{cor:SimplicialHopfMod} one shows that $\underline{\mathbb{Q}}_\bullet^H$ is Reedy cofibrant in $\mathrm{Fun}(\Delta^\mathrm{op}, H\mathrm{-Mod}_u)$, so that the realisation
\[
\mathbb{Q}^H := \underset{\Delta^\mathrm{op}}{\mathrm{colim}}\, \mathbb{Q}[\Delta^n]\otimes H^{\otimes (n+1)} \cong 
\underset{\Delta^\mathrm{op}}{\mathrm{hocolim}}\, H^{\otimes (n+1)}
\]
is a cofibrant object presenting the homotopy colimit
(this is the Bousfield--Kan formula for homotopy colimits, see eg.~\cite{hirschhorn_model_2003}).
The augmentation map  $\varrho$ then induces a map of $H$-modules $\rho\colon \mathbb{Q}^H\to \mathbb{Q}$ which exhibits a cofibrant resolution of the monoidal unit.
Indeed, after forgetting the $H$-actions, the simplicial object $\underline{\mathbb{Q}}^H_\bullet\in \mathrm{Fun}(\Delta^\mathrm{op},\mathrm{sVect}_\mathbb{Q})$ can be endowed with the system of extra degeneracies
\begin{align*}
s_{-1} \colon H^{\otimes (n+1)}&\longrightarrow H^{\otimes (n+2)}
\\
a_0\otimes \dotsb \otimes a_n &\longmapsto 1\otimes a_0\otimes \dotsb \otimes a_n
\end{align*}
defined by inserting the unit $\eta\colon \mathbb{Q}\to H$.
Using these extra degeneracies we construct a simplicial homotopy $\Delta^1\otimes \mathbb{Q}^H_\bullet\to \mathbb{Q}^H_\bullet$ between the identity and $s_{-1}\circ \varrho$, which after taking colimits exhibits $\rho\colon \mathbb{Q}^H\to \mathbb{Q}$ as a deformation retraction in $\mathrm{sVect}_\mathbb{Q}$.
Hence $\rho$ is a weak equivalence of $H$-modules with cofibrant domain.
\end{construction}

Write $H\mathrm{-Mod} := \mathrm{Sp}^\Sigma(H\mathrm{-Mod}_u)$ for the category of symmetric spectrum objects in $H\mathrm{-Mod}_u$ with respect to the suspension endofunctor $\Sigma_\mathbb{Q}$.
\begin{lemma}
\label{lem:SpectralHopfMod}
Let $H$ be a simplicial Hopf algebra over $\mathbb{Q}$.
Then $H\mathrm{-Mod}$ is a stable symmetric monoidal $\mathrm{Sp}^\Sigma$-model category.
\end{lemma}
\begin{proof}
Apply the symmetric stabilisation machine to Lemma \ref{lem:SimplHopfModMonoidal}.
\end{proof}
\begin{remark}
\label{rem:StabHopfModGenCofib}
Generating cofibrations and acyclic cofibrations for $H\mathrm{-Mod}$ are given by the sets
$
\mathcal{I}^\Sigma_H := \bigcup_{k\geq 0} \Sigma^{\infty-k}_\mathbb{Q} \big(\mathcal{I}_H\big)
$ and 
$\mathcal{J}^\Sigma_H := \bigcup_{k\geq 0} \Sigma^{\infty-k}_\mathbb{Q} \big(\mathcal{J}_H\big)
$
respectively, with $\Sigma^{\infty-k}_\mathbb{Q}\colon \mathrm{sVect}_\mathbb{Q}\to \mathrm{Sp}^\Sigma(\mathrm{sVect}_\mathbb{Q})$ the (shifted) symmetric stabilisation functor\footnote{Called $F_k$ in \cite{hovey_spectra_2001}.}.
\end{remark}
\begin{remark}
\label{rem:EquivHopfModels}
Let $H\in\mathrm{sVect}_\mathbb{Q}$ be any associative algebra object (not necessarily a Hopf algebra).
There is a canonical isomorphism of categories between $H\mathrm{-Mod}$ and the category $T_H\mathrm{-Alg}$ of algebras in $\mathrm{Sp}^\Sigma(\mathrm{sVect}_\mathbb{Q})$ for the monad $T_H := \Sigma^\infty_\mathbb{Q} H\otimes (-)$.
The category $T_H\mathrm{-Alg}$ is equipped with a model structure by transferring along the free-forgetful adjunction, which is identified with the stable model structure on $H\mathrm{-Mod}$ considered above.
\end{remark}

In the case that $H = \mathbb{Q}[G]$ is the rational group ring of a simplicial group $G$, Lemma \ref{lem:SpectralHopfMod} furnishes $\mathbb{Q}[G]\mathrm{-Mod}$ with a symmetric monoidal $\mathrm{Sp}^\Sigma$-model structure.
\begin{lemma}
\label{lem:RectGrpRing}
For any simplicial group $G$, the rectification Quillen equivalence 
\[
\begin{tikzcd}
(G_+,H\mathbb{Q})\mathrm{-Bimod}
\ar[rr, shift left=1.1ex, "\mathfrak{R}_G"]
\ar[rr, shift left=-1.1ex, "\mathfrak{M}_G"', "\bot", leftarrow]
&&
\mathbb{Q}[G]\mathrm{-Mod}
\end{tikzcd}
\]
is strongly symmetric monoidal.
\end{lemma}
\begin{proof}
There is a commuting diagram
\[
\begin{tikzcd}
(G_+,H\mathbb{Q})\mathrm{-Bimod} 
\ar[r,"\mathfrak{R}_G"]
\ar[d]
&
\mathbb{Q}[G]\mathrm{-Mod}
\ar[d]
\\
H\mathbb{Q}\mathrm{-Mod}
\ar[r,"\mathfrak{R}"]
&
\mathrm{Sp}^\Sigma (\mathrm{sVect})
\end{tikzcd}
\]
in which the vertical forgetful functors are strictly monoidal.
Since the rectification functor $\mathfrak{R}$ is strongly symmetric monoidal, it follows that $\mathfrak{R}_G$ is too.

To complete the proof we must show that for some (hence any) cofibrant resolution $C\to H\mathbb{Q}$ of the monoidal unit in $(G_+,H\mathbb{Q})\mathrm{-Bimod}$, the comparison morphism 
\[
\gamma_C\colon \mathfrak{R}_G(C)\longrightarrow \mathfrak{R}_G(H\mathbb{Q}) \cong \Sigma^\infty\mathbb{Q}
\]
is a weak equivalence of $\mathbb{Q}[G]$-module spectra.
Taking $C= \Sigma^\infty WG_+$, there is an isomorphism of $\mathbb{Q}[G]$-module spectra $\mathfrak{R}_G (\Sigma^\infty WG_+) \cong \Sigma^\infty_\mathbb{Q} \mathbb{Q}[WG]$ as $\mathbb{Q}[G]$-modules and the underlying morphism of $\gamma_{\Sigma^\infty WG_+}$
coincides with the image of the weak equivalence $WG_+ \to S^0$ under the composite left Quillen functor
\[
\begin{tikzcd}
\mathrm{sSet}_+
\ar[r, "\widetilde{\mathbb{Q}}"]
&
\mathrm{sVect}_\mathbb{Q}
\ar[r, "\Sigma^\infty_\mathbb{Q}"]
&
\mathrm{Sp}^\Sigma(\mathrm{sVect}_{\mathbb{Q}})\,,
\end{tikzcd}
\]
so is a weak equivalence by Ken Brown's lemma.
Since weak equivalences in $\mathbb{Q}[G]\mathrm{-Mod}$ are created by the forgetful functor to $\mathrm{Sp}^\Sigma(\mathrm{sVect}_\mathbb{Q})$, the comparison map $\gamma_{\Sigma^\infty WG_+}$ is a weak equivalence. 
\end{proof}

A morphism of simplicial Hopf algebras $\phi\colon H\to K$ gives rise to a base change adjunction between stable model categories of modules.
If $\phi$ is a weak equivalence, this base change adjunction induces a symmetric monoidal equivalence on homotopy categories.
\begin{lemma}
\label{lem:BaseChangeforSimplHopf}
Let $\phi\colon H\to K$ be a weak equivalence of simplicial Hopf algebras over $\mathbb{Q}$.
Then restriction and extension of scalars define a weakly monoidal $\mathrm{Sp}^\Sigma$-Quillen equivalence
\[
\begin{tikzcd}
H\mathrm{-Mod}
\ar[rr, shift left=1.1ex, "\phi_!"]
\ar[rr, shift left=-1.1ex, "\bot", "\phi^\ast"',leftarrow]
&&
K\mathrm{-Mod}\,.
\end{tikzcd}
\]
\end{lemma}
\begin{proof}
The restriction of scalars functor  $\phi^\ast$ is the identity on underlying objects, so preserves and reflects weak equivalences and fibrations.
The cofibrant object $\Sigma^\infty_\mathbb{Q} H\in H\mathrm{-Mod}$ presents a compact generator of the homotopy category, and the underlying morphism of the derived unit $\eta_{\Sigma^\infty_\mathbb{Q} H}\colon \Sigma^\infty_\mathbb{Q} H \to \phi^\ast \phi_! \Sigma^\infty_\mathbb{Q} H \cong \Sigma^\infty_\mathbb{Q} K$ coincides with the image of $\phi$ under the left Quillen functor $\Sigma^\infty_\mathbb{Q}\colon \mathrm{sVect}_\mathbb{Q}\to \mathrm{Sp}^\Sigma(\mathrm{sVect}_\mathbb{Q})$.
Since $H, K\in \mathrm{sVect}_\mathbb{Q}$ are cofibrant and $\phi$ is a weak equivalence by hypothesis, $\eta_{\Sigma^\infty_\mathbb{Q} H}$ is also a weak equivalence by Ken Brown's lemma.
The argument of Lemma \ref{lem:Rect1} now applies \emph{mutatis mutandis} to show that $(\phi_!\dashv \phi^\ast)$ is a Quillen equivalence.

We now turn to the monoidal properties of the adjunction.
The monoidal structures of $H\mathrm{-Mod}$ and $K\mathrm{-Mod}$ are created in $\mathrm{Sp}^\Sigma(\mathrm{sVect}_\mathbb{Q})$, so that the restriction of scalars functor $\phi^\ast$ is strictly symmetric monoidal.
Write $\nu \colon \phi_! (\Sigma^\infty_\mathbb{Q}\mathbb{Q})\to \Sigma^\infty_\mathbb{Q}\mathbb{Q}$ for the adjunct of the isomorphism of monoidal units $\mathrm{id}\colon \Sigma^\infty_\mathbb{Q}\mathbb{Q}\to \Sigma^\infty_\mathbb{Q}\mathbb{Q} \cong \phi^\ast \Sigma^\infty_\mathbb{Q}\mathbb{Q}$ and consider the natural map
\[
\Lambda_{M,N}\colon \phi_!(M\otimes N)\longrightarrow \phi_!(M)\otimes \phi_!(N)
\] 
adjunct to the composite
$
\begin{tikzcd}[cramped]
M\otimes N 
\ar[r, "\eta_M\otimes \eta_N"]
&
\phi^\ast \phi_! (M)\otimes \phi^\ast \phi_! (M)
=
\phi^\ast \big(\phi_! (M) \otimes \phi_! (N)\big)
\end{tikzcd}
$.
The maps $\nu$ and $\Lambda_{M,N}$ furnish $\phi_!$ with the structure of an oplax symmetric monoidal functor, and in order for $\phi_!$ to be a weakly monoidal Quillen equivalence we must show that:
\begin{enumerate}[label=(\alph*)]
  \item $\Lambda_{M,N}$ is a weak equivalence for all cofibrant $M$ and $N$; and
  \item for some (hence any) cofibrant replacement $C\to \Sigma^\infty_\mathbb{Q}\mathbb{Q}$ of the monoidal unit in $H\mathrm{-Mod}$, the comparison map 
  \[
  \gamma_{C}\colon \phi_! (C)\longrightarrow \phi_! (\Sigma^\infty_\mathbb{Q} \mathbb{Q})
  \xrightarrow{\;\;\nu\;\;}
  \Sigma^\infty_\mathbb{Q} \mathbb{Q}
  \]
  is a weak equivalence in $K\mathrm{-Mod}$.
\end{enumerate}
We prove (a) in a series of steps; first consider the case that $M$ and $N$ are free $H$-modules on shifted suspension spectra, that is
\[
M = \Sigma^{\infty-k}_\mathbb{Q}(H\otimes V)\cong \Sigma^\infty_\mathbb{Q} H\otimes \Sigma^{\infty-k}_\mathbb{Q}V
\qquad
\mbox{ and }
\qquad
N = \Sigma^{\infty-l}_\mathbb{Q}(H\otimes W)\cong \Sigma^\infty_\mathbb{Q} H\otimes \Sigma^{\infty-l}_\mathbb{Q}W
\]
for $k, l\geq 0$ and simplicial rational vector spaces $V$ and $W$.
By general properties of symmetric stabilisation and Corollary \ref{cor:SimplicialHopfMod} there are isomorphisms
\[
M\otimes N \cong \Sigma^{\infty-(k+l)}_\mathbb{Q} \big(H\otimes H \otimes V\otimes W\big)
\cong
\Sigma^\infty_\mathbb{Q} H\otimes \Sigma^{\infty-(k+l)}_\mathbb{Q} \big((H\otimes H)_H \otimes V\otimes W\big)
\]
in $H\mathrm{-Mod}$.
The weak equivalence $\phi\otimes \phi\colon H\otimes H\to K\otimes K$ induces a weak equivalence on coinvariants $\Phi\colon (H\otimes H)_H \to (K\otimes K)_K$, and the oplax structure map $\Lambda_{M,N}$ is isomorphic to the map of $K$-module spectra
\[
\Sigma^\infty_\mathbb{Q} K \otimes  \Sigma^{\infty-(k+l)}_\mathbb{Q} \big((H\otimes H)_H \otimes V\otimes W\big)
\longrightarrow
\Sigma^\infty_\mathbb{Q} K \otimes  \Sigma^{\infty-(k+l)}_\mathbb{Q} \big((K\otimes K)_K \otimes V\otimes W\big)
\]
induced by $\Phi$, so is a weak equivalence by Ken Brown's lemma.
Fixing $M = \Sigma^\infty_\mathbb{Q} H\otimes \Sigma^{\infty-k}_\mathbb{Q}V$, let ${}_M \mathcal{C}$ be the class of objects $B\in H\mathrm{-Mod}$ for which $\Lambda_{M,B}$ is a weak equivalence.
The class ${}_M \mathcal{C}$ is closed under forming retracts, and by above argument contains the domains and codomains of the set of generating cofibrations $\mathcal{I}^\Sigma_H$ (Remark \ref{rem:StabHopfModGenCofib}).
Moreover
\begin{itemize}
  \item ${}_M \mathcal{C}$ is closed under forming pushouts along maps in $\mathcal{I}^\Sigma_H$.
  Indeed, let $i\colon A\to B$ be such a map and consider the pushout 
  \[
  \begin{tikzcd}
  A \ar[r]
  \ar[d, rightarrowtail, "i"'] 
    \ar[dr, phantom, "\lrcorner", very near end]& C\ar[d] 
\\
  B\ar[r] & D
  \end{tikzcd}
  \]
  with $C\in {}_M\mathcal{C}$.
  The oplax structure maps give rise to a map of pushout diagrams
  \[
  \begin{tikzcd}[sep=small]
  & \phi_! (M) \otimes \phi_!(A) \ar[rr]\ar[dd, rightarrowtail] && \phi_! (M) \otimes \phi_!(C)
  \ar[dd]
  \\
  \phi_!(M\otimes A) 
  \ar[ur, "\sim"]
  \ar[rr, crossing over] 
  \ar[dd, rightarrowtail]
  &&
  \phi_!(M\otimes C)
  \ar[ur, "\sim" ] 
  &
  \\
  & \phi_! (M) \otimes \phi_!(B)\ar[rr]&& \phi_! (M) \otimes \phi_!(D)
  \\
  \phi_!(M\otimes B) \ar[rr]
  \ar[ur, "\sim"]&&
  \phi_!(M\otimes D) 
  \ar[ur]
  \ar[from=uu, crossing over]&
  \end{tikzcd}
  \]
  with cofibrations and weak equivalences as marked.
  All objects are cofibrant, so that the morphism on pushouts $\Lambda_{M,D}\colon \phi_! (M\otimes D)\to \phi_!(M)\otimes \phi_!(D)$ is a weak equivalence by Kan's cube lemma \cite[Lemma 5.2.6]{hovey_model_1999}. 
  
  \item For any cofibrant object $B\in {}_M \mathcal{C}$ and transfinite composition 
  \[
  B \longrightarrow
  B(1)
  \longrightarrow
  \dotsb 
  \longrightarrow 
  B(\kappa)
  \]
  of pushouts along maps in $\mathcal{I}^\Sigma_H$, the object $B(\kappa)\in {}_M\mathcal{C}$ by transfinite induction.
  For successor ordinals this follows from the above; if $\alpha_0$ is a limit ordinal then for $\alpha< \alpha_0$ we have weak equivalences of sequences of cofibrations
  \[
  \begin{tikzcd}
  \dotsb
  \ar[r, rightarrowtail]
  &
  \phi_! (M\otimes B(\alpha))
  \ar[r, rightarrowtail]
  \ar[d, "\Lambda_{M,B(\alpha)}"', "\sim"]
  &
  \phi_!(M\otimes B(\alpha+1))
  \ar[r, rightarrowtail]
    \ar[d, "\Lambda_{M,B(\alpha+1)}", "\sim"']
  &
  \dotsb
  \\
    \dotsb
  \ar[r, rightarrowtail]
  &
  \phi_! (M)\otimes \phi_!(B(\alpha))
  \ar[r, rightarrowtail]
  &
  \phi_!(M)\otimes \phi_!(B(\alpha+1))
  \ar[r, rightarrowtail]
  &
  \dotsb
  \end{tikzcd}
  \]
  Both top and bottom sequences are cofibrant for the Reedy model structure on $\alpha$-indexed cotowers, so the morphism of colimits $\Lambda_{M, B(\alpha_0)}$ is a weak equivalence by Ken Brown's lemma.
\end{itemize}
By cofibrant generation of the stable model structure on $H\mathrm{-Mod}$, the class ${}_M \mathcal{C}$ thus contains all cofibrant objects.
For any cofibrant object $N \in H\mathrm{-Mod}$ let $\mathcal{C}_N$ denote the class of objects $B$ for which $\Lambda_{B,N}$ is a weak equivalence.
 $\mathcal{C}_N$ contains the domains and codomains of all morphisms in $\mathcal{I}^\Sigma_H$ by the above, hence $\mathcal{C}_N$ contains all cofibrant objects by a similar argument, which proves (a).

We prove (b) by using the cofibrant replacements $\mathbb{Q}^H\to \mathbb{Q}$ of Construction \ref{cons:HopfCofibUnit}.
After taking symmetric stabilisations we have a cofibrant replacement $\Sigma^\infty_\mathbb{Q} \mathbb{Q}^H\to \Sigma^\infty_\mathbb{Q} \mathbb{Q}$ of the monoidal unit in $H\mathrm{-Mod}$ (similarly replacing $H$ by $K$).
Applying the extension of scalars functor we get
\[
\phi_!(\mathbb{Q}^H) \cong  \underset{\Delta^\mathrm{op}}{\mathrm{colim}}\, \phi_!\big(H^{\otimes(\bullet +1)}\big) \cong \underset{\Delta^\mathrm{op}}{\mathrm{hocolim}}\, \phi_!\big(H^{\otimes(\bullet +1)}\big)\,.
\]
By Corollary \ref{cor:SimplicialHopfMod} there is an equivalence of $H$-modules $H^{\otimes(n+1)}\cong H\otimes (H^{\otimes (n+1)})_H$ for each $n\geq 0$ (and similarly for $K$).
The weak equivalences $\phi^{\otimes(n+1)}\colon H^{\otimes(n+1)}\to K^{\otimes(n+1)}$ give rise to a compatible family of weak equivalences of coinvariant submodules $\Phi_n \colon (H^{\otimes(n+1)})_H \to (K^{\otimes(n+1)})_K$.
From these weak equivalences we obtain a levelwise weak equivalence of Reedy cofibrant simplicial objects in $K\mathrm{-Mod}_u$
\[
\phi_! (H^{\otimes (\bullet+1)}) \cong K\otimes H^{\otimes (\bullet +1)}_H
\xrightarrow{\;\;K\otimes \Phi_\bullet\;\,}
K\otimes K^{\otimes (\bullet +1)}_K\,,
\]
and passing to realisations yields a weak equivalence $\phi_!(\mathbb{Q}^H)\to \mathbb{Q}^K$ in $K\mathrm{-Mod}_u$.
By Ken Brown's lemma, the morphism $\gamma\colon \Sigma^\infty_\mathbb{Q}\phi_!(\mathbb{Q}^H)\cong \phi_!\Sigma^\infty_\mathbb{Q} \mathbb{Q}^H\to \Sigma^\infty_\mathbb{Q}\mathbb{Q}$ is a weak equivalence in $\mathrm{Sp}^\Sigma(\mathrm{sVect}_\mathbb{Q})$, hence also in $K\mathrm{-Mod}$.
Since $\gamma$ is isomorphic to the comparison map $\gamma_{\Sigma^\infty_\mathbb{Q} \mathbb{Q}^H}$, this proves (b).
\end{proof}

\section{Homotopical Lie representations}
\label{sec:LieRep}
Let $\mathfrak{g}$ be a simplicial Lie algebra.
The goal of this section is to prove a zig-zag of (weakly) monoidal Quillen equivalences between symmetric spectra in $\mathfrak{g}$-representations on the one hand and unbounded dg representations of the Lie algebra $N\mathfrak{g}$ of normalised chains on the other.
This plays a key part in the proofs of Section \ref{sec:RatStabParamHom}, providing the link between the simplicial and differential graded settings.

\subsection{The stable Dold--Kan correspondence}
The Dold--Kan correspondence is a classical equivalence between simplicial objects in an abelian category $A$ and non-negatively graded chain complexes in $A$.
In this article we are of course  exclusively interested in the case that $A= \mathrm{Vect}_\mathbb{Q}$ is the category of rational vector spaces, so that the Dold--Kan correspondence
\[
\begin{tikzcd}
\mathrm{sVect}_\mathbb{Q}
\ar[rr, "N", shift left=1.1ex]
\ar[rr, "\Gamma"', shift left=-1.1ex, leftarrow, "\simeq"]
&&
\mathrm{Ch}_+
\end{tikzcd}
\]
relates simplicial rational vector spaces with non-negatively graded rational chain complexes.
The \emph{normalisation functor} $N$ sends a simplicial rational vector space $V_\bullet$ to its \emph{normalised chain complex} $NV$, defined in degree $n\geq 0$ as the joint kernel
\[
NV_n := \bigcap_{i=1}^n \ker(d_i\colon V_n \to V_{n-1})
\]
and equipped with differential by restricting the remaining face map $d_0$ to this mutual kernel.
The inverse equivalence $\Gamma$ sends the non-negatively graded rational chain complex $M_\ast$ to the simplicial vector space
\[
[n]\longmapsto \Gamma(M)_n := \mathrm{Ch}_+ \big(N(\mathbb{Q}[\Delta^n]), M\big)
\]
defined by taking maps out of the normalised chain complex $N(\mathbb{Q}[\Delta^n])$ of the combinatorial $n$-simplex.
\begin{remark}
The Dold--Kan functors $N$ and $\Gamma$ exchange homotopy groups of simplicial vector spaces with homology groups of chain complexes.
The inverse equivalences of categories $N$ and $\Gamma$ can be promoted to Quillen equivalences $(N\dashv \Gamma)$ and $(\Gamma\dashv N)$ (with respect to the projective model structure on $\mathrm{Ch}_+$).
\end{remark}

Though the normalisation functor $N$ does not send the levelwise tensor product of simplicial vector spaces to tensor products of chain complexes, these monoidal structures are related to each in various ways.
The \emph{shuffle map} is a natural quasi-isomorphism
\[
\nabla_{V,W}\colon NV \otimes NW \longrightarrow N(V\otimes W)
\]
exhibiting $N$ as a lax symmetric monoidal functor (details are recalled in \cite[Section 2.3]{schwede_monoidal_2003}).
Passing to adjoints, there are induced natural maps of simplicial rational vector spaces
\[
\Phi_{M,N}\colon \Gamma (M\otimes N)\longrightarrow 
\Gamma M\otimes \Gamma N
\]
that make $\Gamma$ oplax symmetric monoidal.
The $\Phi_{M,N}$ are natural weak equivalences, so that the left derived functor $\mathbf{L}\Gamma$ is strongly symmetric monoidal.
This is the ur-example of a weakly monoidal Quillen equivalence:
\begin{theorem}[{\cite{schwede_monoidal_2003}}]
The Dold--Kan correspondence
\[
\begin{tikzcd}
\mathrm{Ch}_+
\ar[rr, "\Gamma", shift left=1.1ex]
\ar[rr, , "\bot", leftarrow, "N"', shift left=-1.1ex]
&&
\mathrm{sVect}_\mathbb{Q}
\end{tikzcd}
\]
is a weakly monoidal Quillen equivalence.
\end{theorem}
\begin{corollary}
\label{cor:NormLaxClosed}
For simplicial rational vector spaces $V$ and $W$ the shuffle map induces a natural quasi-isomorphism
$
\Xi_{V,W}\colon N([V, W])\to [NV, NW]$
exhibiting $N$ as a lax closed functor.
\end{corollary}
\begin{proof}
For simplicial rational vector spaces $V, W$, the 
\[
\begin{tikzcd}
NV\otimes N([V,W])
\ar[r, "\nabla"]
&
N(V\otimes [V, W])
\ar[r, "N(\mathrm{ev})"]
&
NW
\end{tikzcd}
\]
is the adjunct to a map of non-negatively graded chain complexes $\Xi_{V,W}\colon N([V,W])\to [NV, NW]$.
The map $\Xi_{V,W}$ is manifestly natural in $V$ and $W$, and exhibits a lax closed structure on $N$ by a standard argument.

To complete the proof, we must show that $\Xi_{V,W}$ is a quasi-isomorphism for all simplicial vector spaces $V,W$.
For any non-negatively graded chain complex $M$, arguing on adjoints shows that the natural transformation 
$
\mathrm{Ch}_+\big(M, N([V,W])\big)\to 
\mathrm{Ch}_+ \big(M, [NV, NW]\big)
$
given by postcomposition with $\Xi_{V,W}$
is identified with the natural transformation
$
\mathrm{sVect}_\mathbb{Q}\big(V\otimes \Gamma M, W\big)
\to
\mathrm{sVect}_\mathbb{Q}\big(\Gamma (NV \otimes M), W\big)
$
obtained by precomposition with the natural weak equivalence
\[
\xi_{V,M}
\colon
\begin{tikzcd}
\Gamma (NV\otimes M)
\ar[r, "\Phi_{NV, M}"]
&
\Gamma NV \otimes \Gamma M
\ar[r, "\cong "]
&
V\otimes \Gamma M\,.
\end{tikzcd}
\]
Note that all objects of both $\mathrm{sVect}_\mathbb{Q}$ and $\mathrm{Ch}_+$ are fibrant-cofibrant.
By picking a (functorial) Reedy cosimplicial framing $M^\ast$ of $M$, the map of homotopy function complexes induced by $\Xi_{V,W}$ is presented by
\begin{equation}
\label{eqn:ReedyCof1}
\mathrm{map}_{\mathrm{Ch}_+} \big(M^\ast, \Xi_{V,W}\big)
\colon
\mathrm{map}_{\mathrm{Ch}_+} \big(M^\ast, N([V,W])\big)
\longrightarrow 
\mathrm{map}_{\mathrm{Ch}_+} \big(M^\ast, [NV, NW]\big)\,,
\end{equation}
which is identified with 
\begin{equation}
\label{eqn:ReedyCof2}
\mathrm{map}_{\mathrm{sVect}_\mathbb{Q}} \big(\xi_{V,M^\ast}, W\big)
\colon
\mathrm{map}_{\mathrm{sVect}_\mathbb{Q}} \big( V\otimes \Gamma M^\ast,W\big)
\longrightarrow 
\mathrm{map}_{\mathrm{sVect}_\mathbb{Q}} \big(\Gamma (NV\otimes M^\ast) ,W\big)
\end{equation}
by the above.
The properties of the Dold--Kan correspondence guarantee that $V\otimes \Gamma M^\ast$ and $\Gamma (NV\otimes M^\ast)$ are Reedy cosimplicial coframings of $V\otimes M$ and $\Gamma(NV\otimes M)$ respectively, so that  \eqref{eqn:ReedyCof2} does indeed present a map of homotopy function complexes as indicated by the notation. 
As $\xi_{V,M^\ast}$ is a natural weak equivalence, \eqref{eqn:ReedyCof2} is a natural weak equivalence of simplicial sets for all $V, W$, and $M$.
It follows that \eqref{eqn:ReedyCof1} is a weak equivalence of homotopy function complexes for all $V, W$, and $M$. 

In any model category $\mathcal{M}$, a morphism of fibrant objects $f\colon x\to y$ is a weak equivalence if and only if the map of homotopy function complexes $\mathrm{map}_\mathcal{M} (c,f)$ is a weak equivalence for all cofibrant $c\in \mathcal{M}$.
As observed above, all objects of $\mathrm{sVect}_\mathbb{Q}$ and $\mathrm{Ch}_+$ are fibrant-cofibrant hence $\Xi_{V,W}$ is a quasi-isomorphism for all $V, W\in \mathrm{sVect}_\mathbb{Q}$.
\end{proof}

In \cite{shipley_HZ_2007}, Shipley produces a stabilisation of the Dold--Kan correspondence relating symmetric $H\mathbb{Z}$-module spectra to unbounded chain complexes of abelian groups.
Essentially the  same arguments apply with rational coefficents and we get a zig-zag of (weakly) monoidal Quillen equivalences between $\mathrm{Sp}^\Sigma(\mathrm{sVect}_\mathbb{Q})$ and unbounded rational chain complexes:
\begin{theorem}[Stable Dold--Kan correspondence]
\label{thm:StabDK}
There is a zig-zag of Quillen equivalences
\[
\begin{tikzcd}
\mathrm{Sp}^\Sigma(\mathrm{sVect}_\mathbb{Q})
\ar[rr, shift left=1.1ex, leftarrow, "N_!"]
\ar[rr, shift left=-1.1ex, "N^\ast"', "\bot"]
&&
\mathrm{Sp}^\Sigma(\mathrm{Ch}_+)
\ar[rr, shift left=1.1ex, "\mathfrak{A}"]
\ar[rr, shift left=-1.1ex, leftarrow, "\mathfrak{D}"', "\bot"]
&&
\mathrm{Ch}
\end{tikzcd}
\]
in which $(N_!\dashv N^\ast)$ is a weakly monoidal Quillen equivalence and $(\mathfrak{A}\dashv \mathfrak{D})$ is a strongly monoidal Quillen equivalence.
\end{theorem}
\begin{proof}[Sketch of proof.]
We briefly recall the salient points of \cite[Proposition 2.10]{shipley_HZ_2007}.
Suspension in the homotopy category $Ho(\mathrm{Ch}_+)$  is modelled by the shift endofunctor 
\begin{equation}
\label{eqn:ShiftEndoFun}
s\colon M \longmapsto M[-1]= M\otimes \mathbb{Q}[-1]
\end{equation}
with $\mathbb{Q}[-1]$ the chain complex given by $\mathbb{Q}$ concentrated in degree 1.
The symmetric stabilisation machine gives a symmetric monoidal model category $\mathrm{Sp}^\Sigma(\mathrm{Ch}_+) = \mathrm{Sp}^\Sigma(\mathrm{Ch}_+; s)$ and a left Quillen symmetric stabilisation functor $s^\infty\colon\mathrm{Ch}_+\to \mathrm{Sp}^\Sigma(\mathrm{Ch}_+)$ which is strongly symmetric monoidal.
$\mathrm{Sp}^\Sigma(\mathrm{Ch}_+)$ is identified with the category of modules over the symmetric sequence $s^\infty\mathbb{Q}\colon n \mapsto \mathbb{Q}[-n]$, which is a commutative monoid for the Day convolution product in $\mathrm{Fun}(\mathbf{\Sigma}, \mathrm{Ch}_+)$. (cf~Remark \ref{rem:SymSeqinsVectQ}).
The identification $\mathbb{Q}[-1] \cong N\widetilde{\mathbb{Q}}[S^1]$ gives rise to a morphism of commutative monoids in symmetric sequences
\begin{equation}
\label{eqn:NormalisationMonoid}
\mathbb{Q}[-n]\longrightarrow \mathcal{N}(n) := N\widetilde{\mathbb{Q}}[S^n]\,.
\end{equation}
Levelwise application of $N$ to objects of $\mathrm{Sp}^\Sigma(\mathrm{sVect}_\mathbb{Q})$  gives $\mathcal{N}$-modules in $\mathrm{Fun}(\mathbf{\Sigma}, \mathrm{Ch}_+)$, so that after taking restriction of scalars along \eqref{eqn:NormalisationMonoid} we get a functor
\[
N^\ast \colon \mathrm{Sp}^\Sigma(\mathrm{sVect}_\mathbb{Q})\longrightarrow \mathrm{Sp}^\Sigma(\mathrm{Ch}_+)\,.
\]
The functor $N^\ast$ is thus given on underlying symmetric sequences by $N$; $N^\ast$ preserves fibrations and (stable) weak equivalences.
This functor admits a left adjoint $N_!$ (which in general is \emph{not} given by $\Gamma$ on underlying symmetric sequences) such that the adjunction $(N_!\dashv N^\ast)$ is a weakly monoidal Quillen equivalence.

\begin{construction}
Let $\mathfrak{Fin}$ be the category with objects $\langle n\rangle:= \{1, \dotsc, n\}$ for $n\geq 0$ for which the morphisms $\langle n\rangle \to \langle m\rangle$ are the injections $\{1, \dotsc, n\}\hookrightarrow \{1, \dotsc, m\}$.
Associated to a symmetric spectrum object $P\in \mathrm{Sp}^\Sigma(\mathrm{Ch}_+)$ we construct a functor $\overline{P}\colon \mathfrak{Fin}\to \mathrm{Ch}$ by sending $\langle n\rangle$ to the shifted chain complex $P(n) [n] = P(n) \otimes \mathbb{Q}[n]$.
The functor $\overline{P}$ is uniquely determined on morphisms by the specifications  that
\begin{itemize}
  \item the canonical inclusion $\langle n\rangle\hookrightarrow \langle m\rangle$ of the first $n\leq m$ elements is sent to the composite
  \[
  P(n) [n]
  \xrightarrow{\sigma_n[n+1]}
  P(n+1) [n+1]
  \xrightarrow{\sigma_{n+1}[n+2]}
  \dotsb
  \xrightarrow{\sigma_{m-1}[m]}
  P(k)[k]
  \] 
  of shifted symmetric spectrum structure maps $\sigma_k \colon P(k) \otimes \mathbb{Q}[-1]\to P(k+1)$; and
  \item $\sigma\in \mathfrak{Fin}(\langle n\rangle, \langle n \rangle) =\Sigma_n$ is sent to the automorphism $(-1)^{\mathrm{sgn}(\sigma)}\rho(\sigma)$ of $P(n)[n]$, with $\rho$ the $\Sigma_n$-action on $P(n)$.
\end{itemize}
The \emph{assembly complex} of $P$ is defined as the
unbounded chain complex
$
\mathfrak{A}(P) := \mathrm{colim}_\mathfrak{Fin} \overline{P}
$.
The assembly functor $\mathfrak{A}\colon \mathrm{Sp}^\Sigma(\mathrm{Ch}_+)\to \mathrm{Ch}$ is strongly symmetric monoidal, interchanging the smash-tensor product of symmetric spectra of non-negatively graded chain complexes with the tensor product of unbounded chain complexes.
\end{construction}
\begin{construction}
Given an unbounded chain complex $M\in \mathrm{Ch}$, its \emph{spectral disassembly} $\mathfrak{D}(M)$ is a symmetric spectrum object in $\mathrm{Ch}_+$.
The underlying symmetric sequence of $\mathfrak{D}(M)$ is given at level $n$ by the connective cover 
$
\mathfrak{D}(M)(n) :=  c_0(M[-n])
$,
with $\Sigma_n$ acting in the sign representation.
Symmetric spectrum structure maps $\sigma_n \colon \mathfrak{D}(M)(n)[-1] \to \mathfrak{D}(M)(n+1)$ arise from the natural factorisations
\[
\begin{tikzcd}
c_0\big(M[-n]\big)\otimes \mathbb{Q}[-1]
\ar[r, "\sigma_n"]
\ar[d]
&
c_0\big(M[-n-1]\big)
\ar[d]
\\
\ar[r]
M[-n]\otimes \mathbb{Q}[-1]
\ar[r, "\cong"]
&
M[-n+1]\,.
\end{tikzcd}
\]
\end{construction}
The disassembly functor $\mathfrak{D}\colon \mathrm{Ch}\to \mathrm{Sp}^\Sigma (\mathrm{Ch}_+)$ is right adjoint to $\mathfrak{A}$ and the adjunction $(\mathfrak{A}\dashv \mathfrak{D})$ is a strongly monoidal Quillen equivalence.
\end{proof}
\begin{remark}
\label{rem:StabShuffle}
The stabilised normalisation functor $N^\ast \colon \mathrm{Sp}^\Sigma(\mathrm{sVect}_\mathbb{Q})\to \mathrm{Sp}^\Sigma(\mathrm{Ch}_+)$ is lax symmetric monoidal.
The lax symmetric monoidal transformation $\nabla_{V,W}\colon N^\ast(V)\otimes N^\ast(W)\to N^\ast(V\otimes W)$ is induced by the shuffle map.
\end{remark}
\begin{lemma}
\label{lem:StabDKNorm}
$N^\ast \colon \mathrm{Sp}^\Sigma(\mathrm{sVect}_\mathbb{Q})\to \mathrm{Sp}^\Sigma(\mathrm{Ch}_+)$ preserves and reflects weak equivalences between fibrant objects.
\end{lemma}
\begin{proof}
$N^\ast$ is given on underlying symmetric sequences by levelwise application of the normalisation functor $N$.
Since $N$ preserves and reflects weak equivalences, $N^\ast$ preserves and reflects levelwise weak equivalences.
In both $\mathrm{Sp}^\Sigma(\mathrm{sVect}_\mathbb{Q})$ and $\mathrm{Sp}^\Sigma(\mathrm{Ch}_+)$, the (stable) weak equivalences between (stably) fibrant objects are precisely the levelwise weak equivalences.
\end{proof}

\begin{lemma}
\label{lem:RatSymSpecFibCofib}
For any $V\in \mathrm{sVect}_\mathbb{Q}$, the symmetric spectrum object $\Sigma^\infty_\mathbb{Q} V\in \mathrm{Sp}^\Sigma(\mathrm{sVect}_\mathbb{Q})$ is fibrant-cofibrant.
\end{lemma}
\begin{proof}
All objects of $\mathrm{sVect}_\mathbb{Q}$ are cofibrant and $\Sigma^\infty_\mathbb{Q}\colon \mathrm{sVect}_\mathbb{Q}\to \mathrm{Sp}^\Sigma(\mathrm{sVect}_\mathbb{Q})$ is left Quillen, so that $\Sigma^\infty_\mathbb{Q} V$ is cofibrant.

An object $W$ of $\mathrm{Sp}^\Sigma(\mathrm{sVect}_\mathbb{Q})$ is fibrant if and only if the underlying object of $\mathrm{Sp}^\Sigma$ is fibrant, that is if each $W(n)$ is a Kan complex and the adjoint structure maps $W(n) \to \Omega W(n+1)$ are weak equivalences.
All simplicial rational vector spaces are Kan complexes so the first condition is always satisfied.
For $W= \Sigma^\infty_\mathbb{Q} V$ we have 
\[
W(n) = \widetilde{\mathbb{Q}}[S^n]\otimes V\cong \underbrace{\widetilde{\mathbb{Q}}[S^1]\otimes \dotsb \otimes \widetilde{\mathbb{Q}}[S^1]}_{\text{$n$ times}}\otimes V
\]
and the adjoint spectrum structure maps are
\[
\widetilde{\mathbb{Q}}[S^n] \otimes V
\longrightarrow
\big[
\widetilde{\mathbb{Q}}[S^1],
\widetilde{\mathbb{Q}}[S^{n+1}] \otimes V
\big]\,.
\]
Taking normalised chains and using the shuffle map, Corollary \ref{cor:NormLaxClosed}, and $N\widetilde{\mathbb{Q}}[S^1] = \mathbb{Q}[-1]$, we get a commuting diagram of quasi-isomorphisms
\[
\begin{tikzcd}[row sep =small]
V[-n]
\ar[rr, "\nabla", "\sim"']
\ar[dd, "\cong"']
&&
N\big(
\widetilde{\mathbb{Q}}[S^n]\otimes V
\big)
\ar[d]
\\
&&
N\big(\big[
\widetilde{\mathbb{Q}}[S^1], \widetilde{\mathbb{Q}}[S^{n+1}]\otimes V
\big]\big)
\ar[d, "\sim"]
\\
V[-(n+1)][+1] \cong \big[\mathbb{Q}[-1], V[-(n+1)]\big]
\ar[rr, "{[\mathbb{Q}[-1], \nabla]}", "\sim"']
&&
\big[
\mathbb{Q}[-1], \widetilde{\mathbb{Q}}[S^{n+1}]\otimes V\,.
\big]
\end{tikzcd}
\]
This shows that $\Sigma^\infty_\mathbb{Q} V$ is fibrant.
\end{proof}

\begin{lemma}
\label{lem:StabDKNorm2}
$\mathfrak{D}\colon \mathrm{Ch}\to \mathrm{Sp}^\Sigma(\mathrm{Ch}_+)$ preserves and reflects weak equivalences.
\end{lemma}
\begin{proof}
All objects of $\mathrm{Ch}$ are fibrant, so that $\mathfrak{D}(M)\in \mathrm{Sp}^\Sigma(\mathrm{Ch}_+)$ is (stably) fibrant for any $M\in \mathrm{Ch}$.
Fix a morphism of chain complexes $f\colon M\to N$.
The (stable) weak equivalences between (stably) fibrant objects of $\mathrm{Sp}^\Sigma(\mathrm{Ch}_+)$ are precisely the levelwise quasi-isomorphisms.
Unwinding the definitions, we therefore find that $\mathfrak{D}(f) \colon \mathfrak{D}(M)\to \mathfrak{D}(N)$ is a (stable) weak equivalence precisely if $c_0 (M[-n]) \to c_0(N[-n])$ is a quasi-isomorphism for all $n\geq 0$.
This latter condition is equivalent to the requirement that $f$ is a quasi-isomorphism.
\end{proof}

\begin{remark}
\label{rem:HisPi}
Similarly to the classical case, the stable Dold--Kan correspondence interchanges stable homotopy for homology.
More precisely, it is relatively straightforward to show that there is a diagram of functors
\[
\begin{tikzcd}[row sep = small]
Ho\big(\mathrm{Sp}^\Sigma(\mathrm{sVect}_\mathbb{Q})\big)
\ar[r, "\mathbf{R}N^\ast", "\cong"']
\ar[dr, bend right=20, "\pi^\mathrm{st}_\ast"']
&
Ho\big(\mathrm{Sp}^\Sigma(\mathrm{Ch}_+)\big)
\ar[r, leftarrow, "\mathbf{R}\mathfrak{D}", "\cong"']
&
Ho(\mathrm{Ch})
\ar[dl, bend left=25, "H_\bullet"]
\\
&
\mathrm{Fun}(\mathbb{Z},\mathrm{Vect}_\mathbb{Q})
&
\end{tikzcd}
\]
commuting up to natural isomorphism.
\end{remark}

\subsection{Comparing simplicial and dg Lie representations}
\label{SS:SimplicialDGLieComp}
Let $\mathfrak{g}$ be a simplicial Lie algebra over $\mathbb{Q}$.
Using the properties of the shuffle map, Quillen observed that $N\mathfrak{g}$ is a differential graded Lie algebra
with respect to the bracket
\[
[\![-,- ]\!]
\colon N\mathfrak{g}\otimes N\mathfrak{g}
\xrightarrow{\;\;\nabla_{\mathfrak{g},\mathfrak{g}}\;\;}
N(\mathfrak{g}\otimes \mathfrak{g})
\xrightarrow{\;\;N([-,-])\;\;}
N\mathfrak{g}\,.
\]
The goal of this section is to extend the stable Dold--Kan correspondence (Theorem \ref{thm:StabDK}) to a zig-zag of weakly monoidal Quillen equivalences between stable model categories of $\mathfrak{g}$-representations and $N\mathfrak{g}$-representations.

Given a (simplicial or dg) Lie algebra $\mathfrak{a}$, we denote the universal enveloping algebra by $\mathcal{U}\mathfrak{a}$.
In either setting, $\mathcal{U}\mathfrak{a}$ is a Hopf algebra and we write the multiplication and comultiplication maps  as $\mu_\mathfrak{a}$ and $\kappa_\mathfrak{a}$ respectively.
The first result of this section provides a comparison between the universal enveloping algebra of a simplicial Lie algebra $\mathfrak{g}$ and that of the normalisation $N\mathfrak{g}$:
\begin{lemma}
\label{lem:Chi}
For $\mathfrak{g}$ a simplicial Lie algebra, there is a natural quasi-isomorphism $\chi_\mathfrak{g}\colon \mathcal{U}N\mathfrak{g}\to N\mathcal{U}\mathfrak{g}$.
The weak equivalence $\chi_\mathfrak{g}$ is \emph{multiplicative} in the sense that there is a commuting diagram of chain complexes
\begin{equation}
\label{eqn:ChiMult}
\begin{tikzcd}
\mathcal{U}N\mathfrak{g}\otimes \mathcal{U}N\mathfrak{g}
\ar[rr, "\mu_{N\mathfrak{g}}"]
\ar[d, "\chi_\mathfrak{g}\otimes \chi_\mathfrak{g}"']
&&
\mathcal{U}N\mathfrak{g}
\ar[d, "\chi_\mathfrak{g}"]
\\
N\mathcal{U}\mathfrak{g}\otimes
N\mathcal{U}\mathfrak{g}
\ar[r, "\nabla"]
&
N(\mathcal{U}\mathfrak{g}\otimes \mathcal{U}\mathfrak{g})
\ar[r, "N\mu_\mathfrak{g}"]
&
N\mathcal{U}\mathfrak{g}
\end{tikzcd}
\end{equation}
with $\nabla$ the shuffle map.
Moreover, the map $\chi_\mathfrak{g}$ \emph{comultiplicative} in the sense that there is a commuting diagram of chain complexes  
\begin{equation}
\label{eqn:ChiComult}
\begin{tikzcd}
\mathcal{U}N\mathfrak{g}
\ar[rr, "\chi_\mathfrak{g}"]
\ar[d, "\kappa_{N\mathfrak{g}}"']
&&
N\mathcal{U}\mathfrak{g}
\ar[d, "N\kappa_\mathfrak{g}"]
\\
\mathcal{U}N\mathfrak{g}\otimes 
\mathcal{U}N\mathfrak{g}
\ar[r, "\chi_\mathfrak{g}\otimes \chi_\mathfrak{g}"]
&
N\mathcal{U}\mathfrak{g}\otimes 
N\mathcal{U}\mathfrak{g}
\ar[r, "\nabla"]
&
N(\mathcal{U}\mathfrak{g}\otimes \mathcal{U}\mathfrak{g})\,.
\end{tikzcd}
\end{equation}
\end{lemma}
\begin{proof}
Denoting the free tensor algebra functor by $T$ (in both the dg and simplicial settings), the interated shuffle maps give rise to a natural weak equivalence $\eta_V\colon TNV\to NTV$ for any simplicial rational vector space $V$.
Since the shuffle map is a lax monoidal transformation, $\eta_V$ is multiplicative in the sense of the diagram \eqref{eqn:ChiMult}.

For $V=\mathfrak{g}$,  $\eta_\mathfrak{g}$ sends the two-sided tensor ideal
$
\mathcal{I}_{N\mathfrak{g}} := \big\langle
a\otimes b - (-1)^{|a||b|} b\otimes a -[\![a,b]\!] 
\big\rangle$ of $TN\mathfrak{g}
$
into the the image by $N$ of the two-sided tensor ideal
$
\mathcal{I}_\mathfrak{g} := \big\langle
x\otimes y - y\otimes x - [x,y]\big\rangle\subset T\mathfrak{g}
$,
that is $\eta_\mathfrak{g}(\mathcal{I}_{N\mathfrak{g}})\subset N(\mathcal{I}_\mathfrak{g})$.
Passing to quotients, we get a map $\chi_\mathfrak{g}\colon \mathcal{U}N\mathfrak{g}\to N\mathcal{U}\mathfrak{g}$ that inherits the multiplicativity property from $\eta_\mathfrak{g}$.
The map $\chi_\mathfrak{g}$ is manifestly natural in the simplicial Lie algebra $\mathfrak{g}$.

We now show that $\chi_\mathfrak{g}$ is a quasi-isomorphism.
Working over $\mathbb{Q}$, the quotient map $TV \to SV$ defining the free symmetric algebra on $V$ has a section 
\[
\sigma_V\colon v_1 \dotsb v_n \longmapsto \frac{1}{n!}\sum_{\sigma\in \Sigma_n} \pm x_{\sigma(1)}\otimes \dotsb 
\otimes
x_{\sigma(n)}\,,
\]
with \lq\lq$\pm$'' the Koszul sign (here $V$ is either a chain complex or simplicial vector space).
The shuffle map is a lax symmetric monoidal transformation, so that $\eta_V\colon TNV\to NTV$ descends to define a map $\xi_V \colon SNV\to NSV$.
As $\xi_V$ is a retract of $\eta_V$ by the above, it is also a natural quasi-isomorphism.
For $V=\mathfrak{g}$, there is a commuting diagram of maps of chain complexes 
\[
\begin{tikzcd}
SN\mathfrak{g}
\ar[d, "\xi_\mathfrak{g}"']
\ar[r, "\sigma_{N\mathfrak{g}}"]
\ar[rr, bend left=30, "\mathrm{PBW}_{N\mathfrak{g}}"]
&
TN\mathfrak{g}
\ar[d, "\eta_\mathfrak{g}"]
\ar[r, two heads]
&
\mathcal{U}N\mathfrak{g}
\ar[d, "\chi_\mathfrak{g}"]
\\
NS\mathfrak{g}
\ar[r, "N\sigma_{\mathfrak{g}}"]
\ar[rr, bend right=30,"N(\mathrm{PBW}_{\mathfrak{g}})"']
&
NT\mathfrak{g}
\ar[r, two heads]
&
N\mathcal{U}\mathfrak{g}\,,
\end{tikzcd}
\]
where the maps labelled \lq\lq$\mathrm{PBW}$'' are the Poincar\'{e}--Birkhoff--Witt isomorphisms.
Hence $\chi_\mathfrak{g}$ is a quasi-isomorphism.

To prove comultiplicativity of $\chi_\mathfrak{g}$, recall that the coproduct on $\mathcal{U}N\mathfrak{g}$ is defined on the image of the inclusion $N\mathfrak{g}\hookrightarrow \mathcal{U}N\mathfrak{g}$ by $v\mapsto 1\otimes v + v\otimes 1$ and then extended to all of $\mathcal{U}N\mathfrak{g}$ as an algebra homomorphism (similarly for $\mathcal{U}\mathfrak{g}$).
By construction, the restriction of $\chi_\mathfrak{g}$ to $N\mathfrak{g}\hookrightarrow \mathcal{U}N\mathfrak{g}$ coincides with $N(\mathfrak{g}\hookrightarrow \mathcal{U}\mathfrak{g})$.
It follows that  the comultiplicativity diagram \eqref{eqn:ChiComult} commutes for elements of $N\mathfrak{g}\hookrightarrow \mathcal{U}N\mathfrak{g}$. 
An inductive argument over degree using the  multipliciativity diagram \eqref{eqn:ChiMult} and the properties of the shuffle map shows that \eqref{eqn:ChiComult} commutes on all monomials on elements in $N\mathfrak{g}$.
By the Poincar\'{e}--Birkhoff--Witt theorem, these monomials span $\mathcal{U}N\mathfrak{g}$.
\end{proof}

\begin{definition}
Any simplicial Lie algebra $\mathfrak{g}$ induces a monad $T_\mathfrak{g} := \Sigma^\infty_\mathbb{Q} \mathcal{U}\mathfrak{g}\otimes (-)$ on $\mathrm{Sp}^\Sigma(\mathrm{sVect}_\mathbb{Q})$ and we write $\mathfrak{g}\mathrm{-Rep}_\Delta:= T_\mathfrak{g}\mathrm{-Alg}$ for the category of algebras over this monad.
As $\Sigma^\infty_\mathbb{Q} \mathcal{U}\mathfrak{g}$ is cofibrant in $\mathrm{Sp}^\Sigma(\mathrm{sVect}_\mathbb{Q})$, the free-forgetful adjunction
\[
\begin{tikzcd}
\mathrm{Sp}^\Sigma(\mathrm{sVect}_\mathbb{Q})\ar[rr, "\Sigma^\infty_\mathbb{Q}\mathcal{U}\mathfrak{g}\otimes (-)", shift left=1.1ex]
\ar[rr, leftarrow, "\bot", shift left=-1.1ex]
&&
\mathfrak{g}\mathrm{-Rep}_\Delta
\end{tikzcd}
\]
equips $\mathfrak{g}\mathrm{-Rep}_\Delta$ with a stable model structure, with weak equivalences and fibrations created by the forgetful functor.
\end{definition}
\begin{remark}
As a model category,
$\mathfrak{g}\mathrm{-Rep}_\Delta$ is canonically isomorphic to the symmetric stabilisation of the category $\mathfrak{g}\mathrm{-Rep}_{\Delta}^u$ of $\mathfrak{g}$-modules in simplicial vector spaces with respect the suspension endofunctor \eqref{eqn:RatSuspend}.
\end{remark}

\begin{definition}
For a dg Lie algebra $\mathfrak{h}$, the $\mathcal{U}\mathfrak{h}$-modules in chain complexes organise into a combinatorial model category $\mathfrak{h}\mathrm{-Rep}_\mathrm{dg}$.
Weak equivalences and fibrations in $\mathfrak{h}\mathrm{-Rep}_\mathrm{dg}$ are the quasi-isomorphisms and degreewise epimorphisms respectively, and
the shift endofunctor $M\mapsto M[-1]$ models suspension on $Ho(\mathfrak{h}\mathrm{-Rep}_\mathrm{dg})$.

The non-negatively graded $\mathcal{U}\mathfrak{h}$-modules organise into a combinatoral model category $\mathfrak{h}\mathrm{-Rep}_\mathrm{dg}^+$ with similar features.
The fully faithful functor $\mathfrak{h}\mathrm{-Rep}^+_\mathrm{dg}\hookrightarrow \mathfrak{h}\mathrm{-Rep}_\mathrm{dg}$ is left Quillen with right adjoint $M\mapsto c_0M$ (the connective cover $c_0 M$ inherits a $\mathcal{U}\mathfrak{h}$-module structure since this algebra is non-negatively graded).
Write
$
\mathfrak{h}\mathrm{-Rep}_\mathrm{dg}^\Sigma:= \mathrm{Sp}^\Sigma(\mathfrak{h}\mathrm{-Rep}^+_\mathrm{dg}; s)
$
for the symmetric stabilisation of $\mathfrak{h}\mathrm{-Rep}^+_\mathrm{dg}$ with respect to the shift endofunctor.
\end{definition}
\begin{remark}
$\mathfrak{h}\mathrm{-Rep}_\mathrm{dg}^\Sigma$ is canonically isomorphic to the category of $s^\infty \mathcal{U}\mathfrak{h}$-modules in $\mathrm{Sp}^\Sigma(\mathrm{Ch}_+)$ (recall that $s^\infty\colon \mathrm{Ch}_+\to\mathrm{Sp}^\Sigma(\mathrm{Ch}_+)$ is the functor that sends a connective chain complex to its symmetric suspension spectrum).
This latter category inherits a model structure from $\mathrm{Sp}^\Sigma(\mathrm{Ch}_+)$ via the free-forgetful adjunction of the monad $s^\infty \mathcal{U}\mathfrak{h}$, and the isomorphism of categories identifies the model structures.
\end{remark}

\begin{theorem}
\label{thm:StabDKLie}
For any simplicial Lie algebra $\mathfrak{g}$, there is a zig-zag of Quillen equivalences 
\[
\begin{tikzcd}
\mathfrak{g}\mathrm{-Rep}_\Delta
\ar[rr, shift left=1.1ex, leftarrow, "N_!^\mathfrak{g}"]
\ar[rr, shift left=-1.1ex, "N^\ast_\mathfrak{g}"', "\bot"]
&&
N\mathfrak{g}\mathrm{-Rep}_\mathrm{dg}^\Sigma
\ar[rr, shift left=1.1ex, "\mathfrak{A}_{N\mathfrak{g}}"]
\ar[rr, shift left=-1.1ex, leftarrow, "\mathfrak{D}_{N\mathfrak{g}}"', "\bot"]
&&
N\mathfrak{g}\mathrm{-Rep}_\mathrm{dg}
\end{tikzcd}
\]
inducing an equivalence of homotopy categories $Ho(\mathfrak{g}\mathrm{-Rep}_\Delta)\cong Ho(N\mathfrak{g}\mathrm{-Rep}_\mathrm{dg})$.
\end{theorem}
\begin{proof}
First, let us consider the left hand adjunction. 
Let $M$ be a (left) $\mathcal{U}\mathfrak{g}$-module in simplicial rational vector spaces, then the multiplicativity property of the comparison map $\chi_\mathfrak{g}$ (Lemma \ref{lem:Chi}) allows us to equip $NM$ with a $N\mathcal{U}\mathfrak{g}$-module structure via the map
\[
\mathcal{U}N\mathfrak{g}\otimes NM
\xrightarrow{\;\chi_\mathfrak{g}\otimes NM\;}
N\mathcal{U\mathfrak{g}}\otimes NM
\xrightarrow{\;\nabla_{\mathcal{U}\mathfrak{g}, M}\;}
N(\mathcal{U}\mathfrak{g}\otimes M)
\longrightarrow
N M\,.
\]
The stable Dold--Kan functor $N^\ast \colon \mathrm{Sp}^\Sigma(\mathrm{sVect}_\mathbb{Q})\to \mathrm{Sp}^\Sigma(\mathrm{Ch}_+)$ is determined on underlying symmetric sequences by $N$ so that there is a commuting diagram of functors
\[
\begin{tikzcd}
\mathfrak{g}\mathrm{-Rep}_\Delta
\ar[r, "N^\ast_\mathfrak{g}"]
\ar[d]
&
N\mathfrak{g}\mathrm{-Rep}_\mathrm{dg}^\Sigma
\ar[d]
\\
\mathrm{Sp}^\Sigma(\mathrm{sVect}_\mathbb{Q})
\ar[r, "N^\ast"]
&
\mathrm{Sp}^\Sigma(\mathrm{Ch}_+)\,,
\end{tikzcd}
\]
where the vertical arrows are the obvious forgetful functors.
The functor $N^\ast_\mathfrak{g}$ thus preserves limits and filtered colimits, so has a left adjoint $N^\mathfrak{g}_!$ by the adjoint functor theorem (there is a more explicit description of $N^\mathfrak{g}_!$ using monadicity, but we shall not need it).
The adjunction $(N_!^\mathbb{g}\dashv N^\ast_\mathfrak{g})$ is Quillen by the properties of $N^\ast$.

The right adjoint $N^\ast_\mathfrak{g}$ preserves and reflects stable weak equivalences by Lemma \ref{lem:StabDKNorm}.
To prove the Quillen equivalence, it is sufficient to show that the restriction of the derived unit to cofibrant objects is a natural weak equivalence.
To this end, let $\mathcal{E}\hookrightarrow Ho(N\mathfrak{g}\mathrm{-Rep}_\mathrm{dg}^\Sigma)$ denote the subcategory on objects for which the derived unit is an isomorphism, then it is not hard to see that  $\mathcal{E}$ is a localising subcategory.
The cofibrant object $s^\infty \mathcal{U}N\mathfrak{g}$ presents a compact generator of $Ho(N\mathfrak{g}\mathrm{-Rep}_\mathrm{dg}^\Sigma)$ and the (derived) unit morphism
\[
\eta_{s^\infty \mathcal{U}N\mathfrak{g}}\colon s^\infty \mathcal{U}N\mathfrak{g}
\to N^\ast_\mathfrak{g} N_!^\mathfrak{g} s^\infty \mathcal{U}N\mathfrak{g}
\cong s^\infty (N\mathcal{U} \mathfrak{g})
\]
is isomorphic to the map obtained by applying the stabilisation functor $s^\infty\colon \mathrm{Ch}_+ \to \mathrm{Sp}^\Sigma (\mathrm{Ch}_+)$ to the quasi-isomorphism $\chi_\mathfrak{g}\colon \mathcal{U}N\mathfrak{g}\to N\mathcal{U}\mathfrak{g}$.
By Ken Brown's lemma, $\eta_{s^\infty \mathcal{U}N\mathfrak{g}}$ is a weak equivalence in $\mathrm{Sp}^\Sigma(\mathrm{Ch}_+)$ and thus also in $N\mathfrak{g}\mathrm{-Rep}_\mathrm{dg}^\Sigma$.
This shows that $\mathcal{E}$ contains the compact generator, hence coincides with the full homotopy category and hence $(N_!^\mathfrak{g}\dashv N_!^\mathfrak{g})$ is a Quillen equivalence.

We now turn to the right hand adjunction.
The assembly functor $\mathfrak{A}\colon \mathrm{Sp}^\Sigma(\mathrm{Ch}_+)\to \mathrm{Ch}$ sends $s^\infty \mathcal{U}N\mathfrak{g}$ to $\mathcal{U}N\mathfrak{g}$, hence sends $s^\infty \mathcal{U}N\mathfrak{g}$-modules in $\mathrm{Sp}^\Sigma(\mathrm{Ch}_+)$ to $\mathcal{U}N\mathfrak{g}$-modules in $\mathrm{Ch}$.
Hence we get a functor $\mathfrak{A}_{N\mathfrak{g}}\colon N\mathfrak{g}\mathrm{-Rep}_\mathrm{dg}^\Sigma\to N\mathfrak{g}\mathrm{-Rep}_\mathrm{dg}$ which is given on underlying objects by $\mathfrak{A}$.
Since the dg algebra $\mathcal{U}N\mathfrak{g}$ is connective, the connective cover $c_0 M$ of any $\mathcal{U}N\mathfrak{g}$-module $M$ is naturally a  $\mathcal{U}N\mathfrak{g}$-module.
In view of the construction of the disassembly functor $\mathfrak{D}$, one shows that there is a functor $\mathfrak{D}_{N\mathfrak{g}}\colon N\mathfrak{g}\mathrm{
-Rep}\to N\mathfrak{g}\mathrm{-Rep}_\mathrm{dg}^\Sigma$ that coincides with $\mathfrak{D}$ on underlying objects.
It is easily checked that $\mathfrak{D}_{N\mathfrak{g}}$ is right adjoint to $\mathfrak{A}_{N\mathfrak{g}}$ and the adjunction $(\mathfrak{A}_{N\mathfrak{g}}\dashv \mathfrak{D}_{N\mathfrak{g}})$ is Quillen by the properties of $\mathfrak{D}$.
We prove that $(\mathfrak{A}_{N\mathfrak{g}}\dashv \mathfrak{D}_{N\mathfrak{g}})$ is in fact a Quillen equivalence by an adaptation of the previous argument for the left hand adjunction; this requires Lemma \ref{lem:StabDKNorm2}, the fact that $s^\infty \mathcal{U}N\mathfrak{g}$ presents a compact generator of $Ho(N\mathfrak{g}\mathrm{-Rep}_\mathrm{dg}^\Sigma)$, and the fact that the derived unit at $s^\infty \mathcal{U}N\mathfrak{g}$ is an isomorphism.
\end{proof}

\begin{remark}
For any simplicial Lie algebra $\mathfrak{g}$, the homotopy groups $\pi_\ast (\mathfrak{g})$ assemble into a (non-negatively) graded rational Lie algebra.
By the Dold--Kan correspondence $\pi_\ast(\mathfrak{g})\cong H_\bullet (N\mathfrak{g})$ as graded Lie algebras, and upon passing to stable homotopy and homology groups the above result implies a diagram of functors
\[
\begin{tikzcd}
Ho\big(\mathfrak{g}\mathrm{-Rep}_\Delta\big)
\ar[r, "\mathbf{R}N^\ast_\mathfrak{g}", "\simeq"']
\ar[dr, bend right=20, "\pi^\mathrm{st}_\ast"']
&
Ho\big(N\mathfrak{g}\mathrm{-Rep}_\mathrm{dg}^\Sigma\big)
\ar[r, leftarrow, "\mathbf{R}\mathfrak{D}_{N\mathfrak{g}}", "\simeq"']
&
Ho\big(N\mathfrak{g}\mathrm{-Rep}_\mathrm{dg}\big)
\ar[dl, bend left=19, "H_\bullet"]
\\
&
\pi_\ast(\mathfrak{g})\mathrm{-Rep}
&
\end{tikzcd}
\]
commuting up to natural isomorphism (cf~Remark \ref{rem:HisPi}).
\end{remark}

\subsection{Tensor products and base change}
For a simplicial Lie algebra $\mathfrak{g}$, setting $H=\mathcal{U}\mathfrak{g}$ in Lemma \ref{lem:SpectralHopfMod} shows that $\mathfrak{g}\mathrm{-Rep}_\Delta$ is a symmetric monoidal model category.
Similarly, the tensor product of chain complexes furnishes $N\mathfrak{g}\mathrm{-Rep}_\mathrm{dg}$ with a symmetric monoidal model structure (set $H = \mathcal{U}N\mathfrak{g}$ in Lemma \ref{lem:dgHopfMonoidal} below).
By exploiting the multiplicativity and comultiplicativity properties of $\chi_\mathfrak{g}\colon \mathcal{U}N\mathfrak{g}\to N\mathcal{U}\mathfrak{g}$, we show that the zig-zag of Quillen equivalences of Theorem \ref{thm:StabDKLie} is (weakly) monoidal.
\begin{lemma}
\label{lem:dgHopfMonoidal}
Let $H$ be a non-negatively graded dg Hopf algebra over $\mathbb{Q}$.
Then the categories $H\mathrm{-Mod}_+$ and $H\mathrm{-Mod}$ of, respectively, non-negatively graded and unbounded $H$-modules in chain complexes are symmetric monoidal model categories.
\end{lemma}
\begin{proof}
For non-negatively graded modules the result is essentially the differential graded version of Lemma \ref{lem:SimplHopfModMonoidal} and the same proof applies \emph{mutatis mutandis}; a similar argument also applies in the unbounded case.
\end{proof}

\begin{corollary}
For any non-negatively graded dg Hopf $H$ algebra over $\mathbb{Q}$, the stable model category $H\mathrm{-Mod}^\Sigma := \mathrm{Sp}^\Sigma(H\mathrm{-Mod}_+; s)$ is a symmetric monoidal model category.
\end{corollary}
\begin{proof}
This follows by applying the symmetric stabilisation machine.
\end{proof}
\begin{remark}
The stable model category $\mathrm{Sp}^\Sigma(H\mathrm{-Mod};s)$ is canonically isomorphic to the stable model category of $s^\infty H$-modules in $\mathrm{Sp}^\Sigma(\mathrm{Ch}_+)$.
By a mild abuse of notation $H\mathrm{-Mod}^\Sigma$ is used to refer to either of these isomorphic model categories.
\end{remark}

\begin{theorem}
\label{thm:SimpLieMonoidalEquivs}
For any simplicial Lie algebra $\mathfrak{g}$, $(N_!^\mathfrak{g}\dashv N^\ast_\mathfrak{g})$ is a weakly monoidal Quillen equivalence and $(\mathfrak{A}_{N\mathfrak{g}}\dashv \mathfrak{D}_{N\mathfrak{g}})$ is a strongly monoidal Quillen equivalence.
\end{theorem}
\begin{proof}
The left Quillen functor $\mathfrak{A}_{N\mathfrak{g}}$ is given on underlying objects by $\mathfrak{A}\colon \mathrm{Sp}^\Sigma(\mathrm{Ch}_+)\to \mathrm{Ch}$.
Since $\mathfrak{A}$ is strongly monoidal, so is $\mathfrak{A}_{N\mathfrak{g}}$.
To verify that $(\mathfrak{A}_{N\mathfrak{g}}\dashv \mathfrak{D}_{N\mathfrak{g}})$ is a strongly monoidal Quillen equivalence, it is therefore sufficient to check that for some (hence any) cofibrant replacement $\mathbb{Q}^c\to \Sigma^\infty\mathbb{Q}$ of the monoidal unit in $N\mathfrak{g}\mathrm{-Rep}_\mathrm{dg}^\Sigma$, the morphism $\mathfrak{A}_{N\mathfrak{g}}\mathbb{Q}^c\to \mathbb{Q}$ is a quasi-isomorphism of chain complexes.
For this, note that Construction \ref{cons:HopfCofibUnit} carries over to the dg setting and  so 
\begin{equation}
\label{eqn:CofibUnit}
\mathbb{Q}^{N\mathfrak{g}} := \underset{\Delta^\mathrm{op}}{\mathrm{colim}}\, s^\infty \mathcal{U}N\mathfrak{g}^{\otimes (n+1)} \longrightarrow s^\infty \mathbb{Q}
\end{equation}
is a cofibrant resolution in $N\mathfrak{g}\mathrm{-Rep}_\mathrm{dg}^\Sigma$.
Applying $\mathfrak{A}_{N\mathfrak{g}}$ we get a map of chain complexes 
\[
\underset{\Delta^\mathrm{op}}{\mathrm{colim}}\, \mathcal{U}N\mathfrak{g}^{\otimes (n+1)} \longrightarrow  \mathbb{Q}\,,
\]
which is a quasi-isomorphism by the same argument.

We now turn to the $(N_\mathfrak{g}\dashv N^\ast_\mathfrak{g})$-adjunction.
Let $V$ and $W$ be $\Sigma^\infty_\mathbb{Q}\mathcal{U}\mathfrak{g}$-modules with actions $\rho_V$ and $\rho_W$ respectively.
We write $\otimes$ for the smash-tensor monoidal structures on both $\mathrm{Sp}^\Sigma(\mathrm{sVect}_\mathbb{Q})$ and $\mathrm{Sp}^\Sigma(\mathrm{Ch}_+)$ and abuse notation by writing $\Sigma^\infty_\mathbb{Q} K \otimes V \equiv K\otimes V$ and $s^\infty L \otimes M \equiv L\otimes M$ (the smash-tensor product with a suspension spectrum is the levelwise tensor product). 
The comultiplicativity of $\chi_\mathfrak{g}$ (Lemma \ref{lem:Chi}) and the lax symmetric monoidal structure on $N^\ast$ (Remark \ref{rem:StabShuffle}) imply that we have a commuting diagram:
\[
\begin{tikzcd}[row sep = small, column sep =tiny]
\mathcal{U}N\mathfrak{g}\otimes N^\ast V\otimes N^\ast W
\ar[rr, "\mathrm{id}\otimes \nabla"]
\ar[dd, "\kappa_{N\mathrm{g}}\otimes \mathrm{id}"']
\ar[dr, "\chi_\mathfrak{g}\otimes \mathrm{id}"']
&&
\mathcal{U}N\mathfrak{g}\otimes N^\ast (V\otimes W)
\ar[dd, "\chi_\mathfrak{g}\otimes \mathrm{id}"]
\\
&
N\mathcal{U}\mathfrak{g}\otimes N^\ast V\otimes N^\ast W
\ar[dd, "N\kappa_\mathfrak{g}\otimes \mathrm{id}"]
\ar[dr, "\mathrm{id}\otimes \nabla"]
&
\\
\mathcal{U}N\mathfrak{g}\otimes 
\mathcal{U}N\mathfrak{g}\otimes
N^\ast V\otimes N^\ast W
\ar[dd, "\chi_\mathfrak{g}\otimes \chi_\mathfrak{g}\otimes \mathrm{id}"']
&&
N\mathcal{U}\mathfrak{g}\otimes N^\ast (V\otimes W)
\ar[dd, "N\kappa_\mathfrak{g}\otimes \mathrm{id}"]
\\
&
N(\mathcal{U}\mathfrak{g}\otimes \mathcal{U}\mathfrak{g})\otimes N^\ast(V)\otimes N^\ast(W)
\ar[from=dl, "\nabla\otimes \mathrm{id}"]
\ar[dr,"\mathrm{id}\otimes \nabla"]
&
\\
N\mathcal{U}\mathfrak{g}\otimes 
N\mathcal{U}\mathfrak{g}\otimes
N^\ast V\otimes N^\ast W
\ar[d, "\cong"']
&&
N(\mathcal{U}\mathfrak{g}\otimes \mathcal{U}\mathfrak{g})\otimes N^\ast(V\otimes W)
\ar[d, "\nabla"]
\\
N\mathcal{U}\mathfrak{g}\otimes 
N^\ast V\otimes 
N\mathcal{U}\mathfrak{g}\otimes
N^\ast W 
\ar[d, "\nabla\otimes \nabla"']
&&
N^\ast(\mathcal{U}\mathfrak{g}\otimes \mathcal{U}\mathfrak{g}\otimes V\otimes W)
\ar[d, "\cong"]
\\
N^\ast(\mathcal{U}\mathfrak{g}\otimes V)\otimes
N^\ast(\mathcal{U}\mathfrak{g}\otimes W)
\ar[d, "N^\ast\rho_V\otimes N^\ast \rho_W"']
\ar[rr, "\nabla\;\;"]
&&
N^\ast(\mathcal{U}\mathfrak{g}\otimes V
\otimes \mathcal{U}\mathfrak{g}\otimes W)
\ar[d, "N^\ast(\rho_V\otimes \rho_W)"]
\\
N^\ast V\otimes N^\ast W
\ar[rr, "\;\;\nabla"]
&&
N^\ast(V\otimes W)\,.
\end{tikzcd}
\]
The left vertical composite exhibits the $s^\infty\mathcal{U}N\mathfrak{g}$-action on $N^\ast V\otimes N^\ast W$, whereas the right vertical composite exhibits the $s^\infty\mathcal{U}N\mathfrak{g}$-action on $N^\ast (V\otimes W)$.
Commutativity of the diagram shows that the stabilised shuffle map $\nabla\colon N^\ast V\otimes N^\ast W\to N^\ast(V\otimes W)$ is a morphism of $s^\infty\mathcal{U}N\mathfrak{g}$-modules.
Using the multiplicativity of $\chi_\mathfrak{g}$ and the map \eqref{eqn:NormalisationMonoid}, we deduce from this that $N^\ast_\mathfrak{g}$ is lax symmetric monoidal.

By adjointness (compare Lemma  \ref{lem:BaseChangeforSimplHopf}) the left adjoint $N^\mathfrak{g}_!$ is oplax symmetric monoidal, with structure maps
\begin{align*}
\Lambda_{A,B}\colon N_!^\mathfrak{g}(A\otimes B)&\longrightarrow N_!^\mathfrak{g}A \otimes N_!^\mathfrak{g} B
\\
\lambda\colon N^\mathfrak{g}_! s^\infty \mathbb{Q} & 
\longrightarrow 
\Sigma^\infty_\mathbb{Q} \mathbb{Q}\,.
\end{align*}
Writing $\widetilde{\Omega}^{\infty-k}\colon \mathfrak{g}\mathrm{-Rep}_\Delta\to \mathfrak{g}\mathrm{-Rep}_\Delta^u$ and $\ell^{\infty-k}\colon N\mathfrak{g}\mathrm{-Rep}_\mathrm{dg}^\Sigma \to N\mathfrak{g}\mathrm{-Rep}_\mathrm{dg}^+$ for the functors that extract the $k$-th space of a symmetric spectrum (right adjoint to $\Sigma^{\infty-k}_\mathbb{Q}$ and $s^{\infty-k}$ respectively), we have for each $k\geq 0$ a commuting diagram of functors
\[
\begin{tikzcd}
\mathfrak{g}\mathrm{-Rep}_\Delta
\ar[r," \widetilde{\Omega}^{\infty-k}"]
\ar[d, "N^\ast_\mathfrak{g}"']
&
\mathfrak{g}\mathrm{-Rep}^u_\Delta
\ar[r]
\ar[d]
&
\mathrm{sVect}_\mathbb{Q}
\ar[d, "N"]
\\
N\mathfrak{g}\mathrm{-Rep}_\mathrm{dg}^\Sigma
\ar[r, "\ell^{\infty-k}"]
&
N\mathfrak{g}\mathrm{-Rep}^+_\mathrm{dg}
\ar[r]
&
\mathrm{Ch}_+\,.
\end{tikzcd}
\]
For each $k\geq 0$ and non-negatively graded chain complex $M$, there is a natural isomorphism
$
N^\mathfrak{g}_! \big(s^{\infty-k}(\mathcal{U}N\mathfrak{g}\otimes M)\big) \cong \Sigma^{\infty-k}_\mathbb{Q} \big( \mathcal{U}\mathfrak{g}\otimes \Gamma M\big)
$
by essential uniqueness of adjoints.
For $k,l \geq 0$ and $A,B\in \mathrm{Ch}_+$, the smash-tensor product of $s^{\infty-k}(\mathcal{U}N\mathfrak{g}\otimes A)$ and $s^{\infty-l} (\mathcal{U}N\mathfrak{g}\otimes B)$ in $\mathrm{Sp}^\Sigma(\mathrm{Ch}_+)$ is isomorphic to
\[
s^{\infty-(k+l)}
\big(\mathcal{U}N\mathfrak{g}\otimes \mathcal{U}N\mathfrak{g}\otimes A\otimes B
\big)
\cong 
s^{\infty-(k+l)}
\big(\mathcal{U}N\mathfrak{g}\otimes( \mathcal{U}N\mathfrak{g}\otimes \mathcal{U}N\mathfrak{g})_{\mathcal{U}N\mathfrak{g}} \otimes A\otimes B\big)
\]
using the evident dg version of Lemma \ref{lem:FreeHopf}.
For such objects the oplax structure map is  isomorphic to a map of the form
\begin{equation}
\label{eqn:FreeRectComp}
\Sigma^{\infty-(k+l)}_\mathbb{Q}
\Big(
\mathcal{U}\mathfrak{g}\otimes\Gamma\big(( \mathcal{U}N\mathfrak{g}\otimes \mathcal{U}N\mathfrak{g})_{\mathcal{U}N\mathfrak{g}} \otimes A\otimes B\big)
\xrightarrow{\;\mathrm{id}\otimes \overline{\chi}_\mathfrak{g}^{A,B}\;}
\mathcal{U}\mathfrak{g}\otimes (\mathcal{U}\mathfrak{g}\otimes \mathcal{U}\mathfrak{g})_{\mathcal{U}\mathfrak{g}}\otimes \Gamma A \otimes \Gamma B\Big)\,.
\end{equation}
Let us examine the morphism $\chi_\mathfrak{g}^{A,B}$ in more detail.
Comultiplicativity of $\chi_\mathfrak{g}$ (Lemma \ref{lem:Chi}) implies that $\chi_\mathfrak{g}^{\otimes n}$ preserves coinvariant submodules for each $n\geq 0$, in the sense that there is a commuting diagram of quasi-isomorphisms
\begin{equation}
\label{eqn:ChiBar}
\begin{tikzcd}[row sep =tiny]
\mathcal{U}N\mathfrak{g}^{\otimes n}
\ar[r, "\chi_\mathfrak{g}^{\otimes n}"]
\ar[d, equal, "\wr"']
&
N\mathcal{U}\mathfrak{g}^{\otimes n}
\ar[r, "\nabla"]
&
N(\mathcal{U}\mathfrak{g}^{\otimes n})
\ar[d, equal, "\wr"]
\\
\mathcal{U}N\mathfrak{g}\otimes \mathcal{U}N\mathfrak{g}^{\otimes n}_{\mathcal{U}N\mathfrak{g}}
\ar[r, "\chi_\mathfrak{g}\otimes \overline{\chi}^n_\mathfrak{g}"]
&
N\mathcal{U}\mathfrak{g}
\otimes 
N\big(
\mathcal{U}\mathfrak{g}^{\otimes n}_{\mathcal{U}\mathfrak{g}}
\big)
\ar[r, "\nabla"]
&
N\big(
\mathcal{U}\mathfrak{g}
\otimes 
\mathcal{U}\mathfrak{g}^{\otimes n}_{\mathcal{U}\mathfrak{g}}
\big)\,,
\end{tikzcd}
\end{equation}
in which the vertical isomorphisms come from the fundamental theorem of Hopf modules (in the simplicial and dg settings respectively).
The morphism in parantheses in \eqref{eqn:FreeRectComp} is the composite of $\Gamma(\overline{\chi}^2_\mathfrak{g}\otimes A\otimes B)$ with the oplax monoidal transformation $\Phi$.
Applying the normalisation functor, we get a commuting diagram
\[
\begin{tikzcd}[row sep =small]
N\Big(\mathcal{U}\mathfrak{g}\otimes\Gamma\big(( \mathcal{U}N\mathfrak{g}\otimes \mathcal{U}N\mathfrak{g})_{\mathcal{U}N\mathfrak{g}} \otimes A\otimes B\big)\Big)
\ar[d, "N (\mathrm{id}\otimes\overline{\chi}_\mathfrak{g}^{A,B})"']
\ar[r]
&
N\mathcal{U}\mathfrak{g}\otimes 
( \mathcal{U}N\mathfrak{g}\otimes \mathcal{U}N\mathfrak{g})_{\mathcal{U}N\mathfrak{g}} \otimes A\otimes B
\ar[d, "\nabla\circ (\mathrm{id}\otimes \overline{\chi}^2_\mathfrak{g}\otimes \mathrm{id})"]
\\
N\big(
\mathcal{U}\mathfrak{g}\otimes (\mathcal{U}\mathfrak{g}\otimes \mathcal{U}\mathfrak{g})_{\mathcal{U}\mathfrak{g}}\otimes \Gamma A \otimes \Gamma B
\big)
\ar[r]
&
N\big(\mathcal{U}\mathfrak{g}\otimes (\mathcal{U}\mathfrak{g}\otimes \mathcal{U}\mathfrak{g})_{\mathcal{U}\mathfrak{g}}\big)
\otimes  A\otimes B\,,
\end{tikzcd}
\]
in which the horizontal arrows are the evident composites of the shuffle map and the natural isomorphism $N\Gamma\cong \mathrm{id}$.
Each arrow in this diagram is a quasi-isomorphism:
 for the horizontal arrows this is due to the properties of the shuffle map, the right hand vertical arrow is a quasi-isomorphism by \eqref{eqn:ChiBar}, and the left hand vertical arrow is a quasi-isomorphism by the 2-out-of-3 property.
Normalisation preserves and reflects weak equivalences, so that $\mathrm{id}\otimes \overline{\chi}_\mathfrak{g}^{A,B}$ is a weak equivalence of simplicial vector spaces.
Domain and codomain of this map are cofibrant (in both $\mathfrak{g}\mathrm{-Rep}^u_\Delta$ and $\mathrm{sVect}_\mathbb{Q}$) so that \eqref{eqn:FreeRectComp} is a weak equivalence by Ken Brown's lemma.

The oplax monoidal transformation $\Lambda_{A,B}$ is thus a weak equivalence for $A$, $B$ domains or codomains of generating cofibrations\footnote{We recall that the generating cofibrations of $N\mathfrak{g}\mathrm{-Rep}_\mathrm{dg}^\Sigma$ are obtained by applying the functors $s^{\infty-k}\circ (\mathcal{U}N\mathfrak{g}\otimes(-))$, $k\geq 0$, to the generating cofibrations of $\mathrm{Ch}_+$.}.
Following the argument of Lemma \ref{lem:BaseChangeforSimplHopf}, we deduce that $\Lambda_{A,B}$ is a weak equivalence for all cofibrant $A, B\in N\mathfrak{g}\mathrm{-Rep}_\mathrm{dg}^\Sigma$.

To complete the proof, we must verify that for some (hence any) cofibrant replacement of the monoidal unit $\mathbb{Q}^c\to s^\infty\mathbb{Q}$ in $N\mathfrak{g}\mathrm{-Rep}_\mathrm{dg}^\Sigma$, the composite with the oplax monoidal transformation
$
N_!^\mathfrak{g}\mathbb{Q}^c \to N^\mathfrak{g}_! s^\infty \mathbb{Q}\to
\Sigma^\infty_\mathbb{Q} \mathbb{Q}
$
is a weak equivalence.
Taking $\mathbb{Q}^c= \mathbb{Q}^{N\mathfrak{g}}$ as in \eqref{eqn:CofibUnit} and using the evident dg version of Corollary \ref{cor:HopfPowers}, we calculate
\[
N_!^\mathfrak{g} \mathbb{Q}^{N\mathfrak{g}}
\cong \underset{\Delta^\mathrm{op}}{\mathrm{colim}}\, \Sigma^\infty \left[\mathcal{U}\mathfrak{g} \otimes \Gamma \big(\mathcal{U}N\mathfrak{g}^{\otimes (n+1)}_{\mathcal{U}N\mathfrak{g}}\big)\right]\,.
\]
The maps
$
\mathrm{id}\otimes \Gamma \overline{\chi}^{n+1}_\mathfrak{g}
\colon
\mathcal{U}\mathfrak{g} \otimes \Gamma \big(\mathcal{U}N\mathfrak{g}^{\otimes (n+1)}_{\mathcal{U}N\mathfrak{g}}\big)
\to
\mathcal{U}\mathfrak{g} \otimes \Gamma N \big(\mathcal{U}\mathfrak{g}^{\otimes (n+1)}_{\mathcal{U}\mathfrak{g}}\big)
\cong
\mathcal{U}\mathfrak{g}^{\otimes (n+1)}
$
together determine a levelwise weak equivalence of Reedy cofibrant simplicial objects, so taking colimits and applying the left Quillen functor $\Sigma^\infty_\mathbb{Q}$ gives a weak equivalence $N_!^\mathfrak{g}\mathbb{Q}^{N\mathfrak{g}}\to \Sigma^\infty_\mathbb{Q} \mathbb{Q}^{\mathcal{U}\mathfrak{g}}$ (cf~Construction \ref{cons:HopfCofibUnit}).
But $\Sigma^\infty_\mathbb{Q} \mathbb{Q}^{\mathcal{U}\mathfrak{g}}\to \Sigma^\infty_\mathbb{Q} \mathbb{Q}$ is a cofibrant replacement of the monoidal unit in $\mathfrak{g}\mathrm{-Rep}_\Delta$ and the composite $N_!^\mathfrak{g}\mathbb{Q}^{N\mathfrak{g}}\to \Sigma^\infty_\mathbb{Q} \mathbb{Q}$ factors through the oplax monoidal structure map $\lambda$ by construction.
This completes the proof that $(N^\mathfrak{g}_!\dashv N^\ast_\mathfrak{g})$ is a weakly monoidal Quillen equivalence.
\end{proof}

We will also need the following monoidal base change result:
\begin{lemma}
\label{lem:QIsodgLieBChange}
Let $f\colon \mathfrak{h}\to \mathfrak{k}$ be a quasi-isomorphism of dg Lie algebras.
Then there is a weakly monoidal Quillen equivalence
\[
\begin{tikzcd}
\mathfrak{h}\mathrm{-Rep}_\mathrm{dg}
\ar[rr, shift left =1.1ex, "\mathcal{U}f_!"]
\ar[rr, leftarrow, shift left=-1.1ex, "\mathcal{U}f^\ast"', "\bot"] 
&&
\mathfrak{k}\mathrm{-Rep}_\mathrm{dg}\,.
\end{tikzcd}
\]
\end{lemma}
\begin{proof}
By Lemma \ref{lem:UniEnvAlgQIso} below, the induced map of universal enveloping algebras $\mathcal{U}f\colon \mathcal{U}\mathfrak{h}\to \mathcal{U}\mathfrak{k}$ is a quasi-isomorphism.
The proof Lemma \ref{lem:SimplHopfModMonoidal} carries over to the differential graded setting. 
\end{proof}
\begin{lemma}
\label{lem:UniEnvAlgQIso}
A quasi-isomorphism of dg Lie algebras $f\colon\mathfrak{h}\to \mathfrak{k}$ induces a weak equivalence $\mathcal{U}\mathfrak{g}\colon \mathcal{U}\mathfrak{h}\to \mathcal{U}\mathfrak{k}$ of universal enveloping algebras.
\end{lemma}
\begin{proof}
For any rational chain complex $M$, the K\"{u}nneth theorem implies a natural isomorphism $SH_\bullet M \to H_\bullet SM$, where $S$ is the free symmetric algebra functor.
Using the Poincar\'{e}--Birkhoff--Witt theorem and the universal property of the universal enveloping algebra, there is a commuting diagram of graded rational vector spaces 
\[
  \begin{tikzcd}[column sep=small, row sep=small]
  &
  SH_\bullet\mathfrak{k}
   \ar[rr]\ar[dd] && 
  \mathcal{U}H_\bullet \mathfrak{k}
  \ar[dd]
  \\
  SH_\bullet \mathfrak{h}
  \ar[ur]
  \ar[rr, crossing over] 
  \ar[dd]
  &&
  \mathcal{U}H_\bullet\mathfrak{h}
  \ar[ur] 
  &
  \\
  & 
  H_\bullet S\mathfrak{k}
  \ar[rr]
  && 
  H_\bullet \mathcal{U}\mathfrak{k}
  \\
  H_\bullet S\mathfrak{h}
  \ar[rr]
  \ar[ur]&&
  H_\bullet \mathcal{U}\mathfrak{h}
  \ar[ur, "H_\bullet \mathcal{U}f"']
  \ar[from=uu, crossing over]&
  \end{tikzcd}
  \]
in which all horizontal arrows are PBW isomorphisms or the image of such in homology, so that each arrow in the front and back face of this diagram is an isomorphism.
The hypothesis on $f$ implies that $SH_\bullet \mathfrak{h}\to SH_\bullet\mathfrak{k}$ is an isomorphism, from which it follows that $H_\bullet\mathcal{U}f$ is an isomorphism too.
\end{proof}

\begin{lemma}
\label{lem:LieRepCompPNat}
Let $f\colon \mathfrak{g}\to \mathfrak{l}$ be a morphism of simplicial Lie algebras.
Then there is a diagram of left Quillen functors 
\[
\begin{tikzcd}
\mathfrak{g}\mathrm{-Rep}_\Delta
\ar[r, leftarrow]
\ar[d]
&
N\mathfrak{g}\mathrm{-Rep}_\mathrm{dg}^\Sigma
\ar[r]
\ar[d]
&
N\mathfrak{g}\mathrm{-Rep}_\mathrm{dg}
\ar[d]
\\
\mathfrak{l}\mathrm{-Rep}_\Delta
\ar[r, leftarrow]
&
N\mathfrak{l}\mathrm{-Rep}_\mathrm{dg}^\Sigma
\ar[r]
&
N\mathfrak{l}\mathrm{-Rep}_\mathrm{dg}
\end{tikzcd}
\]
commuting up to natural isomorphism.
If $f$ is a weak equivalence then all of the above arrows are Quillen equivalences.
\end{lemma}
\begin{proof}
For both the left- and right-hand squares is it easier to argue with right adjoints.
We spell this out for the left-hand square, the right-hand square is similar.
For both of the base change adjunctions 
\[
(\mathcal{U}f_!\dashv \mathcal{U}f^\ast) \colon\mathfrak{g}\mathrm{-Rep}_\Delta\to \mathfrak{l}\mathrm{-Rep}_\Delta
\quad 
\mbox{ and }
\quad 
(\mathcal{U}Nf_!\dashv \mathcal{U}Nf^\ast)\colon \mathcal{U}N\mathfrak{g}\mathrm{-Rep}_\mathrm{dg}^\Sigma\to \mathcal{U}N\mathfrak{l}\mathrm{-Rep}_\mathrm{dg}^\Sigma\,,
\] 
the right adjoint commutes with the forgetful functors (to $\mathrm{Sp}^\Sigma(\mathrm{sVect}_\mathbb{Q})$ and $\mathrm{Sp}^\Sigma(\mathrm{Ch}_+)$ respectively).
As a consequence, one checks the diagram of right adjoints
\[
\begin{tikzcd}
\mathfrak{l}\mathrm{-Rep}_\Delta
\ar[r, "\mathcal{U}f^\ast"]
\ar[d, "N^\ast_\mathfrak{l}"']
&
\mathfrak{g}\mathrm{-Rep}_\Delta
\ar[r]
\ar[d, "N^\ast \mathfrak{g}"']
&
\mathrm{Sp}^\Sigma(\mathrm{sVect}_\mathbb{Q})
\ar[d, "N^\ast"]
\\
N\mathfrak{l}\mathrm{-Rep}^\Sigma_\mathrm{dg}
\ar[r, "\mathcal{U}Nf^\ast"]
&
N\mathfrak{g}\mathrm{-Rep}^\Sigma_\mathrm{dg}
\ar[r]
&
\mathrm{Sp}^\Sigma(\mathrm{Ch}_+)
\end{tikzcd}
\]
commutes up to natural isomorphism.
Adjointness implies that the corresponding diagram of left adjoints also commutes up to natural isomorphism.

All horizontal arrows in the diagram are Quillen equivalences by the the results of the previous section.
If $f$ is a weak equivalence, then $Nf \colon N\mathfrak{g}\to N\mathfrak{l}$ is a quasi-isomorphism and the middle vertical arrow represents a Quillen equivalence by Lemma \ref{lem:QIsodgLieBChange}.
By the 2-out-of-3 property, all arrows are Quillen equivalences in this case.
\end{proof}

\section{Koszul duality}
\label{sec:Koszul}
In this section we study the \emph{Koszul duality} between $C$-comodules and modules over the cobar construction $\Omega C$, where $C$ is a $2$-reduced cocommutative dg coalgebra.
Using twisted extensions and their induced spectral sequences we prove Quillen equivalences between model categories of bounded below modules and comodules.
In the unbounded case these spectral sequences generally fail to converge and we must consider the homotopy category of unbounded comodules with respect to a relation that is more discerning than quasi-isomorphism.
We do not claim any great originality for these results---see \cite{positselski_two_2011}, for example, for an alternative approach to the results of Sections \ref{subsec:TwistingChains} and \ref{subsec:KQEquiv}.

\subsection{Twisting chains}
\label{subsec:TwistingChains}
Let $C$ be a non-negatively graded rational dg coalgebra with comultiplication $\Delta$ and counit $\epsilon$.
We shall take $C$ to be strictly coassociative and cocommutative, and we furthermore suppose that we are given a coaugmentation $\eta\colon \mathbb{Q}\to C$.
After some general remarks, we shall impose the additional hypothesis that $C$ is $2$-reduced, that is $C_0 =\mathbb{Q}$ and $C_1 =0$, in which case the coaugmentation is canonical.

\begin{construction}[Cobar construction]
\label{cons:Cobar}
Write $\overline{C} =\mathrm{coker}(\eta)$ with $\pi\colon C\to \overline{C}$ the quotient map.
The \emph{cobar construction} of $C$ is the dg algebra $\Omega C$ with underlying graded algebra $T\overline{C}[1]$.
The differential is defined on elements of the form $tc:=s^{-1}\pi(c)\in \overline{C}[1]$ by
\[
d(tc) = -t (dc) - (-1)^{|c_{(0)}|} tc_{(0)}\otimes tc_{(1)}
\]
and extended to all of $\Omega C$ as an algebra derivation.
The cobar construction has a coproduct defined on algebra generators by $tc\mapsto tc\otimes 1+1\otimes tc$ and extended to all of $\Omega C$ as an algebra homomorphism, with counit determined by the assignment $tc\mapsto 0$.
Since the coproduct of $C$ is coassociative and cocommutative, $\Omega C$ is a dg Hopf algebra, which is non-negatively graded precisely if $C$ is \emph{reduced} (that is, $C_0 =\mathbb{Q}$).
\end{construction}
\begin{definition}
For $A$ a dg algebra, a \emph{twisting chain} $\tau\colon C\rightsquigarrow A$ is a map of graded rational vector spaces of degree $-1$ such that
\[
d(\tau c) = -\tau (dc) -(-1)^{|c_{(0)}|} \tau(c_{(0)})\cdot \tau(c_{(1)})
\]
for all $c\in C$.
The collection of all twisting chains $\tau\colon C\rightsquigarrow A$ forms a rational vector space $\mathrm{Tw}(C,A)$.
\end{definition}
\begin{remark}
It is easy to check that there is a natural bijection $\mathrm{Tw}(C,A)\cong \mathrm{dgAlg}(\Omega C, A)$.
Under this bijection, the identity map on $\Omega C$ corresponds to the \emph{universal twisting chain} $t\colon C\rightsquigarrow \Omega C$ sending $c\mapsto tc$ (compare Construction \ref{cons:Cobar}).
\end{remark}
\begin{construction}
\label{cons:TwistedExt}
Fix a dg algebra $A$ and a twisting chain $\tau\colon C\rightsquigarrow A$. 
For a (left) $A$-module $M$, the \emph{$\tau$-twisted coextension of $M$} is the graded vector space $C\otimes M$ equipped with the differential defined on homogeneous elements by
\[
d(c\otimes m)  = dc\otimes m + (-1)^{|c|}c\otimes dm + (-1)^{|c_{(0)}|} c_{(0)}\otimes \tau c_{(1)}\cdot m\,.
\]
The resulting chain complex is denoted $\tau_\ast M$, or equivalently $C\otimes^\tau M$, and inherits the structure of a left (dg) $C$-comodule from the coproduct of $C$.
The assignment $M\mapsto \tau_\ast M$ defines a functor $\tau_\ast \colon A\mathrm{-Mod}\to C\mathrm{-Comod}$.

For a (left) $C$-comodule $N$, the \emph{$\tau$-twisted extension} is the graded vector space $A\otimes N$ equipped with differential defined on homogeneous elements by
\[
d(a\otimes n) = da \otimes n + (-1)^{|a|} a\otimes dn - (-1)^{|a|} a\cdot \tau n_{(0)}\otimes n_{(1)}\,.
\]
The resulting chain complex is denoted $\tau^! N$ or $A\otimes_\tau N$, and is a (left) dg $A$-module in the obvious way.
The assignment $N\mapsto \tau^! N$ determines a functor $\tau^! \colon C\mathrm{-Comod}\to A\mathrm{-Mod}$.
\end{construction}
\begin{proposition}
For any twisting chain $\tau\colon C\rightsquigarrow A$ there is an adjunction
\[
\begin{tikzcd}
C\mathrm{-Comod}
\ar[rr, shift left =1.1ex, "\tau^!"]
\ar[rr, shift left =-1.1ex, leftarrow, "\tau_\ast"', "\bot"]
&&
A\mathrm{-Mod}
\end{tikzcd}
\]
given by forming $\tau$-twisted extensions and coextensions.
If $A$ is connective, this adjunction restricts to an adjunction between subcategories of $k$-connective objects
\[
\begin{tikzcd}
C\mathrm{-Comod}_{\geq k}
\ar[rr, shift left =1.1ex, "\tau^!"]
\ar[rr, shift left =-1.1ex, leftarrow, "\tau_\ast"', "\bot"]
&&
A\mathrm{-Mod}_{\geq k}
\end{tikzcd}
\]
for any $k\in \mathbb{Z}$.
\end{proposition}
\begin{proof}
For $N\in C\mathrm{-Comod}$ and $M\in A\mathrm{-Mod}$ a modular $\tau$-twist $\theta\colon N\rightsquigarrow M$ is a map of graded vector spaces $\theta\colon N\to M$ such that $d\theta(n) =\theta(dn) - \tau n_{(0)}\cdot \theta (n_{(1)})$ for all $n\in N$.
Writing $\mathrm{MTw}_\tau (N, M)$ for the rational vector space of modular $\tau$-twists $N\rightsquigarrow M$, we get have a bifunctor
\[
\mathrm{MTw}_\tau(-,-)\colon C\mathrm{-Comod}^\mathrm{op}\times A\mathrm{-Mod}\longrightarrow \mathrm{Vect}_\mathbb{Q}\,.
\]
Given a map of $A$-modules $f\colon \tau^! N\to M$, setting $\theta[f](n) := f(1\otimes n)$ determines a modular $\tau$-twist $\theta[f]\colon N\rightsquigarrow M$.
Conversely, given a modular $\tau$-twist $\theta\colon N\rightsquigarrow M$ a straightforward calculation verifies that $a\otimes n \mapsto a\cdot \theta(n)$ defines a morphism $f[\theta]\colon \tau^! N\to M$ in $A\mathrm{-Mod}$.
The assignments $\theta\mapsto f[\theta]$ and $f\mapsto \theta[f]$ are inverse to one another, so we get a natural isomorphism of bifunctors $A\mathrm{-Mod}(\tau^! N, M)\cong \mathrm{MTw}_\tau (N, M)$.

On the other hand, after forgetting differentials $\tau_\ast M$ is the cofree graded $C$-comodule cogenerated by $M$.
It follows that any map of graded $C$-comodules $g\colon N \to C\otimes M$ is necessarily of the form
$
n\mapsto n_{(0)} \otimes \theta(g) (n_{(1)})
$ 
for some map of graded rational vector spaces $\theta(g)\colon N\to M$.
One shows that $g\colon N\to \tau_\ast M$ respects the differentials if and only if $\theta(g)$ is a modular $\tau$-twist.
This establishes a natural isomorphism $\mathrm{MTw}_\tau (N,M) \cong C\mathrm{-Comod}(N, \tau_\ast M)$, which together with the above implies an adjunction $(\tau^!\dashv \tau_\ast)$.

Finally, we remark that $\tau_\ast$ preserves $k$-connectivity since $C$ is connective.
If $A$ is connective then the functor $\tau^!$ also preserves $k$-connectivity, hence the $(\tau^!\dashv \tau_\ast)$-adjunction restricts to full subcategories of $k$-connective objects.
\end{proof}

Throughout the remainder of this section, we assume that the dg coalgebra $C$ is moreover $2$-reduced.
This assumption allows us to construct spectral sequences for $\tau$-twisted extensions and coextensions with particularly nice $E^2$-pages:
\begin{lemma}
\label{lem:TwistingSS}
Let $C$ be a 2-reduced cocommuative dg coalgebra and $\tau\colon C\rightsquigarrow A$ a twisting chain with $A$ connective.
Let $M\in A\mathrm{-Mod}$ and $N\in C\mathrm{-Comod}$ be bounded below.
Then there are convergent spectral sequences
\begin{align*}
H_q(C) \otimes H_q(M) 
&
\Longrightarrow
H_{p+q}(C\otimes^\tau M)
\\
H_k(A) \otimes H_l(N) 
&
\Longrightarrow
H_{k+l}(A\otimes_\tau N)\,.
\end{align*} 
\end{lemma}
\begin{proof}
For the first spectral sequence, consider the filtration $\mathcal{F}_p C :=\bigoplus_{n\leq p} C_n$, which induces the differential filtration $\mathcal{F}_p C\otimes M$ of $\tau_\ast M = C\otimes^\tau M$.
Our hypotheses on $A$ and $C$ imply that the last term in the formula
\[
d(c\otimes m) = dc\otimes m + (-1)^{|c|} c\otimes dm + (-1)^{|c_{(0)}|} c_{(0)}\otimes \tau c_{(1)}\cdot m
\]
is a sum over terms $c_i\otimes m_i$ for which $|c_i|\leq |c|-2$.
We therefore have
$
E^1_{p,q}\cong C_p \otimes H_q(M)
$ and $
E^2_{p,q} \cong H_p(C) \otimes H_q(M)$.
The boundedness assumption on $M$ guarantees convergence of the spectral sequence.

The second spectral sequence arises from the differential filtration $A\otimes \mathcal{F}_k N$ of $\tau^! N$, where $\mathcal{F}_k N := \bigoplus_{p\leq k} N_p$.
Examining the formula for the differential on $\tau^! N$
\[
d(a\otimes n) = da \otimes n + (-1)^{|a|}a\otimes dn + (-1)^{|a|} a\cdot \tau n_{(0)} \otimes n_{(1)}\,,
\] 
our hypotheses on $A$ and $C$ imply that the last term is a sum over elements $a_i\otimes n_i$ such that $|n_i|\leq |n|- 2$.
We can thus identify 
$
E^1_{k,l}\cong H_k(A) \otimes N_l
$ and $
E^2_{k,l} \cong H_k(A) \otimes H_l(N)$
.
Convergence of the spectral sequence is guaranteed by the boundedness assumption on $N$.
\end{proof}

Further specialising to the case that $\tau =t\colon C\rightsquigarrow \Omega C$ is the universal twisting chain, we can say more about the unit and counit of the $(t^!\dashv t_\ast)$-adjunction:
\begin{lemma}
\label{lem:KoszulCounit}
For $C$ a 2-reduced cocommutative dg coalgebra, the counit  $t^!t_\ast M\to M$ is a natural quasi-isomorphism of $\Omega C$-modules.
\end{lemma}
\begin{proof}
%We first remark that for any $\Omega C$-module $M$, the counit $\Omega C\otimes_\tau (C\otimes^\tau M)\to M$ sends 
%\[
%(tc_1\otimes \dotsb \otimes tc_k)\otimes c\otimes m
%\longmapsto
%tc_1\dotsb tc_k \cdot \epsilon(c)m\,.
%\]
The canonical augmentation $\epsilon_{\Omega C}\colon\Omega C\to \mathbb{Q}$ equips $\mathbb{Q}$ with the structure of an $\Omega C$-module.
The $(t^!\dashv t_\ast)$-counit $\eta_\mathbb{Q}\colon t^! t_\ast \mathbb{Q} = \Omega C\otimes_t C \to \mathbb{Q}$ sends is given by $\epsilon_{\Omega C}\otimes \epsilon_C$.
The chain homotopy 
\[
h\colon (tc_1\otimes \dotsb tc_k)\otimes c
\longmapsto
\begin{cases}
\qquad\qquad\qquad\qquad \quad 0& k=0\\
(-1)^{\sum_{i=1}^{k-1} |tc_i|}
(tc_1\otimes\dotsb \otimes tc_{k-1})\otimes \pi(tc_{k})\epsilon(c)
& \mbox{otherwise}
\end{cases}
\] 
shows that $\epsilon_{\mathbb{Q}}$ is a quasi-isomorphism.

Suppose now that the $\Omega C$-module $M$ is bounded below and consider the differential filtration $\mathcal{F}_p:= \bigoplus_{n\leq p}  (\Omega C\otimes C)_n \otimes M$ on $t^! t_\ast M = \Omega C \otimes_{t}C \otimes^{t} M$.
Since $C$ is 2-reduced and $M$ is bounded below, a consideration of the differential as in Lemma \ref{lem:TwistingSS} leads to a convergent spectral sequence
\[
E^2_{p,q} = H_p( \Omega C\otimes_{t} C)\otimes H_q(M) \Longrightarrow H_{p+q}(\Omega C \otimes_{t}C \otimes^{t} M)\,.
\]
The acyclicity of $\Omega C\otimes_{t} C$ thus implies that the counit $t^! t_\ast M\to M$ is a quasi-isomorphism.

The boundedness assumption on $M$ is removed by working stage-by-stage over the complicial Whitehead tower.
As $\Omega C$ is connective, for any $\Omega C$-module $M$ the $k$-connective cover $c_k M$ inherits a canonical $\Omega C$-module structure and we have $M\cong \mathrm{colim}_{k\to -\infty} c_k M$ as $\Omega C$-modules.
The endofunctor $t^! t_\ast \colon \Omega C\mathrm{-Mod}\to \Omega C\mathrm{-Mod}$ preserves filtered colimits, so that the counit map $t^! t_\ast M \cong \mathrm{colim}_{k\to -\infty} t^! t_\ast c_k M \to \mathrm{colim}_{k\to -\infty}  c_k M \cong M$ is a quasi-isomorphism. 
\end{proof}
\begin{lemma}
\label{lem:KoszulUnit}
For $C$ a 2-reduced  cocommutative dg coalgebra, the unit $N\to t_\ast t^! N$ is a natural quasi-isomorphism of $C$-comodules.
\end{lemma}
\begin{proof}
The coaugmentation $\mathbb{Q}\to C$ furnishes $\mathbb{Q}$ with the structure of a (left) $C$-comodule and the $(t^! \dashv t_\ast)$-unit $\eta_\mathbb{Q}\colon \mathbb{Q}\to t_\ast t^! \mathbb{Q} = C\otimes^\tau \Omega C$ sends $1\mapsto 1\otimes 1$.
The chain homotopy
\[
h\colon c\otimes (tc_1\otimes \dotsb \otimes tc_k)
\longmapsto 
\begin{cases}
\qquad\qquad \qquad \quad 0& k =0
\\
(-1)^{|c|}\epsilon(c) \pi(c_1) \otimes (tc_2 \otimes \dotsb\otimes tc_k)
& \mbox{otherwise}
\end{cases}
\]
shows that $\eta_\mathbb{Q}$ is a quasi-isomorphism.

For a bounded below $C$-comodule $N$, the unit $\eta_N\colon N\to t_\ast t^\ast N$ sends $n\mapsto n_{(0)}\otimes 1 \otimes n_{(1)}$.
The differential filtration $\mathcal{F}_k N :=\bigoplus_{l\leq k} N_l$ of $N$ induces a differential filtration $C\otimes \Omega C \otimes \mathcal{F}_k N$ of $t_\ast t^! N$ and the unit $\eta_N$ becomes a map of filtered complexes.
By our hypotheses on $C$ and $N$ we get a morphism of convergent spectral sequences
\[
\begin{tikzcd}[row sep = small]
H_p(\mathbb{Q}) \otimes H_q(N) 
\ar[r, Rightarrow]
\ar[d, "\cong"']
&
H_{p+q}(N)
\ar[d]
\\
H_p(t_\ast t^\ast \mathbb{Q}) \otimes H_q(N) 
\ar[r, Rightarrow]
&
H_{p+q}(t_\ast t^! N)\,,
\end{tikzcd}
\]
so that $N\to t_\ast t^! N$ is a quasi-isomorphism by Zeeman's comparison theorem.

We remove the boundedness condition by working with an analogue of the Whitehead tower for comodules.
For any $k\in \mathbb{Z}$, the fully faithful embedding $C\mathrm{-Comod}_{\geq k}\hookrightarrow C\mathrm{-Comod}$ preserves colimits (which are in either case created by the forgetful functors).
The comonad $V\mapsto C\otimes V$ preserves colimits so the categories $C\mathrm{-Comod}$ and $C\mathrm{-Comod}_{\geq k}$ are locally presentable by  \cite[Proposition A.1]{ching_coalgebraic_2014}.
The adjoint functor theorem supplies a right adjoint $c_{k;C}\colon C\mathrm{-Comod}\to C\mathrm{-Comod}_{\geq k}$; explicitly, this functor sends the comodule $\rho\colon N\to C\otimes N$ to the equaliser of the diagram of cofree $k$-connective $C$-comodules
\[
c_{k;C}N \cong \mathrm{lim}
\left(\!
\begin{tikzcd}
C\otimes c_k N
\ar[r, shift left =0.8ex, "C\otimes c_k\rho"]
\ar[r, shift left = -0.8ex, "\alpha_k"']
&
C\otimes c_k (C\otimes N)
\end{tikzcd}\!
\right)\,,
\]
where $c_k\colon \mathrm{Ch}\to \mathrm{Ch}_{\geq k}$ computes the usual $k$-connective cover and $\alpha_k$ is the composite 
\[
\begin{tikzcd}
C\otimes c_k N
\ar[r, "\Delta \otimes c_k N"]
&
C\otimes C\otimes c_k N
\ar[r]
&
C\otimes c_k(C\otimes N)\,,
\end{tikzcd}
\]
where the monomorphism $C\otimes c_k N\to c_k(C\otimes N)$ is a consequence of the connectivity of $C$.
The functor $C\otimes (-)$ is exact so that finite limits in $C\mathrm{-Comod}$ (and $C\mathrm{-Comod}_{\geq k}$) are created by the forgetful functor.
Finite limits commute with filtered colimits in $\mathrm{Ch}$ (and $\mathrm{Ch}_{\geq k}$), hence there are isomorphisms of $C$-comodules
\[
\underset{k\to -\infty}{\mathrm{colim}}\;c_{k;C} N
\cong 
\mathrm{lim}
\left(\!
\begin{tikzcd}
C\otimes N
\ar[r, shift left =0.8ex, "C\otimes \rho"]
\ar[r, shift left = -0.8ex, "\Delta \otimes N"']
&
C\otimes C\otimes N
\end{tikzcd}\!
\right)\cong N
\]
by comonadicity.
The endofunctor $t_\ast t^! $ commutes with filtered colimits, hence the unit
\[
N 
\cong
\underset{k\to -\infty}{\mathrm{colim}}\,c_{k;C} N
\longrightarrow
\underset{k\to -\infty}{\mathrm{colim}}\,
t_\ast t^! \big(c_{k; C}N\big)
\cong
t_\ast t^!\Big(\underset{k\to -\infty}{\mathrm{colim}}\,c_{k; C}N\Big)
\cong 
t_\ast t^! N
\]
is a quasi-isomorphism.
\end{proof}

\subsection{Koszul--Quillen equivalences}
\label{subsec:KQEquiv}
Let $C$ be a 2-reduced cocommutative dg coalgebra with universal twisting chain $t\colon C\rightsquigarrow \Omega C$.
Using the results of the previous section, we promote the $(t^!\dashv t_\ast)$-adjunction to a Quillen equivalence.
Due to the fact that the spectral sequences of Lemma \ref{lem:TwistingSS} need not converge for unbounded comodules, however, we are forced to tweak the class of weak equivalences on $C\mathrm{-Comod}$.
\begin{definition}
A morphism of $C$-comodules $f\colon M\to N$ is a \emph{$t$-equivalence} if $t^! f \colon t^! M\to t^! N$ is a quasi-isomorphism of $\Omega C$-modules.
\end{definition}
\begin{proposition}
$t_\ast\colon \Omega C\mathrm{-Mod}\to C\mathrm{-Comod}$ preserves quasi-isomorphisms.
\end{proposition}
\begin{proof}
This is an immediate consequence of Lemma \ref{lem:KoszulCounit}, using the naturality square for the adjunction counit.
The spectral sequences of Lemma \ref{lem:TwistingSS} together with Zeeman's comparison theorem imply that $t_\ast$ preserves quasi-isomorphisms of bounded-below $\Omega C$-modules.
Now let $f\colon A\to B$ be a quasi-isomorphism of (not necessarily bounded below) $\Omega C$-modules.
Taking connective covers, we have that $c_k f\colon c_k A\to c_k B$ is a quasi-isomorphism of $\Omega C$-modules for all $k\in \mathbb{Z}$.
The functor $t_\ast$ preserves filtered colimits, hence
\[
t_\ast A \cong \underset{k\to -\infty}{\mathrm{colim}}\, t_\ast c_kA
\longrightarrow 
 \underset{k\to -\infty}{\mathrm{colim}}\, t_\ast c_kB \cong t_\ast B
\]
is a quasi-isomorphism of $C$-comodules.
\end{proof}
\begin{lemma}
\label{lem:TvcQIso}
$t$-equivalences are quasi-isomorphisms.
Furthermore, quasi-isomorphisms of bounded below $C$-comodules are $t$-equivalences.
\end{lemma} 
\begin{proof}
Suppose that $f\colon M\to N$ is a $t$-equivalence of $C$-comodules.
Since $t_\ast$ preserves quasi-isomorphisms and the $(t^!\dashv t_\ast)$-unit is a natural quasi-isomorphism (Lemma \ref{lem:KoszulUnit}), the naturality square 
\[
\begin{tikzcd}
N
\ar[r, "\eta_N"]
\ar[d, "f"']
&
t_\ast t^! N
\ar[d, "t_\ast t^! f"]
\\
M
\ar[r, "\eta_M"]
&
t_\ast t^! M
\end{tikzcd}
\]
shows that $f$ is a quasi-isomorphism.
If $f\colon N\to M$ is a $t$-equivalence of bounded below comodules, then $f$ is also a quasi-isomorphism by Lemma \ref{lem:TwistingSS} and Zeeman's comparison theorem.
\end{proof}
\begin{lemma}
\label{lem:tastQIsotoTequiv}
$t_\ast\colon \Omega C\mathrm{-Mod}\to C\mathrm{-Comod}$ sends quasi-isomorphisms to $t$-equivalences
\end{lemma}
\begin{proof}
This is an immediate consequence of Lemma \ref{lem:KoszulCounit}, using the counit naturality diagrams.
\end{proof}

\begin{remark}
The previous results show that the class of $t$-equivalences is contained in the class of quasi-isomorphisms.
We expect this inclusion to be strict in general, though we do not have an explicit counterexample.
\end{remark}
\begin{lemma}
\label{lem:UnitTEquiv}
The $(t^!\dashv t_\ast)$-unit is a natural $t$-equivalence.
\end{lemma}
\begin{proof}
For any $C$-comodule $N$, consider the commuting  diagram expressing the triangle identity
\[
\begin{tikzcd}
t^! N
\ar[r, "t^! \eta_N"]
\ar[rr, bend right= 15, "\mathrm{id}"']
&
t^! t_\ast t^! N
\ar[r, "\epsilon_{t^! N}"]
&
t^! N\,.
\end{tikzcd}
\]
The counit $\epsilon_{t^! N}$ is a quasi-isomorphism by Lemma \ref{lem:KoszulCounit}, so $\eta_N$ is a $t$-equivalence by the 2-out-of-3 property.
\end{proof}

\begin{thmdef}
\label{thm:KozulQuillen}
The adjunction
\[
\begin{tikzcd}
C\mathrm{-Comod}
\ar[rr, shift left =1.1ex, "t^!"]
\ar[rr, shift left =-1.1ex, leftarrow, "t_\ast"', "\bot"]
&&
\Omega C \mathrm{-Mod}
\end{tikzcd}
\]
induces a combinatorial model structure on $C\mathrm{-Comod}$ called the \emph{$t$-local model structure} and denoted  $C\mathrm{-Comod}_{(t)}$.
The $t$-local model structure has
\begin{itemize}
  \item weak equivalences the $t$-equivalences; and
  \item cofibrations the monomorphisms.
\end{itemize}
The adjunction $(t^! \dashv t_\ast)$ is a Quillen equivalence with respect to the $t$-local model structure.
\end{thmdef}
\begin{proof}
To obtain the $t$-local model structure, we invoke the left transfer theorem (specifically, we use \cite[Theorem A.4]{hess_waldhausen_2016}).
To do this, we first observe that a map of $C$-comodules $f\colon M\to N$ is a monomorphism if and only if $t^! f$ is a cofibration, as is readily checked.
In particular, all objects of $C\mathrm{-Comod}$ are $t^!$-cofibrant (that is, the image of the zero morphism under $t^!$ is a cofibration).

To apply the left transfer theoerem, since $C\mathrm{-Comod}$ is locally presentable (as remarked in the proof of Lemma \ref{lem:KoszulUnit}) it suffices to verify the existence of good cylinder objects
\[
\begin{tikzcd}
N\oplus N
\ar[r, rightarrowtail]
&
\mathrm{Cyl}(N)
\ar[r, "\sim_t"]
&
N
\end{tikzcd}
\]
for all $C$-comodules $N$.
To this end, we set $\mathrm{Cyl}(N) := N\otimes N(\mathbb{Q}[\Delta^1])$ with its induced $C$-coaction.
The vertex inclusions $\{0\}, \{1\}\hookrightarrow \Delta^1$ induce a monomorphism $N\oplus N\to \mathrm{Cyl}(N)$ of comodules, and the terminal map $\Delta^1\to \ast$ induces a comodule map $\mathrm{Cyl}(N)\to N$.
Applying $t^!$, we get $\Omega C\otimes_t (N\otimes N(\mathbb{Q}[\Delta^1]))\to \Omega C\otimes_t N$ and choosing a chain deformation retraction of $N(\mathbb{Q}[\Delta^1])$ onto $\mathbb{Q}$ shows that this map is a quasi-isomorphism.

By definition of the model structure, the left adjoint $t_!$ creates cofibration and weak equivalences so is left Quillen.
The $(t_!\dashv t^\ast)$-counit is a natural weak equivalence by Lemma \ref{lem:KoszulCounit}. It now follows that the adjunction is a Quillen equivalence.
\end{proof}

\begin{remark}
A similar argument shows that for all $k\geq -\infty$, the forgetful-cofree adjunction 
\[
\begin{tikzcd}
C\mathrm{-Comod}_{\geq k}
\ar[rr, shift left=1.1ex]
\ar[rr, shift left=-1.1ex, leftarrow, "\bot", "C\otimes (-)"']
&&
\mathrm{Ch}_{\geq k}
\end{tikzcd}
\]
induces a model structure on $k$-connective comodules for which the weak equivalences are quasi-isomorphisms.
For each $k\in \mathbb{Z}$, there is a commuting diagram of left Quillen functors 
\[
\begin{tikzcd}
C\mathrm{-Comod}
\ar[dr, hookleftarrow, bend right=20]
&
C\mathrm{-Comod}_{(t)}
\ar[r,"t^!", "\sim"']
\ar[l, "\mathrm{id}"']
&
\Omega C\mathrm{-Mod}\phantom{{}_{\geq k}}
\\
&
C\mathrm{-Comod}_{\geq k}
\ar[r,"t^!", "\sim"']
\ar[u, hookrightarrow]
&
\Omega C\mathrm{-Mod}_{\geq k}
\ar[u, hookrightarrow]
\end{tikzcd}
\]
where arrows labelled \lq\lq$\sim$'' are moreover Quillen equivalences.
\end{remark}

\begin{remark}
Let $C=\mathbb{Q}$ regarded as a coalgebra in the obvious way.
The $t$-local model structure on $\mathbb{Q}-\mathrm{Comod} = \mathrm{Ch}$ coincides with the projective model structure on unbounded rational chain complexes.
\end{remark}

Let $\psi\colon C\to D$ be a morphism of (2-reduced) dg coalgebras.
Any $C$-comodule $\rho_N\colon N\to C\otimes N$ becomes a $D$-comodule with respect to the coaction
\[
\begin{tikzcd}
N
\ar[r, "\rho_N"]
&
C\otimes N
\ar[r, "\psi\otimes N"]
&
D\otimes N\,,
\end{tikzcd}
\]
giving a functor $\psi_!\colon C\mathrm{-Comod}\to D\mathrm{-Comod}$.
The functor $\psi_!$ has a right adjoint $\psi^\ast$ that sends the $D$-comodule $\rho_X\colon X\to D\otimes X$ to
\begin{equation}
\label{eqn:ComodMorphismRightAdjoint}
\psi^\ast (X)  =
\mathrm{lim}
\left(\!
\begin{tikzcd}
C \otimes X
\ar[rr,shift left=0.8ex, "C\otimes \rho_X"]
\ar[rr,shift left=-0.8ex, "((C\otimes \psi)\circ \Delta)\otimes X"']
&&
C\otimes D\otimes X
\end{tikzcd}
\!\right)\,
\end{equation}
where $C$ is regarded as a right $D$-comodule via $\psi$. 
Comparing with modules over $\Omega C$ and $\Omega D$ gives a means by which we can analyse the homotopical properties of the $(\psi_!\dashv \psi_\ast)$-adjunction.
\begin{lemma}
\label{lem:KoszulNat}
Let $\psi\colon C\to D$ be a morphism of 2-reduced cocommutative dg coalgebras, with universal twisting chains $t_C\colon C\rightsquigarrow \Omega C$ and $t_D\colon D\rightsquigarrow \Omega D$.
There is a commuting diagram of left Quillen functors
\[
\begin{tikzcd}
C\mathrm{-Comod}_{(t)}
\ar[r, "t_C^!"]
\ar[d, "\psi_!"']
&
\Omega C\mathrm{-Mod}
\ar[d, "\Omega \psi_!"]
\\
D\mathrm{-Comod}_{(t)}
\ar[r, "t_D^!"]
&
\Omega D\mathrm{-Mod}
\end{tikzcd}
\]
commuting up to natural isomorphism.
If $\psi$ is a quasi-isomorphism, this is a diagram of Quillen equivalences.
\end{lemma}
\begin{proof}
The cobar construction for (coaugmented) coalgebras is functorial, hence $\psi$ gives rise to a morphism of cobar constructions $\Omega \psi\colon \Omega C\to \Omega D$.
The dg algebra map $\Omega \psi$ corresponds to a twisting chain $t_\psi\colon C\rightsquigarrow \Omega D$ that factors as
\[
\begin{tikzcd}[cramped, sep = small]
C
\ar[r, rightsquigarrow, "t_C"]
&
\Omega C
\ar[r, "\Omega \psi"]
&
\Omega D
\end{tikzcd}
\qquad 
\mbox{ and as}
\qquad
\begin{tikzcd}[cramped, sep = small]
C
\ar[r, "\psi"]
&
D
\ar[r, rightsquigarrow, "t_D"]
&
\Omega D\,,
\end{tikzcd}
\]
(compare Construction \ref{cons:Cobar}).
For any $C$-comodule $N$, it follows that there are natural isomorphisms
\[
\Omega D\otimes_{t_D} \psi_! N
\cong \Omega D \otimes_{t_\psi} N\cong
\Omega D \otimes_{\Omega C} \big(\Omega C\otimes_{t_C} N\big)\,,
\]
hence the functors $\Omega \psi_!\circ t^!_C$ and $t^!_D \circ \psi_!$ are naturally isomorphic.
The functor $\psi_!$ clearly preserves monomorphisms.
For a $t_C$-equivalence of $C$-comodules $f\colon M\to N$, $\Omega \psi_! \circ t^!_C (f)$ is a quasi-isomorphism of $\Omega D$-modules by Ken Brown's lemma applied to the quasi-isomorphism $t^!_C$.
This shows that $(\psi_!\dashv \psi^\ast)$ is Quillen for the $t$-local model structures.

The (graded) map $\Omega \psi\otimes \psi\colon \Omega C\otimes C\to \Omega D \otimes D$ induces a morphism of spectral sequences
\[
\begin{tikzcd}[row sep = small]
H_p(\Omega C) \otimes H_q(C) 
\ar[r, Rightarrow]
\ar[d, "H_p(\Omega \psi)\otimes H_q(\psi)"']
&
H_{p+q}(\Omega C\otimes_{t_C} C)
\ar[d,"\cong"]
\\
H_p(\Omega D) \otimes H_q(D) 
\ar[r, Rightarrow]
&
H_{p+q}(\Omega D\otimes_{t_D} D)\,,
\end{tikzcd}
\]
where we have an isomorphism of $E^\infty$-pages as both $\Omega C\otimes_{t_C} C$ and $\Omega D \otimes_{t_D} D$ are acyclic (Lemma \ref{lem:KoszulCounit}).
If $\psi$ is a quasi-isomorphism, Zeeman's comparison theorem implies that $\Omega \psi\colon \Omega C\to \Omega D$ is one too and, consequently, that the base change adjunction $(\Omega \psi_!\dashv \Omega \psi^\ast)$ is a Quillen equivalence (compare Lemma \ref{lem:QIsodgLieBChange}).
Then $(\psi_!\dashv \psi^\ast)$ is a Quillen equivalence by the 2-out-of-3 property.
\end{proof}

\subsection{Cotensor products and cotorsion}
\label{subsec:Cotensor}
Let $C$ be a dg coalgebra (coassociative but for the time being not necessarily cocommutative).
The \emph{cotensor product} of a right $C$-comodule $M$ and a left $C$-comodule $N$ is the equaliser
\[
\begin{tikzcd}
M\,\square_C\, N
\ar[r]
&
M\otimes N
\ar[r, shift left =0.8ex, "\vphantom{\big(}M\otimes \rho_N"]
\ar[r, shift left =-0.8ex, "\vphantom{\big(}\rho_M \otimes N"']
&
M\otimes C\otimes N\,,
\end{tikzcd}
\]
with $\rho_M$ and $\rho_N$ the $C$-coactions on $M$ and $N$ respectively (note that \eqref{eqn:ComodMorphismRightAdjoint} expresses $\psi^\ast N$ as the cotensor product $\psi^\ast(X) = C\,\square_D\, X$).
If $C$ is cocommutative, the twist isomorphism allows us to identity left and right $C$-comodules and the cotensor product $N\,\square_C\, M$ becomes a (left) $C$-comodule with respect to the coaction
\[
M\, \square_C\, N \ni
\sum_i m^i\otimes n^i
\longmapsto
\sum_i n_{(0)}^i\otimes n_{(1)}^i\otimes m = \sum_{i} (-1)^{|n^i||m_{(0)}^i|} m_{(0)}^i\otimes n^i\otimes m_{(1)}^i\,.
\]
Forming cotensor products thus determines a bifunctor $C\mathrm{-Comod}\times C\mathrm{-Comod}\to C\mathrm{-Comod}$.

Eilenberg and Moore define the \emph{cotorsion functor} $\mathrm{Cotor}^C_\bullet(M,N)$ for connective $C$-comodules $M$ and $N$ as the right derived functor of the cotensor product \cite{eilenberg_homology_1966}.
We recall that a (connective) $C$-comodule $A$ is \emph{injective} if it is a summand of a cofree $C$-comodule, equivalently if the coaction $\rho_A\colon A\to C\otimes A$ participates in a retract diagram of $C$-comodules
\[
\begin{tikzcd}
A
\ar[rr, bend right=20, "\mathrm{id}"']
\ar[r, "\rho_A"]
&
C\otimes  A
\ar[r, "f"]
&
A\,.
\end{tikzcd}
\]
Given injective resolutions $M\to A_\bullet $ and $N \to B_\bullet $ the cotorsion functor is then defined as the homology of any of the quasi-isomorphic complexes
\[
\mathrm{Tot}^{\,\prod} \big(A_\bullet \,\square_C\, N\big)
\longrightarrow 
\mathrm{Tot}^{\,\prod}\big(A_\bullet\,\square_C\, B_\bullet\big)
\longleftarrow
\mathrm{Tot}^{\,\prod} \big(M \,\square_C\, B_\bullet\big)\,.
\]
\begin{remark}
If $C$ is 2-reduced and cocommuative, the adjunction unit $\eta_N\colon N \to t_\ast t^! N$ is an injective resolution. 
Indeed, $\eta_N$ is a quasi-isomorphism by Lemma \ref{lem:KoszulUnit} and the comodule map
\begin{align*}
C\otimes \big(C\otimes^{t} \Omega C\otimes_t N \big)& \longrightarrow C\otimes^{t} \Omega C\otimes_t N
\\
c\otimes \big(c' \otimes (tc_1\otimes\dotsb \otimes tc_k)\otimes n\big)
&\longmapsto
\epsilon(c)\cdot c' \otimes (tc_1\otimes\dotsb \otimes tc_k)\otimes n
\end{align*}
witnesses injectivity of $t_\ast t^! N = C\otimes^t \Omega C\otimes_t N$.
For $M$ and $N$ connective, we therefore have $H_\bullet(t_\ast t^! M\,\square_C\, t_\ast t^! N) \cong \mathrm{Cotor}_\bullet^C(M,N)$.
\end{remark}
\begin{remark}
More generally, if $C$ is equipped with a coaugmentation then the cotensor product $t_\ast t^! M \,\square_C\, t_\ast t^! N$  is isomorphic to the  two-sided cobar construction
$
\Omega (M;N;C) = M\otimes_t \Omega C \otimes_t N
$,
with $M$ regarded as a right $C$-comodule by the symmetry isomorphism.
\end{remark}

We are now able show that the Eilenberg--Moore cotorsion functor extends to a symmetric monoidal structure on the homotopy category of $t$-local $C$-comodules. 
\begin{theorem}
\label{thm:DerivedCotensor}
For $C$ a 2-reduced cocommutative dg coalgebra, the assignment 
\[
\square_C^\mathbf{R}\colon (M,N)\longmapsto t_\ast t^! M\,\square_C\, t_\ast t^! N
\]
determines a symmetric monoidal structure
\[
\square_C^\mathbf{R}\colon
Ho(C\mathrm{-Comod}_{(t)})\times
Ho(C\mathrm{-Comod}_{(t)})
\longrightarrow
Ho(C\mathrm{-Comod}_{(t)})
\]
such that $H_\bullet (M\,\square_C^\mathbf{R}\, N) = \mathrm{Cotor}^C_\bullet(M,N)$.
\end{theorem}
\begin{proof}
The cobar construction $\Omega C$ is a dg Hopf algebra, so by Lemma \ref{lem:dgHopfMonoidal} $\Omega C\mathrm{-Mod}$ is a symmetric monoidal model category with respect to the tensor product.
Concretely, for $A, B\in \Omega C\mathrm{-Mod}$ the tensor product $A\otimes B$ is an $\Omega C$-module with respect to the action determined on generators by
\[
tc\cdot (a\otimes b) := tc\cdot a\otimes b + (-1)^{|a|(|c|-1)} c\otimes tc\cdot b\,.
\]
There is a natural isomorphism of $C$-comodules  $(C\otimes^t A)\,\square_C\, (C\otimes^t B) \cong C\otimes^t(A\otimes B)$,  that is $t_\ast (A\otimes B)\cong t_\ast A\,\square_C\, t_\ast B$, and hence for $M, N\in C\mathrm{-Comod}$ we get a natural isomorphism
\begin{equation}
\label{eqn:CotorvsTensor}
t_\ast t^! M\,\square_C\, t_\ast t^! N
\cong
t_\ast \big(t^! M\otimes t^! N\big)\,.
\end{equation}
Let $f\colon M\to M'$ be a $t$-equivalence of $C$-comodules, so that $t^! M\otimes t^! N \to t^! M'\otimes t^!N$ is a quasi-isomorphism of $\Omega C$-modules by the K\"{u}nneth theorem.
By 
Lemma \ref{lem:tastQIsotoTequiv} and \eqref{eqn:CotorvsTensor}, we get that $t_\ast t^! M\,\square_C\, t_\ast t^! N \to t_\ast t^! M'\,\square_C\, t_\ast t^! N$ is a $t$-equivalence.
Similarly, the bifunctor $\square^\mathbf{R}_C$ preserves $t$-equivalences in the second argument hence there is a diagram of functors commuting up to natural isomorphism
\[
\begin{tikzcd}[column sep=small]
C\mathrm{-Comod}\times C\mathrm{-Comod}\ar[rr, "{\square^\mathbf{R}_C}"]
\ar[d]
&&
C\mathrm{-Comod}
\ar[d]
\\
Ho\big(C\mathrm{-Comod}_{(t)}\big)\times
Ho\big(C\mathrm{-Comod}_{(t)}\big)
\ar[rr, dashed, "{\exists}"]
&&
Ho\big(C\mathrm{-Comod}_{(t)}\big)\,.
\end{tikzcd}
\]
For $\Omega C$-modules $A$ and $B$ the $(t^!\dashv t_\ast)$-counit is a cofibrant resolution and hence
\[
\mathbf{R}t_\ast \big(A\otimes^\mathbf{L} B  \big)
\cong t_\ast \big(t^!t_\ast A\otimes t^! t_\ast B \big)
\cong 
(t_\ast t^!t_\ast A)\,\square_C\, (t_\ast t^! t_\ast B )
= t_\ast A \,\square^\mathbf{R}_C\, t_\ast B\,.
\]
The derived equivalence of homotopy categories  $(\mathbf{L}t^!\dashv \mathbf{R}t_\ast)$ identifies $\otimes^\mathbf{L}$ with $\square^\mathbf{R}_C$, so the latter indeed determines a symmetric monoidal structure.
\end{proof}

\section{Stable rational parametrised homotopy theory}
\label{sec:RatStabParamHom}
Quillen's rational homotopy theory is motivated by the mod-torsion Hurewicz theorem:
\begin{theorem}
For a map $f\colon X\to Y$ of simply-connected spaces, the following are equivalent:
\begin{itemize}
  \item $f$ is a rational homotopy equivalence;
  \item $f$ is a rational homology equivalence.
\end{itemize}
\end{theorem}
For $X$ simply-connected, the groups $\pi_{\ast +1 }(X)\otimes_\mathbb{Z}\mathbb{Q}$ form a graded rational Lie algebra with respect to the Whitehead bracket---the \emph{Whitehead Lie algebra of $X$}, denoted $\pi^\mathbb{Q}_{\ast+1} X$.
The mod-torsion Hurewicz theorem can be interpreted as saying that the Whitehead Lie algebra and the rational homology coalgebra are equivalent as algebraic invariants of the homotopy type.

In his foundational work \cite{quillen_rational_1969}, Quillen explains this equivalence of algebraic invariants as a consequence of equivalences of homotopy theories
\begin{equation}
\label{eqn:QuillensEq1}
Ho(\mathrm{sSet}_{\geq 2})_\mathbb{Q}
\xrightarrow{\;\cong\;}
Ho(\mathrm{dgLie}_{\geq 1})
\xrightarrow{\;\cong\;}
Ho(\mathrm{dgCoalg}_{\geq 2})\,,
\end{equation}
where
\begin{itemize}
  \item $\mathrm{sSet}_{\geq 2}$ is the category of 2-reduced simplicial sets (cf~Remark \ref{rem:RedsSet}).
  The homotopy category $Ho(\mathrm{sSet}_{\geq 2})_\mathbb{Q}$ is obtained by localising the class of rational homotopy equivalences. 
  
  \item $\mathrm{dgLie}_{\geq 1}$ is the category of reduced dg Lie algebras over $\mathbb{Q}$.
  The homotopy category is obtained by localising at the class of quasi-isomorphisms of dg Lie algebras.
  
  \item $\mathrm{dgCoalg}_{\geq 2}$ is the category of 2-reduced cocommutative dg coalgebras over $\mathbb{Q}$.
  The homotopy category is obtained by localising at the class of quasi-isomorphisms of dg coalgebras.
\end{itemize}
The upshot of the equivalences \eqref{eqn:QuillensEq1} is that the rational homotopy theory of simply-connected spaces is completely controlled by relatively simple algebraic objects.
A consequence of immense practical import is that topological constructions admit translations into algebra, where the analysis is usually much simpler.

Quillen's proof is rather convoluted and relies on the machinery of model categories.
Quillen constructs a sequence of model categories and Quillen equivalences between them:
\begin{equation}
\label{eqn:QuillenEquiv2}
\begin{tikzcd}
\mathrm{sSet}_{\geq 2}^\mathbb{Q}
\ar[r, shift left =1.1ex, "\mathbb{G}" ]
\ar[r, leftarrow, shift left =-1.1ex, "\bot", "\overline{W}"']
&
\mathrm{sGrp}_{\geq 1}^\mathbb{Q}
\ar[r, shift left =1.1ex, "\widehat{\mathbb{Q}}"]
\ar[r, leftarrow, shift left =-1.1ex, "\bot", "\mathcal{G}"']
&
\mathrm{scHopf}_{\geq 1}
\ar[r, leftarrow, shift left =1.1ex, "\widehat{\mathcal{U}}"]
\ar[r, shift left =-1.1ex, "\bot", "\mathcal{P}"']
&
\mathrm{sLie}_{\geq 1}
\ar[r, shift left =1.1ex, leftarrow, "N^!"]
\ar[r, shift left =-1.1ex, "\bot", "N"']
&
\mathrm{dgLie}_{\geq 1}
\ar[r, shift left =1.1ex, leftarrow, "\mathcal{L}"]
\ar[r, shift left =-1.1ex, "\bot", "\mathcal{C}"']
&
\mathrm{dgCoalg}_{\geq 2}\,.
\end{tikzcd}
\end{equation}
The derived functors participate in adjoint equivalences of homotopy categories, proving \eqref{eqn:QuillensEq1}.
Let us briefly recall the model categories appearing in this zig-zag of adjunctions (only two out of the three classes of weak equivalences, cofibrations, and fibrations is specified as any two determines the third).
\begin{itemize}
  \item $\mathrm{sSet}_{\geq 2}^\mathbb{Q}$ is the category of 2-reduced simplicial sets. Weak equivalences are rational homotopy equivalences and cofibrations are the monomorphisms.
  
  \item $\mathrm{sGrp}_{\geq 1}^\mathbb{Q}$ is the category of 1-reduced simplicial groups.
  Weak equivalences are rational homotopy equivalences and cofibrations are retracts of free maps of simplicial groups.
  
  \item $\mathrm{scHopf}_{\geq 1}$ is the category of reduced complete simplicial Hopf algebras over $\mathbb{Q}$.
  Reduced in this case means that the unit $\eta\colon \mathbb{Q}\to H$ is an isomorphism in simplicial degree zero; complete that the filtration by powers of the augmentation ideal $\overline{H}= \mathrm{ker}(\epsilon\colon H\to \mathbb{Q})$ is complete.
  Weak equivalences and fibrations are created by the functor $\mathcal{P}$ computing simplicial Lie algebras of primitives.
  
  \item $\mathrm{sLie}_{\geq 1}$ is the category of reduced simplicial Lie algebras over $\mathbb{Q}$.
  Reduced in this case means that the inclusion of the zero vector is an isomorphism in simplicial degree zero.
  Weak equivalences and fibrations are created by the forgetful functor $\mathrm{sLie}_{\geq 1}\to \mathrm{sSet}_{\geq 1}$.
  
  \item $\mathrm{dgLie}_{\geq 1}$ is the category of reduced dg Lie algebras.
  A map of reduced dg Lie algebras is a  fibration if and only if it is an epimorphism in degrees $>1$ and is a weak equivalence precisely if the map of underlying chain complexes is a quasi-isomorphism.
  
  \item $\mathrm{dgCoalg}_{\geq 2}$ is the category of 2-reduced cocommutative dg coalgebras.
  Weak equivalences and cofibrations are created by the forgetful functor to rational chain complexes, hence are respectively the quasi-isomorphisms and monomorphisms.
\end{itemize}
The adjunctions in the diagram \eqref{eqn:QuillenEquiv2} are 
\begin{description}
  \item[$(\mathbb{G}\dashv \overline{W})$ --] this is the restriction of Kan's adjunction to 2-reduced simplicial sets and reduced simplicial groups as used in Section \ref{SS:StableLoopSpace} (see \cite{goerss_simplicial_2009} for a modern account).
  The functor $\mathbb{G}$ sends a (2-)reduced simplicial set to its Kan simplicial loop group and $\overline{W}$ sends a (reduced) simplicial group to its classifying space.
  
  \item[$(\widehat{\mathbb{Q}}\dashv \mathcal{G})$ --] the functor $\widehat{\mathbb{Q}}$ sends a reduced simplicial group $G$ to the rational group ring $\mathbb{Q}[G]$ which is then completed at the augmentation ideal. 
  The right adjoint $\mathcal{G}$ computes the simplicial group of group-like elements.
  
  \item[$(\widehat{\mathcal{U}}\dashv \mathcal{P})$ --] the functor $\widehat{\mathcal{U}}$ sends a reduced simplicial Lie algebra $\mathfrak{g}$ to the completion of the universal enveloping algebra $\mathcal{U}\mathfrak{g}$ at the augmentation ideal.
  The right adjoint $\mathcal{P}$ computes simplicial Lie algebras of primitives.
  
  \item[$(N^!\dashv N)$ --] the functor $N$ sends a simplicial Lie algebra $\mathfrak{g}$ to its normalised chain complex, which inherits the structure of a dg Lie algebra (see Section \ref{SS:SimplicialDGLieComp}).
  The left adjoint $N^!$ is defined on free dg Lie algebras via the Dold--Kan correspondence and extended to all dg Lie algebras by monadicity.
  
  \item[$(\mathcal{L}\dashv \mathcal{C})$ --] the functor $\mathcal{L}$ sends a 2-reduced dg coalgebra $C$ to the primitive Lie algebra of the cobar construction $\Omega C$.
  The right adjoint $\mathcal{C}$ sends a dg Lie algebra to its Chevalley--Eilenberg coalgebra, which computes Lie algebra homology.
\end{description}
For what follows, it is useful to have specific algebraic models of the rational homotopy type to hand.
\begin{definition}[Quillen models]
Let $X$ be a 2-reduced simplicial set and define
\begin{equation}
\label{eqn:QuillenModels}
\Lambda_X := N\mathcal{P}\widehat{\mathbb{Q}}[\mathbb{G}X]
\;\;\;\;
\mbox{ and }
\;\;\;\;
C_X :=  \mathcal{C}\Lambda_X\,.
\end{equation}
Both $\Lambda_X$ and $C_X$ correspond to $X$ under the equivalences of rational homotopy theories \eqref{eqn:QuillensEq1} (this is an easy consequence of the fact that all objects of $\mathrm{sSet}_{\geq 2}^\mathbb{Q}$ are cofibrant and all objects of $\mathrm{scHopf}_{\geq 1}$ are fibrant); $\Lambda_X$ computes the Whitehead Lie algebra of $X$ whereas $C_X$ computes the rational homology coalgebra.
$\Lambda_X$ and $C_X$ are called \emph{Quillen's rational models for $X$}.
Both of the assignments $X\mapsto \Lambda_X$ and $X\mapsto C_X$ are functorial.
\end{definition}

For any space $X$, rational stable equivalences of $X$-parametrised spectra are detected on homotopy fibre spectra.
For connected $X$ these fibre spectra carry $\Omega X$-actions and their rational stable homotopy groups are acted upon by $H_\bullet(\Omega X; \mathbb{Q})$ with the Pontrjagin product.
If $X$ is moreover simply-connected the Milnor--Moore theorem says that this $H_\bullet(\Omega X; \mathbb{Q})$-action restricts to define representations of the Whitehead Lie algebra $\pi^\mathbb{Q}_{\ast+1}(X)$.
Fixing a point $x\colon \ast \to X$, any $X$-parametrised spectrum $P$ determines a $\pi^\mathbb{Q}_{\ast+1}(X)$-representation $\mathbold{\theta}_\ast(P) = \pi^\mathrm{st}_\ast (x^\ast P)\otimes_\mathbb{Z}\mathbb{Q}$ by rationalising the fibre spectrum at $x$.
But the rational homotopy theory of spectra is encoded by rational chain complexes and Quillen's rational homotopy theory provides a dg Lie algebra $\Lambda_X$ encoding the rational homotopy type of $X$.
One is naturally led to expect that there is some $\Lambda_X$-representation encoding the rational homotopy type of $P$ whose homology is the $\pi^\mathbb{Q}_{\ast+1}(X)$-representation $\mathbold{\theta}_\ast(P)$.

Another way of extracting algebraic data from an $X$-parametrised spectrum $P$ is to consider the image under the pushforward functor $X_!\colon Ho(\mathrm{Sp}_X)\to Ho(\mathrm{Sp})$.
The pushforward $X_! P$ is an $X_+$-comodule spectrum and the rational stable homotopy groups $\pi^\mathrm{st}_\ast (X_! P)\otimes_\mathbb{Z}\mathbb{Q}$ are coacted upon by rational homology coalgebra $H_\bullet (X; \mathbb{Q})$; we write $\mathbold{H}_\bullet^X(P)$ for this comodule.
When $X$ is simply-connected its rational homotopy type is encoded by the dg coalgebra $C_X$.
Since the rational homotopy theory of spectra is encoded by rational chain complexes, it is natural to expect that the rational homotopy type $P$ is encoded by some $C_X$-comodule whose homology is the $H_\bullet(X;\mathbb{Q})$-comodule $\mathbold{H}_\bullet^X(P)$.

We can now state the main result of this article, which lifts Quillen's algebraic models for rational homotopy theory to categories of representations.
\begin{theorem}
\label{thm:RatParStaHom}
Let $X$ be a simply-connected space and let $\Lambda_X$ and $C_X$ be Quillen's rational models for $X$.
Then there are pseudonatural strongly symmetric monoidal equivalences of categories between
\begin{enumerate}[label=(\roman*)]
  \item $(Ho(\mathrm{Sp}_X)_\mathbb{Q}, \wedge_X)$, the rational homotopy category of $X$-parametrised spectra with the fibrewise smash product,
  \item $(Ho(\Lambda_X\mathrm{-Rep}), \otimes^\mathbf{L})$, the derived category of unbounded $\Lambda_X$-representations with the derived tensor product, and
  \item $(Ho(C_X\mathrm{-Comod}_{(t)}), \square^\mathbf{R}_C)$, the $t$-local derived category of unbounded $C_X$-comodules with the cotorsion monoidal structure.
\end{enumerate}
Moreover, the diagrams 
\[
\begin{tikzcd}
Ho(\mathrm{Sp}_X)_\mathbb{Q}
\ar[r, "\sim"]
\ar[dr, bend right =15, "\mathbold{\theta}_\ast"']
&
Ho(\Lambda_X\mathrm{-Rep})
\ar[d, "H_\bullet"]
\\
&
\pi^\mathbb{Q}_{\ast+1}(X)\mathrm{-Rep}
\end{tikzcd}
\qquad
\mbox{ and }
\qquad
\begin{tikzcd}
Ho(\mathrm{Sp}_X)_\mathbb{Q}
\ar[r, "\sim"]
\ar[dr, bend right =15, "\mathbold{H}_\bullet^X"']
&
Ho(C_X\mathrm{-Comod}_{(t)})
\ar[d, "H_\bullet"]
\\
&
H_\bullet(X)\mathrm{-Comod}
\end{tikzcd}
\]
commute up to natural isomorphism.
\end{theorem}
\begin{proof}
The pseudonatural monoidal equivalences are proven during the course of Section \ref{sec:Equivalences} (the diagram \eqref{masterplan} is a schematic overview of the argument).
The parts of the Theorem pertaining to algebraic invariants are proven in Sections \ref{SSec:Whiteheadrep} and \ref{SSec:Comod} as Theorems \ref{thm:LieRep} and \ref{thm:Comod} respectively.
\end{proof}

\subsection{Equivalences of rational stable homotopy theories}
\label{sec:Equivalences}
Let $X$ be a simply-connected simplicial set.
In order to access strict algebraic models for the rational homotopy type of $X$, we first replace $X$ by a weakly equivalent 2-reduced model.
One way to do this is the following.
Take a Kan fibrant replacement $X\to X'$ and choose $x\in X'$. 
Let $X'' =E_2(X',x)$ be the second Eilenberg subcomplex of $X'$ at $x$, which is the pullback
\[
\begin{tikzcd}
E_2(X',x)
\ar[r]
\ar[d]
\ar[dr, phantom, "\ulcorner", very near start]
&
X'
\ar[d]
\\
\ast 
\ar[r, "x"]
&
\mathrm{cosk}_1 X'\,.
\end{tikzcd}
\]
Then $X''$ is 2-reduced and $X''\to X'$ is a weak equivalence by fibrancy of $X'$ \cite[Theorem 8.4]{may_simplicial_1967}.
The zig-zag of weak equivalences $X\to X'\leftarrow X''$ gives rise to a zig-zag of Quillen equivalences
\[
\begin{tikzcd}
\mathrm{Sp}^\Sigma_X
\ar[rr, "", shift left=1ex]
\ar[rr, leftarrow, "\bot", shift left=-1ex]
&&
\mathrm{Sp}^\Sigma_{X'}
\ar[rr, leftarrow,"", shift left=1ex]
\ar[rr,  "\bot", shift left=-1ex]
&&
\mathrm{Sp}^\Sigma_{X''}\,.
\end{tikzcd}
\]
Since we are concerned with the homotopy category $Ho(\mathrm{Sp}^\Sigma_{X})$ and its rationalisation, we may take $X$ to be 2-reduced without loss of generality. 

Let $\mathcal{Q}$ be a cofibrant replacement functor on $\mathrm{sLie}_{\geq 1}$ and
for a 2-reduced simplicial set $X$ write
\begin{equation}
\label{eqn:CofibRepl}
\varrho_X\colon \mathfrak{g}_X = \mathcal{Q}(\mathcal{P}\widehat{\mathbb{Q}}[\mathbb{G} X])\xrightarrow{\;\;\sim\;\;} \mathcal{P}\widehat{\mathbb{Q}}[\mathbb{G}X]
\end{equation}
for the corresponding cofibrant resolution.
The assignment $X\mapsto \mathfrak{g}_X$ is functorial and gives rise to a quasi-isomorphism of dg Lie algebras $N\mathfrak{g}_X\to \Lambda_X$ after taking normalised chains.

The main goal of this section is to show that for every 2-reduced simplicial set $X$, there is a diagram of Quillen equivalences (arrows are left Quillen functors)
\begin{equation}
\label{masterplan}
\begin{tikzcd}
H\mathbb{Q}\mathrm{-Mod}_{X}
\ar[r,
"{
\hyperlink{1stStep}{\textcircled{\raisebox{0.3pt}{\tiny{1}}}}
}"
]
&
\widehat{
\mathbb{Q}}[\mathbb{G}X]\mathrm{-Mod}
\ar[r, leftarrow, 
"{
\hyperlink{2ndStep}{\textcircled{\raisebox{0.3pt}{\tiny{2}}}}
}"
]
&
N\mathfrak{g}_X\mathrm{-Rep}^\Sigma_\mathrm{dg}
\ar[r, "\hyperlink{3rdStep}{\textcircled{\raisebox{0.3pt}{\tiny{3}}}}"]
&
N\mathfrak{g}_X\mathrm{-Rep}_\mathrm{dg}
\ar[d,
"\hyperlink{4thStep}{\textcircled{\raisebox{0.3pt}{\tiny{4}}}}"
]
\\
&
C_X\mathrm{-Comod}_{(t)}
\ar[r, "\hyperlink{6thStep}{\textcircled{\raisebox{0.3pt}{\tiny{6}}}}"]
&
\mathcal{L}C_X\mathrm{-Rep}_\mathrm{dg}
\ar[r, "\hyperlink{5thStep}{\textcircled{\raisebox{0.3pt}{\tiny{5}}}}"]
&
\Lambda_X\mathrm{-Rep}_\mathrm{dg}
\end{tikzcd}
\end{equation}
inducing strongly symmetric monoidal equivalences of homotopy categories.
We argue in steps as indicated by the labels in the diagram.

\paragraph*{\raisebox{0.5pt}{{\textcircled{\raisebox{0.3pt}{\tiny{1}}}}}\,-\,\hypertarget{1stStep}{Rational representations of the loop group}.}
As $X$ is 2-reduced, the Kan simplicial loop group $\mathbb{G}X$ is reduced.
Taking rational chains, the rational group ring $\mathbb{Q}[\mathbb{G}X]$ is itself reduced and hence canonically augmented.
The Kan simplicial loop group $\mathbb{G}X$ is free by construction and hence $\mathbb{Q}[\mathbb{G}X]$ is a free simplicial algebra.
Completion at the augmentation ideal
\begin{equation}
\label{eqn:CompletionAtAugIdeal}
\kappa_X
\colon \mathbb{Q}[\mathbb{G}X]\longrightarrow
\widehat{\mathbb{Q}}[\mathbb{G}X]
\end{equation}
is then a weak equivalence by \cite[Theorem I.3.6]{quillen_rational_1969}.
Note that $\widehat{\mathbb{Q}}[\mathbb{G}X]$ is \emph{not} a Hopf algebra in simplicial rational vector spaces as the completed coproduct
\[
\widehat{\mathbb{Q}}[\mathbb{G}X]\longrightarrow 
(\mathbb{Q}[\mathbb{G}X]\otimes \mathbb{Q}[\mathbb{G}X])\,\,{\widehat{}}\,
\cong
\widehat{\mathbb{Q}}[\mathbb{G}X]\,\widehat{\otimes}\,
\widehat{\mathbb{Q}}[\mathbb{G}X]
\]
does not factor through the inclusion $\widehat{\mathbb{Q}}[\mathbb{G}X]\otimes
\widehat{\mathbb{Q}}[\mathbb{G}X]\hookrightarrow (\mathbb{Q}[\mathbb{G}X]\otimes \mathbb{Q}[\mathbb{G}X])\,\,{\widehat{}}$.
Nevertheless, 
$\widehat{\mathbb{Q}}[\mathbb{G}X]$ is still an algebra object in $\mathrm{sVect}_\mathbb{Q}$ and the $\Sigma^\infty \widehat{\mathbb{Q}}[\mathbb{G}X]$-modules in $\mathrm{Sp}^\Sigma(\mathrm{sVect}_\mathbb{Q})$ organise into a stable model category $\widehat{\mathbb{Q}}[\mathbb{G}X]\mathrm{-Mod}$.
The same argument as in Lemma \ref{lem:BaseChangeforSimplHopf} shows that the base change adjunction associated to $\kappa_X$ is a Quillen equivalence.
Combined with Lemmas \ref{lem:RatParamSpec} and \ref{lem:RectGrpRing} we get a composite Quillen equivalence
\[
\begin{tikzcd}
H\mathbb{Q}\mathrm{-Mod}_{X}
\ar[r, shift left=1ex, "\mathfrak{f}_X"]
\ar[r, leftarrow, shift left=-1ex, "\mathfrak{b}_X"', "\bot"]
&
(\mathbb{G}X_+,H\mathbb{Q})\mathrm{-Bimod}
\ar[r, shift left=1ex, "\mathfrak{R}_{\mathbb{G}X}"]
\ar[r, leftarrow, shift left=-1ex, "\mathfrak{M}_{\mathbb{G}X}"', "\bot"]
&
\mathbb{Q}[\mathbb{G}X]\mathrm{-Mod}
\ar[r, shift left=1ex, "(\kappa_X)_!"]
\ar[r, leftarrow, shift left=-1ex, "\kappa_X^\ast"', "\bot"]
&
\widehat{\mathbb{Q}}[\mathbb{G}X]\mathrm{-Mod}\,,
\end{tikzcd}
\]
where left-hand and middle adjunctions induce symmetric monoidal equivalences of homotopy categories (by Remark \ref{rem:FibrewiseRatSmash} and Lemma \ref{lem:RectGrpRing} respectively). 

\paragraph*{\raisebox{0.5pt}{{\textcircled{\raisebox{0.4pt}{\tiny{2}}}}}\,-\,\hypertarget{2ndStep}{Comparison with simplicial Lie representations}.}
The map $\mathfrak{g}_X\to \mathcal{P}\widehat{\mathbb{Q}}[\mathbb{G}X]$ is the $(\widehat{\mathcal{U}}\dashv \mathcal{P})$-adjunct of a morphism of complete simplicial Hopf algebras
\[
\theta_X\colon \widehat{\mathcal{U}}\mathfrak{g}_X\longrightarrow \widehat{\mathbb{Q}}[\mathbb{G}X]\,.
\]
As $\mathfrak{g}_X$ is cofibrant, hence free, $\widehat{\mathcal{U}}\mathfrak{g}_X$ is a \lq\lq free'' completed simplicial Hopf algebra in the sense of Quillen (see \cite[Example A.2.11]{quillen_rational_1969}).
The completed simplicial Hopf algebra $\widehat{\mathbb{Q}}[\mathbb{G}X]$ is also free in this sense since $\mathbb{Q}[\mathbb{G}X]$ is a free algebra in the usual sense.
\begin{lemma}
\label{lem:ThetaisWE}
$\theta_X$ is a natural weak homotopy equivalence.
\end{lemma}
\begin{proof}
We first show that $\mathcal{P}\theta_X$ is a weak equivalence of simplicial Lie algebras of primitives.
The morphism $\theta_X$ factors via the $(\widehat{\mathcal{U}}\dashv \mathcal{P})$-counit as
\[
\theta_X\colon \widehat{\mathcal{U}}\mathfrak{g}_X\longrightarrow
\widehat{\mathcal{U}}\mathcal{P}\widehat{\mathbb{Q}}[\mathbb{G}X]
\longrightarrow 
\widehat{\mathbb{Q}}[\mathbb{G}X]\,.
\]
Applying the functor $\mathcal{P}$ and using the triangle identity of the adjunction, we get a commuting diagram
\[
\begin{tikzcd}
\mathcal{P}\widehat{\mathcal{U}}\mathfrak{g}_X
\ar[r]
\ar[rr, bend left=20, "\mathcal{P}\theta_X"]
\ar[d, leftarrow]
&
\mathcal{P}\widehat{\mathcal{U}}\mathcal{P}\widehat{\mathbb{Q}}[\mathbb{G}X]
\ar[r]
\ar[d, leftarrow]
&
\mathcal{P}\widehat{\mathbb{Q}}[\mathbb{G}X]
\\
\mathfrak{g}_X
\ar[r, "\sim"]
&
\mathcal{P}\widehat{\mathbb{Q}}[\mathbb{G}X]
\ar[ur, equal]
\end{tikzcd}
\]
where the vertical arrows are components of the $(\widehat{\mathcal{U}}\dashv \mathcal{P})$-unit.
Cofibrancy of $\mathfrak{g}_X$ implies that the left-hand vertical arrow is a weak equivalence, from which  it follows from the 2-out-of-3 property that the same is true of $\mathcal{P}\theta_X$.

Any reduced complete simplicial Hopf algebra $A\in \mathrm{scHopf}_{\geq 1}$ is canonically augmented.
Taking powers of the augmentation ideal gives rise to a filtration on $A$ and a homologically graded spectral sequence 
\[
E^1_{p,q}=\pi_{p+q} (\mathrm{gr}_q A)\Longrightarrow \pi_{p+q} (A)\,.
\]
If $A$ is free this spectral sequence is concentrated in the first quadrant and hence convergent.
Using the induced filtration on $\mathcal{P}A$,
Quillen proves \cite[Proposition I.6.9]{quillen_rational_1969}
 that the following conditions on a map $\psi\colon A \to B$ of free reduced simplicial complete Hopf algebras are equivalent:
\begin{itemize}
  \item  $\mathrm{gr}_1\psi$ is a weak equivalence,
  \item $\mathcal{P}\psi$ is a weak equivalence.
\end{itemize}
For $A$ free we find that $\mathrm{gr}_p A$ is the $p$-th tensor power $T^p \mathrm{gr}_1 A$.
With Zeeman's comparison theorem and the K\"{u}nneth theorem, the map of spectral sequences induced by $\psi$ shows that $\mathrm{gr}_1\psi$ is a weak equivalence if and only if $\psi$ is.

The result now follows from the fact that $\mathcal{P}\theta_X$ is a weak equivalence, since both the domain and codomain of $\theta_X$ are free. 
\end{proof}

\begin{remark}
If $A=\widehat{\mathbb{Q}}[G]$ is the completed rational group ring of a connected simplicial group then the Lie algebra of primitives $\mathcal{P}A$ models the the Whitehead Lie algebra of $BG$.
With this in mind, the above spectral sequence argument can be interpreted as an algebraic incarnation of the mod-torsion Hurewicz theorem.
\end{remark}

Completion at the augmentation ideal $\kappa_{\mathfrak{g}_X}\colon \mathcal{U} \mathfrak{g}_X\to \widehat{\mathcal{U}}\mathfrak{g}_X$ is a weak equivalence of rational simplicial algebras by \cite[Theorem I.3.6]{quillen_rational_1969}.
 The weak equivalences $\kappa_{\mathfrak{g}_X}$ and $\theta_X$ induce base change Quillen equivalences and together with Theorem \ref{thm:StabDKLie} this implies a composite Quillen equivalence
\[
\begin{tikzcd}
N\mathfrak{g}_X\mathrm{-Rep}_\mathrm{dg}^\Sigma
\ar[r, shift left=1ex, "N_!^\mathfrak{g}"]
\ar[r, shift left=-1ex, leftarrow, "\bot", "N^\ast_\mathfrak{g}"']
&
\mathfrak{g}_X\mathrm{-Rep}_\Delta
=\mathcal{U}\mathfrak{g}_X\mathrm{-Mod}
\ar[r, shift left=1ex, "(\kappa_{\mathfrak{g}_X})_!"]
\ar[r, shift left=-1ex, leftarrow, "\bot", "\kappa_{\mathfrak{g}_X}^\ast"']
&
\widehat{\mathcal{U}}\mathfrak{g}_X\mathrm{-Mod}
\ar[r, shift left=1ex, "(\theta_X)_!"]
\ar[r, shift left=-1ex, leftarrow, "\bot", "\theta_X^\ast"']
&
\widehat{\mathbb{Q}}[\mathbb{G}X]\mathrm{-Mod}\,.
\end{tikzcd}
\]
The first of these Quillen equivalences is weakly monoidal by Theorem \ref{thm:SimpLieMonoidalEquivs}.

\paragraph*{Monoidal matters.}
We must now pause to clarify how the zig-zag of Quillen equivalences
\[
\begin{tikzcd}
H\mathbb{Q}\mathrm{-Mod}_X
\ar[rr, shift left=1ex]
\ar[rr, leftarrow, shift left=-1ex, "\bot"]
&&
\widehat{\mathbb{Q}}[\mathbb{G}X]\mathrm{-Mod}
\ar[rr, leftarrow, shift left=1ex]
\ar[rr, shift left=-1ex, "\bot"]
&&
N\mathfrak{g}_X\mathrm{-Rep}_\mathrm{dg}^\Sigma
\end{tikzcd}
\]
obtained thus far induces strongly monoidal equivalences of homotopy categories.
The issue is that the completed simplicial Hopf algebras $\widehat{\mathbb{Q}}[\mathbb{G}X]$ and $\widehat{\mathcal{U}}\mathfrak{g}_X$ are \emph{not} Hopf algebras in simplicial vector spaces (the coproduct involves the completed tensor product in an essential way) so we cannot apply the results of Section \ref{ssec:RectandTens}.
Nonetheless, we have the following
\begin{proposition}
The zig-zag of Quillen equivalences
\[
\begin{tikzcd}
\mathcal{U}\mathfrak{g}_X\mathrm{-Mod}
\ar[rr, shift left=1ex, "(\kappa_{\mathfrak{g}_X})_!"]
\ar[rr, shift left=-1ex, leftarrow, "\kappa_{\mathfrak{g}_X}^\ast"', "\bot"]
&&
\widehat{\mathcal{U}}\mathfrak{g}_X\mathrm{-Mod}
\ar[rr, shift left=1ex, "(\theta_X)_!"]
\ar[rr, shift left=-1ex, leftarrow, "\theta_X^\ast"', "\bot"]
&&
\widehat{\mathbb{Q}}[\mathbb{G}X]\mathrm{-Mod}
\ar[rr, leftarrow, shift left=1ex, "(\kappa_X)_!"]
\ar[rr, shift left=-1ex, "\bot", "\kappa_X^\ast"']
&&
\mathbb{Q}[\mathbb{G}X]\mathrm{-Mod}
\end{tikzcd}
\]
identifies the symmetric monoidal structures of $Ho(\mathfrak{g}_X\mathrm{-Rep}_\Delta)$ and $Ho(\mathbb{Q}[\mathbb{G}X]\mathrm{-Mod})$.
\end{proposition}
\begin{proof}
For cofibrant objects $A, B \in \mathbb{Q}[\mathbb{G}X]\mathrm{-Mod}$ take fibrant resolutions
\[
\begin{tikzcd}
(\kappa_X)_! A 
\ar[r, rightarrowtail, "\sim"]
&
A^f
\ar[r, two heads]
&
0
\end{tikzcd}
\qquad
\mbox{ and }
\qquad
\begin{tikzcd}
(\kappa_X)_! B
\ar[r, rightarrowtail, "\sim"]
&
B^f
\ar[r, two heads]
&
0
\end{tikzcd}
\]
in $\widehat{\mathbb{Q}}[\mathbb{G}X]\mathrm{-Mod}$.
The forgetful functor $\widehat{\mathbb{Q}}[\mathbb{G}X]\mathrm{-Mod}\to \mathrm{Sp}^\Sigma(\mathrm{sVect}_\mathbb{Q})$ preserves cofibrations, fibrations and weak equivalences from which it follows that
$
A\to (\kappa_X)_! A
\to A^f
$
is a sequence of stable weak equivalences of cofibrant objects of $\mathrm{Sp}^\Sigma(\mathrm{sVect}_\mathbb{Q})$ compatible with $\mathbb{Q}[\mathbb{G}X]$-actions,
and likewise for $B$.
Applying Lemma \ref{lem:SmashTensCofibObject} we get a stable weak equivalence of $\mathbb{Q}[\mathbb{G}X]$-module spectra $A\otimes B\to A^f\otimes B^f$.

On the other hand, let $A^c \to \kappa^\ast_{\mathfrak{g}_X}\theta^\ast_X A^f = A^f$ be a cofibrant resolution in $\mathcal{U}\mathfrak{g}_X\mathrm{-Mod}$, similarly $B^c \to B^f$.
Then the same argument as above shows that $A^c \otimes B^c \to A^f \otimes B^f$ is a stable weak equivalence of $\mathcal{U}\mathfrak{g}_X$-module spectra.
We now have a zig-zag of stable weak equivalences in $\mathrm{Sp}^\Sigma(\mathrm{sVect}_\mathbb{Q})$
\[
(\kappa_X)_! (A\otimes B)
\longleftarrow
A\otimes B
\longrightarrow 
A^f\otimes B^f
\longleftarrow
A^c\otimes B^c
\]
in which the left-hand and middle arrows commute with $\mathbb{Q}[\mathbb{G}X]$-actions and the right-hand arrow commutes with $\mathcal{U}\mathfrak{g}_X$-actions (we have omitted explicit reference to forgetful functors in the above).
Hence the zig-zag of Quillen equivalences identifies the monoidal product bifunctors on $Ho(\mathbb{Q}[\mathbb{G}X]\mathrm{-Mod})$ and $Ho(\mathfrak{g}_X\mathrm{-Rep}_\Delta)$. A purely formal argument shows that the composite derived functor $Ho(\mathbb{Q}[\mathbb{G}X]\mathrm{-Mod})\to Ho(\mathfrak{g}_X\mathrm{-Rep}_\Delta)$ preserves the monoidal unit up to natural isomorphism, which proves the claim.
\end{proof}

\paragraph*{\raisebox{0.5pt}{{\textcircled{\raisebox{0.4pt}{\tiny{3}}}}}\,-\,\hypertarget{3rdStep}{Assembly}.}
Use the assembly functor $\mathfrak{A}_{N\mathfrak{g}_X}\colon N\mathfrak{g}_X \mathrm{-Rep}_\mathrm{dg}^\Sigma\to N\mathfrak{g}_X\mathrm{-Rep}_\mathrm{dg}$, which participates in a strongly monoidal Quillen equivalence (Theorem \ref{thm:SimpLieMonoidalEquivs}).

\paragraph*{\raisebox{0.5pt}{{\textcircled{\raisebox{0.4pt}{\tiny{4}}}}}\,-\,\hypertarget{4thStep}{Comparing Lie models}.}
The normalisation functor sends the weak equivalence of simplicial Lie algebras $\mathfrak{g}_X\to \mathcal{P}\widehat{\mathbb{Q}}[\mathbb{G}X]$ to a quasi-isomorphism of dg Lie algebras $N\mathfrak{g}_X\to \Lambda_X$.
By Lemma \ref{lem:QIsodgLieBChange} there is a weakly monoidal Quillen equivalence
between
$
N\mathfrak{g}_X\mathrm{-Rep}_\mathrm{dg}
$ and
$ 
\Lambda_X\mathrm{-Rep}_\mathrm{dg}
$.

\paragraph*{\raisebox{0.5pt}{{\textcircled{\raisebox{0.4pt}{\tiny{5}}}}}\,-\,\hypertarget{5thStep}{Comparing Lie models, redux}.}
The counit of the adjunction $(\mathcal{L}\dashv \mathcal{C})\colon \mathrm{dgCoalg}_{\geq 2}\to \mathrm{dgLie}_{\geq 1}$ is a natural quasi-isomorphism.
The component of the counit at $\Lambda_X$ is $\mathcal{L}\mathcal{C}\Lambda_X = \mathcal{L}C_X\to \Lambda_X$.
Lemma \ref{lem:QIsodgLieBChange} provides the desired weakly monoidal Quillen equivalence.

\paragraph*{\raisebox{0.5pt}{{\textcircled{\raisebox{0.4pt}{\tiny{6}}}}}\,-\,\hypertarget{6thStep}{Koszul duality}.}
The cobar construction of $C_X$ is canonically isomorphic to the universal enveloping algebra $\mathcal{U}\mathcal{L}C_X$ (essentially by construction of $\mathcal{L}$, see \cite[Proposition B.6.2]{quillen_rational_1969}).
Forming twisted extensions and coextensions with respect to the universal twisting chain $t\colon C_X\rightsquigarrow \Omega C_X$ gives a Quillen equivalence
\[
\begin{tikzcd}
C_X\mathrm{-Comod}_{(t)}
\ar[rr, "t^!", shift left=1ex]
\ar[rr, "t_\ast"', "\bot", leftarrow, shift left=-1ex]
&&
\Omega C_X\mathrm{-Mod} = \mathcal{L}C_X\mathrm{-Rep}_\mathrm{dg}
\end{tikzcd}
\]
by
 Theorem \ref{thm:KozulQuillen} which induces a strongly symmetric monoidal equivalence of homotopy categories (by the proof of Theorem \ref{thm:DerivedCotensor}).

\paragraph*{Pseudonaturality.}
The assignment
\[
X
\longmapsto
\left(
\!
\begin{tikzcd}[sep =small]
H\mathbb{Q}\mathrm{-Mod}_{X}
\ar[r]
&
\widehat{
\mathbb{Q}}[\mathbb{G}X]\mathrm{-Mod}
\ar[r, leftarrow]
&
N\mathfrak{g}_X\mathrm{-Rep}_\mathrm{dg}^\Sigma
\ar[r]
&
N\mathfrak{g}_X\mathrm{-Rep}_\mathrm{dg}
\ar[d]
\\
&
C_X\mathrm{-Comod}_{(t)}
\ar[r]
&
\mathcal{L}C_X\mathrm{-Rep}_\mathrm{dg}
\ar[r]
&
\Lambda_X\mathrm{-Rep}_\mathrm{dg}
\end{tikzcd}
\!
\right)
\]
is pseudonatural in the sense that a map of 2-reduced simplicial sets $f\colon X\to Y$ induces Quillen adjunctions at each stage of the above diagram of model categories.
These Quillen adjunctions are such that all possible squares of left Quillen functors commute up to natural isomorphism.
To show this, one uses functoriality of the various constructions involved (eg~$\mathbb{Q}[\mathbb{G}X]$, $\mathfrak{g}_X$, and $C_X$) together with Lemmas \ref{lem:RatParamSpec}, \ref{lem:PseudoNatRect}, \ref{lem:LieRepCompPNat}, and \ref{lem:KoszulNat}.
The precise details are left to the reader.
\begin{remark}
If $f\colon X\to Y$ is a rational homotopy equivalence of 2-reduced simplicial sets then the diagrams of model categories \eqref{masterplan} associated to $X$ and $Y$ are pseudonaturally Quillen equivalent.
\end{remark}

\subsection{Representations of the Whitehead Lie algebra}
\label{SSec:Whiteheadrep}
In this section we prove the part of Theorem \ref{thm:RatParStaHom} relating the rational stable homotopy groups of stable homotopy fibres of parametrised spectra to representations of the rational Whitehead Lie algebra.

Consider first the case of fibrewise suspension spectra over 2-reduced base spaces. 
Unwinding the definitions, we find that the composite left Quillen functor
\[
\begin{tikzcd}[sep =large]
\mathrm{sSet}_{\dslash X}
\ar[r,
"\Sigma^\infty_X"]
&
\mathrm{Sp}^\Sigma_X
\ar[r, "H\mathbb{Q}\owedge_X (-)"]
&
H\mathbb{Q}\mathrm{-Mod}_X
\ar[r, "\mathfrak{R}_{\mathbb{G}X}\circ \mathfrak{f}_X"]
&
\mathbb{Q}[\mathbb{G}X]\mathrm{-Mod}
\end{tikzcd}
\]
sends the retractive space $Y\in\mathrm{sSet}_{\dslash X}$ to $\Sigma^\infty \widetilde{\mathbb{Q}}[\mathbb{P}X_!\pi^\ast_X Y]$.
Recall that for an $X$-parametrised spectrum $P$, we write $\mathbold{\theta}_\ast (P)$ for the rational homology of the stable homotopy fibre  as a $\pi^\mathbb{Q}_{\ast+1}(X)$-representation.
\begin{lemma}
\label{lem:FibSusSpecTheta}
For any retractive space $Y$ over $X$, the image of the $\mathbb{Q}[\mathbb{G}X]$-module $\widetilde{\mathbb{Q}}[\mathbb{P}X_! \pi^\ast_X Y]$ under taking homotopy groups is isomorphic to $\mathbold{\theta}_\ast (\Sigma^\infty_X Y)$.
\end{lemma}
\begin{proof}
Pulling back $Y\to X$ along the fibration $\pi_X\colon \mathbb{P}X\to X$ yields a $\mathbb{G}X$-space $\pi^\ast_X Y$ modelling the $\Omega X$-action on the homotopy fibre $\mathrm{fib}(Y\to X)$.
The pushout square of $\mathbb{G}X$-spaces
\[
\begin{tikzcd}
\mathbb{P}X
\ar[r]
\ar[d]
\ar[dr, phantom, "\lrcorner", very near end]
&
\ast
\ar[d]
\\
\pi^\ast_X Y
\ar[r]
&
\mathbb{P}X_! \pi^\ast_X Y
\end{tikzcd}
\]
is also a homotopy pushout, in particular the bottom horizontal arrow is an isomorphism in reduced rational homology. 
The homotopy groups of $\widetilde{\mathbb{Q}}[\mathbb{P}X_! \pi^\ast_X Y]$ compute $\widetilde{H}_\bullet(\mathbb{P}X_! \pi^\ast_X Y;\mathbb{Q})\cong \widetilde{H}_\bullet (\pi^\ast_X Y; \mathbb{Q})$ so the result is proven.
\end{proof}

Observe that $\Sigma^\infty \widetilde{\mathbb{Q}}[\mathbb{P}X_! \pi^\ast_X Y]$ is cofibrant as an object of $\mathbb{Q}[\mathbb{G}X]\mathrm{-Mod}$.
Under the composite Quillen equivalence
\[
\begin{tikzcd}
\mathbb{Q}[\mathbb{G}X]\mathrm{-Mod}
\ar[r, shift left=1ex, "\text{(i)}"]
\ar[r, shift left=-1ex, leftarrow, "\bot"]
&
\widehat{\mathbb{Q}}[\mathbb{G}X]\mathrm{-Mod}
\ar[r, leftarrow, shift left=1ex, "\text{(ii)}"]
\ar[r, shift left=-1ex, "\bot"]
&
N\mathfrak{g}_X\mathrm{-Rep}_\mathrm{dg}^\Sigma
\ar[r, shift left =1ex, "\text{(iii)}"]
\ar[r, leftarrow, shift left=-1ex, "\bot"]
&
\Lambda_X\mathrm{-Rep}_\mathrm{dg}
\end{tikzcd}
\]
we have that
\begin{enumerate}[label=(\roman*)]
  \item $\Sigma^\infty \widetilde{\mathbb{Q}}[\mathbb{P}X_! \pi^\ast_X Y]$ is sent to the cofibrant object $\Sigma^\infty (\kappa_X)_! \widetilde{\mathbb{Q}}[\mathbb{P}X_! \pi^\ast_X Y]$ of $\widehat{\mathbb{Q}}[\mathbb{G}X]\mathrm{-Mod}$.
  \item By Lemma \ref{lem:RatSymSpecFibCofib} $\Sigma^\infty (\kappa_X)_! \widetilde{\mathbb{Q}}[\mathbb{P}X_! \pi^\ast_X Y]$ is fibrant in $\widehat{\mathbb{Q}}[\mathbb{G}X]\mathrm{-Mod}$ and is sent under the right Quillen functor to the fibrant object 
  \[
  N^\ast 
  \big(
  \Sigma^\infty_\mathbb{Q}
  \widetilde{\mathbb{Q}}[S^n]\otimes 
  (\kappa_X)_! \widetilde{\mathbb{Q}}[\mathbb{P}X_! \pi^\ast_X Y]
  \big)\colon
  n \longmapsto 
  N \big(
  \widetilde{\mathbb{Q}}[S^n]\otimes 
  (\kappa_X)_! \widetilde{\mathbb{Q}}[\mathbb{P}X_! \pi^\ast_X Y]
  \big)
  \]
  of $N\mathfrak{g}_X\mathrm{-Rep}_\mathrm{dg}^\Sigma$.
  Using the levelwise weak equivalence of commutative monoids \eqref{eqn:NormalisationMonoid} induced by the shuffle map, we find that there is a levelwise weak equivalence 
  \[
  s^\infty 
  N(\kappa_X)_! \widetilde{\mathbb{Q}}[\mathbb{P}X_! \pi^\ast_X Y]
  \longrightarrow
  \widetilde{\mathbb{Q}}[S^n]\otimes 
  (\kappa_X)_! \widetilde{\mathbb{Q}}[\mathbb{P}X_! \pi^\ast_X Y]
  \]
  of objects in $N\mathfrak{g}_X\mathrm{-Rep}_\mathrm{dg}^\Sigma$.
  Let $\Theta(Y)\to N(\kappa_X)_! \widetilde{\mathbb{Q}}[\mathbb{P}X_! \pi^\ast_X Y]$ be a cofibrant replacement in $N\mathfrak{g}_X\mathrm{-Rep}^+_\mathrm{dg}$ so that we have a stable weak equivalence
  \[
  s^\infty \Theta (Y)
  \longrightarrow 
  s^\infty 
  N(\kappa_X)_! \widetilde{\mathbb{Q}}[\mathbb{P}X_! \pi^\ast_X Y]
  \longrightarrow
  \widetilde{\mathbb{Q}}[S^n]\otimes 
  (\kappa_X)_! \widetilde{\mathbb{Q}}[\mathbb{P}X_! \pi^\ast_X Y]
  \]
  in $N\mathfrak{g}_X\mathrm{-Rep}_\mathrm{dg}^\Sigma$
  with cofibrant domain.
  
  \item The assembly functor $\mathfrak{A}_{N\mathfrak{g}_X}\colon N\mathfrak{g}_X\mathrm{-Rep}_\mathrm{dg}^\Sigma\to N\mathfrak{g}_X\mathrm{-Rep}_\mathrm{dg}$ sends $s^\infty\Theta(Y)$ to $\Theta(Y)$, regarded as a cofibrant object of the category of unbounded $N\mathfrak{g}_X$-representations.
  Under the  Quillen equivalence
  \[
  \begin{tikzcd}
  N\mathfrak{g}_X\mathrm{-Rep}_\mathrm{dg}
  \ar[rr, shift left=1ex, "(\mathcal{U}N\varrho_X)_!"]
  \ar[rr, shift left=-1ex, leftarrow, "\mathcal{U}N\varrho_X^\ast"', "\bot"]
  &&
  \Lambda_X \mathrm{-Rep}_\mathrm{dg}
  \end{tikzcd}
  \]
  induced by the quasi-isomorphism of dg Lie algebras $N\varrho_X\colon N\mathfrak{g}_X\to \Lambda_X$ \eqref{eqn:CofibRepl}, the component of the unit $\Theta(Y)\to \mathcal{U}N\varrho_X^\ast(\mathcal{U}N\varrho_X)_!\Theta(Y) = (\mathcal{U}N\varrho_X)_!\Theta(Y)$ is a quasi-isomorphism.
\end{enumerate}
In summary, together with Lemma \ref{lem:FibSusSpecTheta} this shows that the diagram of functors 
\[
\begin{tikzcd}[sep=tiny]
&
Ho(\mathrm{Sp}_X)_\mathbb{Q}
\ar[rr, "\sim"]
&&
Ho(\Lambda_X\mathrm{-Rep}_\mathrm{dg})
\ar[dr, bend left=10, "H_\bullet"]
&
\\
Ho(\mathrm{sSet}_{\dslash X})
\ar[ur, bend left=10, "\Sigma^\infty_X"]
\ar[drr, bend left=-10, "\Sigma^\infty_X"']
&&&&
\pi^\mathbb{Q}_{\ast+1}(X)\mathrm{-Rep}
\\
&&
Ho(\mathrm{Sp}_X)_\mathbb{Q}
\ar[urr, bend left=-10, "\mathbold{\theta}_\ast"']
&&
\end{tikzcd}
\]
commutes up to natural isomorphism.
$Ho(\mathrm{Sp}_X)$ is generated by fibrewise suspension spectra under shifts and sequential homotopy colimits (see \cite[Lemma 2.19]{braunack-mayer_combinatorial_2019} for a proof of this fact in terms of sequential parametrised spectra) and all of the functors in the above diagram commute with these operations, we have proven the following
\begin{theorem}
\label{thm:LieRep}
The diagram of functors
\[
\begin{tikzcd}
Ho(\mathrm{Sp}_X)_\mathbb{Q}
\ar[r, "\sim"]
\ar[dr, bend right =15, "\mathbold{\theta}_\ast"']
&
Ho(\Lambda_X\mathrm{-Rep})
\ar[d, "H_\bullet"]
\\
&
\pi^\mathbb{Q}_{\ast+1}(X)\mathrm{-Rep}
\end{tikzcd}
\]
commutes up to natural isomorphism.
\end{theorem}

\subsection{Rational homology comodules}
\label{SSec:Comod}
Let $X$ be a 2-reduced simplicial set and $P$ an $X$-spectrum.
In this section we prove the part of Theorem \ref{thm:RatParStaHom} relating the homology of the $C_X$-comodule associated to $P$ with the $H_\bullet(X;\mathbb{Q})$-comodule structure of $\widetilde{H}_\bullet (X_!P)$.
The method is to use two-sided cobar constructions to compare various models for $\widetilde{H}_\bullet (X_!P)$ in the case that $P=\Sigma^\infty_X Y$ is a fibrewise suspension spectrum, though the argument is somewhat involved.
Along the way we obtain a new proof of the fact that the dg coalgebra $C_X$ models the rational homology coalgebra of $X$ (Corollary \ref{cor:HomologyModels}).

For a simplicial group $G$ and a simplicial set $K$, recall that a \emph{twisting function} $\varphi\colon K\rightsquigarrow G$ is a family of maps $\varphi_n\colon K_n \to G_{n-1}$ such that 
\begin{align*}
\qquad\qquad
 d_0 (\varphi (k)) \varphi(d_0k)  &= \varphi(d_1 k)
\\
d_i \varphi(k) &= \varphi(d_{i+1} k)\,,&i>0\\
s_i\varphi(k) &= \varphi(s_{i+1} k)\,,& i> 0
\\
\varphi(s_0k) &= e_G\,,
\end{align*}
with $e_G\in G_{n-1}$ the neutral element.
Twisting functions $K\rightsquigarrow G$ are in bijective correspondence with maps of simplicial sets $K\to \overline{W}G$ and hence also with principal simplicial $G$-bundles with base $K$.
To a twisting function $\varphi\colon K\rightsquigarrow G$ is assigned the simplicial principal bundle $G\times_\varphi K$ with $n$-simplices $(G\times_\varphi K)_n = G_n \times K_n$, simplicial structure maps
\begin{align*}
\qquad\qquad\qquad\qquad 
d_i(g,k) &= (d_i g, d_i k)\,,& i>0
\qquad\\
d_0(g,k) &= ( d_0(g)\varphi(k) , d_0(k))\\
s_i(g,k) &= (s_i(g), s_i(k))\,,
\end{align*}
and left $G$-action.
The twisted Cartesian product $G\times_\varphi K$ fits into the pullback diagram of simplicial sets
\[
\begin{tikzcd}
G\times_\varphi K
\ar[r]
\ar[d]
\ar[dr, phantom, "\ulcorner", very near start]
&
WG
\ar[d]
\\
K
\ar[r, "\varphi"]
&
\overline{W}G\,,
\end{tikzcd}
\]
where the right-hand vertical arrow is Kan's model for the universal simplicial principal $G$-bundle.

According to a theorem of E.~Brown, any twisting function $\varphi\colon K\rightsquigarrow G$ gives rise to a natural twisting chain $\tau(\varphi)\colon C_\bullet (K) \rightsquigarrow C_\bullet(K)$ between integral simplicial chain complexes that is compatible with $\varphi$ in an appropriate sense.
The following simplicial version of Brown's theorem was proven by Szczarba \cite{szczarba_homology_1961}:
\begin{theorem}
For each twisting function $\varphi\colon K\rightsquigarrow G$ there is a twisting chain $\tau(\varphi)\colon C_\bullet(K)\rightsquigarrow C_\bullet (G)$ and a strong deformation retract
\[
\begin{tikzcd}
C_\bullet (G)\otimes_{\tau(\varphi)}
C_\bullet (K)
\ar[r, shift left =0.8ex, "\nabla_\varphi"]
\ar[r, leftarrow, shift left=-0.8ex, "f_\varphi"']
&
C_\bullet (G\times_\varphi K)
\arrow[loop right, distance=2em, start anchor={[yshift=1ex]east}, end anchor={[yshift=-1ex]east}]{}{h_\varphi}
\end{tikzcd}
\]
Furthermore, the choices of $\tau(\varphi)$, $\nabla_\varphi$, $f_\varphi$, and $h_\varphi$ can be made naturally. 
\end{theorem}
\begin{remark}
We write $N_\bullet K = N\mathbb{Q}[K]$ as shorthand for the complex of normalised rational chains on the simplicial set $K$. 
\end{remark} 
Tensoring with $\mathbb{Q}$, since the subcomplex of normalised chains is naturally a retract we get the following
\begin{corollary}
\label{cor:Brown}
For each $\varphi\colon K\rightsquigarrow G$ there is a natural twisting chain $\tau(\varphi)\colon N_\bullet K\rightsquigarrow N_\bullet G$ and a strong deformation retract
\[
\begin{tikzcd}
N_\bullet G\otimes_{\tau(\varphi)}
N_\bullet K
\ar[r, shift left =0.8ex, "\nabla_\varphi"]
\ar[r, leftarrow, shift left=-0.8ex, "f_\varphi"']
&
N_\bullet (G\times_\varphi K)
\arrow[loop right, distance=2em, start anchor={[yshift=1ex]east}, end anchor={[yshift=-1ex]east}]{}{h_\varphi}
\end{tikzcd}
\]
\end{corollary}

\begin{lemma}
\label{lem:LoopSpaceModels}
For any 2-reduced simplicial set $X$ there is a natural diagram 
\[
\begin{tikzcd}
\Omega N_\bullet X
\ar[r, "\sim"]
&
N\widehat{\mathbb{Q}}[\mathbb{G}X]
\ar[r, leftarrow, "\sim"]
&
\Omega C_X
\end{tikzcd}
\]
of quasi-isomorphisms of dg algebras.
\end{lemma}
\begin{proof}
The $(\mathbb{G}\dashv \overline{W})$-unit $X\to \overline{W}\mathbb{G}X$ corresponds to a twisting function $\varphi_X\colon X\rightsquigarrow \mathbb{G}X$ which in turn induces a twisting chain $\tau(\varphi_X)\colon N_\bullet X \rightsquigarrow N_\bullet \mathbb{G}X$ and a quasi-isomorphism
\[
N_\bullet \mathbb{G}X \otimes_{\tau(\varphi_X)}N_\bullet X\longrightarrow
N_\bullet(\mathbb{G}X\times_{\varphi_X} X)\cong N_\bullet \mathbb{P}X\,.
\]
The twisting chain $\tau(\varphi_X)$ also gives rise to a morphism of dg algebras $\Omega N_\bullet X\to N_\bullet\mathbb{G}X$ and hence a composite map of chain complexes
$
\Omega N_\bullet X\otimes_{\tau(\varphi_X)}
N_\bullet X
\to
N_\bullet \mathbb{G}X\otimes_{\tau(\varphi_X)}
N_\bullet X
\to
N_\bullet \mathbb{P}X
$.
Since the domain is acyclic (see eg~Lemma \ref{lem:KoszulCounit}) this is a quasi-isomorphism.
Filtering the complexes $\Omega N_\bullet X\otimes_{\tau(\varphi_X)}
N_\bullet X$ and $N_\bullet \mathbb{G}X\otimes_{\tau(\varphi_X)}
N_\bullet X$ as in Lemma \ref{lem:TwistingSS} and using the fact that $N_\bullet X$ is 2-reduced we get a morphism of convergent spectral sequences
\[
\begin{tikzcd}[row sep = small]
H_p(\Omega N_\bullet X) \otimes H_q(X) 
\ar[r, Rightarrow]
\ar[d]
&
H_{p+q}(\Omega N_\bullet X\otimes_{\tau(\varphi_X)}
N_\bullet X)
\ar[d, "\cong"]
\\
H_p(\mathbb{G}X) \otimes H_q(X) 
\ar[r, Rightarrow]
&
H_{p+q}(N_\bullet \mathbb{G}X\otimes_{\tau(\varphi_X)}
N_\bullet X)\,.
\end{tikzcd}
\]
Zeeman's comparison theorem implies that $\Omega N_\bullet X\to N_\bullet \mathbb{G}X$ is a quasi-isomorphism---note that the hypothesis that $X$ is 2-reduced is crucial to obtaining amenable $E^2$-pages.
Recall that completion at the augmentation ideal $\mathbb{Q}[\mathbb{G}X]\to \widehat{\mathbb{Q}}[\mathbb{G}X]$ is a natural weak equivalence of monoid objects in simplicial rational vector spaces. Taking normalised chains we have a composite quasi-isomorphism of dg algebras $\Omega N_\bullet X\to N_\bullet \mathbb{G}X = N\mathbb{Q}[\mathbb{G}X]\to N\widehat{\mathbb{Q}}[\mathbb{G}X]$.

For the right-hand quasi-isomorphism, first observe that the $(\mathcal{L}\dashv \mathcal{C})$-counit $\mathcal{L}C_X\to \Lambda_X$ is a quasi-isomorphism.
After taking universal enveloping algebras we get a quasi-isomorphism $\Omega C_X =\mathcal{U}\mathcal{L} C_X\to \mathcal{U}\Lambda_X$ (Lemma \ref{lem:UniEnvAlgQIso}).
Making use of the cofibrant replacement $\varrho_X\colon \mathfrak{g}_X\to \mathcal{P}\widehat{\mathbb{Q}}[\mathbb{G}X]$ of simplicial Lie algebras \eqref{eqn:CofibRepl}, we have a commuting diagram of morphisms of dg algebras
\[
\begin{tikzcd}
\mathcal{U}N\mathfrak{g}_X
\ar[r, "\chi", "\sim"']
\ar[d, "\mathcal{U}N \varrho_X"', "\sim"]
&
N\mathcal{U}\mathfrak{g}_X
\ar[r, "N\kappa", "\sim"']
\ar[d, "N\mathcal{U}\varrho_X"', "\sim"]
&
N\widehat{\mathcal{U}}\mathfrak{g}_X
\ar[dr, bend left=10, "N\theta_X", "\sim"']
\\
\mathcal{U}N\mathcal{P}\widehat{\mathbb{Q}}[\mathbb{G}X]
\ar[r, "\chi", "\sim"']
&
N\mathcal{U}\mathcal{P}\widehat{\mathbb{Q}}[\mathbb{G}X]
\ar[r, "N\kappa"]
&
N\widehat{\mathcal{U}}\mathcal{P}\widehat{\mathbb{Q}}[\mathbb{G}X]
\ar[r, "N\epsilon"]
&
N\widehat{\mathbb{Q}}[\mathbb{G}X]\,.
\end{tikzcd}
\]
The quasi-isomorphisms labelled \lq\lq$\chi$'' are from Lemma \ref{lem:Chi}, the quasi-isomorphisms labelled \lq\lq$\kappa$'' arise from taking normalised chains of completion maps, $\theta_X$ is the weak equivalence of Lemma \ref{lem:ThetaisWE}, and the left-hand vertical arrow is a quasi-isomorphism by Lemma \ref{lem:UniEnvAlgQIso}.
The bottom composite $\mathcal{U}N\mathcal{P}\widehat{\mathbb{Q}}[\mathbb{G}X] =\mathcal{U}\Lambda_X\to N\widehat{\mathbb{Q}}[\mathbb{G}X]$ is thus a quasi-isomorphism, hence so too is the composite $\Omega C_X \to \mathcal{U}\Lambda_X \to N\widehat{\mathbb{Q}}[\mathbb{G}X]$.
The reader can check that $X\mapsto (\Omega N_\bullet X\to N\widehat{\mathbb{Q}}[\mathbb{G}X]\leftarrow \Omega C_X)$ is natural.
\end{proof}

\begin{construction}[Bar construction]
Let $A$ be an augmented dg algebra with augmentation ideal $\overline{A}= \ker(A\to \mathbb{Q})$. 
Recall that the \emph{bar construction} of $A$ is the dg coalgebra $BA$ whose underlying graded coalgebra is the (cofree) tensor coalgebra $T \overline{A}[-1]$ on $\overline{A}[-1]$.
The differential on $BA$ is the sum of the differential $d_\otimes$ on $T\overline{A}[-1]$ induced from that of $\overline{A}$ with the unique coderivation $\delta$ for which the following diagram commutes:
\[
\begin{tikzcd}
BA 
\ar[r, "\delta"]
\ar[d, "\text{proj.}"']
&
BA
\ar[d, "\text{proj.}"]
\\
\overline{A}[-1]\otimes \overline{A}[-1]
\ar[r, "s \mu (s^{-1}\otimes s^{-1})"]
&
\overline{A}[-1]\,.
\end{tikzcd}
\]
Here $s\colon \overline{A}\to \overline{A}[-1]$ is the shift functor, $s^{-1}$ is its inverse, and $\mu \colon \overline{A}\otimes \overline{A}\to \overline{A}$ is induced from the multiplication of $A$.
Thus, the differential on $BA$ has the form
\[
d(sa_1\otimes \dotsb sa_n)
= d_{\otimes}
(sa_1\otimes \dotsb \otimes sa_n)
+
\sum_{i=1}^{n-1} (-1)^{i-1 + \sum_{k=1}^i|a_k|}
sa_1\otimes \dotsb \otimes s \mu (a_i, a_{i+1})\otimes
\dotsb
\otimes sa_n\,.
\]
The assignment $A\mapsto BA$ is functorial, preserves quasi-isomorphisms, 
and determines a right adjoint to the cobar construction $\Omega$ (for not necessarily cocommutative dg coalgebras).
In the case that $A$ is connective and $C$ is 2-reduced both the unit $C\to B\Omega C$ and counit $\Omega B A\to A$  are natural quasi-isomorphisms (\cite{loday_algebraic_2012} is a good reference).
\end{construction}

\begin{corollary}
\label{cor:HomologyModels}
Let $X$ be a 2-reduced simplicial set. Then for any 2-reduced simplicial set $X$ there is a natural diagram
\[
\begin{tikzcd}
N_\bullet X
\ar[r, "\sim"]
&
\mathbf{B} X:= B\big(N\widehat{\mathbb{Q}}[\mathbb{G}X]\big)
\ar[r, leftarrow, "\sim"]
&
C_X
\end{tikzcd}
\]
of quasi-isomorphisms of 2-reduced dg coalgebras.
\end{corollary}
\begin{proof}
Apply the bar construction to the zig-zag of quasi-isomorphisms of Lemma \ref{lem:LoopSpaceModels} and compose with $(\Omega\dashv B)$-units.
Both $N_\bullet X$ and $C_X$ are 2-reduced by construction; the dg algebra $N\widehat{\mathbb{Q}}[\mathbb{G}X]$ is reduced so its bar construction is 2-reduced. 
\end{proof}
\begin{construction}
Let $C$ be a (not necessarily cocommutative) dg coalgebra and $N$, $M$ respectively a right and left $C$-comodule.
The \emph{two-sided cobar construction} is 
\[
\Omega(N;C;M) := \big(N\otimes_{t_C} \Omega C\big)\otimes_{\Omega C}\big( \Omega C\otimes_{t_C} M\big)\cong N\otimes_{t_C} \Omega C\otimes_{t_C} M\,,
\]
with $t_C\colon C\rightsquigarrow \Omega C$ the universal twisting chain (and using the evident notion of $t_C$-twisted extensions of right $C$-comodules, compare Construction \ref{cons:TwistedExt}).
\end{construction}
\begin{lemma}
\label{lem:2SidedCobar}
Let $Z\to X\leftarrow Y$ be a diagram of 2-reduced simplicial sets. 
Then there is a zig-zag of quasi-isomorphisms
\[
\Omega\big(N_\bullet Z; N_\bullet X; N_\bullet Y\big)
\longrightarrow
\Omega\big(\mathbf{B} Z;\mathbf{B} X;\mathbf{B} Y\big)
\longleftarrow
\Omega
 \big(
 C_Z;C_X;C_Y\big)\,.
\]
If $Z=\ast$ then the left-hand arrow is  a map of $\Omega N_\bullet X$-modules and the right-hand arrow is a map of $\Omega C_X$-modules.
\end{lemma}
\begin{proof}
For $C$ a 2-reduced dg coalgebra, the filtration $\mathcal{F}_p \Omega C =\big\{\sum tc_1\otimes \dotsb \otimes tc_n \;\big|\; n\geq -p\big\}$ is exhaustive and bounded below and induces a filtration on the two-sided bar construction 
\[
\mathcal{F}_p \Omega (N;C;M) = \Big\{
\sum n\otimes t \otimes m \;\Big| \; t\in \mathcal{F}_p\Omega C\Big\}
\]
with the same properties.

The naturality clause of Corollary \ref{cor:HomologyModels} implies that we have a commuting diagram of coalgebra morphisms
\[
\begin{tikzcd}
N_\bullet Z
\ar[r]
\ar[d, "f_Z"', "\sim"]
&
N_\bullet X
\ar[r, leftarrow]
\ar[d, "f_X"', "\sim"]
&
N_\bullet Y
\ar[d, "\sim"', "f_Y"]
\\
\mathbf{B} Z
\ar[r]
&
\mathbf{B} X
\ar[r, leftarrow]
&
\mathbf{B} Y
\end{tikzcd}
\]
in which all vertical arrows are quasi-isomorphisms.
The left-hand square induces a map of right $\Omega N_\bullet X$-modules
$
N_\bullet Z\otimes_{t_{N_\bullet X}} \Omega N_\bullet X
\to
N_\bullet Z\otimes_{t_{\mathbf{B} X}} \Omega \mathbf{B} X
\to
\mathbf{B}Z\otimes_{t_{\mathbf{B} X}} \Omega \mathbf{B} X
$,
similarly the right-hand square induces a map of left $\Omega N_\bullet X$-modules
$
\Omega N_\bullet X\otimes_{t_{N_\bullet X}} N_\bullet Y
\to
\Omega \mathbf{B}X\otimes_{t_{\mathbf{B}X}} \mathbf{B} Y
$.
From this, we obtain a map of two-sided cobar constructions
\[
\Omega\big(N_\bullet Z; N_\bullet X; N_\bullet Y\big)
\longrightarrow
\Omega\big(\mathbf{B} Z;\mathbf{B} X;\mathbf{B} Y\big)\,.
\]
This morphism respects filtrations so determines a map of spectral sequences.
On $E^0$-pages, this map takes the form
\[
N_\bullet Z\otimes \big(s^{-1} \overline{N_\bullet X}\big)^{\otimes p}
\otimes N_\bullet Y
\xrightarrow{
\;
f_Z\otimes (s^{-1}\circ \overline{f}_X \circ s)^{\otimes p}\otimes f_Y
\;}
\mathbf{B} Z\otimes \big(s^{-1} \overline{\mathbf{B} X}\big)^{\otimes p}
\otimes \mathbf{B}Y\,,
\]
where the overline denotes taking the coaugmentation coideal.
The differentials on this page are untwisted (all terms in the differentials generated by twisting chains decrease the filtration degree) hence the induced map on $E^0$-pages is a quasi-isomorphism by Corollary \ref{cor:HomologyModels}.
By convergence, it follows that $\Omega\big(N_\bullet Z; N_\bullet X; N_\bullet Y\big)
\to
\Omega\big(\mathbf{B} Z;\mathbf{B} X;\mathbf{B} Y\big)$
is a quasi-isomorphism.
The quasi-isomorphism $\Omega\big(C_Z;C_X;C_Y\big)
\to
\Omega\big(\mathbf{B} Z;\mathbf{B} X;\mathbf{B} Y\big)$ is derived in a similar fashion.

Finally, if $Z = \ast$ then $N_\bullet Z = \mathbf{B} Z = C_Z = \mathbb{Q}$.
It follows that 
\[
\Omega\big(N_\bullet Z; N_\bullet X; N_\bullet Y\big)= \Omega N_\bullet X\otimes_{t_{N_\bullet X}} N_\bullet Y\longrightarrow
\Omega\big(\mathbf{B} Z; \mathbf{B} X; \mathbf{B} Y\big)=
 \Omega \mathbf{B}X\otimes_{t_{\mathbf{B}X}} \mathbf{B} Y
\] respect the left $\Omega
 N_\bullet X$-actions.
 Similarly for the right-hand map.
\end{proof}
\begin{remark}
The Eilenberg--Moore theorem allows us to interpret this last result as giving three equivalent models for the rational homology of the fibre product $Z\times_X Y$.
\end{remark}

Fix a retractive space $(X\xrightarrow{i}Y\xrightarrow{r}X)$ over $X$.
The pushout
\[
\begin{tikzcd}
X\ar[r]
\ar[d, "i"] 
\ar[dr, phantom, "\lrcorner", very near end]
&
\ast
\ar[d]
\\
Y
\ar[r, "\pi"]
&
X_! Y= Y/X
\end{tikzcd}
\]
is an $X_+$-comodule in $(\mathrm{sSet}_\ast, \wedge)$ with coaction determined by the commuting diagram
\[
\begin{tikzcd}
Y
\ar[r, "\Delta_Y"]
\ar[d, "\pi"']
&
Y\times Y
\ar[d, "r\wedge \pi"]
\\
Y/X
\ar[r, "\rho"]
&
X_+\wedge Y/X
\end{tikzcd}
\]
so that $\rho([y]) = r(y) \wedge [y]$.
Taking (reduced) rational homology, $\widetilde{H}_\bullet (Y/X)$ thus inherits the structure of a $H_\bullet(X)$-comodule.
\begin{lemma}
\label{lem:RatHomComodules2Red}
Let $Y\in \mathrm{sSet}_{\dslash X}$ be 2-reduced.
Then the image of $Y$ under the composite derived functor
\[
\begin{tikzcd}
Ho(\mathrm{sSet}_{\dslash X})
\ar[r, "\Sigma^\infty_X"]
&
Ho(\mathrm{Sp}^\Sigma_X)_\mathbb{Q}
\ar[r, "\sim"]
&
Ho(C_X\mathrm{-Comod}_{(t)})
\ar[r, "H_\bullet"]
&
H_\bullet(X)\mathrm{-Comod}
\end{tikzcd}
\]
is naturally isomorphic to $\widetilde{H}_\bullet(Y/X)$.
\end{lemma}
\begin{proof}
The diagram $X\to Y\to X$ is represented in rational homotopy theory by the diagram of 2-reduced dg coalgebras $C_X \to C_Y \to C_X$, where the composite is the identity.
Consider the iterated pushout diagram of chain complexes
\begin{equation}
\label{eqn:CoalgPushout}
\begin{tikzcd}
C_X
\ar[r]
\ar[d]
\ar[dr, phantom, "\lrcorner", very near end]
&
\mathbb{Q}
\ar[d]
\ar[r]
\ar[dr, phantom, "\lrcorner", very near end]
&
0
\ar[d]
\\
C_Y
\ar[r]
&
C_{Y/X}
\ar[r]
&
\overline{C}_{Y/X}\,.
\end{tikzcd}
\end{equation}
The map of coalgebras $C_X\to C_Y$ is injective, hence a cofibration, so that the left-hand square exhibits a homotopy cofibre sequence of 2-reduced dg coalgebras.
In particular, by Quillen's rational homotopy equivalences $C_{Y/X}$ computes computes the rational homology of $Y/X$.

As $C_X$ is a retract of $C_Y$, the exact sequence $0\to C_X\to C_Y\to \overline{C}_{Y/X}\to 0$ splits.
The coalgebra map $C_Y\to C_X$ makes $C_Y$ a left $C_X$-comodule, and $\overline{C}_{Y/X}$ inherits the structure of a $C_X$-comodule via the splitting $C_Y\cong C_X \oplus \overline{C}_{Y/X}$.
As the iterated pushout diagram \eqref{eqn:CoalgPushout} is compatible with $C_X$-coactions in the obvious sense, we have that  $H_\bullet(\overline{C}_{Y/X})\cong\widetilde{H}_\bullet(Y/X)$ and $H_\bullet(X)$-comodules.

Using Lemma \ref{lem:2SidedCobar} we construct a commuting diagram of chain complexes
\begin{equation}
\label{eqn:ComodCoherence}
\begin{tikzcd}
\Omega N_\bullet X\otimes_{t_{N_\bullet X}} N_\bullet X
\ar[r, "\sim"]
\ar[d]
&
\Omega \mathbf{B}X\otimes_{t_{\mathbf{B} X}}
\mathbf{B} X
\ar[r, leftarrow, "\sim"]
\ar[d]
&
\Omega C_X 
\otimes_{t_{C_X}} C_X
\ar[d]
\\
\Omega N_\bullet X\otimes_{t_{N_\bullet X}} N_\bullet Y
\ar[r, "\sim"]
&
\Omega \mathbf{B}X\otimes_{t_{\mathbf{B} X}}
\mathbf{B} Y
\ar[r, leftarrow, "\sim"]
&
\Omega C_X 
\otimes_{t_{C_X}} C_Y
\end{tikzcd}
\end{equation}
in which all horizontal arrows are quasi-isomorphisms, all arrows on the left-hand side are maps of $\Omega N_\bullet X$-modules, and all maps on the right-hand side are maps of $\Omega C_X$-modules.
The complexes in the top row are all acyclic (cf~Lemma \ref{lem:KoszulCounit}).

Composing the twisting function $\varphi_X\colon X\rightsquigarrow \mathbb{G}X$ with the retraction $r\colon Y\to X$ gives a new twisting function $\varphi_r\colon Y\rightsquigarrow \mathbb{G}X$ whose associated simplicial principal bundle is the iterated pullback
\[
\begin{tikzcd}
\pi^\ast_X Y
\cong
\mathbb{G}X\times_{\varphi_r} Y  \ar[r]
\ar[d]
\ar[dr, phantom, "\ulcorner", very near start]
&
\mathbb{P}X \cong \mathbb{G}X\times_{\varphi_X} X
\ar[d]
\ar[r]
\ar[dr, phantom, "\ulcorner", very near start]
&
W\mathbb{G}X
\ar[d]
\\
Y
\ar[r, "r"]
&
X
\ar[r]
&
\overline{W}\mathbb{G}X\,.
\end{tikzcd}
\]
By Brown's theorem (Corollary \ref{cor:Brown}) there is a quasi-isomorphism
$
N_\bullet \mathbb{G} X\otimes_{\tau(\varphi_r)} N_\bullet Y \to N_\bullet \pi^\ast_X Y 
$
commuting with left $N_\bullet \mathbb{G}X$-actions.
Arguing as in Lemma \ref{lem:LoopSpaceModels} we get a quasi-isomorphism of $\Omega N_\bullet X$-modules
$ 
\Omega N_\bullet X\otimes_{t_{N_\bullet X}} N_\bullet Y\to  N_\bullet \pi^\ast_X Y$  (this requires that $Y$ is 1-connected).

Recall \eqref{eqn:CompletionAtAugIdeal} that completion at the augmentation ideal $\kappa_X \colon \mathbb{Q}[\mathbb{G}X]\to \widehat{\mathbb{Q}}[\mathbb{G}X]$ is a weak equivalence.
Since the base space $X$ is fixed throughout, let us write $\kappa = \kappa_X$ to ease the notation in what follows.
As $\mathbb{Q}[\pi^\ast_X Y]$ is a cofibrant $\mathbb{Q}[\mathbb{G}X]$-module, the base change unit map $\mathbb{Q}[\pi^\ast_X Y]\to \kappa^\ast\kappa_! \mathbb{Q}[\pi^\ast_X Y]$ is a weak equivalence of $\mathbb{Q}[\mathbb{G}X]$-modules.
Passing to normalised chains we get a quasi-isomorphism of $N_\bullet \mathbb{G}X$-modules $N_\bullet
 \pi^\ast_X Y \to N \kappa_! \mathbb{Q}[\pi^\ast_X Y]$.
We can thus extend the diagram \eqref{eqn:ComodCoherence} as:
\begin{equation}
\label{eqn:ComodCoherenceII}
\begin{tikzcd}
N \kappa_!\mathbb{Q}[\mathbb{P}X]
\ar[r, leftarrow, "\sim"]
\ar[d]
&
\Omega N_\bullet X\otimes_{t_{N_\bullet X}} N_\bullet X
\ar[r, "\sim"]
\ar[d]
&
\Omega \mathbf{B}X\otimes_{t_{\mathbf{B} X}}
\mathbf{B} X
\ar[r, leftarrow, "\sim"]
\ar[d]
&
\Omega C_X 
\otimes_{t_{C_X}} C_X
\ar[d]
\\
N \kappa_!\mathbb{Q}[\pi^\ast_X Y]
\ar[r, leftarrow, "\sim"]
&
\Omega N_\bullet X\otimes_{t_{N_\bullet X}} N_\bullet Y
\ar[r, "\sim"]
&
\Omega \mathbf{B}X\otimes_{t_{\mathbf{B} X}}
\mathbf{B} Y
\ar[r, leftarrow, "\sim"]
&
\Omega C_X 
\otimes_{t_{C_X}} C_Y\,.
\end{tikzcd}
\end{equation}
There is a split exact sequence of $\mathbb{Q}[\mathbb{G}X]$-modules $
\mathbb{Q}[\mathbb{P}X]\to
\mathbb{Q}[\pi^\ast_X Y]
\to
\widetilde{\mathbb{Q}}[\mathbb{P}X_!\pi^\ast_XY]$ giving rise to a split cofibre sequence of $N\widehat{\mathbb{Q}}[\mathbb{G}X]$-modules
$
N\kappa_!\mathbb{Q}[\mathbb{P}X]\to
N\kappa_!\mathbb{Q}[\pi^\ast_X Y]
\to
N\kappa_!\widetilde{\mathbb{Q}}[\mathbb{P}X_!\pi^\ast_XY]
$.
The cokernels of the vertical maps in \eqref{eqn:ComodCoherenceII} are related by a zig-zag of quasi-isomorphisms
\begin{equation}
\label{eqn:ComodCoherenceIII}
\begin{tikzcd}[sep = tiny]
N\kappa_!\widetilde{\mathbb{Q}}[\mathbb{P}X_!\pi^\ast_XY]
\ar[r, leftarrow]
&
\underbrace{\Omega N_\bullet X\otimes_{t_{N_\bullet X}} \big(N_\bullet Y/N_\bullet X\big)}_{=:\mathbf{N}[Y/X]}
\ar[r]
&
\underbrace{\Omega\mathbf{B}X\otimes_{t_{\mathbf{B}X}}\big(\mathbf{B} Y/\mathbf{B}X\big)}_{=: \mathbf{B}[Y/X]}
\ar[r, leftarrow]
&
\Omega C_X \otimes_{t_{C_X}} \overline{C}_{Y/X}
\end{tikzcd}
\end{equation}
 By restricting  along the algebra quasi-isomorphism $\mathcal{U}\Lambda_X \to N\widehat{\mathbb{Q}}[\mathbb{G}X]$, $N\kappa_!\widetilde{\mathbb{Q}}[\mathbb{P}X_!\pi^\ast_XY]$ inherits the structure of a $\Lambda_X$-representation.
By results of the previous section this representation models the image of $\Sigma^\infty_X Y$ under the equivalence $Ho(\mathrm{Sp}_X)_\mathbb{Q}\to Ho(\Lambda_X\mathrm{-Rep}_\mathrm{dg})$.
Hence, the image of $\Sigma_X^\infty Y$ under the composite equivalence $Ho(\mathrm{Sp}_X)_\mathbb{Q}\to Ho(\Lambda_X\mathrm{-Rep}_\mathrm{dg})\to Ho(C_X\mathrm{-Comod}_{(t)})$ is modelled by the $C_X$-comodule
\[
C_X \otimes^{t_{C_X}}N\kappa_!\widetilde{\mathbb{Q}}[\mathbb{P}X_!\pi^\ast_XY]\,.
\]
Using the twisting chains corresponding to the homomorphisms of dg algebras $\Omega C_X\to N\widehat{\mathbb{Q}}[\mathbb{G}X]$, $\Omega \mathbf{B} X\to N\widehat{\mathbb{Q}}[\mathbb{G}X]$, and $\Omega N_\bullet X\to N\widehat{\mathbb{Q}}[\mathbb{G}X]$ we construct a diagram of twisted coextensions
\[
\begin{tikzcd}[sep=small]
C_X \otimes^{t_{C_X}} N\kappa_! \widetilde{\mathbb{Q}}[\mathbb{P}X_! \pi^\ast_X Y]
\ar[r]
&
\mathbf{B}X\otimes^{t_{\mathbf{B}X}} N\kappa_! \widetilde{\mathbb{Q}}[\mathbb{P}X_! \pi^\ast_X Y]
\ar[r, leftarrow]
&
N_\bullet X\otimes^{t_{N_\bullet X}} N\kappa_! \widetilde{\mathbb{Q}}[\mathbb{P}X_! \pi^\ast_X Y]
\ar[d, leftarrow]
\\
C_X \otimes^{t_{C_X}} \Omega C_X\otimes_{t_{C_X}}\overline{C}_{Y/X}
\ar[d]
&&
N_\bullet X \otimes^{t_{N_\bullet X}} \mathbf{N}[Y/X]
\ar[d]
\\
C_X\otimes^{t_{C_X}}
\mathbf{B}[Y/X]
\ar[r]
&
\mathbf{B} X\otimes^{t_{\mathbf{B} X}}
\mathbf{B}[Y/X]\ar[r,leftarrow]
&
N_\bullet X\otimes^{t_{N_\bullet X}}
\mathbf{B}[Y/X]\end{tikzcd}
\]
in which the vertical arrows are induced by the morphisms \eqref{eqn:ComodCoherenceIII} and each arrow is colinear in the appropriate sense.
Each of the coalgebras $C_X$, $\mathbf{B} X$ and $N_\bullet X$ is 2-reduced so applying the comparison theorem to the spectral sequences of Lemma \ref{lem:TwistingSS} shows that each of the above arrows is a quasi-isomorphism.
Since the unit morphism $\overline{C}_{Y/X} \to C_X\otimes^{t_{C_X}}\Omega C_X\otimes_{t_{C_X}} \overline{C}_{Y/X}$ is a quasi-isomorphism of $C_X$-comodules, the result is proven.
\end{proof}

The previous proof depends crucially on the condition that $Y$ is 2-reduced.
More generally, given an arbitrary retractive space $Z\in \mathrm{sSet}_{\dslash X}$ we reduce to the 2-reduced case as follows.
Assume without loss of generality that $X$ and $Z\in \mathrm{sSet}_{\dslash X}$ are fibrant.
Applying fibrewise suspension to the homotopy fibre sequence $F\to Z\to X$ we get the homotopy fibre sequence
$
\Sigma^2 F\to \Sigma^2_X Z\to X
$, from which it follows that $\Sigma^2_X Z$ is simply-connected.
Take a fibrant replacement $\begin{tikzcd}[cramped, sep=small](\Sigma^2_X Z)^f \ar[r, two heads] &X \end{tikzcd}$ of $\Sigma^2_X Z\to X$ and choose any vertex $z\in (\Sigma^2_X Z)^f$.
By fibrancy, the map from the second Eilenberg subcomplex $E_2((\Sigma^2_X Z)^f, z)\to (\Sigma^2_X Z)^f$ is a weak equivalence (cf~Remark \ref{rem:RedsSet} and Section \ref{sec:Equivalences}).
Applying Lemma \ref{lem:RatHomComodules2Red} to $E_2((\Sigma^2_X Z)^f, z)$ we now find that the image of $\Sigma^2_X Z$ under the composite derived functor 
\begin{equation}
\label{eqn:ComoduleCompositeFun}
\begin{tikzcd}
Ho(\mathrm{sSet}_{\dslash X})
\ar[r, "\Sigma^\infty_X"]
&
Ho(\mathrm{Sp}^\Sigma_X)_\mathbb{Q}
\ar[r, "\sim"]
&
Ho(C_X\mathrm{-Comod}_{(t)})
\ar[r, "H_\bullet"]
&
H_\bullet(X)\mathrm{-Comod}
\end{tikzcd}
\end{equation}
is naturally isomorphic to $\widetilde{H}_\bullet (\Sigma^2_X Z/X)$.
But there are isomorphisms $X_!(\Sigma^2_X Z)\cong \Sigma^2 (Z/X)$ and $\Sigma^\infty_X \Sigma^2_X Z\cong \Sigma^2_X \Sigma^\infty_X Z$, so we find that the composite derived functor \eqref{eqn:ComoduleCompositeFun} sends $\Sigma^2_X Z$ to the $H_\bullet(X)$-comodule $\widetilde{H}_\bullet(\Sigma^2 (Z/X)) \cong \widetilde{H}_{\bullet+2} (Z/X)$.
By the stability of $Ho(\mathrm{Sp}^\Sigma_X)_\mathbb{Q}$, there are equivalences
$
\Sigma^\infty_X Z \cong \Omega^2_X\Sigma^2_X \Sigma^\infty_X Z
\cong 
\Omega^2_X \Sigma^\infty_X \Sigma^2_X Z
$.
The equivalence of categories $Ho(\mathrm{Sp}^\Sigma)_\mathbb{Q}\to Ho(C_X\mathrm{-Comod}_{(t)})$ commutes with shifts, from which it follows that the retractive space $Z$  is sent under the composite derived functor \eqref{eqn:ComoduleCompositeFun} to the $H_\bullet(X)$-comodule
\[
\widetilde{H}_{\bullet-2}\big(\Sigma^2(Z/X)\big) \cong \widetilde{H}_\bullet\big(Z/X\big)\,.
\]
Thus the diagram of functors 
\[
\begin{tikzcd}[sep=tiny]
&
Ho(\mathrm{Sp}_X)_\mathbb{Q}
\ar[rr, "\sim"]
&&
Ho(C_X\mathrm{-Comod}_{(t)})
\ar[dr, bend left=10, "H_\bullet"]
&
\\
Ho(\mathrm{sSet}_{\dslash X})
\ar[ur, bend left=10, "\Sigma^\infty_X"]
\ar[drr, bend left=-10, "\Sigma^\infty_X"']
&&&&
H_\bullet(X)\mathrm{-Comod}
\\
&&
Ho(\mathrm{Sp}_X)_\mathbb{Q}
\ar[urr, bend left=-10, "\mathbold{H}_\bullet^X"']
&&
\end{tikzcd}
\]
commutes up to natural isomorphism.
$Ho(\mathrm{Sp}_X)$ is generated by fibrewise suspension spectra under shifts and sequential homotopy colimits and all of the functors in the above diagram commute with these operations.
This completes the proof of the following
\begin{theorem}
\label{thm:Comod}
The diagram of functors
\[
\begin{tikzcd}
Ho(\mathrm{Sp}_X)_\mathbb{Q}
\ar[r, "\sim"]
\ar[dr, bend right =15, "\mathbold{H}_\bullet^X"']
&
Ho(C_X\mathrm{-Comod}_{(t)})
\ar[d, "H_\bullet"]
\\
&
H_\bullet(X)\mathrm{-Comod}
\end{tikzcd}
\]
commutes up to natural isomorphism.
\end{theorem}

\subsection{Rational homotopy classes of fibrewise stable maps}
Using the equivalences of Theorem \ref{thm:RatParStaHom}, in this section we show that rational homotopy classes of fibrewise stable maps over a simply-connected base can be computed either as
\begin{itemize}
  \item $\mathrm{Ext}$-groups in the homotopy category of modules over the Whitehead Lie algebra (Theorem \ref{thm:RatStabMapExt}), or
  \item Closed colinear maps between homology coalgebra comodules (Theorem \ref{thm:RatStabMapCoext}).
\end{itemize}
For both of these algebraic models we construct a spectral sequence that converges under some additional hypotheses.

Let $X$ be a simplicial set and $P$ and $Q$ respectively cofibrant and fibrant symmetric $X$-spectra, so that
\[
\pi^\mathrm{st}_\ast \underline{\mathrm{Sp}}^\Sigma_X (P,Q) = \{P,Q\}_X^\ast
\]
computes homotopy classes of  fibrewise stable maps $P\to Q$. 
Per the discussion of Section \ref{S:ratsmash}, rational homotopy classes of fibrewise stable maps are computed in terms of fibrewise $H\mathbb{Q}$-module spectra as
\[
\{P,Q\}_X^\ast\otimes_\mathbb{Z}\mathbb{Q}
=
\pi^\mathrm{st}_\ast\big( \underline{H\mathbb{Q}\mathrm{-Mod}}_X (H\mathbb{Q}\wedge P,H\mathbb{Q}\wedge Q)\big)
\,.
\]
\begin{remark}
Suppose that $X$ is connected so that $Q\in \mathrm{Sp}^\Sigma_X$ is equivalent to a fibrant $\Omega X_+$-module spectrum $\mathcal{E}$ (the connectedness constraint is imposed on $X$ in order to reduce the notational burden and is not essential for what follows).
In the case that $P=\Sigma^\infty_{X,+} Y$ is the pointed fibrewise suspension spectrum of a map $\tau\colon Y\to X$, homotopy classes of fibrewise stable maps $P\to Q$ compute the $\tau$-twisted $\mathcal{E}$-cohomology groups of $Y$, that is
$
\{\Sigma^\infty_{X,+} Y, Q\}_X^{k}
\cong \mathcal{E}^{\tau-k} (Y)
$\footnote{this is the very definition of twisted cohomology---see also \cite[Remark 2.68]{braunack-mayer_combinatorial_2019}.}.
\end{remark}

Let $X$ now be 2-reduced with Quillen Lie and coalgebraic models $\Lambda_X$ and $C_X$ respectively.
By Theorem \ref{thm:RatParStaHom} the equivalences of homotopy categories
\[
Ho(\mathrm{Sp}_X)_\mathbb{Q}\cong Ho(H\mathbb{Q}\mathrm{-Mod}_X) \cong Ho(\Lambda_X\mathrm{-Rep}_\mathrm{dg})\cong 
Ho(C_X\mathrm{-Comod}_{(t)})
\]
are strongly symmetric monoidal and pseudonatural in $X$.
\begin{theorem}
\label{thm:RatStabMapExt}
Let the $\Lambda_X$-representations $V$ and $W$ correspond to the $X$-parametrised spectra $P$ and $Q$ respectively under the equivalence of categories $Ho(\mathrm{Sp}_X)_\mathbb{Q}\cong Ho(\Lambda_X\mathrm{-Rep})$.
Then there is an isomorphism of graded rational vector spaces
\[
\mathrm{Ext}^\bullet_{\mathcal{U}\Lambda_X}(V,W)  \cong \{P,Q\}^\bullet_X\otimes_\mathbb{Z}\mathbb{Q}\,.
\]
\end{theorem}
\begin{proof}
Theorem \ref{thm:RatParStaHom} implies that the tensoring
\begin{align*}
Ho(\mathrm{Sp})\times Ho(\mathrm{Sp}_X)
&\longrightarrow
Ho(\mathrm{Sp}_X)
\\
(\mathcal{E}, P)&\longmapsto 
X^\ast \mathcal{E}\wedge_{X} P
\end{align*}
is modelled in terms of Whitehead Lie representations by 
\begin{align*}
\mathrm{Ch} \times \Lambda_X\mathrm{-Rep}_\mathrm{dg}
&\longrightarrow
\Lambda_X\mathrm{-Rep}_\mathrm{dg}
\\
(M, V)&\longmapsto V\otimes M\,,
\end{align*}
where $V\otimes M$ is regarded as a $\Lambda_X$-representation in the obvious way;
note that $V\otimes M$ generally fails to be cofibrant in $\Lambda_X\mathrm{-Rep}_\mathrm{dg}$.
To remedy this issue we make use of the bar resolution $B (\mathcal{U}\Lambda_X, V\otimes M)$, constructed as the realisation of the augmented simplicial object in $\Lambda_X\mathrm{-Rep}_\mathrm{dg}$
\[
\begin{tikzcd}
\dotsb 
\ar[r]
\ar[r,shift left=1.5ex]
\ar[r,shift left=-1.5ex]
&
\mathcal{U}\Lambda_X
\otimes 
\mathcal{U}\Lambda_X
\otimes
V\otimes M
\ar[r,shift left=0.75ex]
\ar[r,shift left=-0.75ex]
&
\mathcal{U}\Lambda_X
\otimes V\otimes M
\ar[r]
&
V\otimes M\,.
\end{tikzcd}
\]
The face maps are induced by the product of $\mathcal{U}\Lambda_X$ and the degeneracy maps are given by insertions of the unit $\mathbb{Q}\to \mathcal{U}\Lambda_X$.
The bar resolution $B(\mathcal{U}\Lambda_X, V\otimes M)\cong B(\mathcal{U}\Lambda_X, V)\otimes M$ is a cofibrant $\Lambda_X$-representation.
By taking adjoints 
\[
\Lambda_X\mathrm{-Rep}_\mathrm{dg}\Big( B(\mathcal{U}\Lambda_X, V\otimes M), W\Big) 
\cong 
\mathrm{Ch}\Big(
M, [B(\mathcal{U}\Lambda_X, V), W]_{\mathcal{U}\Lambda_X} \Big)\,.
\]
We have that
\[
(P, Q)\longmapsto \underline{\mathrm{Sp}}^\Sigma_X(P,Q)
\]
is modelled in terms of Whitehead Lie representations by the chain complex $[B(\mathcal{U}\Lambda_X, V), W]_{\mathcal{U}\Lambda_X}$ of $\mathcal{U}\Lambda_X$-linear maps $B(\mathcal{U}\Lambda_X, V)\to W$.
Hence 
\[
\mathrm{Ext}^\bullet_{\mathcal{U}\Lambda_X}(V,W) =
H_\bullet \Big([B(\mathcal{U}\Lambda_X, V), W]_{\mathcal{U}\Lambda_X} \Big)
\cong
\{P, Q\}^\bullet_X\otimes_\mathbb{Z}\mathbb{Q}
\]
as claimed.
\end{proof}

\begin{remark}
When $X$ is of finite type, Theorem \ref{thm:RatStabMapExt} specialises to the characterisation of rational homotopy classes of fibrewise stable maps given in \cite{felix_fibrewise_2010}.
The result of \emph{loc.~cit.}~requires all spaces involved to have finite rational type, however it can deal with nilpotent $\pi_1$-actions not accessible with the present methods.
\end{remark}

Under favourable circumstances we can compute rational homotopy classes of fibrewise stable maps by means of a hyper-$\mathrm{Ext}$ spectral sequence.
\begin{theorem}[Hyper-$\mathrm{Ext}$ spectral sequence]
Let the $\Lambda_X$-representations $V$ and $W$ be bounded below and bounded above respectively.
Then there is a convergent spectral sequence
\[
\mathrm{Ext}^{p,q}_{H_\bullet(\mathcal{U}\Lambda_X)} (H_\bullet (V), H_\bullet(W))
\Longrightarrow
\mathrm{Ext}^{q-p}_{\mathcal{U}\Lambda_X}(V,W)\,.
\]
\end{theorem}
\begin{proof}
Filtering the simplicial object
\[
\begin{tikzcd}
\dotsb 
\ar[r]
\ar[r,shift left=1.5ex]
\ar[r,shift left=-1.5ex]
&
\mathcal{U}\Lambda_X
\otimes 
\mathcal{U}\Lambda_X
\otimes
V
\ar[r,shift left=0.75ex]
\ar[r,shift left=-0.75ex]
&
\mathcal{U}\Lambda_X
\otimes V
\end{tikzcd}
\]
by skeleta and passing to the normalised chain complexes gives rise to a colimiting sequence of cofibrations of $\Lambda_X$-representations
$
\begin{tikzcd}[cramped, sep =small]
B_0 (\mathcal{U}\Lambda_X, V)
\ar[r, rightarrowtail]
&
B_1 (\mathcal{U}\Lambda_X, V)
\ar[r, rightarrowtail]
&
\dotsb
\ar[r, rightarrowtail]
&
B (\mathcal{U}\Lambda_X, V)
\end{tikzcd}
$
that exhibits $B (\mathcal{U}\Lambda_X, V)\to V$ as a cofibrant resolution.
For $p\geq 0$ there is a (homotopy) cofibre sequence of $\Lambda_X$-representations
\[
\begin{tikzcd}
B_p (\mathcal{U}\Lambda_X, V)
\ar[r, rightarrowtail]
&
B_{p+1}(\mathcal{U}\Lambda_X, V)
\ar[r]
&
\mathcal{U}\Lambda_X\otimes \overline{\mathcal{U}\Lambda}_X^{\otimes(p+1)}\otimes V[-p-1]\,,
\end{tikzcd}
\]
where $\overline{\mathcal{U}\Lambda}_X = \mathrm{coker}(\mathbb{Q}\to \mathcal{U}\Lambda_X)$.
Taking $\mathcal{U}\Lambda_X$-linear hom complexes, we get a limiting tower of fibrations
\[
\begin{tikzcd}
\dotsb
\ar[r, two heads]
&
{[B_2 (\mathcal{U}\Lambda_X, V), W]_{\mathcal{U}\Lambda_X}}
\ar[r, two heads]
&
{[B_1 (\mathcal{U}\Lambda_X, V), W]_{\mathcal{U}\Lambda_X}
\ar[r, two heads]}
&
{[B_0 (\mathcal{U}\Lambda_X, V ), W]_{\mathcal{U}\Lambda_X}}
\\
&
{[\overline{\mathcal{U}\Lambda}_X^{\otimes 2}\otimes V[-2], W]}
\ar[u]
&
{[\overline{\mathcal{U}\Lambda}_X[-1], W]}
\ar[u]
\end{tikzcd}
\]
with homotopy fibres as indicated (using the free-forgetful adjunction).
For rational chain complexes $M, N$ there is an isomorphism $H_q[M,N] \cong [H_\bullet M, H_\bullet N]_q$ so that passing to homology yields an exact couple
\begin{equation}
\label{eqn:HyperExt}
\begin{tikzcd}[sep=small]
H_q
{[B_p (\mathcal{U}V, W), W]_{\mathcal{U}\Lambda_X}}
\ar[r]
&
H_q 
{[B_{p-1} (\mathcal{U}V, W), W]_{\mathcal{U}\Lambda_X}}
\ar[dl, bend left =15, dashed]
\\
{[H_\bullet \overline{\mathcal{U}\Lambda}_X^{\otimes p}\otimes H_\bullet V, H_\bullet W]}_{q+p}
\ar[u]
\end{tikzcd}
\end{equation}
whose spectral sequence has $E_1^{p,q} \cong {[H_\bullet \overline{\mathcal{U}\Lambda}_X^{\otimes p}\otimes H_\bullet V, H_\bullet W]}_{q}$ with differential $d_1$ given as the composite
\[
{[H_\bullet \overline{\mathcal{U}\Lambda}_X^{\otimes p}\otimes H_\bullet V, H_\bullet W]}_{q}
\longrightarrow
H_{q-p} [B_p(\mathcal{U}\Lambda_X, V), W]_{\mathcal{U}\Lambda_X}
\longrightarrow
{[H_\bullet \overline{\mathcal{U}\Lambda}_X^{\otimes {p+1}}\otimes H_\bullet V, H_\bullet W]}_{q}\,.
\]
In order to identify the $E_2$-term, observe that the realisation of the augmented simplicial object
\[
\begin{tikzcd}
\dotsb 
\ar[r]
\ar[r,shift left=1.5ex]
\ar[r,shift left=-1.5ex]
&
H_\bullet \mathcal{U}\Lambda_X
\otimes 
H_\bullet \mathcal{U}\Lambda_X
\otimes
H_\bullet V
\ar[r,shift left=0.75ex]
\ar[r,shift left=-0.75ex]
&
H_\bullet \mathcal{U}\Lambda_X
\otimes H_\bullet V
\ar[r]
&
H_\bullet V
\end{tikzcd}
\]
is a cofibrant resolution of $H_\bullet V$ in $H_\bullet \mathcal{U}\Lambda_X\mathrm{-Mod}$ whose normalised chain complex
\[
\begin{tikzcd}
\dotsb
\ar[r]
&
H_\bullet \mathcal{U}\Lambda_X
\otimes 
H_\bullet\overline{\mathcal{U}\Lambda}_X
\otimes H_\bullet V
\ar[r]
&
H_\bullet \mathcal{U}\Lambda_X
\otimes H_\bullet V
\ar[r]
&
H_\bullet V
\end{tikzcd}
\]
exhibits a free resolution of $H_\bullet V$.
Taking $H_\bullet\mathcal{U}\Lambda_X$-linear maps to $H_\bullet W$ yields the cochain complex in graded rational vector spaces
\[
\dotsb
\longrightarrow
{[H_\bullet \overline{\mathcal{U}\Lambda}_X^{\otimes p}\otimes H_\bullet V, H_\bullet W]}_\bullet
\longrightarrow
{[H_\bullet \overline{\mathcal{U}\Lambda}_X^{\otimes (p+1)}\otimes H_\bullet V, H_\bullet W]}_\bullet
\longrightarrow
\dotsb
\]
and a careful analysis (cf~\cite[Lemma VIII.1.17]{goerss_simplicial_2009}) shows that the differential on this chain complex coincides with $d_1$.
It follows that $E^{p,q}_2 \cong \mathrm{Ext}^{p,q}_{H_\bullet (\mathcal{U}\Lambda_X)} (H_\bullet V, H_\bullet W)$.

Suppose that $V$ is $k$-connected and $W$ is $l$-coconnected; that is $H_{\bullet\leq k } V = 0$ and $H_{\bullet\geq l} W =0$.
Since $H_\bullet\overline{\mathcal{U}\Lambda}_X$ is connected we have $E^{p,q}_1 =0$ for $p+q >l-k$.
The complete convergence lemma \cite[Lemma VI.2.20]{goerss_simplicial_2009} then implies
\[
\mathrm{Ext}^{p,q}_{H_\bullet(\mathcal{U}\Lambda_X)} (H_\bullet (V), H_\bullet(W))
\Longrightarrow
H_{p-q}\big(\underset{\longleftarrow}{\mathrm{lim}}\, [B_n(\mathcal{U}\Lambda , V), W]_{\mathcal{U}\Lambda_X}\big)\,.
\]
But $\underset{\longleftarrow}{\mathrm{lim}}\, [B_n(\mathcal{U}\Lambda , V), W]_{\mathcal{U}\Lambda_X} \cong [B(\mathcal{U}\Lambda_X, V),W)]_{\mathcal{U}\Lambda_X}$ so that 
\[
\mathrm{Ext}^{p,q}_{H_\bullet(\mathcal{U}\Lambda_X)} (H_\bullet (V), H_\bullet(W))
\Longrightarrow \mathrm{Ext}^{q-p}_{\mathcal{U}\Lambda_X}(V,W)
\]
as claimed.
\end{proof}

To describe rational homotopy classes of fibrewise stable maps in terms of homology coalgebra comodules consider the tensoring bifunctor
\begin{align*}
\mathrm{Ch}\times C_X\mathrm{-Comod}
&\longrightarrow
C_X\mathrm{-Comod}
\\
(W,N) &\longmapsto N\otimes W\,,
\end{align*}
where $N\otimes W$ inherits a left $C_X$-coaction from $N$ in the obvious way.
\begin{proposition}
Let $W\in \mathrm{Ch}$ and $N\in C_X\mathrm{-Comod}_{(t)}$ correspond to $\mathcal{E}\in \mathrm{Sp}$ and $P\in \mathrm{Sp}_X$ in rational homotopy theory.
Then $N\otimes W$ corresponds to $X^\ast \mathcal{E}\wedge_X P$.
\end{proposition}
\begin{proof}
Theorem \ref{thm:RatParStaHom} implies that $X^\ast \mathcal{E}\wedge_X P$ is modelled in terms of $C_X$-comodules by the derived cotensor product $(t_\ast t^! (C_X\otimes W))\,\square_{C_X}\,(t_\ast t^! N)$.
There is a natural isomorphism of $\Omega C_X$-modules 
\[
t^! (C_X\otimes W) \otimes t^! N \cong (\Omega C_X\otimes_t C_X)\otimes W\otimes (\Omega C_X \otimes_t N)
\]
and the $(t^!\dashv t_\ast)$-counit furnishes a natural quasi-isomorphism of $\Omega C_X$-modules
\[
(\Omega C_X \otimes_t C_X)\otimes W \otimes (\Omega C_X\otimes_t N) \longrightarrow \mathbb{Q}\otimes M\otimes (\Omega C_X\otimes_t N)\cong \Omega C_X \otimes _t (N\otimes W)\,.
\]
Applying $t_\ast \colon \Omega C_X\mathrm{-Mod}\to C_X\mathrm{-Comod}_{(t)}$ we get a zig-zag of $t$-equivalences of $C_X$-comodules
\[
\big(t_\ast t^! (C_X\otimes W)\big)\,\square_{C_X}\, \big(t_\ast t^! N\big)
\cong 
t_\ast \big(
t^! (C_X\otimes W) \otimes t^! N
\big)
\longrightarrow 
C_X\otimes^t \Omega C_X \otimes_t (N\otimes W)
\longleftarrow
N \otimes W\,.
\]
The second arrow is the $(t^! \dashv t_\ast)$-unit, which is a natural $t$-equivalence.
\end{proof}
\begin{lemma}
\label{lem:ComodSmash}
The tensoring bifunctor
\begin{align*}
\mathrm{Ch}\times C_X\mathrm{-Comod}_{(t)}
&\longrightarrow
C_X\mathrm{-Comod}_{(t)}
\\
(W,N) &\longmapsto N\otimes W
\end{align*}
satisfies the pushout-product axiom.
\end{lemma}
\begin{proof}
The forgetful functor $C_X\mathrm{-Comod}_{(t)}\to \mathrm{Ch}$ creates colimits and cofibrations. It is immediate that the pushout-product of a pair of cofibrations is again a cofibration.
Using the K\"{u}nneth theorem it is straightforward to check that $(W,N )\mapsto N\otimes W$ sends both quasi-isomorphisms in the first variable and $t$-equivalences in the second variable to $t$-equivalences.
A formal argument (see the diagram \eqref{eqn:PPAcycDiag}, for instance) completes the proof of the pushout-product axiom.
\end{proof}

For $C_X$-comodules $M$, $N$ we write
\[
\{N,M\}_{C_X}:=
\mathrm{lim}
\left(
\begin{tikzcd}[row sep =tiny]
{[N,M]}
\ar[rr, "{[N, \rho_M]}"]
\ar[dr, bend left =-15]
&&
{[N, C_X\otimes M]}
\\
&
{[C_X\otimes N, C_X\otimes M]}
\ar[ur, bend left = -15, "{[\rho_N, C_X\otimes M]}"']
\end{tikzcd}
\right)
\]
for the limit in chain complexes.
By construction there is a natural isomorphism
\[
\mathrm{Ch}\big(W, \{N,M\}_{C_X}\big)
\cong C_X\mathrm{-Comod}\big(N\otimes W, M\big)\,,
\]
and a straighforward adjointness argument applied to Lemma \ref{lem:ComodSmash} yields the following
\begin{corollary}
\label{cor:DerivedComodHoms}
Given a cofibration $f\colon \begin{tikzcd}[sep=small, cramped] K \ar[r, rightarrowtail]& L\end{tikzcd}$ and a fibration $g\colon \begin{tikzcd}[sep=small, cramped] N\ar[r, two heads] & M\end{tikzcd}$ in $C_X\mathrm{-Comod}_{(t)}$, the induced map
\[
\{L, N\}_{C_X}\longrightarrow 
\{L, M\}_{C_X}\times_{\{K,M\}_{C_X}}\{K,N\}_{C_X}
\]
is a fibration in $\mathrm{Ch}$, which is moreover a quasi-isomorphism if either of $f$, $g$ is a $t$-equivalence.
In particular, 
\begin{enumerate}[label=\emph{(\roman*)}]
  \item For any $C_X$-comodule $N$ the functor $M\mapsto \{N,M\}_{C_X}$ sends $t$-equivalences between fibrant $C_X$-comodules to quasi-isomorphisms.
  \item For fibrant $M$ the functor $N\mapsto \{N,M\}_{C_X}$ sends $t$-equivalences to quasi-isomorphisms.
\end{enumerate}
\end{corollary}
\begin{definition}
For $C_X$-comodules $N, M$ we define the \emph{Coext-groups} as
\[
\mathrm{Coext}^\bullet_{C_X} (N,M) := H_\bullet\big(\{N, M^f\}_{C_X}\big)
\]
where $M\to M^f$ is any fibrant replacement; this is well-defined by
Corollary \ref{cor:DerivedComodHoms}.
\end{definition}
Recall that the unit morphism $M\to t_\ast t^! M$ is a canonical fibrant replacement.
For any chain complex $W$ we have natural isomorphisms
\begin{align*}
\mathrm{Ch}\big(W, \{N, t_\ast t^! M\}_{C_X}\big)
&\cong 
C_X\mathrm{-Comod}\big(
N\otimes W, t_\ast t^! M
\big)\\
&\cong
\Omega C_X\mathrm{-Mod}\big(
t^! N\otimes W, t^! M
\big)\\
&
\cong
\mathrm{Ch}\big(
W, [t^! N, t^!M]_{\Omega C_X}\big)\,,
\end{align*}
where we have used $t^! (N\otimes W) \cong t^! N\otimes W$.
Since $\Omega C_X = \mathcal{U}\Lambda_X$ and every object in the image of $t^!$ is cofibrant we have
\[
\mathrm{Coext}^\bullet_{C_X}(N,M) \cong 
\mathrm{Ext}^\bullet_{\mathcal{U}\Lambda_X}(t^! N, t^! M)
\]
for all $C_X$-comodules $N$, $M$.
By combining 
Theorems \ref{thm:RatParStaHom} and \ref{thm:RatStabMapExt} we obtain
\begin{theorem}
\label{thm:RatStabMapCoext}
Let the $C_X$-comodules $N$ and $M$ correspond to the $X$-parametrised spectra $P$ and $Q$ respectively under the equivalence of categories $Ho(\mathrm{Sp}_X)_\mathbb{Q}\cong Ho(C_X\mathrm{-Comod}_{(t)})$.
Then there is an isomorphism of graded rational vector spaces 
\[
\mathrm{Coext}^{\bullet}_{C_X} (N,M) \cong \{P,Q\}^\bullet_X\otimes_\mathbb{Z}\mathbb{Q}\,.
\]
\end{theorem}

Under some additional hypotheses on the $C_X$-comodules $N, M$ there is a spectral sequence computing the Coext-groups $\mathrm{Coext}_{C_X}^\bullet(N,M)$.
Setting $\mathcal{F}_p := \bigoplus_{k\leq p} N_k$ we obtain a sequence of cofibrations of $C_X$-comodules
$
\begin{tikzcd}[cramped, sep=small]
\dotsb \ar[r,rightarrowtail]
&
\mathcal{F}_{p-1} N
\ar[r, rightarrowtail]
&
\mathcal{F}_p N
\ar[r, rightarrowtail]
&
\mathcal{F}_{p+1} N
\ar[r, rightarrowtail]
&
\dotsb
\end{tikzcd}
$
such that $N = \mathrm{colim}_{p}\mathcal{F}_p N$ (this requires connectivity of $C_X$). 
Let $M \to t_\ast t^! M = C\otimes^t \Omega C\otimes_t M$ be the canonical fibrant replacement of $M$  as a $C_X$-comodule, then
by Corollary \ref{cor:DerivedComodHoms} we get a sequence of fibrations of chain complexes
\[
\begin{tikzcd}
\dotsb
\ar[r, two heads]
& \{\mathcal{F}_{p+1} N, t_\ast t^! M\}_{C_X}
\ar[r, two heads]
&
\{\mathcal{F}_{p} N, t_\ast t^! M\}_{C_X}
\ar[r, two heads]
&
\{\mathcal{F}_{p-1} N, t_\ast t^! M\}_{C_X}
\ar[r, two heads]
&
\dotsb
\end{tikzcd}
\]
For each $p\in \mathbb{Z}$ there is a fibre sequence
\[
\{N_p[-p] , t_\ast t^! M\}_{C_X}
\longrightarrow
\{\mathcal{F}_{p} N, t_\ast t^! M\}_{C_X}
\longrightarrow\{\mathcal{F}_{p-1} N, t_\ast t^! M\}_{C_X}\,,
\]
and the $C_X$-coaction on $N_p[-p]$ is trivial for degree reasons. Applying Lemma \ref{lem:KoszulNat} to the coalgebra morphism $\mathbb{Q}\to C_X$ yields  $\{N_p[-p], t_\ast t^! M\}_{C_X} \cong [N_p[-p], (t_\ast t^! M)_{C_X}]$, with
$
(t_\ast t^! M)_{C_X} \cong \Omega C_X\otimes_t M
$ the $C_X$-coinvariant submodule.
Taking homology we get an unrolled exact couple
\[
\begin{tikzcd}
\dotsb
\ar[r]
&
H_\bullet 
\{\mathcal{F}_{p+1}, M\}_{C_X}
\ar[r]
&
H_\bullet 
\{\mathcal{F}_{p}, M\}_{C_X}
\ar[r]
\ar[dl, bend left =15, dashed]
&
H_\bullet 
\{\mathcal{F}_{p-1}, M\}_{C_X}
\ar[r]
\ar[dl, bend left =15, dashed]
&
\dotsb
\\
&
H_\bullet {[N_{p+1}, M_{C_X}]}
\ar[u]
&
H_\bullet {[N_{p}, M_{C_X}]}
\ar[u]&&
\end{tikzcd}
\]
The associated spectral sequence has
$E_1^{p,q}\cong [N_p, H_q (\Omega C_X\otimes_t M)]$ with differential $d_1$ given induced by the boundary map $\partial\colon N_{p+1} \to N_{p}$ so that
$
E_2^{p,q} \cong [H_p (N), H_q (\Omega C_X\otimes_t M)]
$.
Note that although this description of the $E_2$-page does not depend on the $C_X$-comodule structure of $N$, the differentials $d_k$ for $k\geq 2$ certainly do.
\begin{theorem}[$\mathrm{Coext}$ spectral sequence]
Let $N$ be a $C_X$-comodule with bounded homology.
Then for any $C_X$-comodule $M$ there is a convergent spectral sequence
\[
[H_p (N), H_q (\Omega C_X\otimes_t M)]
\Longrightarrow
\mathrm{Coext}^{q-p}_{C_X}(N,M)\,.
\]
\end{theorem}
\begin{proof}
Suppose that $H_p N =0$ for  $p>k$ and $p<l$ so that $E_2^{p,q} = 0$ for $p>k$ and $p<l$.
By the complete convergence lemma we then have
\[
E_2^{p,q}\Longrightarrow
H_{q-p} \Big(\underset{\longleftarrow}{\mathrm{lim}}\, \{\mathcal{F}_n N, t_\ast t^! M\}_{C_X}\Big) = 
H_{q-p} \{N, t_\ast t^! M\}_{C_X} \cong \mathrm{Coext}^{q-p}_{C_X}(N,M)
\]
as claimed.
\end{proof}

\begin{example}[Twisted Atiyah--Hirzebruch spectral sequence]
Let $X$ be simply-connected and $Y$ a finite dimensional cell complex with a map $\tau\colon Y\to X$.
For an $X$-parametrised spectrum $A$, write $\mathcal{E}$ for the stable homotopy fibre of $A$ (at a generic point of $X$).
Suppose that the $C_X$-comodules $N$ and $M$ respectively correspond to the parametrised $X$-spectra $\Sigma^\infty_X Y_+$ and $A$ under Theorem \ref{thm:RatParStaHom}.
The (derived) $C_X$-coinvariant submodule $\Omega C_X \otimes_t M$ models the rational homotopy type of the stable homotopy fibre $\mathcal{E}$; in particular $H_\bullet (\Omega C_X \otimes_t M) = \pi^\mathrm{st}_\ast \mathcal{E}\otimes_\mathbb{Z}\mathbb{Q}$.
Using the universal coefficient theorem the $E_2$-page of the $\mathrm{Coext}$ spectral sequence can be recast in the form
\[
H^p( Y; \pi^\mathrm{st}_q\mathcal{E}\otimes_\mathbb{Z}\mathbb{Q})
\Longrightarrow
\{\Sigma^\infty_X Y_+, A\}^{q-p}_X\otimes_\mathbb{Z}\mathbb{Q}\,.
\]
But $\mathcal{E}^{\tau+p-q}(Y)= \{\Sigma^\infty_X Y_+, A\}^{q-p}_X$ is the $\tau$-twisted $\mathcal{E}$-cohomology of $Y$, so we have obtained the rationalised twisted Atiyah--Hirzebruch spectral sequence.
\end{example}

%%%%%%%%%%%%%%%%%%%%
%%%%%%%%%%%%%%%%%%%%
%%%%%%%%%%%%%%%%%%%%

%%%%%%%%%%%%%%%%%

\end{document}